\documentclass[11pt]{amsart}

\usepackage{
    amsmath,
    amsfonts,
    amssymb,
    amsthm,
    amscd,
    comment,
    enumitem,
    etoolbox,
    gensymb,    
    mathtools,
    mathdots,
    stmaryrd,
    booktabs
}
\usepackage[usenames,dvipsnames]{xcolor}
\usepackage[all]{xy}

\usepackage[T1]{fontenc}
\usepackage{bbm}                     
\usepackage[colorlinks=true, linkcolor=blue, citecolor=blue, urlcolor=blue, breaklinks=true]{hyperref}

\DeclareFontFamily{OT1}{pzc}{}
\DeclareFontShape{OT1}{pzc}{m}{it}{<-> s * [1.10] pzcmi7t}{}
\DeclareMathAlphabet{\mathpzc}{OT1}{pzc}{m}{it}

\usepackage[capitalize]{cleveref}   

\crefname{defin}{Definition}{Definitions}
\crefname{eg}{Example}{Examples}
\crefname{egs}{Examples}{Examples}
\crefname{convention}{Convention}{Convention}
\crefname{lem}{Lemma}{Lemmas}
\crefname{prop}{Proposition}{Propositions}
\crefname{theo}{Theorem}{Theorems}
\crefname{equation}{}{}
\crefname{enumi}{}{}

\newcommand\C{\mathbb{C}}
\newcommand\DD{\mathbb{D}}
\newcommand\N{\mathbb{N}}
\newcommand\HH{\mathbb{H}}

\newcommand\R{\mathbb{R}}
\newcommand\Z{\mathbb{Z}}
\newcommand\kk{\Bbbk}
\newcommand\one{\mathbbm{1}}

\newcommand\bB{\mathbf{B}}
\newcommand\bD{\mathbf{D}}
\newcommand\obD{\vec{\bD}}

\newcommand\fg{\mathfrak{g}}
\newcommand\fgl{\mathfrak{gl}}           

\newcommand\fp{\mathfrak{p}}            
\newcommand\fq{\mathfrak{q}}
\newcommand\osp{\mathfrak{osp}}         
\newcommand\fpsl{\mathfrak{psl}}

\newcommand\fsl{\mathfrak{sl}}

\newcommand\fsp{\mathfrak{sp}}

\newcommand\fu{\mathfrak{u}}

\newcommand\Cl{\mathrm{Cl}}             
\newcommand\GL{\mathrm{GL}}
\newcommand\op{\mathrm{op}}
\newcommand\OSp{\mathrm{OSp}}           
\newcommand\rQ{\mathrm{Q}}

\newcommand\rO{\mathrm{O}}              
\newcommand\SL{\mathrm{SL}}
\newcommand\SO{\mathrm{SO}}             
\newcommand\Sp{\mathrm{Sp}}             
\newcommand\SU{\mathrm{SU}}

\newcommand\rU{\mathrm{U}}              
\newcommand\rUQ{\mathrm{UQ}}

\newcommand\md{\textup{-mod}}
\newcommand\rd{\textup{red}}            
\newcommand\smod{\textup{-smod}}
\newcommand\st{\textup{st}}             

\newcommand\tmod{\textup{-tmod}}         
\newcommand\tsmod{\textup{-tsmod}}       
\newcommand\transpose{\textup{t}}

\newcommand\sC{\mathsf{C}}
\newcommand\sD{\mathsf{D}}
\newcommand\sE{\mathsf{E}}              
\newcommand\sF{\mathsf{F}}

\newcommand\sS{\mathsf{S}}

\newcommand\sG{\mathsf{G}}
\newcommand\sM{\mathsf{M}}
\newcommand\sR{\mathsf{R}}

\newcommand\form{\tau}                  
\newcommand\inv{\diamond}                  
\newcommand\Nak{\zeta}                  

\newcommand\even{0}
\newcommand\odd{1}

\newcommand\blue[1]{{\color{blue} \boldsymbol{#1}}}
\newcommand\red[1]{{\color{red} #1}}

\newcommand{\downobj}{{\mathord{\downarrow}}}
\newcommand{\upobj}{{\mathord{\uparrow}}}

\newcommand\AOB{\mathpzc{AOB}}          
\newcommand\Brauer{\mathpzc{B}}         
\newcommand\cC{\mathcal{C}}
\newcommand\cD{\mathcal{D}}
\newcommand\cON{\mathpzc{ON}}           
\newcommand\cN{\mathpzc{N}}

\newcommand\OB{\mathpzc{OB}}            
\newcommand\OBC{\mathpzc{OBC}}
\newcommand\SCat{\mathpzc{SCat}}

\newcommand\go{{\mathsf{I}}}            
\newcommand\gob{{\blue{\mathsf{I}}}}
\newcommand\gor{{\red{\mathsf{I}}}}

\DeclareMathOperator{\Ad}{Ad}
\DeclareMathOperator{\Add}{Add}
\DeclareMathOperator{\Aut}{Aut}
\DeclareMathOperator{\coev}{coev}
\DeclareMathOperator{\End}{End}
\DeclareMathOperator{\ev}{ev}
\DeclareMathOperator{\flip}{flip}
\DeclareMathOperator{\Hom}{Hom}

\DeclareMathOperator{\id}{id}
\DeclareMathOperator{\im}{im}      

\DeclareMathOperator{\Kar}{Kar}
\DeclareMathOperator{\Mat}{Mat}

\DeclareMathOperator{\proj}{proj}

\DeclareMathOperator{\RP}{Re}       
\DeclareMathOperator{\sdim}{sdim}       

\DeclareMathOperator{\str}{str}         

\DeclareMathOperator{\tr}{tr}

\usepackage{tikz}
\usetikzlibrary{arrows.meta}
\usetikzlibrary{decorations.markings}
\usetikzlibrary{calc}
\usetikzlibrary{snakes}
\usepackage{tikz-cd}

\tikzset{anchorbase/.style={>=To,baseline={([yshift=-0.5ex]current bounding box.center)}}}
\tikzset{ 
    centerzero/.style={>=To,baseline={([yshift=-0.5ex](#1))}},
    centerzero/.default={0,0}
}
\tikzset{wipe/.style={white,line width=3pt}}
\tikzset{bcolor/.style={blue,very thick}}
\tikzset{rightb/.style={blue,very thick,-{>[scale=0.7]}}}
\tikzset{leftb/.style={blue,very thick,{<[scale=0.7]}-}}
\tikzset{rcolor/.style={red}}

\newcommand\braidup{to[out=up,in=down]}
\newcommand\braiddown{to[out=down,in=up]}

\newcommand\shiftlabel[1]{$\color{purple} \scriptstyle{#1}$}
\newcommand\shiftline[3]{\draw[purple] (#1) to (#2) node[anchor=west] {\shiftlabel{#3}}}

\newcommand\dotlabel[1]{$\scriptstyle{#1}$}
\newcommand\strandlabel[1]{$\scriptstyle{#1}$}
\newcommand\token[3]{
    \filldraw[black] (#1) circle (1.5pt) node[anchor=#2] {\dotlabel{#3}}
}
\newcommand\tokenb[3]{
    \filldraw[blue] (#1) circle (1.5pt) node[anchor=#2] {\dotlabel{#3}}
}

\newcommand\bub[1]{
    \draw (#1)++(0,0.2) arc(90:-270:0.2)
}
\newcommand\cbub[1]{
    \draw[->] (#1)++(0,0.2) arc(90:-270:0.2)
}
\newcommand\ccbub[1]{
    \draw[->] (#1)++(0,0.2) arc(90:450:0.2)
}

\newcommand\bubble[1]{
    \begin{tikzpicture}[centerzero]
        \bub{0,0};
        \token{0.2,0}{west}{#1};
    \end{tikzpicture}
}
\newcommand\cbubble[1]{
    \begin{tikzpicture}[centerzero]
        \cbub{0,0};
        \token{0.2,0}{west}{#1};
    \end{tikzpicture}
}
\newcommand\ccbubble[1]{
    \begin{tikzpicture}[centerzero]
        \ccbub{0,0};
        \token{0.2,0}{west}{#1};
    \end{tikzpicture}
}
\newcommand\idstrand{
    \begin{tikzpicture}[centerzero]
        \draw[-] (0,-0.2) -- (0,0.2);
    \end{tikzpicture}
}
\newcommand\idstrandb{
    \begin{tikzpicture}[centerzero]
        \draw[bcolor] (0,-0.2) -- (0,0.2);
    \end{tikzpicture}
}
\newcommand\idstrandr{
    \begin{tikzpicture}[centerzero]
        \draw[rcolor] (0,-0.2) -- (0,0.2);
    \end{tikzpicture}
}
\newcommand\idstrandbr{
    \begin{tikzpicture}[centerzero]
        \draw[bcolor] (0,-0.2) -- (0,0);
        \draw[rcolor] (0,0) -- (0,0.2);
    \end{tikzpicture}
}
\newcommand\idstrandrb{
    \begin{tikzpicture}[centerzero]
        \draw[rcolor] (0,-0.2) -- (0,0);
        \draw[bcolor] (0,0) -- (0,0.2);
    \end{tikzpicture}
}
\newcommand\upstrand{
    \begin{tikzpicture}[centerzero]
        \draw[->] (0,-0.2) -- (0,0.2);
    \end{tikzpicture}
}

\newcommand\downstrand{
    \begin{tikzpicture}[centerzero]
        \draw[<-] (0,-0.2) -- (0,0.2);
    \end{tikzpicture}
}

\newcommand\tokstrand[1][a]{
    \begin{tikzpicture}[centerzero]
        \draw[-] (0,-0.2) -- (0,0.2);
        \token{0,0}{west}{#1};
    \end{tikzpicture}
}
\newcommand\uptokstrand[1][a]{
    \begin{tikzpicture}[centerzero]
        \draw[->] (0,-0.2) -- (0,0.2);
        \token{0,0}{west}{#1};
    \end{tikzpicture}
}

\newcommand\downtokstrand[1][a]{
    \begin{tikzpicture}[centerzero]
        \draw[<-] (0,-0.2) -- (0,0.2);
        \token{0,0}{west}{#1};
    \end{tikzpicture}
}
\newcommand\crossmor{
    \begin{tikzpicture}[centerzero]
        \draw[-] (0.2,-0.2) -- (-0.2,0.2);
        \draw[-] (-0.2,-0.2) -- (0.2,0.2);
    \end{tikzpicture}
}
\newcommand\crossmorbb{
    \begin{tikzpicture}[centerzero]
        \draw[bcolor] (0.2,-0.2) -- (-0.2,0.2);
        \draw[bcolor] (-0.2,-0.2) -- (0.2,0.2);
    \end{tikzpicture}
}
\newcommand\crossmorrr{
    \begin{tikzpicture}[centerzero]
        \draw[rcolor] (0.2,-0.2) -- (-0.2,0.2);
        \draw[rcolor] (-0.2,-0.2) -- (0.2,0.2);
    \end{tikzpicture}
}
\newcommand\crossmorbr{
    \begin{tikzpicture}[centerzero]
        \draw[rcolor] (0.2,-0.2) -- (-0.2,0.2);
        \draw[bcolor] (-0.2,-0.2) -- (0.2,0.2);
    \end{tikzpicture}
}
\newcommand\crossmorrb{
    \begin{tikzpicture}[centerzero]
        \draw[bcolor] (0.2,-0.2) -- (-0.2,0.2);
        \draw[rcolor] (-0.2,-0.2) -- (0.2,0.2);
    \end{tikzpicture}
}
\newcommand\upcross{
    \begin{tikzpicture}[centerzero]
        \draw[->] (0.2,-0.2) -- (-0.2,0.2);
        \draw[->] (-0.2,-0.2) -- (0.2,0.2);
    \end{tikzpicture}
}

\newcommand\rightcross{
    \begin{tikzpicture}[centerzero]
        \draw[<-] (0.2,-0.2) -- (-0.2,0.2);
        \draw[->] (-0.2,-0.2) -- (0.2,0.2);
    \end{tikzpicture}
}
\newcommand\downcross{
    \begin{tikzpicture}[centerzero]
        \draw[<-] (0.2,-0.2) -- (-0.2,0.2);
        \draw[<-] (-0.2,-0.2) -- (0.2,0.2);
    \end{tikzpicture}
}
\newcommand\leftcross{
    \begin{tikzpicture}[centerzero]
        \draw[->] (0.2,-0.2) -- (-0.2,0.2);
        \draw[<-] (-0.2,-0.2) -- (0.2,0.2);
    \end{tikzpicture}
}
\newcommand{\cupmor}{
    \begin{tikzpicture}[anchorbase]
        \draw[-] (-0.15,0.15) -- (-0.15,0) arc(180:360:0.15) -- (0.15,0.15);
    \end{tikzpicture}
}
\newcommand{\cupmorbb}{
    \begin{tikzpicture}[anchorbase]
        \draw[bcolor] (-0.15,0.15) -- (-0.15,0) arc(180:360:0.15) -- (0.15,0.15);
    \end{tikzpicture}
}
\newcommand{\cupmorbr}{
    \begin{tikzpicture}[anchorbase]
        \draw[bcolor] (-0.15,0.15) -- (-0.15,0) arc(180:270:0.15);
        \draw[rcolor] (0,-0.15) arc(270:360:0.15) -- (0.15,0.15);
    \end{tikzpicture}
}

\newcommand{\cupmorrb}{
    \begin{tikzpicture}[anchorbase]
        \draw[rcolor] (-0.15,0.15) -- (-0.15,0) arc(180:270:0.15);
        \draw[bcolor] (0,-0.15) arc(270:360:0.15) -- (0.15,0.15);
    \end{tikzpicture}
}
\newcommand{\rightcup}{
    \begin{tikzpicture}[anchorbase]
        \draw[->] (-0.15,0.15) -- (-0.15,0) arc(180:360:0.15) -- (0.15,0.15);
    \end{tikzpicture}
}

\newcommand{\leftcup}{
    \begin{tikzpicture}[anchorbase]
        \draw[<-] (-0.15,0.15) -- (-0.15,0) arc(180:360:0.15) -- (0.15,0.15);
    \end{tikzpicture}
}

\newcommand{\capmor}{
    \begin{tikzpicture}[anchorbase]
        \draw[-] (-0.15,-0.15) -- (-0.15,0) arc(180:0:0.15) -- (0.15,-0.15);
    \end{tikzpicture}
}
\newcommand{\capmorbb}{
    \begin{tikzpicture}[anchorbase]
        \draw[bcolor] (-0.15,-0.15) -- (-0.15,0) arc(180:0:0.15) -- (0.15,-0.15);
    \end{tikzpicture}
}
\newcommand{\capmorbr}{
    \begin{tikzpicture}[anchorbase]
        \draw[bcolor] (-0.15,-0.15) -- (-0.15,0) arc(180:90:0.15);
        \draw[rcolor] (0,0.15) arc(90:0:0.15) -- (0.15,-0.15);
    \end{tikzpicture}
}

\newcommand{\capmorrb}{
    \begin{tikzpicture}[anchorbase]
        \draw[rcolor] (-0.15,-0.15) -- (-0.15,0) arc(180:90:0.15);
        \draw[bcolor] (0,0.15) arc(90:0:0.15) -- (0.15,-0.15);
    \end{tikzpicture}
}
\newcommand{\rightcap}{
    \begin{tikzpicture}[anchorbase]
        \draw[->] (-0.15,-0.15) -- (-0.15,0) arc(180:0:0.15) -- (0.15,-0.15);
    \end{tikzpicture}
}

\newcommand{\leftcap}{
    \begin{tikzpicture}[anchorbase]
        \draw[<-] (-0.15,-0.15) -- (-0.15,0) arc(180:0:0.15) -- (0.15,-0.15);
    \end{tikzpicture}
}

\newtheorem{theo}{Theorem}[section]

\newtheorem{prop}[theo]{Proposition}
\newtheorem{lem}[theo]{Lemma}
\newtheorem{cor}[theo]{Corollary}

\theoremstyle{definition}
\newtheorem{defin}[theo]{Definition}
\newtheorem{rem}[theo]{Remark}
\newtheorem{eg}[theo]{Example}
\newtheorem{egs}[theo]{Examples}
\newtheorem{convention}[theo]{Convention}

\numberwithin{equation}{section}
\allowdisplaybreaks

\setenumerate[1]{label=(\alph*)}          

\newtoggle{comments}
\newtoggle{details}
\newtoggle{detailsnote}

\iftoggle{comments}{%
    \usepackage[notref,notcite]{showkeys}   
    \newcommand{\acomments}[1]{
        \ \\
        {\color{red}
            \textbf{AS:} #1
        }
        \ \\
    }
    \newcommand{\saima}[1]{
        \ \\
        {\color{purple}
            \textbf{For Saima:} #1
        }
        \ \\
    }
}{%
    \newcommand{\acomments}[1]{}
    \newcommand{\saima}[1]{}
}

\iftoggle{details}{%
    \newcommand{\details}[1]{
        \ \\
        {\color{OliveGreen}
            \textbf{Details:} #1
        }
        \\
    }
}{%
    \newcommand{\details}[1]{}
}

\begin{document}

\title{Diagrammatics for real supergroups}

\author{Saima Samchuck-Schnarch}
\address[S.S.]{
  Department of Mathematics and Statistics \\
  University of Ottawa \\
  Ottawa, ON K1N 6N5, Canada
}
\email{ssamc090@uottawa.ca}

\author{Alistair Savage}
\address[A.S.]{
  Department of Mathematics and Statistics \\
  University of Ottawa \\
  Ottawa, ON K1N 6N5, Canada
}
\urladdr{\href{https://alistairsavage.ca}{alistairsavage.ca}, \textrm{\textit{ORCiD}:} \href{https://orcid.org/0000-0002-2859-0239}{orcid.org/0000-0002-2859-0239}}
\email{alistair.savage@uottawa.ca}

\begin{abstract}
    We introduce two families of diagrammatic monoidal supercategories.  The first family, depending on an associative superalgebra, generalizes the oriented Brauer category.  The second, depending on an involutive superalgebra, generalizes the unoriented Brauer category.  These two families of supercategories admit natural superfunctors to supercategories of supermodules over general linear supergroups and supergroups preserving superhermitian forms, respectively.  We show that these superfunctors are full when the superalgebra is a central real division superalgebra.  As a consequence, we obtain first fundamental theorems of invariant theory for all real forms of the general linear, orthosymplectic, periplectic, and isomeric supergroups.  We also deduce equivalences between monoidal supercategories of tensor supermodules over the real forms of a complex supergroup.
\end{abstract}



\ifboolexpr{togl{comments} or togl{details}}{%
  {\color{magenta}DETAILS OR COMMENTS ON}
}{%
}

\maketitle
\thispagestyle{empty}

\tableofcontents

\section{Introduction}

Many recent developments in representation theory involve one or more of the following interrelated concepts:
\begin{enumerate}
    \item \emph{Dual pairs}.  The classic examples are Schur--Weyl duality, which yields a precise relationship between the symmetric group and the general linear group, and the analogue for the orthogonal and symplectic groups, where the symmetric groups are replaced by Brauer algebras.
    \item \emph{Invariant theory}.  This amounts to giving explicit descriptions of invariants in tensor products of certain modules, such as the natural modules for classical Lie groups.
    \item \emph{Interpolating categories}.  Here one aims to give uniform descriptions of representations of families of groups, such as symmetric groups, general linear groups, orthogonal groups, and symplectic groups.  Highly influential in this approach are the interpolating categories introduced by Deligne \cite{Del07}.  Such interpolating categories can often be given nice diagrammatic descriptions, leading to intuitive topological arguments.
\end{enumerate}

In the case of the general linear group, the connection between the above concepts is as follows.  The \emph{oriented Brauer category} $\OB(d)$ is the free rigid symmetric $\C$-linear monoidal category on a generating object of categorical dimension  $d$.  Since the category of modules over the general linear group $\GL(m,\C)$, $m \in \N$, is rigid symmetric monoidal, there exists a functor
\[
    \sG \colon \OB(m) \to \GL(m,\C)\md
\]
sending the generating object $\upobj$ of $\OB(m)$ to the natural $\GL(m,\C)$-module $V$.  The additive Karoubi envelope of $\OB(d)$ is Deligne's interpolating category for the general linear groups.  The endomorphism algebra $\End_{\OB(d)}(\upobj^{\otimes r})$ is isomorphic to the group algebra of the symmetric group $\mathfrak{S}_r$, and so the functor $\sG$ yields an algebra homomorphism
\begin{equation} \label{SW}
    \kk \mathfrak{S}_r \cong \End_{\OB(m)}(\upobj^{\otimes r}) \to \End_{\GL(m,\C)}(V^{\otimes r}).
\end{equation}
One half of Schur--Weyl duality is that the homomorphism \cref{SW} is surjective.  From this, one is able to deduce that the functor $\sG$ is full.  The connection to invariant theory comes from the fact that $\sG$ also induces a surjective homomorphism
\begin{equation} \label{FFT}
    \Hom_{\OB(m)}(\one,\upobj^{\otimes r} \otimes \downobj^{\otimes s})
    \to \Hom_{\GL(m,\C)}(\C, V^{\otimes r} \otimes (V^*)^{\otimes s}),
\end{equation}
where $V^*$ is the $\GL(m,\C)$-module dual to $V$, and $\downobj$ is the object of $\OB(m)$ dual to $\upobj$.  Thus, all $\GL(m,\C)$-invariant elements of $V^{\otimes r} \otimes (V^*)^{\otimes s}$ lie in the image under $\sG$ of morphisms in $\OB(m)$.  The fullness of $\sG$, or of \cref{FFT}, is sometimes referred to as the \emph{first fundamental theorem} of invariant theory.  (Describing the kernel is the \emph{second fundamental theorem}.)

An analogous picture exists for the orthogonal and symplectic groups.  In these cases, the natural module is self-dual.  Thus, the \emph{oriented} Brauer category is replaced by the \emph{unoriented} Brauer category $\Brauer(d)$ of \cite{LZ15}, which is the free rigid symmetric $\kk$-linear monoidal category on a symmetrically self-dual object of categorical dimension $d$.  Then, for $m \in \N$, one has a full functors
\[
    \Brauer(m) \to \rO(m,\C)\md
    \qquad \text{and} \qquad
    \Brauer(-2m) \to \Sp(2m,\C)\md.
\]
Here the endomorphism algebras are Brauer algebras, which surject onto the endomorphism algebras of tensor powers of the natural module.

In fact, it turns out that the most natural setting for the above picture is that of categories of \emph{super}modules over \emph{super}groups.  There are full functors
\[
    \OB(m-n) \to \GL(m|n,\C)\smod
    \qquad \text{and} \qquad
    \Brauer(m-2n) \to \OSp(m|2n,\C)\smod,
\]
where $\GL(m|n,\C)$ and $\OSp(m|2n,\C)$ are the general linear and orthosymplectic supergroups, respectively \cite{CW12,BS12,LSM02,LZ17,LSM02}.  The move to the super world also leads to additional free categories.  First, one observes that an isomorphism of a module with its dual can be even or odd.  The even case corresponds to the Brauer category.  The odd case leads to the \emph{periplectic Brauer supercategory} $\Brauer^1$, which is the free rigid symmetric $\kk$-linear monoidal supercategory on an odd-self-dual object (which necessarily has categorical dimension zero).  Then there is a full superfunctor
\[
    \Brauer^1 \to \mathrm{P}(m)\smod,
\]
where $\mathrm{P}(m)$ is the periplectic supergroup \cite{KT17,CE21,DLZ18,Moo03}.  Another free supercategory arises from the super version of Schur's lemma.  Since we work over the complex numbers, Schur's lemma implies that the endomorphism algebra of a simple module is a complex division superalgebra.  In the non-super setting, the only possibility is $\C$.  However, in the super setting, there is one additional possibility, which is the two-dimensional complex Clifford superalgebra $\Cl(\C)$.  This observation leads to the definition of the \emph{oriented Brauer-Clifford category} $\OBC$ of \cite{BCK19}, which is the free rigid symmetric monoidal supercategory on a generating object whose endomorphism algebra is $\Cl(\C)$.  (As in the periplectic case, the categorical dimension must be zero.)  There is a full superfunctor
\[
    \OBC \to \rQ(m)\smod,
\]
where $\rQ(m)\smod$ is the isomeric supergroup (also known as the queer supergroup).

Despite the great success of the above-mentioned approaches to the representation theory of some of the most important groups and supergroups appearing in mathematics and physics, surprisingly little is known when we work with \emph{real} supergroups instead of complex ones.  The goal of the current paper is to initiate this line of research.  Let us now describe our main results.

To any associative superalgebra $A$ over a field $\kk$, we define a diagrammatic supercategory $\OB_\kk(A)$, which is the free rigid symmetric monoidal supercategory on an object with endomorphism superalgebra $A$.  Imposing a condition on the categorical dimension yields a quotient category $\OB_\kk(A;d)$, for $d \in \kk$.  This category has essentially appeared in \cite{Sav19,BSW-foundations,MS21}, although our definition is slightly more general.  When $A=\kk$, $\OB_\kk(\kk;d)$ is the oriented Brauer category (over a general field $\kk$) mentioned above.  The universal property of $\OB_\kk(A)$ implies that, if $\fg$ is any Lie superalgebra, and $V$ is a $(\fg,A)$-superbimodule, then there is an \emph{oriented incarnation superfunctor}
\[
    \OB_\kk(A^\op) \to \fg\smod,
\]
sending the generating object of $\OB_\kk(A^\op)$ to $V$, where $A^\op$ denotes the superalgebra opposite to $A$.  When we work over the ground field $\kk=\R$, Schur's lemma implies that the endomorphism algebra of a \emph{simple} supermodule must be one of the ten real division superalgebras.  Our first main result (\cref{OBrealfull}) is that, when $A$ is a central real division superalgebra and $V = A^{m|n}$, the functor
\[
    \OB_\R(A^\op;m-n) \to \fgl(m|n,A)\smod
\]
is full.  (Note that, since the general linear groups are connected, we can freely replace the general linear supergroups by the general linear Lie superalgebras.)  The method of proof is to pass to complexifications and use known results over the complex numbers.

We then turn our attention to the unoriented (i.e.\ self-dual) cases.  Here the situation is a bit more involved, since we must carefully analyze which types of self-duality can arise.  The natural setting for such self-dualities is over superalgebras equipped with an anti-involution $a \mapsto a^\inv$.  As mentioned above in the complex setting, the self-duality also has a parity $\sigma \in \Z_2$.  To any $\kk$-superalgebra $A$ with anti-involution $\inv$, and $\sigma \in \Z_2$, we assign a supercategory $\Brauer_\kk^\sigma(A,\inv)$ and quotient supercategories $\Brauer_\kk^\sigma(A,\inv;d)$ for $d \in \kk$.  When $A = \kk$ and the anti-involution is trivial, $\Brauer_\kk^0(\kk,\id;d)$ is the usual Brauer category, while $\Brauer_\kk^1(\kk,\id;0)$ is the periplectic Brauer category.  We deduce a basis theorem (\cref{basisthm}) for the morphism spaces of $\Brauer_\kk^\sigma(A,\inv;d)$ by embedding it into the superadditive envelope of $\OB_\kk(A)$.

Self-duality of a supermodule is realized by a superhermitian or skew-superhermitian form $\Phi$.  To such a form, we can associate the supergroup $G(\Phi)$ preserving the form.  We then define an \emph{unoriented incarnation superfunctor}
\[
    \sF_\Phi \colon \Brauer_\kk^\sigma(A,\inv) \to G(\Phi)\smod.
\]
It turns out that only four of the ten real division superalgebras admit anti-involutions: the real numbers, the complex numbers, the quaternions, and the two-dimensional complex Clifford superalgebra.  Our second main result (\cref{divfull}) is that, in these cases, the functor $\sF_\Phi$ is full.  The proof, which occupies \cref{sec:real,sec:complex,sec:quaternionic}, is much more involved than in the oriented case.  We must treat each of the involutive division superalgebras separately, since each one behaves quite differently.

Taking the oriented and unoriented cases together, our results handle real supergroups corresponding to \emph{all} real forms of the general linear, orthosymplectic, periplectic, and isomeric Lie superalgebras.  (We give a classification of these real forms in \cref{breakthrough}.)  Looking at endomorphism algebras, as explained above, one immediately obtains analogues of Schur--Weyl duality, or first fundamental theorems, for these real supergroups.  Such results seem to be rare in the literature.  (See \cite{Cal22} for some partial results for certain real groups.)  Even in the non-super setting, we obtain new results, corresponding, for example, to the indefinite orthogonal, unitary, and symplectic groups; see \cref{pike}.  In fact, in these cases where the module categories are semisimple, we show that these module categories are isomorphic to quotients of the additive Karoubi envelopes of our diagrammatic categories by tensor ideals of negligible morphisms; see \cref{storm,pike}.  These are real analogues of the some of the main results concerning Deligne's interpolating categories in the complex case.  As another application, we deduce equivalences between supercategories of tensor supermodules over the different real forms of a complex supergroup; see \cref{surprise,iceR,iceC,iceH}.

\subsection*{Further directions}

We conclude this introduction with a brief discussion of some of the future research directions that stem from the current work.  Many of these are real analogues of promising work that has been done in the complex case.

While our results show that the oriented and unoriented incarnation functors are \emph{full}, we leave a description of the kernels of these functors, also known as the \emph{second fundamental theorem}, for future work.  When the target module supercategory is semisimple, the kernel is the tensor ideal of negligible morphisms; see \cref{storm,pike}.  However, this is not the case in general. For the usual oriented and unoriented Brauer categories, kernels have been described explicitly in \cite{CW12,LZ21}.

Since the target module supercategories of incarnation functors are idempotent complete, one has induced functors
\[
    \Kar(\OB_\kk(A^\op;m-n)) \to \fgl(m|n;A)\smod
    \qquad \text{and} \qquad
    \Kar(\Brauer_\kk^\sigma(A,\inv)) \to G(\Phi)\smod,
\]
where $\Kar(\cC)$ denotes the additive Karoubi envelope of $\cC$.  The supermodules that appear in the image of these functors are the summands of the tensor powers of the natural module (and, in the oriented case, its dual).  It would be interesting to give a more precise description of these supermodules.  For the usual oriented and unoriented Brauer categories, results in this direction have been obtained in \cite{BS12,CH17,CW12,Hei17}.

The supercategories introduced here have affine analogues \cite{MS21,Sam22}, generalizing the affine oriented Brauer category of \cite{BCNR17} and the affine Brauer category of \cite{RS19}.  These affine supercategories act naturally on categories of supermodules over supergroups.  We plan to investigate these actions in future work.

There exist quantum analogues of the oriented and unoriented Brauer categories.  These are the framed HOMFLYPT skein and Kauffman skein categories, respectively.  One has analogues of the results mentioned above, but with supergroups replaced by quantized enveloping superalgebras.  We expect that one can also define quantum analogues of the more general supercategories introduced in the current paper.  When the ground field is $\R$, these should be related to the representation theory of real quantum groups.

\subsection*{Acknowledgements}

This research of A.S.\ was supported by NSERC Discovery Grant RGPIN-2017-03854.  We thank Jon Brundan, Inna Entova-Aizenbud, Thorsten Heidersdorf, Allan Merino, Hadi Salmasian, Nolan Wallach, and Ben Webster for helpful discussions.

\section{Monoidal supercategories\label{sec:monsupcat}}

In this paper, we will work with \emph{strict monoidal supercategories} in the sense of \cite{BE17}.  In this section, we review a few of the more important ideas that are crucial for our exposition and somewhat less well known.  Throughout this section we work over an arbitrary ground field $\kk$.

A \emph{supercategory} is a category enriched in the monoidal category of superspaces and parity preserving linear maps. Thus, its morphism spaces are vector superspaces and composition is parity-preserving; that is, $\overline{f \circ g} = \bar{f} + \bar{g}$, where $\bar{f}$ denotes the parity of $f$.  A \emph{superfunctor} between supercategories induces a parity-preserving linear map between morphism superspaces.  For superfunctors $F,G \colon \cC \to \cD$, a \emph{supernatural transformation} $\alpha \colon F \Rightarrow G$ of \emph{parity $r\in\Z_2$} is a family of morphisms $\alpha_X\in \Hom_{\cD}(FX, GX)_r$, $X \in \cC$, such that $Gf \circ \alpha_X = (-1)^{r \bar f}\alpha_Y\circ Ff$ for each homogeneous $f \in \Hom_{\cC}(X, Y)$.  Note when $r$ is odd that $\alpha$ is \emph{not} a natural transformation in the usual sense due to the sign. A \emph{supernatural transformation} $\alpha \colon F \Rightarrow G$ is a sum $\alpha = \alpha_\even + \alpha_\odd$, where $\alpha_r$ is a supernatural transformation of parity $r$.

In a strict monoidal supercategory, the \emph{super interchange law}, which follows from the fact that $\otimes$ is a superbifunctor, is
\begin{equation}\label{interchange}
    (f' \otimes g) \circ (f \otimes g')
    = (-1)^{\bar f \bar g} (f' \circ f) \otimes (g \circ g').
\end{equation}
We denote the unit object by $\one$ and the identity morphism of an object $X$ by $1_X$.  We will use the usual calculus of string diagrams, representing the tensor product $f \otimes g$ of morphisms $f$ and $g$ diagrammatically by drawing $f$ to the left of $g$, and the composition $f \circ g$ by drawing $f$ above $g$.  Care is needed with horizontal levels in such diagrams due to
the signs arising from the super interchange law:
\begin{equation}\label{intlaw}
    \begin{tikzpicture}[anchorbase]
        \draw (-0.5,-0.5) -- (-0.5,0.5);
        \draw (0.5,-0.5) -- (0.5,0.5);
        \filldraw[fill=white,draw=black] (-0.5,0.15) circle (5pt);
        \filldraw[fill=white,draw=black] (0.5,-0.15) circle (5pt);
        \node at (-0.5,0.15) {$\scriptstyle{f}$};
        \node at (0.5,-0.15) {$\scriptstyle{g}$};
    \end{tikzpicture}
    \quad=\quad
    \begin{tikzpicture}[anchorbase]
        \draw (-0.5,-0.5) -- (-0.5,0.5);
        \draw (0.5,-0.5) -- (0.5,0.5);
        \filldraw[fill=white,draw=black] (-0.5,0) circle (5pt);
        \filldraw[fill=white,draw=black] (0.5,0) circle (5pt);
        \node at (-0.5,0) {$\scriptstyle{f}$};
        \node at (0.5,0) {$\scriptstyle{g}$};
    \end{tikzpicture}
    \quad=\quad
    (-1)^{\bar f\bar g}\
    \begin{tikzpicture}[anchorbase]
        \draw (-0.5,-0.5) -- (-0.5,0.5);
        \draw (0.5,-0.5) -- (0.5,0.5);
        \filldraw[fill=white,draw=black] (-0.5,-0.15) circle (5pt);
        \filldraw[fill=white,draw=black] (0.5,0.15) circle (5pt);
        \node at (-0.5,-0.15) {$\scriptstyle{f}$};
        \node at (0.5,0.15) {$\scriptstyle{g}$};
    \end{tikzpicture}
    \ .
\end{equation}

\begin{defin}\label{pienv}
    For a supercategory $\cC$, its \emph{$\Pi$-envelope} $\cC_\pi$ is the supercategory with objects given by formal symbols $\{ \Pi^r X : X \in \cC,\ r \in \Z_2\}$ and morphisms defined by
    \begin{equation}
        \Hom_{\cC_\pi}( \Pi^r X, \Pi^s Y )
        := \Pi^{s-r} \Hom_{\cC}(X,Y),
    \end{equation}
    where, on the right-hand side, $\Pi$ denotes the parity shift operator determined by $(\Pi V)_r := V_{r-\odd}$ for a vector superspace $V$. The composition law in $\cC_\pi$ is induced in the obvious way from the one in $\cC$: writing $f_r^s$ for the morphism in $\Hom_{\cC_\pi}(\Pi^r X, \Pi^s Y)$
    of parity $\bar{f}+r-s$ defined by $f \in \Hom_{\cC}(X,Y)$, we have that $f_s^u \circ g_r^s = (f \circ g)_r^u$.
\end{defin}

A \emph{$\Pi$-supercategory} $(\cD,\Pi,\zeta)$ is a supercategory $\cD$, together with the extra data of a superfunctor $\Pi \colon \cD \to \cD$, called the \emph{parity shift}, and an odd supernatural isomorphism $\zeta$ from $\Pi$ to the identity superfunctor; see \cite[Def.~1.7]{BE17}.  The $\Pi$-envelope $\cC_\pi$ from \cref{pienv} is a $\Pi$-supercategory with parity shift superfunctor $\Pi \colon \cC_\pi\rightarrow \cC_\pi$ sending object $\Pi^r X$ to $\Pi^{r+\odd}X$ and morphism $f_r^s$ to $f_{r+\odd}^{s+\odd}$.  Viewing $\cC$ as a full subcategory of its $\Pi$-envelope $\cC_\pi$ via the canonical embedding
\begin{equation}\label{flowers}
    \cC \rightarrow \cC_\pi, \qquad X \mapsto \Pi^\even X, \quad f \mapsto f_\even^\even,
\end{equation}
the $\Pi$-envelope satisfies a universal property: any superfunctor $F \colon \cC \rightarrow \cD$ to a $\Pi$-supercategory $\cD$ extends in a canonical way to a superfunctor $\tilde F \colon \cC_\pi \rightarrow \cD$ such that $\tilde F \circ \Pi = \Pi \circ \tilde F$.  In turn, any supernatural transformation $\theta \colon F \Rightarrow G$ between superfunctors $F, G \colon \cC\rightarrow \cD$ extends in a unique way to a supernatural transformation $\tilde\theta \colon \tilde F \Rightarrow \tilde G$; see \cite[Lem.~4.2]{BE17}.

Given superalgebras $A$ and $B$, the supercategory of $(A,B)$-superbimodules is a $\Pi$-supercategory, as explained in \cite[Example~1.8]{BE17}.  If $V$ is an $(A,B)$-supermodule, we will denote its parity shift by
\begin{equation} \label{apple}
    \Pi V := \{\pi v : v \in V\}
    \qquad \text{with} \quad
    \overline{\pi v} = \overline{v} + 1.
\end{equation}
Here $\pi$ is a formal symbol to remind us that $\pi f$ is an element of $\Pi V$.  In order to unify some expressions where the parity shift may or may not be present, we define $\pi^\sigma v$, for $\sigma \in \Z_2$, by
\[
    \pi^\sigma v =
    \begin{cases}
        \pi v & \text{if } \sigma = 1,\\
        v & \text{if } \sigma = 0.
    \end{cases}
\]
For a morphism $f \colon V \to W$, we have
\[
    \Pi f \colon \Pi V \to \Pi W,\qquad
    (\Pi f)(\pi v) = (-1)^{\bar{f}} \pi f(v).
\]
The isomorphism
\begin{equation} \label{quirk}
    \Pi^2 V \xrightarrow{\cong} V,\qquad
    \pi \pi v \mapsto -v
\end{equation}
is denoted $\xi_V$ in the notation of \cite[Definition~1.7]{BE17}.

If $\cC$ is a supercategory, we let $\Add(\cC)$ denote its additive envelope.  This is the supercategory whose objects are formal finite direct sums of objects in $\cC$, and whose morphisms are identified with matrices of morphisms in $\cC$ in the usual way.  The supercategory $\Add(\cC_\pi)$, which is the additive envelope of the $\Pi$-envelope of $\cC$, is sometimes referred to as the \emph{superadditive envelope} of $\cC$.  The \emph{Karoubi envelope} $\Kar(\cC)$ of $\cC$ is the completion of its additive envelope $\Add(\cC)$ at all homogeneous idempotents.  Thus, objects of $\Kar(\cC)$ are pairs $(X,e)$ consisting of a finite direct sum $X$ of objects of $\cC$ together with a homogeneous idempotent $e \in \End_{\Add(\cC)}(X)$.  Morphisms $(X,e) \rightarrow (Y,f)$ are elements of $f\Hom_{\Add(\cC)}(X,Y)e$.  If $\cC$ is a $\Pi$-supercategory, the parity shift superfunctors extend by the usual universal property of Karoubi envelopes to make $\Kar(\cC)$ into a $\Pi$-supercategory too.

Now we consider the monoidal situation.  We make the category $\SCat$ of supercategories and superfunctors into a symmetric monoidal category following the general construction of \cite[$\S$1.4]{Kel05}.  In particular, for supercategories $\cC$ and $\cD$, their $\kk$-linear product, denoted $\cC\boxtimes\cD$, has as objects pairs $(X,Y)$ for $X \in \cC$ and $Y \in \cD$, and
\begin{equation}
    \Hom_{\cC\boxtimes\cD}((X,Y), (X',Y'))
    = \Hom_{\cC}(X,X') \otimes \Hom_{\cD}(Y,Y')
\end{equation}
with composition defined via \cref{interchange}.  A \emph{strict monoidal supercategory} is a supercategory $\cC$ with an associative, unital tensor functor $-\otimes- \colon \cC \boxtimes \cC \rightarrow \cC$.  See \cite[Def.~1.4]{BE17} for the appropriate notions of (not necessarily strict) \emph{monoidal superfunctors} between strict monoidal supercategories, and of \emph{monoidal natural transformations} between monoidal superfunctors (which are required to be even).

There is also a notion of \emph{strict monoidal $\Pi$-supercategory}; see \cite[Def.~1.12]{BE17}.  Such a category is a $\Pi$-supercategory in the earlier sense with $\Pi := \pi \otimes-$ for a distinguished object $\pi$ admitting an odd isomorphism $\zeta \colon \pi \xrightarrow{\sim} \one$.  The \emph{$\Pi$-envelope} $\cC_\pi$ of a strict monoidal supercategory $\cC$ is the $\Pi$-supercategory from \cref{pienv}, viewed as a strict monoidal $\Pi$-supercategory with $\pi := \Pi\one$, tensor product of objects defined by
\begin{equation}\label{curtain}
    \left( \Pi^r X \right) \otimes \left( \Pi^s Y \right)
    := \Pi^{r+s} (X \otimes Y),
\end{equation}
and tensor product (horizontal composition) of morphisms defined by
\begin{equation}\label{pole}
    f_r^s \otimes g_u^v := (-1)^{r(\bar{g}+u+v) +\bar{f} v}(f \otimes g)_{r+u}^{s+v}
\end{equation}
for homogeneous morphisms $f$ and $g$ in $\cC$.  See \cite[Def.~1.16]{BE17} for more details and discussion of its universal property.  When working with string diagrams, we will represent the morphism $f_{r}^{s}$ in $\cC_{\pi}$ by adding horizontal lines labelled by $r$ and $s$ at the bottom and top of the diagram for $f \colon X \to Y$:
\begin{equation}\label{covid1}
    \begin{tikzpicture}[anchorbase]
        \draw (0,-0.4) -- (0,0.4);
        \filldraw[fill=white,draw=black] (0,0) circle(4pt);
        \node at (0,0) {\dotlabel{f}};
        \shiftline{-0.3,-0.4}{0.3,-0.4}{r};
        \shiftline{-0.3,0.4}{0.3,0.4}{s};
    \end{tikzpicture}
    \colon \Pi^r X \to \Pi^s Y.
\end{equation}
Then the rules for horizontal and vertical composition in $\mathcal{C}_\pi$ become
\begin{equation} \label{slush}
    \begin{tikzpicture}[centerzero]
        \draw (0,-0.4) -- (0,0.4);
        \filldraw[fill=white,draw=black] (0,0) circle(0.18);
        \node at (0,0) {\dotlabel{f}};
        \shiftline{-0.3,-0.4}{0.3,-0.4}{r};
        \shiftline{-0.3,0.4}{0.3,0.4}{s};
    \end{tikzpicture}
    \otimes
    \begin{tikzpicture}[centerzero]
        \draw (0,-0.4) -- (0,0.4);
        \filldraw[fill=white,draw=black] (0,0) circle(0.18);
        \node at (0,0) {\dotlabel{g}};
        \shiftline{-0.3,-0.4}{0.3,-0.4}{u};
        \shiftline{-0.3,0.4}{0.3,0.4}{v};
    \end{tikzpicture}
    = (-1)^{r(\bar{g}+u+v)+\bar{f}v}
    \begin{tikzpicture}[centerzero]
        \draw (-0.3,-0.4) -- (-0.3,0.4);
        \filldraw[fill=white,draw=black] (-0.3,0) circle(0.18);
        \node at (-0.3,0) {\dotlabel{f}};
        \draw (0.3,-0.4) -- (0.3,0.4);
        \filldraw[fill=white,draw=black] (0.3,0) circle(0.18);
        \node at (0.3,0) {\dotlabel{g}};
        \shiftline{-0.6,-0.4}{0.6,-0.4}{r+u};
        \shiftline{-0.6,0.4}{0.6,0.4}{s+v};
    \end{tikzpicture}
    ,\qquad
    \begin{tikzpicture}[centerzero]
        \draw (0,-0.4) -- (0,0.4);
        \filldraw[fill=white,draw=black] (0,0) circle(0.18);
        \node at (0,0) {\dotlabel{f}};
        \shiftline{-0.3,-0.4}{0.3,-0.4}{s};
        \shiftline{-0.3,0.4}{0.3,0.4}{t};
    \end{tikzpicture}
    \circ
    \begin{tikzpicture}[centerzero]
        \draw (0,-0.4) -- (0,0.4);
        \filldraw[fill=white,draw=black] (0,0) circle(0.18);
        \node at (0,0) {\dotlabel{g}};
        \shiftline{-0.3,-0.4}{0.3,-0.4}{r};
        \shiftline{-0.3,0.4}{0.3,0.4}{s};
    \end{tikzpicture}
    =
    \begin{tikzpicture}[centerzero]
        \draw (0,-0.6) -- (0,0.6);
        \filldraw[fill=white,draw=black] (0,-0.25) circle(0.18);
        \filldraw[fill=white,draw=black] (0,0.25) circle(0.18);
        \node at (0,0.25) {\dotlabel{f}};
        \node at (0,-0.25) {\dotlabel{g}};
        \shiftline{-0.3,-0.6}{0.3,-0.6}{r};
        \shiftline{-0.3,0.6}{0.3,0.6}{t};
    \end{tikzpicture}
    .
\end{equation}
The \emph{Karoubi envelope} $\Kar(\cC)$ of a strict monoidal $\Pi$-supercategory is a strict monoidal $\Pi$-supercategory.

If $\kk'$ is a field extension of $\kk$ and $\cC$ is a supercategory over $\kk$, then we define $\cC^{\kk'}$ to be the supercategory obtained from $\cC$ by extension of scalars.  Precisely, the objects of $\cC^{\kk'}$ are the same of those of $\cC$, and we have
\[
    \Hom_{\cC^{\kk'}}(X,Y) := \Hom_\cC(X,Y) \otimes_\kk \kk',\qquad
    X,Y \in \cC,
\]
with composition extended in the natural way.  Any superfunctor $F \colon \cC \to \cD$ naturally extends to a superfunctor $F^{\kk'} \colon \cC^{\kk'} \to \cD^{\kk'}$.  If $\cC$ is a (strict) monoidal supercategory or a $\Pi$-supercategory, then so is $\cC^{\kk'}$.  In the special case where $\kk=\R$, we call $\cC^\C$ the \emph{complexification} of $\cC$.

\section{Superalgebras and supermodules\label{sec:superalgebra}}

In this section, we review some basic properties of superalgebras and supermodules that will be used in the current paper.  We work over a ground field $\kk$.

\subsection{Associative superalgebras\label{subsec:assocsup}}

All vector superspaces and superalgebras are over $\kk$ unless otherwise indicated.  We also assume that all $\kk$-supermodules are finite dimensional.  We let $V_0$ and $V_1$ denote the even and odd parts, respectively, of a $\kk$-supermodule $V$.  Then its \emph{superdimension} is
\[
    \sdim_\kk(V) = \dim_\kk(V_0) - \dim_\kk(V_1).
\]
We let $\bar{v}$ denote the parity of a homogeneous element $v$ of a $\kk$-supermodule $V$.  When we write equations involving parities of elements, we implicitly assume these elements are homogeneous; we then extend by linearity.

The term \emph{superalgebra} refers to a unital associative superalgebra.  If $A$ is a superalgebra, its \emph{opposite superalgebra} $A^\op = \{a^\op : a \in A\}$ has multiplication given by
\[
    a^\op b^\op = (-1)^{\bar{a}\bar{b}} (ba)^\op.
\]

\subsection{Supermodules\label{subsec:supermodules}}

Throughout this subsection, $A$ denotes a superalgebra and $V,W$ denote right $A$-supermodules.  We also let $\fg$ denote a Lie superalgebra.

We let $\Hom_A(V,W)$ denote the $\kk$-supermodule of all (that is, not necessarily parity-preserving) morphisms of $A$-supermodules.  We also define $\End_A(V) := \Hom_A(V,V)$.  Thus, for example, $\Hom_A(V,W)_0$ denotes the $\kk$-module of all parity-preserving $A$-linear maps from $V$ to $W$.

For $a \in A$, define
\begin{equation} \label{rho}
    \rho_a \colon V \to V,\quad
    v \mapsto (-1)^{\bar{a}\bar{v}} v a.
\end{equation}
We also define
\[
    \flip \colon V \otimes W \to W \otimes V,\qquad
    v \otimes w \mapsto (-1)^{\bar{v}\bar{w}} w \otimes v.
\]
If $V$ and $W$ are $(\fg,A)$-superbimodules, then $\rho_a$ and $\flip$ are homomorphisms of $\fg$-supermodules.

Let
\[
    V^* = \Hom_\kk(V,\kk)
\]
denote the $\kk$-dual of $V$.  This is a left $A$-module with action given by
\begin{equation} \label{dualaction}
    (af)(v) := (-1)^{\bar{a}\bar{f}} f(va).
\end{equation}
If $V$ is also a left $\fg$-supermodule, then $V^*$ is a left $\fg$-supermodule, with action given by
\begin{equation} \label{smooth}
    (Xf)(v) = -(-1)^{\bar{X}\bar{f}} f(Xv),\qquad
    X \in \fg,\ f \in V^*,\ v \in V.
\end{equation}
This action supercommutes with the left $A$-action given in \cref{dualaction}.

Let $\bB_V$ be a $\kk$-basis for $V$, which we will sometimes denote by $\bB_V^\kk$ when there is some possibility of confusion about the ground field.  Let $\{v^* : v \in \bB_V\}$ be the dual basis of $V^*$ given by
\[
    v^*(w) = \delta_{vw},\qquad v,w \in \bB_V.
\]
We have the evaluation map
\[
    \ev \colon V^* \otimes V \to \kk,\qquad
    f \otimes v \mapsto f(v),
\]
and the coevaluation map
\[
    \coev \colon \kk \to V \otimes V^*,\qquad
    1 \mapsto \sum_{v \in \bB_V} v \otimes v^*.
\]
The map $\coev$ is independent of the choice of basis $\bB_V$.  If $V$ is a left $\fg$-supermodule, then $\ev$ and $\coev$ are both homomorphisms of $\fg$-supermodules.

\subsection{Frobenius superalgebras\label{subsec:FrobAlg}}

We now recall some basic definitions and facts about Frobenius superalgebras.  For more details, including proofs in the super case considered here, we refer the reader to \cite{PS16}.  A \emph{Frobenius superalgebra} is a superalgebra $A$ equipped with a parity-preserving $\kk$-linear map $\form = \form_A \colon A \to \kk$, called the \emph{Frobenius form}, such that the induced bilinear form
\[
    A \times A \to \kk, \qquad (a,b) \mapsto \form(ab),
\]
is nondegenerate.  If $(A,\form)$ and $(A,\form')$ are two Frobenius superalgebras with the same underlying superalgebra $A$, then there exists an even invertible element $u \in A$ such that
\begin{equation} \label{goal}
    \form'(a) = \form(au)
    \qquad \text{for all } a \in A.
\end{equation}

Every Frobenius superalgebra has a \emph{Nakayama automorphism} $\Nak$, which is a superalgebra automorphism of $A$ satisfying
\begin{equation} \label{Nakayama}
    \form(ab) = (-1)^{\bar{a}\bar{b}} \form(b\Nak(a))
    = (-1)^{\bar{a}} \form(b\Nak(a))
    = (-1)^{\bar{b}} \form(b\Nak(a))
    \qquad \text{for all } a,b \in A,
\end{equation}
where the last two equalities follow from the fact that $\form(ab)=0$ unless $\bar{a}=\bar{b}$.  A Frobenius superalgebra is said to be \emph{supersymmetric} if its Nakayama automorphism is the identity map.  We will \emph{not} assume that Frobenius superalgebras are supersymmetric in this paper.  We will often refer to $A$ itself as a Frobenius superalgebra, leaving the Frobenius form implied.  Our main sources of examples of Frobenius superalgebras will be the real division superalgebras, to be discussed in detail in \cref{sec:divalg}.  See \cref{sky} for additional examples.  If $A$ is a Frobenius superalgebra, then so is $A^\op$ with Frobenius form $\form_{A^\op}(a^\op) = \form_A(a)$, $a \in A$.

If $\bB_A$ is a homogeneous $\kk$-basis of a Frobenius superalgebra $A$, we let $\bB_A^\vee := \{b^\vee : b \in \bB_A\}$ be the left dual basis, defined by
\[
    \form(b^\vee c) = \delta_{bc},\qquad b,c \in \bB_A.
\]
It follows that, for all $a \in A$, we have
\begin{equation} \label{adecomp}
    a
    = \sum_{b \in \bB_A} \form(b^\vee a)b
    = \sum_{b \in \bB_A} \form(ab) b^\vee.
\end{equation}
We also have
\begin{equation} \label{barcheck}
    \overline{b^\vee} = \bar{b}
    \qquad \text{for all } b \in \bB_A.
\end{equation}
and
\begin{equation} \label{choice}
    \sum_{b \in \bB_A^\kk} b \otimes b^\vee
    \text{ is independent of the choice of basis } \bB_A.
\end{equation}
The basis left dual to $\{ \Nak(b) : b \in \bB_A \}$ is given by
\begin{equation} \label{Nakcheck}
    \Nak(b)^\vee = \Nak(b^\vee)
    \qquad \text{for all } b \in \bB_A.
\end{equation}
\details{
    For $b,c \in \bB_A^\kk$, we have
    \[
        \form(\Nak(b^\vee) \Nak(a))
        = (-1)^{\bar{a}\bar{b}} \form(a \Nak(b^\vee)
        = \form(b^\vee a)
        = \delta_{ab}.
    \]
}

If $V$ is a right $A$-supermodule, then the \emph{supertrace} of the action of $a$ on $V$ is
\begin{equation} \label{crazy}
    \str_V(a) = \str_V^\kk(a) := \sum_{v \in \bB_V} (-1)^{\bar{v}} v^*(va),
\end{equation}
where $\bB_V$ is a $\kk$-basis of $V$ and $\{v^* : v \in \bB_V\}$, is the dual basis of $V_\kk^*$.  We use the superscript $\kk$ on $\str_V^\kk(a)$ when there is the possibility of confusion about the ground field.  We have
\[
    \str_V(ab) = (-1)^{\bar{a}\bar{b}} \str_V(ba)
    \qquad \text{for all } a,b \in A.
\]
\details{
    Since $V$ is a finite-dimensional, we can write the action of $a$ as multiplication by a supermatrix with entries in $\kk$.  Then the above identity follows from the standard one for supertraces of supermatrices.
}
It is clear that
\begin{equation} \label{break}
    \str_V(a) = (m-n)\str_A(a)
    \qquad \text{for } V = A^{m|n}.
\end{equation}

\begin{lem}
    If $(A,\form)$ is a Frobenius superalgebra, then
    \begin{equation} \label{essex}
        \str_A(a)
        = \str_A(\Nak(a))
        = \sum_{b \in \bB_A} (-1)^{\bar{b}} \form(b^\vee b a)
        = \sum_{b \in \bB_A} (-1)^{\bar{b}} \form(a b^\vee b),\qquad a \in A.
    \end{equation}
\end{lem}

\begin{proof}
    For all $b \in \bB_A$, we have
    \[
        \form(b^\vee a) = b^*(a).
    \]
    It follows immediately that $\str_A(a)$ is equal to the first sum in \cref{essex}.

    Next, note that $\form(c) = \form(\Nak(c))$ for all $c \in A$.  Thus, for all $a \in A$,
    \begin{multline*}
        \str_A(a)
        = \sum_{b \in \bB_A} (-1)^{\bar{b}} \form(b^\vee b a)
        = \sum_{b \in \bB_A} (-1)^{\bar{b}} \form \left( \Nak(b^\vee) \Nak(b) \Nak(a) \right)
        \overset{\cref{Nakayama}}{=} \sum_{b \in \bB_A} (-1)^{\bar{b}} \form \left( a \Nak(b^\vee) \Nak(b) \right)
        \\
        = \sum_{b \in \bB_A} (-1)^{\bar{b}} \form \left( a b^\vee b \right)
        \overset{\cref{Nakayama}}{=} \sum_{b \in \bB_A} (-1)^{\bar{b}} (b^\vee b \Nak(a))
        = \str_A(\Nak(a)).
    \end{multline*}
    where, in the fourth equality we changed to a sum over the basis $\{\Nak(b) : b \in \bB_A\}$ and used \cref{Nakcheck}.
\end{proof}

\subsection{Supermatrices}

We will use the term \emph{supermatrix} to denote a supermatrix with entries in a superalgebra $A$.  We let $\Mat_{p|q,r|s}(A)$ denote the $\kk$-supermodule of $(p|q) \times (r|s)$ supermatrices, and set $\Mat_{p|q}(A) := \Mat_{p|q,p|q}(A)$.  We write a supermatrix $X \in \Mat_{p|q,r|s}(A)$ in block form as
\[
    X =
    \begin{pmatrix}
        X_{00} & X_{01} \\
        X_{10} & X_{11}
    \end{pmatrix},
\]
where $X_{00}$ is $p \times r$, $X_{01}$ is $p \times s$, $X_{10}$ is $q \times r$  and $X_{11}$ is $q \times s$.  The even elements of $\Mat_{p|q,r|s}(A)$ are those supermatrices $X$ where $X_{00}$, $X_{11}$ have even entries and $X_{01}$, $X_{10}$ have odd entries.  The odd elements of $\Mat_{p|q,r|s}(A)$ are those supermatrices $X$ where $X_{00}$, $X_{11}$ have odd entries and $X_{01}$, $X_{10}$ have even entries.

We view elements of $A^{m|n}$ as column supermatrices, that is, as $(m|n) \times (1,0)$ supermatrices.  Similarly, we view elements of $A$ as $(1|0) \times (1|0)$ supermatrices.  Then the right action $(v,a) \mapsto va$ of $A$ on $A^{m|n}$ can be viewed as matrix multiplication and we can identify elements of $\End_A(A^{m|n})$ with $\Mat_{m|n}(A)$ in the usual way.

The \emph{supertranspose} of a supermatrix is given by
\begin{equation} \label{supertranspose}
    X^\st
    :=
    \begin{pmatrix}
        X_{00}^\transpose & (-1)^{\bar{X}} X_{10}^\transpose \\
        -(-1)^{\bar{X}} X_{01}^\transpose & X_{11}^\transpose
    \end{pmatrix},
\end{equation}
where $X^\transpose$ denotes the usual transpose of a matrix.  Note that
\begin{equation} \label{quiet}
    (X^\st)^\st =
    \begin{pmatrix}
        X_{00} & -X_{01} \\
        -X_{10} & X_{11}
    \end{pmatrix},
\end{equation}
so that the supertranspose has order four in general.  The \emph{supertrace} of a square supermatrix $X \in \Mat_{p|q}(A)$ is given by
\[
    \str(X) = \tr(X_{00}) - (-1)^{\bar{X}} \tr(X_{11}),
\]
where $\tr$ denotes the usual matrix trace.

For a supermatrix $X \in \Mat_{p|q,r|s}(A)$, let $X_\op \in \Mat_{p|q,r|s}(A^\op)$ denote the matrix obtained from $X$ by replacing each entry $a$ by $a^\op$.  Then define $X_\op^\st := (X_\op)^\st = (X^\st)_\op$.

\begin{lem}
    We have an isomorphism of $\kk$-superalgebras
    \begin{equation} \label{hopping}
        \Mat_{m|n}(A)^\op \xrightarrow{\cong} \Mat_{m|n}(A^\op),\quad
        X^\op \mapsto X_\op^\st.
    \end{equation}
\end{lem}

\begin{proof}
    The map is clearly an isomorphism of $\kk$-vector superspaces.  It is also a straightforward computation to prove that it respects multiplication.
    \details{
        We have
        \begin{align*}
            X_\op^\st &Y_\op^\st
            =
            \begin{pmatrix}
                (X_{00})_\op^\transpose & (-1)^{\bar{X}} (X_{10})_\op^\transpose \\
                -(-1)^{\bar{X}} (X_{01})_\op^\transpose & (X_{11})_\op^\transpose
            \end{pmatrix}
            \begin{pmatrix}
                (Y_{00})_\op^\transpose & (-1)^{\bar{Y}} (Y_{10})_\op^\transpose \\
                -(-1)^{\bar{Y}} (Y_{01})_\op^\transpose & (Y_{11})_\op^\transpose
            \end{pmatrix}
            \\
            &=
            \begin{pmatrix}
                (X_{00})_\op^\transpose (Y_{00})_\op^\transpose - (-1)^{\bar{X}+\bar{Y}} (X_{10})_\op^\transpose (Y_{01})_\op^\transpose
                & (-1)^{\bar{Y}} (X_{00})_\op^\transpose (Y_{10})_\op^\transpose + (-1)^{\bar{X}} (X_{10})_\op^\transpose (Y_{11})_\op^\transpose
                \\
                -(-1)^{\bar{X}} (X_{01})_\op^\transpose (Y_{00})_\op^\transpose - (-1)^{\bar{Y}} (X_{11})_\op^\transpose (Y_{01})_\op^\transpose
                &
                -(-1)^{\bar{X}+\bar{Y}} (X_{01})_\op^\transpose (Y_{10})_\op^\transpose + (X_{11})_\op^\transpose (Y_{11})_\op^\transpose
            \end{pmatrix}
            \\
            &= (-1)^{\bar{X}\bar{Y}}
            \begin{pmatrix}
                (Y_{00} X_{00})_\op^\transpose + (Y_{01}X_{10})_\op^\transpose
                & (-1)^{\bar{X}+\bar{Y}} \big( (Y_{10}X_{00})_\op^\transpose + (Y_{11}X_{10})_\op^\transpose \big)
                \\
                -(-1)^{\bar{X}+\bar{Y}} \big( (Y_{00}X_{01})_\op^\transpose + (Y_{01}X_{11})_\op^\transpose \big)
                &
                (Y_{10}X_{01})^{\op,\transpose} + (Y_{11}X_{11})^{\op,\transpose}
            \end{pmatrix}
            \\
            &= (-1)^{\bar{X}\bar{Y}} (YX)_\op^\st.
        \end{align*}
        To simplify signs above, we have used the fact that the parity of the entries of $X_{ij}$ are $\bar{X} + i + j$.
    }
\end{proof}

\begin{lem}
    If $A$ is a superalgebra, then
    \begin{equation} \label{blink}
        \str_{\Mat_{m|n}(A)}^\kk(X) = (m-n) \str_A^\kk \circ \str(X)
        \qquad \text{for all } X \in \Mat_{m|n}(A).
    \end{equation}
\end{lem}

\begin{proof}
    Choose the basis $\{E_{rs} b : 1 \le r,s \le m+n,\ b \in \bB_A\}$ of $\Mat_{m|n}(A)$, where $E_{rs}$ denotes the matrix with a $1$ in position $(r,s)$, and a $0$ in all other positions.  The dual basis of $\Mat_{m|n}(A)^*$ is given by
    \[
        (E_{rs} b)^*(X) = (-1)^{p(s)} b^*(\str(E_{sr}X)).
    \]
    Then we have
    \begin{multline*}
        \str_{\Mat_{m|n}(A)}^\kk(X)
        = \sum_{r,s=1}^{m+n} \sum_{b \in \bB_A} (-1)^{\bar{b} + p(r)} b^*(\str(E_{sr}E_{rs}bX))
        = (m-n) \sum_{b \in \bB_A} (-1)^{\bar{b}} b^*(\str(bX)) \\
        = (m-n) \sum_{b \in \bB_A} (-1)^{\bar{b}} b^*(b\str(X))
        = (m-n) \str_A^\kk \circ \str(X).
        \qedhere
    \end{multline*}
\end{proof}

\subsection{Lie superalgebras}

If $\fg$ is a Lie superalgebra over $\kk$, we let $\fg\smod_\kk$ denote the supercategory of finite-dimensional $\fg$-supermodules over $\kk$ with arbitrary (i.e.\ not necessarily parity-preserving) homomorphisms.  We will be particularly interested in the cases where $\kk$ is $\R$ or $\C$.  If $\kk = \R$, our notation $\fg\smod_\R$ is designed to emphasize that we are speaking of \emph{real} supermodules, as opposed to \emph{complex} supermodules.

The \emph{general linear Lie superalgebra} $\fgl(V_A)$ associated to a right $A$-supermodule $V$ is equal to $\End_A(V)$ as a $\kk$-supermodule, with Lie superbracket given by
\begin{equation} \label{toreba}
    [X,Y] := XY - (-1)^{\bar{X}\bar{Y}} YX.
\end{equation}
In the special case that $V = A^{m|n}$, we introduce the notation $\fgl(m|n,A) = \fgl(V_A)$.  Identifying $\End_A(V)$ with $\Mat_{m|n}(A)$ in the usual way, we have that $\fgl(m|n,A)$ is equal to $\Mat_{m|n}(A)$ as a $\kk$-vector superspace, with Lie superbracket given by \cref{toreba}.  By convention, we define $\fgl(0|0,A)$ to be the zero Lie superalgebra, so that $\fgl(0|0,A)\smod_\kk$ is the supercategory $\kk\smod$ of $\kk$-supermodules.

\begin{cor} \label{cut}
    We have an isomorphism of Lie $\kk$-superalgebras
    \[
        \fgl(m|n,A) \xrightarrow{\cong} \fgl(m|n,A^\op),\qquad
        X \mapsto - X_\op^\st.
    \]
\end{cor}

\begin{proof}
    The given map is clearly an isomorphism of $\kk$-vector spaces.  To verify that it also respects the Lie superbracket, we compute
    \[
        [-X_\op^\st, - Y_\op^\st]
        = X_\op^\st Y_\op^\st - (-1)^{\bar{X}\bar{Y}} Y_\op^\st X_\op^\st
        \overset{\cref{hopping}}{=} (-1)^{\bar{X}\bar{Y}} (YX)_\op^\st - (XY)_\op^\st
        = -[X,Y]_\op^\st.
        \qedhere
    \]
\end{proof}

\subsection{Harish--Chandra superpairs\label{subsec:HCpair}}

We will sometimes need to work with supergroups instead of Lie superalgebras.  The reason for this is that the incarnation functors to be defined in \cref{sec:Oinc,sec:Unic} will be full onto a category of supermodules for a supergroup, but, outside of the connected case, they will not necessarily be full onto the category of supermodules for the associated Lie superalgebra.  We review here some basic facts, referring the reader to \cite{DM99} for a more detailed overview.  Instead of directly working with supergroups, it will be simpler to work with the equivalent category of Harish-Chandra superpairs.  We refer the reader to \cite{Gav20} and the references cited therein for a proof of this equivalence.  A \emph{Harish-Chandra superpair} over $\kk$ is a pair $G = (G_\rd,\fg)$, where $G_\rd$ is an algebraic group over $\kk$, $\fg$ is a Lie superalgebra over $\kk$,
\begin{itemize}
    \item $\fg_0$ is the Lie algebra of $G_\rd$,
    \item $G_\rd$ acts algebraically on $\fg$ by $\kk$-linear transformations, and
    \item the differential of the action of $G_\rd$ on $\fg$ coincides with the action of $\fg_0$ on $\fg$ via the superbracket.
\end{itemize}
Suppose $G = (G_\rd,\fg)$ is a Harish-Chandra superpair.  A $G$-supermodule is a $\kk$-supermodule $V$ that is both a $G_\rd$-supermodule and a $\fg$-supermodule, and such that the differential of the action of $G_\rd$ coincides with the action of $\fg_0$.  The finite-dimensional $G$-supermodules form a monoidal supercategory, which we denote by $G\smod_\kk$.

For $G$-supermodules $V$ and $W$, we have
\[
    \Hom_G(V,W) = \Hom_{G_\rd}(V,W) \cap \Hom_\fg(V,W).
\]
If $G_\rd$ is connected, then $\Hom_{\fg_0}(V,W) = \Hom_{G_\rd}(V,W)$, and so $\Hom_G(V,W) = \Hom_\fg(V,W)$.  In this case, the forgetful superfunctor $G\smod_\kk \to \fg\smod_\kk$ is full and faithful.  On the other hand, if $\fg_1 = 0$, so that we are working in the purely even setting, then the forgetful functor $G\md_\kk \to G_\rd\md_\kk$ is full and faithful.  (In fact, it is an equivalence of categories.)  In this case, we will often identity $G$ and $G_\rd$.

On the other hand, suppose $G_\rd$ has $r+1$ connected components, and let $X_1,\dotsc,X_r$ be elements of $G_\rd$, one from each of the $r$ connected components not containing the identity.  Then $G_\rd = H \cup H X_1 \cup \dotsb H X_r$, where $H$ is the identity component of $G_\rd$, and we have
\begin{multline} \label{breath}
    \Hom_G(V,W) = \Hom_{X_1,\dotsc,X_r,\fg}(V,W)
    \\
    := \{ f \in \Hom_\fg(V,W) : f(X_t v) = X_t f(v) \ \forall\ v \in V,\ 1 \le t \le r\}.
\end{multline}
For example, if $\kk=\C$, and $G_\rd = \rO(m,\C)$ is the complex orthogonal group, then we have
\[
    \Hom_G(V,W) = \Hom_{X,\fg}(V,W),
\]
where $X$ is any element of $\rO(m,\C)$ with $\det(X) = -1$.

\subsection{Complexification\label{subsec:complexification}}

If $V$ is a real vector superspace, its \emph{complexification} is
\[
    V^\C := V \otimes_\R \C.
\]
We view $V$ as an $\R$-vector subspace of $V^\C$ by identifying $v \in V$ with $v \otimes 1$.  If $A$ is a real superalgebra, then $A^\C$ is a complex superalgebra, with product
\[
    (a \otimes y) (b \otimes z) = ab \otimes yz,\qquad a,b \in A,\ y,z \in \C.
\]
Similarly, if $\fg$ is a Lie superalgebra over $\R$, then its complexification $\fg^\C$ is a Lie superalgebra over $\C$.  A \emph{real form} of a complex vector superspace $W$ is a real vector superspace $V$ such that $V^\C \cong W$ as complex vector superspaces.  We define real forms of complex associative superalgebras and complex Lie superalgebras similarly.

Suppose $R$ is either a real associative superalgebra or a real Lie superalgebra.  If $V$ is a left (respectively, right) $R$-supermodule, then $V^\C$ is a left (respectively, right) $R^\C$-supermodule with the natural action.  Furthermore, every $f \in \Hom_R(V,W)$ induces an element $f^\C \in \Hom_{R^\C}(V^\C,W^\C)$ given by
\[
    f^\C(v \otimes z) = f(v) \otimes z,\qquad v \in V,\ z \in \C.
\]
We have
\begin{align} \label{funk}
    \ker \big( f^\C \big) &= \ker(f)^\C,&
    \ker(f) &= \ker \big( f^\C \big) \cap V,
    \\
    \im \big( f^\C \big) &= \im(f)^\C,&
    \im(f) &= \im \big( f^\C \big) \cap V.
\end{align}
In particular
\begin{align*}
    f \text{ is injective} &\iff f^\C \text{ is injective}
    \qquad \text{and} \\
    f \text{ is surjective} &\iff f^\C \text{ is surjective}.
\end{align*}
The above constructions yield a superfunctor $R\smod \to R^\C\smod$, which induces a full and faithful \emph{complexification superfunctor}
\begin{equation} \label{golem}
    \sC_R \colon (R\smod)^\C \to R^\C\smod.
\end{equation}
In particular, if $\fg$ is a real Lie superalgebra and $V,W \in \fg\smod_\R$, then we have a canonical isomorphism of $\C$-supermodules
\begin{equation} \label{skool}
    \Hom_\fg(V,W)^\C \xrightarrow{\cong} \Hom_{\fg^\C}(V^\C, W^\C).
\end{equation}
\details{
    We have
    \begin{multline*}
        \Hom_{\fg,\R}(V,W) \otimes_\R \C
        \cong (V^* \otimes_\R W)^\fg \otimes_\R \C
        = \big( (V^* \otimes_\R W) \otimes_\R \C \big)^{\fg^\C}
        \\
        \cong \big( (V^\C)^* \otimes_\C W^\C \big)^{\fg^\C}
        \cong \Hom_{\fg^\C,\C}(U^\C,W^\C).
    \end{multline*}
}

If $H$ is a real supergroup acting on a real vector space $V$, then $H$ also acts on $V^\C = V \otimes_\R \C$ by acting on the first factor.  If $V$ and $W$ are supermodules over a real Harish-Chandra superpair $G=(G_\rd,\fg$), then we have an isomorphism of $\C$-supermodules
\begin{multline*}
    \Hom_G(V,W)^\C \cong \Hom_{G_\rd,\fg^\C}(V^\C,W^\C)
    \\
    := \{X \in \Hom_{\fg^\C}(V^\C,W^\C) : f(Xv) = Xf(v) \ \forall\ X \in G_\rd,\ v \in V^\C\}.
\end{multline*}
If $G_\rd$ has $r+1$ connected components, and $X_1,\dotsc,X_r$ are elements of $G_\rd$, one from each connected component not containing the identity, then, using \cref{breath}, we have
\begin{multline} \label{skool2}
    \Hom_G(V,W)^\C \cong \Hom_{X_1,\dotsc,X_r,\fg^\C}(V^\C,W^\C)
    \\
    := \{f \in \Hom_{\fg^\C}(V^\C,W^\C) : f(X_tv) = X_tf(v) \ \forall\ v \in V^\C,\ 1 \le t \le r\}.
\end{multline}

If $A$ is a Frobenius $\R$-superalgebra with Frobenius form $\form$, then its complexification $A^\C$ is a Frobenius $\C$-superalgebra with Frobenius form (which we continue to denote by the same symbol)
\begin{equation} \label{floor}
    \form \colon A^\C \to \C,\qquad
    a \otimes z \mapsto \form(a)z,\quad a \in A,\ z \in \C.
\end{equation}

It is straightforward to verify that
\begin{equation} \label{chain}
    \Mat_{m|n}(A)^\C \cong \Mat_{m|n}(A^\C)
\end{equation}
as $\C$-superalgebras and
\begin{equation} \label{link}
    \fgl(m|n,A)^\C \cong \fgl(m|n,A^\C)
\end{equation}
as complex Lie superalgebras.

\section{Real division superalgebras\label{sec:divalg}}

In this section, we discuss real division superalgebras.  These will play a key role in our main applications to the representation theory of real supergroups.

\subsection{Real division superalgebras}

For our purposes, one of the most important classes of examples of Frobenius superalgebras are the real division superalgebras, which were classified by Wall \cite{Wal64}. (See also \cite{Bae20} for a short exposition.)

\begin{prop} \label{realdivalg}
    Every real division superalgebra is isomorphic to exactly one of the following, where the $\Z_2$-grading is given by declaring $\varepsilon$ to be odd, and $\star$ denotes complex conjugation.
    \begin{itemize}
        \item $\Cl_0(\R) = \R$;
        \item $\Cl_1(\R) := \R \oplus \varepsilon \R$, with $\varepsilon^2 = 1$;
        \item $\Cl_2(\R) := \C \oplus \varepsilon \C$, with $\varepsilon^2 = 1$ and $z \varepsilon = \varepsilon z^\star$ for all $z \in \C$;
        \item $\Cl_3(\R) := \HH \oplus \varepsilon \HH$, with $\varepsilon^2=-1$ and $z \varepsilon = \varepsilon z$ for all $z \in \HH$;
        \item $\Cl_4(\R) := \HH$;
        \item $\Cl_5(\R) := \HH \oplus \varepsilon \HH$, with $\varepsilon^2=1$ and $z \varepsilon = \varepsilon z$ for all $z \in \HH$;
        \item $\Cl_6(\R) := \C \oplus \varepsilon \C$, with $\varepsilon^2=-1$ and $z \varepsilon = \varepsilon z^\star$ for all $z \in \C$;
        \item $\Cl_7(\R) := \R \oplus \varepsilon \R$, with $\varepsilon^2 = -1$;
        \item $\C$;
        \item $\Cl(\C) := \C \oplus \varepsilon \C$, with $\varepsilon^2=1$ and $z \varepsilon = \varepsilon z$ for all $z \in \C$.
    \end{itemize}
    For $0 \le r \le 8$, we have $\Cl_r(\R)^\op \cong \Cl_{-r}(\R)$ as superalgebras, where subscripts are considered modulo $8$.
\end{prop}

The notation in \cref{realdivalg} is inspired by the fact that $\Cl_r(\R) \otimes_\R \Cl_s(\R)$ is Morita equivalent to $\Cl_{r+s}(\R)$.  Note that $\C$ and the complex Clifford algebra $\Cl(\C)$ are the only complex division superalgebras.  The $\Cl_r(\R)$, $0 \le r \le 7$, are real Clifford superalgebras.  Recall that a $\kk$-superalgebra is \emph{central} if its center is $\kk$.  Thus, the central real division superalgebras are those real division superalgebras isomorphic to $\Cl_r(\R)$ for $0 \le r \le 7$.

\begin{rem} \label{complexdivalg}
      The complex division superalgebras are isomorphic to their own opposite superalgebras.  For $\C$, this follows from the fact that $\C$ is commutative.  For $\Cl(\C)$, we have $\Cl(\C)^\op = \C \oplus \varepsilon \C$, with $\varepsilon^2 = -1$, and an isomorphism of $\C$-superalgebras
    \[
        \Cl(\C) \xrightarrow{\cong} \Cl(\C)^\op,\qquad \varepsilon \mapsto \varepsilon i.
    \]
\end{rem}

\begin{convention}[Frobenius forms on division superalgebras] \label{amongus}
    Note that $\C$ and $\Cl(\C)$ are complex Frobenius superalgebras, with Frobenius form given by projection $\proj_\C$ onto their even part.  (They are also real Frobenius superalgebras with Frobenius form given by projection onto the real part of their even part.)  We will always view the central real division superalgebras as superalgebras over $\R$.  They are real Frobenius superalgebras with Frobenius form $\RP \colon \DD \to \R$ given by taking the real part of the even part of $\DD$.  In all of these cases, the Nakayama automorphism is given by
    \begin{equation} \label{divNak}
        \Nak(a) = (-1)^{\bar{a}} a,\qquad a \in \DD.
    \end{equation}
    In particular, a real division superalgebra is supersymmetric if and only if it is purely even.
\end{convention}

\begin{lem} \label{doubledual}
    If $\kk \in \{\R,\C\}$ and $\bB_\DD$ is a $\kk$-basis for a division $\kk$-superalgebra $\DD$, then the basis left dual to $\bB_\DD^\vee = \{b^\vee : b \in \bB_\DD\}$ is given by
    \begin{equation}
        (b^\vee)^\vee = b,\qquad b \in \bB_\DD.
    \end{equation}
\end{lem}

\begin{proof}
    For all $b,c \in \bB_\DD$, we have
    \[
        \form(c b^\vee)
        \overset{\cref{Nakayama}}{=} (-1)^{\bar{b}\bar{c}} \form(b^\vee \Nak(c))
        \overset{\cref{divNak}}{=} (-1)^{\bar{b}\bar{c} + \bar{c}} \form(b^\vee c)
        = \delta_{bc}.
        \qedhere
    \]
\end{proof}

If $\DD$ is a real division superalgebra, we will use the term \emph{$\DD$-vector superspace} to denote a finite-dimensional \emph{right} $\DD$-supermodule.  Recall the definition of $\str_A$ given in \cref{crazy}.

\begin{lem} \label{delay}
    If $\DD$ is a real or complex division superalgebra with standard Frobenius form $\form$, then $\str_\DD = (\sdim_\kk \DD) \form$.  In particular, $\str_\DD=0$ whenever $\DD$ is a real or complex division superalgebra with nonzero odd part.
\end{lem}

\begin{proof}
    If $\DD$ is a real division superalgebra, choose the basis
    \[
        \bB_\DD =
        \begin{cases}
            \{1\} & \text{if } \DD = \R, \\
            \{1,i\} & \text{if } \DD = \C, \\
            \{1,i,j,k\} & \text{if } \DD = \HH, \\
            \{1,\varepsilon\} & \text{if } \DD \in \{ \Cl_1(\R), \Cl_7(\R) \}, \\
            \{1,i,\varepsilon,\varepsilon i\} & \text{if } \DD \in \{ \Cl_2(\R), \Cl_6(\R) \}, \\
            \{1,i,j,k,\varepsilon, \varepsilon i, \varepsilon j, \varepsilon k \} & \text{if } \DD \in \{ \Cl_3(\R), \Cl_5(\R) \}.
        \end{cases}
    \]
    If $\DD = \C$, considered as a complex division superalgebra choose $\bB_\DD = \{1\}$.  Finally, if $\DD = \Cl(\C)$, considered as a complex division superalgebra, choose $\bB_\DD = \{1,\varepsilon\}$.  Then we have $b^\vee = b^{-1}$ for all $b \in \bB_\DD$.  Hence, using \cref{essex}, we have
    \[
        \str_\DD(a)
        = \sum_{b \in \bB_\DD} (-1)^{\bar{b}} \form(a)
        = (\sdim_\kk \DD) \form(a).
        \qedhere
    \]
\end{proof}

\subsection{Complexification of real division superalgebras}

\begin{lem} \label{ride}
    We have the following injections of superalgebras, where $\star$ denotes complex conjugation,
    \begin{align}
        \R &\hookrightarrow \C,&
        a &\mapsto a,
        \quad a \in \R,
        \\ \label{Pauli}
        \imath \colon \HH &\hookrightarrow \Mat_2(\C),&
        i &\mapsto \begin{pmatrix} i & 0 \\ 0 & -i \end{pmatrix},\quad
        j \mapsto \begin{pmatrix} 0 & -1 \\ 1 & 0 \end{pmatrix},\quad
        k \mapsto \begin{pmatrix} 0 & -i \\ -i & 0 \end{pmatrix},
        \\
        \Cl_1(\R) &\hookrightarrow \Cl(\C),&
        a + \varepsilon b &\mapsto a + \varepsilon b,
        \quad a,b \in \R,
        \\
        \Cl_2(\R) &\hookrightarrow \Mat_{1|1}(\C),&
        a + \varepsilon b &\mapsto \begin{pmatrix} a & b^\star \\ b & a^\star \end{pmatrix},\quad
        a,b \in \C,
        \\
        \Cl_3(\R) &\hookrightarrow \Mat_2(\Cl(\C)),&
        a + \varepsilon b &\mapsto \imath(a) + \varepsilon \imath(b) i,
        \quad a,b \in \HH,
        \\
        \Cl_5(\R) &\hookrightarrow \Mat_2(\Cl(\C)),&
        a + \varepsilon b &\mapsto \imath(a) + \varepsilon \imath(b),\quad a,b \in \HH,
        \\
        \Cl_6(\R) &\hookrightarrow \Mat_{1|1}(\C),&
        a + \varepsilon b &\mapsto \begin{pmatrix} a & -b^\star \\ b & a^\star \end{pmatrix},\quad
        a,b \in \C,
        \\
        \Cl_7(\R) &\hookrightarrow \Cl(\C),&
        a + \varepsilon b &\mapsto a + \varepsilon i b,
        \quad a,b \in \R.
    \end{align}
\end{lem}

\begin{proof}
    These are all straightforward verifications.
\end{proof}

\begin{rem} \label{cliffmat}
    We have an injection of complex superalgebras
    \[
        \Cl(\C) \hookrightarrow \Mat_2(\C),\qquad
        a + \varepsilon b \mapsto
        \begin{pmatrix}
            a & b \\
            b & a
        \end{pmatrix}
        ,\quad a,b \in \C.
    \]
    Combined with \cref{ride}, this shows that all of the real division superalgebras can be embedded in complex supermatrix superalgebras.
\end{rem}

The following result gives the complexification of the central real division superalgebras.

\begin{lem} \label{sail}
    The inclusions of \cref{ride} induce isomorphisms of complex superalgebras
    \begin{gather*}
        \R^\C \cong \C,\qquad
        \HH^\C \cong \Mat_2(\C),\qquad
        \Cl_1(\R)^\C \cong \Cl_7(\R)^\C \cong \Cl(\C),
        \\
        \Cl_2(\R)^\C \cong \Cl_6(\R)^\C \cong \Mat_{1|1}(\C),\qquad
        \Cl_3(\R)^\C \cong \Cl_5(\R)^\C \cong \Mat_2(\Cl(\C)).
    \end{gather*}
    In particular, for every central real division superalgebra $\DD$, its complexification $\DD^\C$ is a simple complex superalgebra.
\end{lem}

\begin{proof}
    It is straightforward to verify that, for each of the injections in \cref{ride}, every matrix in the codomain can be written uniquely in the form $X + iY$, where $X$ and $Y$ are in the image of the injection.
\end{proof}

\subsection{General linear Lie superalgebras over division superalgebras}

\begin{lem} \label{fold}
    If $\DD$ is a real division superalgebra with $\DD_1 \ne 0$, then
    \begin{enumerate}
        \item $\DD^{m|n} \cong \DD^{m+n}$ as $\DD$-vector superspaces,

        \item $\Mat_{m|n}(\DD) \cong \Mat_{m+n}(\DD)$ as superalgebras,

        \item $\fgl(m|n,\DD) \cong \fgl(m+n,\DD)$ as Lie superalgebras.
    \end{enumerate}
\end{lem}

\begin{proof}
    Suppose $\DD$ is a real division superalgebra with $\DD_1 \ne 0$.  Then left multiplication by any nonzero odd element gives an isomorphism $\DD^{0|n} \cong \DD^{n|0}$.  Hence $\DD^{m|n} \cong \DD^{m|0} \oplus \DD^{0|n} \cong \DD^{m|0} \oplus \DD^{n|0} \cong \DD^{m+n}$.  This induces isomorphisms of superalgebras
    \[
        \Mat_{m+n}(\DD)
        \cong \End_\DD(\DD^{m|n})
        \cong \End_\DD(\DD^{m+n})
        \cong \Mat_{m+n}(\DD).
    \]
    Passing to the associated Lie superalgebras then gives the isomorphism $\fgl(m|n,\DD) \cong \fgl(m+n,\DD)$.
\end{proof}

In light of \cref{fold}, we will often consider only $\fgl(m,\DD)$, as opposed to $\fgl(m|n,\DD)$, when $\DD_1 \ne 0$.

\begin{rem}
    The general linear superalgebras over the central real division superalgebras are often known by different names and notation.
    \begin{itemize}
        \item $\fgl(m,\Cl_1(\R))$ is the \emph{split real isomeric Lie superalgebra}, often called the \emph{split real queer Lie superalgebra}.  It is usually denoted $\fq(m,\R)$.
        \item $\fgl(m,\Cl(\C))$ is the \emph{complex isomeric Lie superalgebra}, often called the \emph{complex queer Lie superalgebra}.  It is usually denoted $\fq(m,\C)$.
        \item $\fgl(m,\Cl_2(\R))$ is sometimes denoted $\fq^0(m,\R)$.
        \item $\fgl(m,\Cl_5(\R))$ is sometimes denoted $\fq^*(2m)$.
        \item $\fgl(m|n,\HH)$ is sometimes denoted $\fu^*(2m|2n)$.
    \end{itemize}
    (Recall that $\fgl(m|n,\DD) \cong \fgl(m|n,\DD^\op)$ by \cref{cut}.)  Many references focus on the realization of real Lie superalgebras in terms of complex matrices, using the inclusions of \cref{ride,cliffmat}.  However, we feel that the realization in terms of general linear Lie superalgebras over real division superalgebras is more natural and leads to more uniform and easy-to-understand notation.  We will only use the notation $\fq(m,\C)$ in the complex case:
    \[
        \fq(m,\C) = \fgl(m,\Cl(\C)).
    \]
\end{rem}

The following proposition shows that the general linear Lie superalgebras over central real division superalgebras are real forms of general linear and isomeric Lie superalgebras.

\begin{prop} \label{glcomplex}
    We have isomorphisms of complex Lie superalgebras
    \begin{enumerate}
        \item $\fgl(m|n,\R)^\C \cong \fgl(m|n,\C)$,
        \item \label{glcomplex:H} $\fgl(m|n,\HH)^\C \cong \fgl(2m|2n,\C)$,
        \item $\fgl(m,\Cl_1(\R))^\C \cong \fgl(m,\Cl_7(\R))^\C \cong \fq(m,\C)$,
        \item $\fgl(m,\Cl_2(\R))^\C \cong \fgl(m,\Cl_6(\R))^\C \cong \fgl(m|m,\C)$,
        \item $\fgl(m,\Cl_3(\R))^\C \cong \fgl(m,\Cl_5(\R))^\C \cong \fq(2m,\C)$.
    \end{enumerate}
\end{prop}

\begin{proof}
    This follows from \cref{link}, \cref{sail}, and the fact that we have canonical isomorphisms of complex Lie superalgebras
    \begin{equation} \label{nested}
        \fgl(m|n,\Mat_{r|s}(A)) \cong \fgl((mr+ns)|(ms+nr),A)
    \end{equation}
    for any superalgebra $A$.
\end{proof}

\section{The oriented supercategory\label{sec:OBC}}

In this section, we introduce the first of our two main diagrammatic supercategories.  After defining the supercategory, we prove a basis theorem for its morphism spaces.  In \cref{subsec:OBCdef,subsec:OBbasis}, $\kk$ denotes an arbitrary field.  In \cref{subsec:OBcomplexification}, we discuss the special cases $\kk \in \{\R,\C\}$.

\subsection{Definition of the supercategory\label{subsec:OBCdef}}

\begin{defin}[{\cite[Def.~4.1]{MS21}}] \label{OBC}
    For an associative superalgebra $A$, we define $\OB_\kk(A)$ to be the strict monoidal supercategory generated by objects $\upobj$ and $\downobj$ and morphisms
    \begin{gather*}
        \upcross \colon \upobj \otimes \upobj \to \upobj \otimes \upobj
        \ ,\quad
        \uptokstrand \colon \upobj \to \upobj
        \ ,\ a \in A,
        \\
        \leftcap \colon \downobj \otimes \upobj \to \one
        , \quad
        \leftcup \colon \one \to \upobj \otimes \downobj
        , \quad
        \rightcap \colon \upobj \otimes \downobj \to \one
        , \quad
        \rightcup \colon \one \to \downobj \otimes \upobj,
    \end{gather*}
    subject to the relations
    \begin{gather} \label{toklin}
        \begin{tikzpicture}[centerzero]
            \draw[->] (0,-0.35) -- (0,0.35);
            \token{0,0}{west}{1};
        \end{tikzpicture}
        =
        \begin{tikzpicture}[centerzero]
            \draw[->] (0,-0.35) -- (0,0.35);
        \end{tikzpicture}
        ,\quad
        \lambda\
        \begin{tikzpicture}[centerzero]
            \draw[->] (0,-0.35) -- (0,0.35);
            \token{0,0}{west}{a};
        \end{tikzpicture}
        + \mu\
        \begin{tikzpicture}[centerzero]
            \draw[->] (0,-0.35) -- (0,0.35);
            \token{0,0}{west}{b};
        \end{tikzpicture}
        =
        \begin{tikzpicture}[centerzero]
            \draw[->] (0,-0.35) -- (0,0.35);
            \token{0,0}{west}{\lambda a + \mu b};
        \end{tikzpicture}
        ,\quad
        \begin{tikzpicture}[centerzero]
            \draw[->] (0,-0.35) -- (0,0.35);
            \token{0,-0.15}{east}{b};
            \token{0,0.15}{east}{a};
        \end{tikzpicture}
        =
        \begin{tikzpicture}[centerzero]
            \draw[->] (0,-0.35) -- (0,0.35);
            \token{0,0}{west}{ab};
        \end{tikzpicture}
        \ ,
        \\ \label{wreath}
        \begin{tikzpicture}[centerzero]
            \draw[->] (0.2,-0.4) to[out=135,in=down] (-0.15,0) to[out=up,in=-135] (0.2,0.4);
            \draw[->] (-0.2,-0.4) to[out=45,in=down] (0.15,0) to[out=up,in=-45] (-0.2,0.4);
        \end{tikzpicture}
        =
        \begin{tikzpicture}[centerzero]
            \draw[->] (-0.15,-0.4) -- (-0.15,0.4);
            \draw[->] (0.15,-0.4) -- (0.15,0.4);
        \end{tikzpicture}
        \ ,\quad
        \begin{tikzpicture}[centerzero]
            \draw[->] (0.3,-0.4) -- (-0.3,0.4);
            \draw[->] (0,-0.4) to[out=135,in=down] (-0.25,0) to[out=up,in=-135] (0,0.4);
            \draw[->] (-0.3,-0.4) -- (0.3,0.4);
        \end{tikzpicture}
        =
        \begin{tikzpicture}[centerzero]
            \draw[->] (0.3,-0.4) -- (-0.3,0.4);
            \draw[->] (0,-0.4) to[out=45,in=down] (0.25,0) to[out=up,in=-45] (0,0.4);
            \draw[->] (-0.3,-0.4) -- (0.3,0.4);
        \end{tikzpicture}
        \ ,\quad
        \begin{tikzpicture}[centerzero]
            \draw[->] (0.3,-0.4) -- (-0.3,0.4);
            \draw[->] (-0.3,-0.4) -- (0.3,0.4);
            \token{-0.15,-0.2}{east}{a};
        \end{tikzpicture}
        =
        \begin{tikzpicture}[centerzero]
            \draw[->] (0.3,-0.4) -- (-0.3,0.4);
            \draw[->] (-0.3,-0.4) -- (0.3,0.4);
            \token{0.15,0.2}{west}{a};
        \end{tikzpicture}
        \ ,
        \\ \label{inversion}
        \begin{tikzpicture}[centerzero]
            \draw[<-] (0.2,-0.4) to[out=135,in=down] (-0.15,0) to[out=up,in=-135] (0.2,0.4);
            \draw[->] (-0.2,-0.4) to[out=45,in=down] (0.15,0) to[out=up,in=-45] (-0.2,0.4);
        \end{tikzpicture}
        =
        \begin{tikzpicture}[centerzero]
            \draw[<-] (-0.15,-0.4) -- (-0.15,0.4);
            \draw[->] (0.15,-0.4) -- (0.15,0.4);
        \end{tikzpicture}
        \ ,\quad
        \begin{tikzpicture}[centerzero]
            \draw[->] (0.2,-0.4) to[out=135,in=down] (-0.15,0) to[out=up,in=-135] (0.2,0.4);
            \draw[<-] (-0.2,-0.4) to[out=45,in=down] (0.15,0) to[out=up,in=-45] (-0.2,0.4);
        \end{tikzpicture}
        =
        \begin{tikzpicture}[centerzero]
            \draw[->] (-0.15,-0.4) -- (-0.15,0.4);
            \draw[<-] (0.15,-0.4) -- (0.15,0.4);
        \end{tikzpicture}
        \ ,\quad
        \begin{tikzpicture}[centerzero]
            \draw[->] (0,-0.4) to[out=up,in=0] (-0.25,0.15) to[out=180,in=up] (-0.4,0) to[out=down,in=180] (-0.25,-0.15) to[out=0,in=down] (0,0.4);
        \end{tikzpicture}
        =
        \begin{tikzpicture}[centerzero]
            \draw[->] (0,-0.4) -- (0,0.4);
        \end{tikzpicture}
        =
        \begin{tikzpicture}[centerzero]
            \draw[->] (0,-0.4) to[out=up,in=180] (0.25,0.15) to[out=0,in=up] (0.4,0) to[out=down,in=0] (0.25,-0.15) to[out=180,in=down] (0,0.4);
        \end{tikzpicture}
        \ ,
        \\ \label{leftadj}
        \begin{tikzpicture}[centerzero]
            \draw[<-] (-0.3,-0.4) -- (-0.3,0) arc(180:0:0.15) arc(180:360:0.15) -- (0.3,0.4);
        \end{tikzpicture}
        =
        \begin{tikzpicture}[centerzero]
            \draw[<-] (0,-0.4) -- (0,0.4);
        \end{tikzpicture}
        \ ,\qquad
        \begin{tikzpicture}[centerzero]
            \draw[<-] (-0.3,0.4) -- (-0.3,0) arc(180:360:0.15) arc(180:0:0.15) -- (0.3,-0.4);
        \end{tikzpicture}
        =
        \begin{tikzpicture}[centerzero]
            \draw[->] (0,-0.4) -- (0,0.4);
        \end{tikzpicture}
        \ ,
    \end{gather}
    for all $a,b \in A$ and $\lambda,\mu \in \kk$.  In the above, the left and right crossings are defined by
    \begin{equation} \label{windmill}
        \rightcross
        :=
        \begin{tikzpicture}[centerzero]
            \draw[->] (0.2,-0.3) \braidup (-0.2,0.3);
            \draw[->] (-0.4,0.3) -- (-0.4,0.1) to[out=down,in=left] (-0.2,-0.2) to[out=right,in=left] (0.2,0.2) to[out=right,in=up] (0.4,-0.1) -- (0.4,-0.3);
        \end{tikzpicture}
        \ ,\qquad
        \leftcross
        \ :=\
        \begin{tikzpicture}[centerzero]
            \draw[->] (-0.2,-0.3) \braidup (0.2,0.3);
            \draw[->] (0.4,0.3) -- (0.4,0.1) to[out=down,in=right] (0.2,-0.2) to[out=left,in=right] (-0.2,0.2) to[out=left,in=up] (-0.4,-0.1) -- (-0.4,-0.3);
        \end{tikzpicture}
        \ .
    \end{equation}
    The parity of $\uptokstrand$ is $\bar{a}$, and all the other generating morphisms are even.  We refer to the morphisms $\uptokstrand$ as \emph{tokens}.

    For $d \in \kk$, we define $\OB_\kk(A;d)$ to be the quotient of $\OB_\kk(A)$ by the relations
    \begin{equation}
        \ccbubble{a} = d \str_A(a) 1_\one,\qquad a \in A,
    \end{equation}
    where $\str_A$ is given by \cref{crazy}.  We call $d$ the \emph{specialization parameter}.
\end{defin}

When $A$ is a Frobenius superalgebra, $\OB_\kk(A)$ was called the \emph{oriented Frobenius Brauer supercategory} in \cite[Def.~4.1]{MS21}.  The Frobenius structure on $A$ allows one to enlarge it to the \emph{affine oriented Frobenius Brauer category} of \cite[Def.~4.3]{MS21}, which is the central charge zero special case of the \emph{Frobenius Heisenberg supercategory} introduced in \cite{Sav19}, and further studied in \cite{BSW-foundations,MS21}.  We refer the reader to these papers for proofs omitted here, none of which use the Frobenius structure on $A$.  Our presentation of $\OB_\kk(A)$ is slightly different from the one given in \cite[Def.~4.1]{MS21}.  Precisely, the relations \cref{leftadj} are the reflections in the vertical axis of the ones in \cite[(4.4)]{MS21}.  However, $\OB_\kk(A)$ has a symmetry given by reflecting diagrams in the vertical axis.  (This is the composition of the isomorphisms (5.16) and (5.17) in \cite{BSW-foundations}.)  Hence, the two definitions are equivalent.

\begin{rem}
    \begin{enumerate}[wide]
        \item When $A=\kk$, we have $\uptokstrand = a\, \upstrand$ for all $a \in \kk$.  Thus, we can omit the generators $\uptokstrand$ and all the relations involving them.  Then we see that $\OB_\kk(\kk)$ is the \emph{oriented Brauer category}, which is the free rigid symmetric $\kk$-linear monoidal category generated by a single object.  This is the motivation for the notation $\OB_\kk(A)$.

        \item The supercategory $\OB_\C(\Cl(\C),0)$ is the oriented Brauer--Clifford supercategory introduced in \cite[Def.~3.2]{BCK19}.
    \end{enumerate}
\end{rem}

\begin{rem} \label{tuna}
    When $A$ is a real or complex division superalgebra with nonzero odd part, it follows from \cref{delay} that $\OB_\kk(A;d) = \OB_\kk(A;0)$ for all $d \in \R$.
\end{rem}

The relations \cref{leftadj} means that $\downobj$ is left dual to $\upobj$.  In fact, we also have
\begin{equation} \label{rightadj}
    \begin{tikzpicture}[centerzero]
        \draw[->] (-0.3,-0.4) -- (-0.3,0) arc(180:0:0.15) arc(180:360:0.15) -- (0.3,0.4);
    \end{tikzpicture}
    \ =\
    \begin{tikzpicture}[centerzero]
        \draw[->] (0,-0.4) -- (0,0.4);
    \end{tikzpicture}
    \ ,\qquad
    \begin{tikzpicture}[centerzero]
        \draw[->] (-0.3,0.4) -- (-0.3,0) arc(180:360:0.15) arc(180:0:0.15) -- (0.3,-0.4);
    \end{tikzpicture}
    \ =\
    \begin{tikzpicture}[centerzero]
        \draw[<-] (0,-0.4) -- (0,0.4);
    \end{tikzpicture}
    \ ,
\end{equation}
and so $\downobj$ is also right dual to $\upobj$.  Thus $\OB_\kk(A)$ is \emph{rigid}.  Furthermore, we have that
\begin{equation} \label{stake}
    \downcross
    :=
    \begin{tikzpicture}[centerzero]
        \draw[<-] (0.2,-0.3) \braidup (-0.2,0.3);
        \draw[->] (-0.4,0.3) -- (-0.4,0.1) to[out=down,in=left] (-0.2,-0.2) to[out=right,in=left] (0.2,0.2) to[out=right,in=up] (0.4,-0.1) -- (0.4,-0.3);
    \end{tikzpicture}
    =
    \begin{tikzpicture}[centerzero]
        \draw[<-] (-0.2,-0.3) \braidup (0.2,0.3);
        \draw[->] (0.4,0.3) -- (0.4,0.1) to[out=down,in=right] (0.2,-0.2) to[out=left,in=right] (-0.2,0.2) to[out=left,in=up] (-0.4,-0.1) -- (-0.4,-0.3);
    \end{tikzpicture}
    \ ,\quad
    \begin{tikzpicture}[centerzero]
        \draw[<-] (0,-0.4) -- (0,0.4);
        \token{0,0}{east}{a};
    \end{tikzpicture}
    :=
    \begin{tikzpicture}[centerzero]
        \draw[->] (-0.4,0.4) -- (-0.4,-0.05) arc(180:360:0.2) -- (0,0.05) arc(180:0:0.2) -- (0.4,-0.4);
        \token{0,0}{east}{a};
    \end{tikzpicture}
    =
    \begin{tikzpicture}[centerzero]
        \draw[->] (0.4,0.4) -- (0.4,-0.05) arc(360:180:0.2) -- (0,0.05) arc(0:180:0.2) -- (-0.4,-0.4);
        \token{0,0}{west}{a};
    \end{tikzpicture}
    \ ,\quad a \in A.
\end{equation}
These relations mean that tokens and crossings slide over all cups and caps in the sense that, for all orientations of the strands, we have
\begin{equation} \label{ruby}
    \begin{tikzpicture}[anchorbase]
        \draw (-0.2,-0.2) -- (-0.2,0) arc (180:0:0.2) -- (0.2,-0.2);
        \token{-0.2,0}{east}{a};
    \end{tikzpicture}
    \ =\
    \begin{tikzpicture}[anchorbase]
        \draw (-0.2,-0.2) -- (-0.2,0) arc (180:0:0.2) -- (0.2,-0.2);
        \token{0.2,0}{west}{a};
    \end{tikzpicture}
    \ ,\qquad
    \begin{tikzpicture}[anchorbase]
        \draw (-0.2,0.2) -- (-0.2,0) arc (180:360:0.2) -- (0.2,0.2);
        \token{-0.2,0}{east}{a};
    \end{tikzpicture}
    \ =\
    \begin{tikzpicture}[anchorbase]
        \draw (-0.2,0.2) -- (-0.2,0) arc (180:360:0.2) -- (0.2,0.2);
        \token{0.2,0}{west}{a};
    \end{tikzpicture}
    \ ,\qquad
    \begin{tikzpicture}[centerzero]
        \draw (-0.2,0.3) -- (-0.2,0.1) arc(180:360:0.2) -- (0.2,0.3);
        \draw (-0.3,-0.3) to[out=up,in=down] (0,0.3);
    \end{tikzpicture}
    =
    \begin{tikzpicture}[centerzero]
        \draw (-0.2,0.3) -- (-0.2,0.1) arc(180:360:0.2) -- (0.2,0.3);
        \draw (0.3,-0.3) to[out=up,in=down] (0,0.3);
    \end{tikzpicture}
    \ ,\qquad
    \begin{tikzpicture}[centerzero]
        \draw (-0.2,-0.3) -- (-0.2,-0.1) arc(180:0:0.2) -- (0.2,-0.3);
        \draw (-0.3,0.3) \braiddown (0,-0.3);
    \end{tikzpicture}
    =
    \begin{tikzpicture}[centerzero]
        \draw (-0.2,-0.3) -- (-0.2,-0.1) arc(180:0:0.2) -- (0.2,-0.3);
        \draw (0.3,0.3) \braiddown (0,-0.3);
    \end{tikzpicture}
    \ .
\end{equation}
More precisely, the cups and caps equip $\OB_\kk(A)$ with the structure of a \emph{strict pivotal} supercategory; see \cite[(5.16)]{BSW-foundations}.  It follows from the definition of the tokens on downward strands that
\[
    \begin{tikzpicture}[centerzero]
        \draw[<-] (0,-0.35) -- (0,0.35);
        \token{0,-0.15}{east}{b};
        \token{0,0.15}{east}{a};
    \end{tikzpicture}
    = (-1)^{\bar{a}\bar{b}}\
    \begin{tikzpicture}[centerzero]
        \draw[<-] (0,-0.35) -- (0,0.35);
        \token{0,0}{west}{ba};
    \end{tikzpicture}
    \ .
\]
We also have
\begin{equation} \label{flippy}
    \rightcup
    =
    \begin{tikzpicture}[anchorbase]
        \draw[->] (-0.15,0.3) to[out=-45,in=90] (0.15,0) arc(360:180:0.15) to[out=90,in=225] (0.15,0.3);
    \end{tikzpicture}
    ,\qquad
    \leftcup
    =
    \begin{tikzpicture}[anchorbase]
        \draw[<-] (-0.15,0.3) to[out=-45,in=90] (0.15,0) arc(360:180:0.15) to[out=90,in=225] (0.15,0.3);
    \end{tikzpicture}
    ,\qquad
    \rightcap
    =
    \begin{tikzpicture}[anchorbase]
        \draw[<-] (-0.15,-0.3) to[out=45,in=-90] (0.15,0) arc(0:180:0.15) to[out=-90,in=135] (0.15,-0.3);
    \end{tikzpicture}
    ,\qquad \text{and} \qquad
    \leftcap
    =
    \begin{tikzpicture}[anchorbase]
        \draw[<-] (-0.15,-0.3) to[out=45,in=-90] (0.15,0) arc(0:180:0.15) to[out=-90,in=135] (0.15,-0.3);
    \end{tikzpicture}
    \ .
\end{equation}

\subsection{The basis theorem\label{subsec:OBbasis}}

We now describe bases for the morphism spaces of $\OB_\kk(A)$.  Let $X = X_1 \otimes \dotsb \otimes X_r$ and $Y = Y_1 \otimes \dotsb \otimes Y_s$ be objects of $\OB_\kk(A)$ for $X_t, Y_t \in \{\upobj, \downobj\}$.  An \emph{$(X,Y)$-matching} is a bijection between the sets
\begin{equation} \label{firefly}
    \{t : X_t = \upobj \} \sqcup \{t : Y_t = \downobj \}
    \quad \text{and} \quad
    \{t : X_t = \downobj \} \sqcup \{t : Y_t = \upobj \}.
\end{equation}
A \emph{reduced lift} of an $(X,Y)$-matching is a string diagram representing a morphism $X \to Y$ such that
\begin{itemize}
    \item the endpoints of each string are points which correspond under the given matching;
    \item there are no floating bubbles (i.e.\ strings with no endpoints) and no tokens on any string;
    \item there are no self-intersections of strings and no two strings cross each other more than once.
\end{itemize}
Fix a set $\obD(X,Y)$ consisting of a choice of reduced lift for each $(X,Y)$-matching.  Then let $\obD^\bullet(X,Y)$ denote the set of all morphisms that can be obtained from the elements of $\obD(X,Y)$ by adding one token to each string according to the following convention.
\begin{convention} \label{jiggy}
    Tokens are placed such that:
    \begin{itemize}
        \item each token is labelled by an element of $\bB_A$;
        \item if a string has endpoints at the top and bottom of the diagram, then its token appears near the bottom of the string (below all crossings);
        \item if a string has both endpoints at the top of the diagram, then its token appears near the left endpoint (above all crossings);
        \item if a string has both endpoints at the bottom of the diagram, then its token appears near the right endpoint (below all crossings);
        \item all tokens near top endpoints are at the same height, all tokens near bottom endpoints are at the same height, and the tokens near top endpoints are above the tokens near bottom endpoints.
    \end{itemize}
\end{convention}
For example, for $X = \downobj \otimes \upobj \otimes \downobj \otimes \downobj \otimes \upobj$ and $Y = \downobj \otimes \downobj \otimes \downobj \otimes \upobj \otimes \downobj \otimes \upobj \otimes \upobj$,
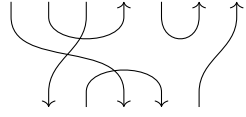
\[
    \begin{tikzpicture}[anchorbase]
        \draw[->] (0.5,-0.2) -- (0.5,0) to[out=up,in=up] (1.5,0) -- (1.5,-0.2);
        \draw[<-] (1,-0.2) -- (1,0) \braidup (-0.5,1) -- (-0.5,1.2);
        \draw[<-] (0,-0.2) -- (0,0) \braidup (0.5,1) -- (0.5,1.2);
        \draw[->] (0,1.2) -- (0,1) to[out=down,in=down] (1,1) -- (1,1.2);
        \draw[->] (2,-0.2) -- (2,0) \braidup (2.5,1) -- (2.5,1.2);
        \draw[->] (1.5,1.2) -- (1.5,1) to[out=down,in=down,looseness=2] (2,1) -- (2,1.2);
    \end{tikzpicture}
\]
is a possible element of $\obD(X,Y)$ and
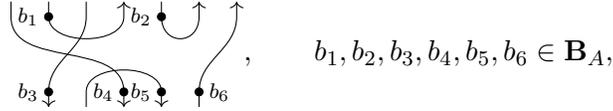
\[
    \begin{tikzpicture}[anchorbase]
        \draw[->] (0.5,-0.2) -- (0.5,0) to[out=up,in=up] (1.5,0) -- (1.5,-0.2);
        \draw[<-] (1,-0.2) -- (1,0) \braidup (-0.5,1) -- (-0.5,1.2);
        \draw[<-] (0,-0.2) -- (0,0) \braidup (0.5,1) -- (0.5,1.2);
        \draw[->] (0,1.2) -- (0,1) to[out=down,in=down] (1,1) -- (1,1.2);
        \draw[->] (2,-0.2) -- (2,0) \braidup (2.5,1) -- (2.5,1.2);
        \draw[->] (1.5,1.2) -- (1.5,1) to[out=down,in=down,looseness=2] (2,1) -- (2,1.2);
        \token{0,1}{east}{b_1};
        \token{1.5,1}{east}{b_2};
        \token{0,0}{east}{b_3};
        \token{1,0}{east}{b_4};
        \token{1.5,0}{east}{b_5};
        \token{2,0}{west}{b_6};
    \end{tikzpicture}
    ,\qquad b_1,b_2,b_3,b_4,b_5,b_6 \in \bB_A,
\]
are the corresponding elements of $\obD^\bullet(X,Y)$.

While we expect the following theorem to hold for an arbitrary associative superalgebra $A$, our proof assumes that $A$ is a Frobenius superalgebra.  As explained in \cref{amongus}, this assumption holds whenever $A$ is a real or complex division superalgebra.

\begin{theo} \label{Obasisthm}
    Let $d \in \kk$.  For $X,Y \in \OB_\kk(A)$, the morphism space $\Hom_{\OB_\kk(A;d)}(X,Y)$ is a free $\kk$-supermodule with basis $\obD^\bullet(X,Y)$.
\end{theo}

\begin{proof}
    The proof that $\obD^\bullet(X,Y)$ is a spanning set for $\Hom_{\OB_\kk(A;d)}(X,Y)$ is standard; see the argument in the proof of \cite[Th.~7.2]{BSW-foundations}, ignoring the dots.  To prove linear independence, we consider the \emph{affine} Frobenius Brauer supercategory $\AOB(A)$ defined in \cite[Def.~4.3]{MS21}.  This is the central charge $k=0$ case of the Frobenius Heisenberg supercategory introduced in \cite{Sav19} and further studied in \cite{BSW-foundations}.  Since $\AOB(A)$ is obtained from $\OB_\kk(A;d)$ by adjoining an additional generator and imposing some extra relations, it follows immediately that there is a functor from $\OB_\kk(A;d)$ to $\AOB(A)$.  It follows from the basis theorem \cite[Th.~4.7]{MS21} for $\AOB(A)$, which is a special case of the basis theorem \cite[Th.~7.2]{BSW-foundations} for Frobenius Heisenberg supercategories, that the elements of the set $\obD^\bullet(X,Y)$ are sent to elements of a basis for $\AOB(A)$.  It follows that the elements of $\obD^\bullet(X,Y)$ are linearly independent.

    In \cite[Th.~7.2]{BSW-foundations}, the basis elements carry tokens near the terminus of each strand, which differs from the placement of tokens in the elements of the $\obD^\bullet(X,Y)$.  However, it follows from the relations in $\OB_\kk(A)$ that this difference in placement changes the corresponding diagrams by at most a sign.  In addition, \cite[Th.~7.2]{BSW-foundations} assumes the Frobenius superalgebra is supersymmetric.  However, the same proof given there works without this assumption, using the defining property of the Nakayama automorphism wherever supersymmetry is needed, and tracking these applications throughout the calculations.  (See, for example, \cite{Sav19}, which works in this generality.)
\end{proof}

\subsection{Complexifications\label{subsec:OBcomplexification}}

Our proof of fullness of the oriented incarnation superfunctor when $\kk=\R$ and $A$ is a real division superalgebra (\cref{OBrealfull}) will involve the complexification of $\OB_\kk(A)$.  In this subsection we state some results about this complexification that we will need.

\begin{prop} \label{crystal}
    For any superalgebra $A$ over $\kk=\R$, and $d \in \R$, there are isomorphisms of monoidal supercategories
    \[
        \sR \colon \OB_\R(A)^\C \xrightarrow{\cong} \OB_\C(A^\C)
        \qquad \text{and} \qquad
        \sR \colon \OB_\R(A;d)^\C \xrightarrow{\cong} \OB_\C(A^\C;d),
    \]
    given on objects by $\upobj \mapsto \upobj$, $\downobj \mapsto \downobj$ and on morphisms by
    \[
        \upcross \mapsto \upcross,\qquad
        \leftcap \mapsto \leftcap,\qquad
        \leftcup \mapsto \leftcup,\qquad
        \rightcap \mapsto \rightcap,\qquad
        \rightcup \mapsto \rightcup,\qquad
        \uptokstrand \mapsto \uptokstrand[a \otimes 1],\quad a \in A.
    \]
\end{prop}

\begin{proof}
    It is clear that the superfunctor $\sR$ is well-defined.  The inverse functor is given on morphisms by
    \[
        \upcross \mapsto \upcross,\quad
        \leftcap \mapsto \leftcap,\quad
        \leftcup \mapsto \leftcup,\quad
        \rightcap \mapsto \rightcap,\quad
        \rightcup \mapsto \rightcup,\quad
        \uptokstrand[a \otimes z] \mapsto \left( \uptokstrand[a] \right) \otimes z,\quad a \in A,\ z \in \C.
        \qedhere
    \]
\end{proof}

By \cref{crystal,sail}, the complexifications $\OB_\R(\DD)^\C$ and $\OB_\R(\DD;d)^\C$, where $\DD$ is a central real division superalgebra, are related to $\OB_\C(R)$, where $R$ is a supermatrix superalgebra over a complex division superalgebra $A$.  The following result, which is formulated more generally, relates these to $\OB_\C(A)$.  Recall, from \cref{sec:monsupcat}, the superadditive envelope $\Add(\cC_\pi)$ of a supercategory $\cC$.  We write morphisms in superadditive envelopes as sums of their components.

\begin{prop} \label{sunrise}
    For any superalgebra $A$ and $r,s \in \N$, $r+s \ge 1$, there is a unique monoidal superfunctor
    \[
        \sM \colon \OB_\kk(\Mat_{r|s}(A)) \to \Add(\OB_\kk(A)_\pi)
    \]
    given on objects by $\upobj \mapsto \upobj^{\oplus r} \oplus \Pi \upobj^{\oplus s}$, $\downobj \mapsto \downobj^{\oplus r} \oplus \Pi \downobj^{\oplus s}$, and on morphisms by
    \begin{gather*}
        \upcross \mapsto \sum_{t,u=1}^{r+s} (-1)^{p(t)p(u)}
        \begin{tikzpicture}[centerzero]
            \draw[->] (-0.2,-0.2) node[anchor=north] {\strandlabel{t}} -- (0.2,0.2) node[anchor=south] {\strandlabel{t}};
            \draw[->] (0.2,-0.2) node[anchor=north] {\strandlabel{u}} -- (-0.2,0.2) node[anchor=south] {\strandlabel{u}};
            \shiftline{-0.3,0.2}{0.3,0.2}{p(t)+p(u)};
            \shiftline{-0.3,-0.2}{0.3,-0.2}{p(t)+p(u)};
        \end{tikzpicture}
        ,\qquad
        \uptokstrand[E_{tu}a] \mapsto
        \begin{tikzpicture}[centerzero]
            \draw[->] (0,-0.2) node[anchor=north] {\strandlabel{u}} -- (0,0.2) node[anchor=south] {\strandlabel{t}};
            \shiftline{-0.1,0.2}{0.1,0.2}{p(t)};
            \shiftline{-0.1,-0.2}{0.1,-0.2}{p(u)};
            \token{0,0}{east}{a};
        \end{tikzpicture}
        ,
        \\
        \leftcup \mapsto \sum_{t=1}^{r+s}
        \begin{tikzpicture}[centerzero]
            \draw[<-] (-0.15,0.15) node[anchor=south] {\strandlabel{t}} -- (-0.15,0) arc(180:360:0.15) -- (0.15,0.15) node[anchor=south] {\strandlabel{t}};
            \shiftline{-0.25,0.15}{0.25,0.15}{0};
            \shiftline{-0.2,-0.25}{0.25,-0.25}{0};
        \end{tikzpicture}
        ,\quad
        \leftcap \mapsto \sum_{t=1}^{r+s}
        \begin{tikzpicture}[centerzero]
            \draw[<-] (-0.15,-0.15) node[anchor=north] {\strandlabel{t}} -- (-0.15,0) arc(180:0:0.15) -- (0.15,-0.15) node[anchor=north] {\strandlabel{t}};
            \shiftline{-0.25,-0.15}{0.25,-0.15}{0};
            \shiftline{-0.25,0.25}{0.25,0.25}{0};
        \end{tikzpicture}
        ,\quad
        \rightcup \mapsto \sum_{t=1}^{r+s} (-1)^{p(t)}
        \begin{tikzpicture}[centerzero]
            \draw[->] (-0.15,0.15) node[anchor=south] {\strandlabel{t}} -- (-0.15,0) arc(180:360:0.15) -- (0.15,0.15) node[anchor=south] {\strandlabel{t}};
            \shiftline{-0.25,0.15}{0.25,0.15}{0};
            \shiftline{-0.2,-0.25}{0.25,-0.25}{0};
        \end{tikzpicture}
        ,\quad
        \rightcap \mapsto \sum_{t=1}^{r+s} (-1)^{p(t)}
        \begin{tikzpicture}[centerzero]
            \draw[->] (-0.15,-0.15) node[anchor=north] {\strandlabel{t}} -- (-0.15,0) arc(180:0:0.15) -- (0.15,-0.15) node[anchor=north] {\strandlabel{t}};
            \shiftline{-0.25,-0.15}{0.25,-0.15}{0};
            \shiftline{-0.25,0.25}{0.25,0.25}{0};
        \end{tikzpicture}
        ,
    \end{gather*}
    where
    \begin{equation} \label{pdef}
        p(t) =
        \begin{cases}
            0 & \text{if } 1 \le t \le r, \\
            1 & \text{if } r < t \le r+s.
        \end{cases}
    \end{equation}
    This superfunctor is full and faithful.  For $d \in \kk$, it induces equivalences of monoidal supercategories
    \begin{align*}
        \Add \left( \OB_\kk(\Mat_{r|s}(A))_\pi \right) &\xrightarrow{\cong} \Add(\OB_\kk(A)_\pi),
        \\
        \Add \left( \OB_\kk(\Mat_{r|s}(A);d)_\pi \right) &\xrightarrow{\cong} \Add(\OB_\kk(A;(r-s)d)_\pi).
    \end{align*}
\end{prop}

\begin{proof}
    We first consider the non-specialized supercategories.  To prove that $\sM$ is well defined, we must show that it respects the relations of \cref{OBC}.  These are all straightforward verifications, which we leave to the reader.  (See the proof of \cref{bulb} for the details of a similar, but slightly less straightforward, verification.)

    Next we prove that $\sM$ is full and faithful.  Suppose $X_1,\dotsc,X_v,Y_1,\dotsc,Y_w \in \{\upobj,\downobj\}$, and let $X = X_1 \otimes \dotsb \otimes X_v$, $Y = Y_1 \otimes \dotsb \otimes Y_w$.  Then $\sM$ induces a $\kk$-linear map
    \begin{equation} \label{water}
        \Hom_{\OB_\kk(\Mat_{r|s}(A))}(X,Y)
        \to \bigoplus_{t_1,\dotsc,t_v,u_1,\dotsc,u_w=1}^{r+s} \Hom_{\OB_\kk(A)} \left( \Pi^{p(t_1)+\dotsb+p(t_v)} X, \Pi^{p(u_1)+\dotsb+p(u_w)} Y \right).
    \end{equation}
    It suffices to assume that $v+w$ is even and
    \[
        \# \{a : X_a = \upobj \} + \# \{a : Y_a = \downobj \}
        = \frac{v+w}{2}
        = \# \{a : X_a = \downobj \} + \# \{a : Y_a = \upobj \},
    \]
    otherwise both the domain and image of \cref{water} have dimension zero.  (Here, $\# S$ denotes the cardinality of a set $S$.)  By \cref{Obasisthm},
    \[
        \dim_\kk \Hom_{\OB_\kk(\Mat_{r|s}(A))}(X,Y)
        = \left( \tfrac{v+w}{2} \right)! \left( (r+s)^2 \dim_\kk A \right)^{(v+w)/2}.
    \]
    We use here the fact there the number of $(X,Y)$-matchings is $(\tfrac{v+w}{2})!$, and that $\dim_\kk \Mat_{r|s}(A) = (r+s)^2 \dim_\kk A$.  On the other hand, \cref{Obasisthm} implies that the codomain of the map \cref{water} has the same dimension.  Thus, it suffices to prove that the map \cref{water} is surjective.  This follows from the fact that any string diagram in the summand
    \[
        \Hom_{\OB_\kk(A)} \left( \Pi^{p(t_1)+\dotsb+p(t_v)} X, \Pi^{p(u_1)+\dotsb+p(u_w)} Y \right)
    \]
    is the image under \cref{water} (up to a sign) of the same diagram with appropriate tokens $\uptokstrand[E_{tu}]$ placed near the endpoints of strands.

    Finally, we show that $\sM$ is essentially surjective.  This follows from the fact that the generating objects $\upobj$ and $\downobj$ of $\OB_\kk(A)$ are the images of $(\upobj,\uptokstrand[E_{11}])$ and $(\downobj,\downtokstrand[E_{11}])$, respectively, if $m \ge 1$, and the images of $(\Pi \upobj,\uptokstrand[E_{11}])$ and $(\Pi \downobj,\downtokstrand[E_{11}])$, respectively, if $m=0$.

    It remains to prove the statement about the specialized supercategories.  For $1 \le t,u \le r+s$ and $a \in A$, we have
    \begin{multline*}
        \sM \left( \ccbubble{E_{tu} a} \right)
        = \sM(\leftcap) \circ \sM(\downstrand \otimes \uptokstrand[E_{tu} a]) \circ(\rightcup)
        \overset{\cref{slush}}{=} \delta_{tu} (-1)^{p(t)+p(t)\bar{a}}\,
        \begin{tikzpicture}[centerzero]
            \ccbub{0,0};
            \token{0.2,0}{west}{a};
            \shiftline{-0.3,0.3}{0.3,0.3}{0};
            \shiftline{-0.3,-0.3}{0.3,-0.3}{0};
        \end{tikzpicture}
        \\
        = \delta_{tu} (-1)^{p(t)} (r-s) d \str_A^\kk(a)
        = (r-s) d \str_A^\kk \circ \str (E_{tu}a)
        \overset{\cref{blink}}{=} d \str_{\Mat_{r|s}(A)}^\kk (E_{tu} a),
    \end{multline*}
    where, in the third equality, we used the fact that $\str_A^\kk(a) = 0$ unless $\bar{a}=0$.
\end{proof}

\section{The oriented incarnation superfunctor\label{sec:Oinc}}

In this section we introduce the main application of the supercategory $\OB_\kk(A)$ to the representation theory of Lie superalgebras.  We begin by defining a very general \emph{oriented incarnation superfunctor}.  We then turn our attention to the special cases where $\kk \in \{\R,\C\}$ and $A$ is a division superalgebra over $\kk$.  When $\kk=\C$, fullness of the incarnation functor follows from known results.  When $\kk=\R$, we give a proof of fullness using the complexification of the supercategories involved.

\subsection{Definition of the superfunctor}

Throughout this subsection, $\kk$ denotes an arbitrary field.  Recall the maps $\flip$, $\ev$, $\coev$, and $\rho_a$ from \cref{subsec:supermodules}.  The following result is the main motivation for the definition of the supercategory $\OB_\kk(A)$.

\begin{theo} \label{tiger}
    Suppose that $A$ is an associative superalgebra, $\fg$ is a Lie superalgebra, and $V$ is a $(\fg,A)$-superbimodule.  There exists a unique monoidal superfunctor, which we call the \emph{oriented incarnation superfunctor},
    \[
        \sG = \sG_V \colon \OB_\kk(A^\op) \to \fg\smod_\kk.
    \]
    such that $\sG(\upobj) = V$, $\sG(\downobj) = V^*$, and
    \begin{equation} \label{zelda}
        \sG(\upcross) = \flip,\qquad
        \sG(\leftcap) = \ev,\qquad
        \sG(\uptokstrand[a^\op]) = \rho_a,\quad a \in A.
    \end{equation}
    This superfunctor also satisfies the following:
    \begin{gather} \label{make}
        \sG(\leftcup) = \coev,\qquad
        \sG(\rightcap) = \ev \circ \flip,\qquad
        \sG(\rightcup) = \flip \circ \coev,\qquad
        \\ \label{believe}
        \sG \left( \bubble{a} \right) = \str_V(a),
        \qquad a \in A.
    \end{gather}
    If $V = A^{m|n}$ and $\fg = \fgl(m|n,A)$ for some $m,n \in \N$, then $\sG_V$ induces a monoidal superfunctor
    \[
        \sG_{m|n} \colon \OB_\kk(A^\op;m-n) \to \fgl(m|n,A)\smod_\kk.
    \]
\end{theo}

\begin{proof}
    We first show that \cref{zelda,make} indeed yield a superfunctor $\sG$.   We must show that it respects the relations \cref{toklin,wreath,inversion,leftadj}.  The first two relations in \cref{toklin} are straightforward.  For the third relation in \cref{toklin}, we have
    \[
        \sG
        \left(
            \begin{tikzpicture}[centerzero]
                \draw[->] (0,-0.35) -- (0,0.35);
                \token{0,-0.15}{east}{b^\op};
                \token{0,0.15}{east}{a^\op};
            \end{tikzpicture}
        \right) (v)
        = (-1)^{(\bar{a}+\bar{b})\bar{v} + \bar{a}\bar{b}} vba
        = (-1)^{\bar{a}\bar{b}} \sG ( \uptokstrand[(ba)^\op] ) (v)
        = \sG (\uptokstrand[a^\op b^\op]) (v).
    \]

    Next, we show that
    \[
        \sG(\leftcross) = \flip_{V^*,V},\qquad
        \sG(\rightcross) = \flip_{V,V^*},\qquad
        \sG(\downcross) = \flip_{V^*,V^*}.
    \]
    Using the definition \cref{windmill} of the left crossing, the map
    \[
        \sG(\leftcross)
        = \sG
        \left(
            \begin{tikzpicture}[centerzero]
                \draw[->] (-0.2,-0.3) \braidup (0.2,0.3);
                \draw[->] (0.4,0.3) -- (0.4,0.1) to[out=down,in=right] (0.2,-0.2) to[out=left,in=right] (-0.2,0.2) to[out=left,in=up] (-0.4,-0.1) -- (-0.4,-0.3);
            \end{tikzpicture}
        \right)
        \colon V^* \otimes V \to V \otimes V^*
    \]
    is given by
    \begin{multline*}
        f \otimes v \mapsto \sum_{v \in \bB_V^\kk} f \otimes v \otimes w \otimes w^*
        \mapsto \sum_{v \in \bB_V^\kk} (-1)^{\bar{v}\bar{w}} f \otimes w \otimes v \otimes w^*
        \\
        \mapsto v \otimes \sum_{v \in \bB_V^\kk} (-1)^{\bar{v}\bar{w}} f(w) w^*
        = (-1)^{\bar{f}\bar{v}} v \otimes f,
    \end{multline*}
    where we use the fact that $\bar{w} = \bar{f}$ whenever $f(w) \ne 0$.  The proofs for $\rightcross$ and $\downcross$ are analogous.
    \details{
        Using the definition \cref{windmill} of the right crossing, the map
        \[
            \sG( \rightcross )
            = \sG
            \left(
                \begin{tikzpicture}[centerzero]
                    \draw[->] (0.2,-0.3) \braidup (-0.2,0.3);
                    \draw[->] (-0.4,0.3) -- (-0.4,0.1) to[out=down,in=left] (-0.2,-0.2) to[out=right,in=left] (0.2,0.2) to[out=right,in=up] (0.4,-0.1) -- (0.4,-0.3);
                \end{tikzpicture}
            \right)
            \colon V \otimes V^* \to V^* \to V
        \]
        is given by
        \begin{multline*}
            v \otimes f
            \mapsto \sum_{w \in \bB_V^\kk} (-1)^{\bar{w}} w^* \otimes w \otimes v \otimes f
            \mapsto \sum_{w \in \bB_V^\kk} (-1)^{\bar{w}+\bar{v}\bar{w}} w^* \otimes v \otimes w \otimes f \\
            \mapsto \sum_{w \in \bB_V^\kk} (-1)^{\bar{v}\bar{f}} f(w) w^* \otimes v
            = (-1)^{\bar{v}\bar{f}} f \otimes v.
        \end{multline*}

        Then, using the definition \cref{stake} of the down crossing, the map
        \[
            \sG( \downcross )
            = \sG
            \left(
                \begin{tikzpicture}[centerzero]
                    \draw[<-] (-0.2,-0.3) \braidup (0.2,0.3);
                    \draw[->] (0.4,0.3) -- (0.4,0.1) to[out=down,in=right] (0.2,-0.2) to[out=left,in=right] (-0.2,0.2) to[out=left,in=up] (-0.4,-0.1) -- (-0.4,-0.3);
                \end{tikzpicture}
            \right)
            \colon V^* \otimes V^* \to V^* \otimes V^*
        \]
        is given by
        \begin{multline*}
            f \otimes g
            \mapsto \sum_{v \in \bB_V^\kk} f \otimes g \otimes v \otimes v^*
            \mapsto \sum_{v \in \bB_V^\kk} (-1)^{\bar{g}\bar{v}} f \otimes v \otimes g \otimes v^* \\
            \mapsto \sum_{v \in \bB_V^\kk} (-1)^{\bar{g}\bar{v}} f(v) g \otimes v^*
            = (-1)^{\bar{f}\bar{g}} g \otimes f,
        \end{multline*}
        where, in the second-to-last equality, we used the fact that $f(v)=0$ unless $\bar{v}=\bar{f}$.
    }
    The relations \cref{wreath} and the first two relations in \cref{inversion} are then straightforward to verify.

    For the fourth equality in \cref{inversion}, we have
    \[
        \sG
        \left(
            \begin{tikzpicture}[centerzero]
                \draw[->] (0,-0.4) to[out=up,in=180] (0.25,0.15) to[out=0,in=up] (0.4,0) to[out=down,in=0] (0.25,-0.15) to[out=180,in=down] (0,0.4);
            \end{tikzpicture}
        \right)
        \colon
        v \mapsto \sum_{w \in \bB_V^\kk} v \otimes w \otimes w^*
        \mapsto \sum_{w \in \bB_V^\kk} (-1)^{\bar{v}\bar{w}} w \otimes v \otimes w^*
        \mapsto \sum_{w \in \bB_V^\kk} w^*(v) w
        = v
        = \sG
        \left(
            \begin{tikzpicture}[centerzero]
                \draw[->] (0,-0.4) -- (0,0.4);
            \end{tikzpicture}
        \right)
        (v).
    \]
    The verification of the third equality in \cref{inversion} is analogous.
    \details{
        We have
        \begin{multline*}
            \sG
            \left(
                \begin{tikzpicture}[centerzero]
                    \draw[->] (0,-0.4) to[out=up,in=0] (-0.25,0.15) to[out=180,in=up] (-0.4,0) to[out=down,in=180] (-0.25,-0.15) to[out=0,in=down] (0,0.4);
                \end{tikzpicture}
            \right)
            \colon
            v \mapsto \sum_{w \in \bB_V^\kk} (-1)^{\bar{w}} w^* \otimes w \otimes v
            \mapsto \sum_{w \in \bB_V^\kk} (-1)^{\bar{w} + \bar{v}\bar{w}} w^* \otimes v \otimes w
            \\
            \mapsto \sum_{w \in \bB_V^\kk} (-1)^{\bar{w} + \bar{v}\bar{w}} w^*(v) w
            = v
            = \sG
            \left(
                \begin{tikzpicture}[centerzero]
                    \draw[->] (0,-0.4) -- (0,0.4);
                \end{tikzpicture}
            \right)
            (v),
        \end{multline*}
        where we used the fact that $\bar{v} = \bar{w}$ whenever $w^*(v) \ne 0$.
    }
    Verification of the relations \cref{leftadj} is straightforward.

    To show \cref{believe}, we compute that
    \[
        \sG \left( \ccbubble{a^\op} \right) \colon \kk \mapsto \kk
    \]
    is the map
    \[
        1 \mapsto \sum_{v \in \bB_V^\kk} (-1)^{\bar{v}} v^* \otimes v
        \mapsto \sum_{v \in \bB_V^\kk} (-1)^{\bar{v}} v^* \otimes v a
        \mapsto \sum_{v \in \bB_V^\kk} (-1)^{\bar{v}} v^*(va)
        \overset{\cref{crazy}}{=} \str_V(a).
    \]
    The fact that $\sG$ factors through $\OB_\kk(A^\op;m-n)$ when $V = A^{m|n}$ then follows from \cref{break}.

    It remains to prove that, for any functor as in the first sentence of the theorem, we have \cref{make}. Suppose that
    \[
        \sG(\leftcup) \colon 1
        \mapsto \sum_{u,v \in \bB_V} a_{uv} u \otimes v^*,\qquad
        a_{uv} \in \kk.
    \]
    Then, for all $v \in \bB_V$, we have
    \[
        v =
        \sG
        \left(\
            \begin{tikzpicture}[centerzero]
                \draw[->] (0,-0.4) -- (0,0.4);
            \end{tikzpicture}
        \ \right)
        (v)
        =
        \sG
        \left(
            \begin{tikzpicture}[centerzero]
                \draw[<-] (-0.3,0.4) -- (-0.3,0) arc(180:360:0.15) arc(180:0:0.15) -- (0.3,-0.4);
            \end{tikzpicture}
        \right)
        (v)
        \xmapsto{\sG \left( \leftcup\, \otimes\ \upstrand\ \right)}
        \sum_{u,w \in \bB_V} a_{uw} u \otimes w^* \otimes v
        \xmapsto{1_V \otimes \ev}
        \sum_{u \in \bB_V} a_{uv}  u.
    \]
    It follows that $a_{uv} = \delta_{uv}$ for all $u,v \in \bB_V$, and so $\sG(\leftcup) = \coev$.  The other two equalities in \cref{make} then follow from \cref{flippy}.
\end{proof}

\begin{rem}
    \Cref{tiger} holds in greater generality.  If $\cC$ is any rigid symmetric monoidal supercategory (e.g.\ the category of supermodules over a triangular Hopf superalgebra) with an object $V$ that has the structure of a right $A$-supermodule, then \cref{zelda} defines a unique monoidal superfunctor $\sG \colon \OB_\kk(A^\op) \to \cC$, and \cref{make,believe} hold.  The proof of this more general statement is exactly the same as the proof of \cref{tiger}.  We chose to state \cref{tiger} with the choice $\cC = \fg\smod_\kk$ since that will be our main application.
\end{rem}

\begin{rem}
    When $A$ is a Frobenius superalgebra, $\sG$ is essentially the functor of \cite[Th.~5.1]{MS21}.  The paper \cite{MS21} works with \emph{right} $\fgl(m|n,A)$-supermodules and \emph{left} $A$-supermodules.  In \cref{tiger}, we have translated to the setting of \emph{right} $A$-supermodules by considering $\OB_\kk(A^\op)$ instead of $\OB_\kk(A)$ and to the setting of \emph{left} $\fgl(m|n,A)$-supermodules using the involution $X \mapsto -X$ of $\fgl(m|n,A)$.  The natural module $V$ is denoted by $V_+$ in \cite{MS21}.  Furthermore, in \cite{MS21}, the dual module $V^*$ is replaced by a supermodule $V_-$, together with a nondegenerate bilinear form $V_- \otimes V_+ \to \kk$.  This form identifies $V_-$ with $V^*$.
\end{rem}

\begin{rem} \label{dubpi}
    By the universal property of $\Pi$-envelopes mentioned in \cref{sec:monsupcat}, we have an induced monoidal superfunctor
    \[
        \sG \colon \OB_\kk(A^\op)_\pi \to \fgl(V_A)\smod_\kk.
    \]
    The coherence maps of this monoidal superfunctor involve some signs.  For example, we have the coherence map
    \begin{gather*}
        \sG(\upobj) \otimes \sG(\upobj)
        = \Pi V \otimes \Pi V
        \xrightarrow{\cong} \Pi^2 V \otimes V
        \xrightarrow[\cref{quirk}]{\cong} V \otimes V
        = \sG(\upobj \otimes \upobj)
        = V \otimes V,
        \\
        \pi v \otimes \pi w
        \mapsto (-1)^{\bar{v}} \pi^2 v \otimes w
        \mapsto -(-1)^{\bar{v}} v \otimes w.
    \end{gather*}
\end{rem}

The following result will be useful in later computations.

\begin{lem}
    We have
    \begin{equation} \label{batty}
        \sG( \downtokstrand[a^\op] ) \colon V^* \to V^*,\qquad
        f \mapsto af,\qquad
        f \in V^*,\ a \in A,
    \end{equation}
    where $af$ is defined as in \cref{dualaction}.
\end{lem}

\begin{proof}
    We have
    \begin{multline*}
        \sG( \downtokstrand[a^\op] )
        = \sG
        \left(
            \begin{tikzpicture}[centerzero]
                \draw[->] (0.6,0.4) -- (0.6,-0.05) arc(360:180:0.3) -- (0,0.05) arc(0:180:0.15) -- (-0.3,-0.4);
                \token{0,0}{west}{a^\op};
            \end{tikzpicture}
        \right)
        \colon f
        \mapsto \sum_{v \in \bB_V^\kk} f \otimes v \otimes v^*
        \\
        \mapsto \sum_{v \in \bB_V^\kk} (-1)^{\bar{a}(\bar{f}+\bar{v})} f \otimes va \otimes v^*
        \mapsto \sum_{v \in \bB_V^\kk} (-1)^{\bar{a}(\bar{f}+\bar{v})} f(va)v^*
        = \sum_{v \in \bB_V^\kk} (af)(v)v^*
        = af.
        \qedhere
    \end{multline*}
\end{proof}

\subsection{Fullness over the complex numbers}

The remainder of this section is dedicated to proving that the oriented incarnation superfunctor of \cref{tiger} is full in certain important special cases.  In this subsection we consider the case where $\kk=\C$ and $A$ is a matrix superalgebra over a complex division superalgebra.  We begin with the case where $A$ is complex division superalgebra, which follows from results in the literature.

\begin{prop} \label{pyramid}
    If $\kk = \C$ and $A$ is a complex division superalgebra, then the oriented incarnation functor $\sG_{m|n}$ of \cref{tiger} is full for all $m,n \in \N$.
\end{prop}

\begin{proof}
    As explained in \cref{sec:divalg}, the only complex division superalgebras are $\C$ and $\Cl(\C)$.  When $A = \C$, the supercategory $\OB_\C(\C;m-n)$ is the usual oriented Brauer category, and the result was proved in \cite[\S8.3]{CW12}.  (Closely related results were obtained in \cite[Th.~7.8]{BS12} and \cite[Th.~3.5]{LSM02}.)  On the other hand, $\OB_\C(\Cl(\C),0)$ is the oriented Brauer--Clifford supercategory.  (Recall that, by \cref{tuna}, we may assume the specialization parameter is zero.)  In this case, fullness was proved in \cite[Th.~4.1]{BCK19}.
\end{proof}

In the remainder of this subsection, our goal is to show that the oriented incarnation superfunctor $\sG_{m|n}$ is full when $A$ is the superalgebra of supermatrices over a complex division superalgebra.  This will be key in our proof that it is also full when $A$ is a real division superalgebra (\cref{OBrealfull}).  We begin with a result that holds in a more general setup.

Let $A$ be a superalgebra.  Fix $m,n,r,s \in \N$ with $m+n,r+s \ge 1$.  In what follows, we will identify
\[
    \Mat_{m|n}(\Mat_{r,s}(A))
    \qquad \text{and} \qquad
    \Mat_{(mr+ns)|(ms+nr)}(A)
\]
in the natural way.  This induces a natural identification of
\[
    \fgl(m|n,\Mat_{r,s}(A)) \qquad \text{and} \qquad
    \fgl((mr+ns|ms+nr),A),
\]
and we denote this Lie superalgebra by $\fg$.  Let
\[
    W = \Mat_{r|s}(A)^{m|n}
    \qquad \text{and} \qquad
    V = A^{(mr+ns|ms+nr)}.
\]
We have an isomorphism of $(\fg,A)$-superbimodules
\begin{equation} \label{obscure1}
    W
    \xrightarrow{\cong}
    V^{\oplus r} \oplus \Pi V^{\oplus s},\qquad
    v \mapsto \left( (-1)^{p(t) \overline{v_t}} \pi^{p(t)} v_t \right)_{t=1}^{r+s},
\end{equation}
where $v_t \in V$ is the $t$-th column of $v$, and $p(t)$ is defined as in \cref{pdef}.
\details{
    The sign of $(-1)^{p(t)\overline{v_t}}$ appearing in \cref{obscure1} arises from the fact that $\Pi V$ has $\fg$-action $x \cdot v = (-1)^{\bar{x}} v$, whereas the $\fg$-action on columns $r+1,\dotsc,r+s$ of $\Mat_{r|s}(A)^{m|n}$ does not have this sign of $(-1)^{\bar{x}}$.
}
Similarly, we have an isomorphism of $(\fg,A)$-superbimodules
\begin{equation} \label{obscure2}
    W^* \xrightarrow{\cong} (V^*)^{\oplus r} \oplus (\Pi V^*)^{\oplus s},
    \qquad
    f \mapsto (\pi^{p(t)} f_t)_{t=1}^{r+s},
\end{equation}
where $f_t \in V^*$ denotes the restriction of $f$ to the $t$-th summand in \cref{obscure1}.
\details{
    The sign of $(-1)^{p(t)\overline{f_t}}$ that appears for the same reason that it did in \cref{obscure1} is cancelled by the sign appearing in the isomorphism $(\Pi V)^* \xrightarrow{\cong} \Pi V^*$, $f \mapsto (-1)^{\bar{f}} f$.
}

The next result shows that the diagram of superfunctors
\[
    \begin{tikzcd}
        \OB_\kk(\Mat_{r|s}(A)^\op) \arrow[r, "\sM"] \arrow[rd, "\sG_{m|n}"'] & \Add(\OB_\kk(A^\op)_\pi) \arrow[d, "\sG_{(mr+ns|ms+nr)}"] \\
        & \fg\smod_\kk
    \end{tikzcd}
\]
commutes up to natural isomorphism, where $\sM$ is the superfunctor of \cref{sunrise}.

\begin{prop} \label{pound}
    The isomorphisms \cref{obscure1,obscure2} induce a monoidal supernatural isomorphism of superfunctors $\sG_{m|n} \xrightarrow{\cong} \sG_{(mr+ns)|(ms+nr)} \sM$.
\end{prop}

\begin{proof}
    Let $\omega$ denote the supernatural transformation induced by \cref{obscure1,obscure2}.  To simplify notation, let $\sG = \sG_{m|n}$ and $\sG' = \sG_{(mr+ns)|(ms+nr)}$.  We need to show that
    \[
        \omega_Y \circ \sG(f)
        = \sG' \sM(f) \circ \omega_X
    \]
    for every generating morphism $f \in \{ \upcross, \leftcup, \leftcap, \rightcup, \rightcap, \uptokstrand : a \in A\}$, where $X$ and $Y$ denote the domain and codomain of $f$, respectively.  These are all straightforward verifications, although care is needed to keep careful track of signs.  We give the details for $\upcross$ and $\leftcap$, since the others are similar.

    First consider the case $f = \upcross$.  We have
    \begin{align*}
        \omega_{\upobj \otimes \upobj} \colon W \otimes W &\to \bigoplus_{t,u=1}^{r+s} \Pi^{p(t)+p(u)} V \otimes V,\\
        v \otimes w &\mapsto \left( (-1)^{p(t)\overline{v_t} + p(u)\overline{w_u} + p(u) \overline{v_t} + p(t)p(u)} v_t \otimes w_u \right)_{t,u=1}^{r+s},
    \end{align*}
    where $p(t)$ is defined as in \cref{pdef}.  Then we compute that
    \[
        \sG' \sM(\upcross) \circ \omega_{\upobj \otimes \upobj}
        = \sum_{t,u=1}^{r+s} (-1)^{p(t)p(u)} \left( \Pi^{p(t)+p(u)} \flip \right) \circ \omega_{\upobj \otimes \upobj}
    \]
    and
    \[
        \omega_{\upobj \otimes \upobj} \circ \sG(\upcross)
        = \omega_{\upobj \otimes \upobj} \circ \flip
    \]
    are both the map (see \cref{dubpi})
    \begin{align*}
        W \otimes W &\to \bigoplus_{t,u=1}^{r+s} \Pi^{p(t)+p(u)} V \otimes V
        \\
        v \otimes w
        &\mapsto \left( (-1)^{p(t)\overline{v_t} + p(u)\overline{w_u} + p(u) \overline{v_t} + \overline{v_t} \overline{w_u}} v_t \otimes w_u \right)_{t,u=1}^{r+s}.
    \end{align*}

    Now consider $f = \leftcap$.  We have
    \begin{align*}
        \omega_{\downobj \otimes \upobj} \colon W^* \otimes W
        &\to \bigoplus_{t,u=1}^{r+s} \Pi^{p(t)+p(u)} V^* \otimes V,
        \\
        f \otimes v
        &\mapsto \left( (-1)^{p(u)\overline{v_u} + p(u)\overline{f_t} + p(t)p(u)} f_t \otimes v_u \right)_{t,u=1}^{r+s}.
    \end{align*}
    In addition, $\omega_\one \colon \kk \to \kk$ is the identity map.  Then we compute that
    \[
        \sG' \sM (\leftcup) \circ \omega_{\downobj \otimes \upobj}
        = \sum_{t=1}^{r+s} \sG'(\leftcap) \circ \omega_{\downobj \otimes \upobj}
        \qquad \text{and} \qquad
        \omega_\one \circ \sG (\leftcup)
        = \sG (\leftcup)
    \]
    are both the map
    \[
        W^* \otimes W \to \kk,\qquad
        f \otimes v \mapsto \sum_{t=1}^{r+s} (-1)^{p(t)} f_t(v_t).
        \qedhere
    \]
\end{proof}

\begin{prop} \label{picnic}
    If $\DD$ is a complex division superalgebra, then the superfunctor
    \[
        \sG_{m|n} \colon \OB_\C(\Mat_{r|s}(\DD)^\op) \to \fgl((mr+ns|ms+nr),\DD)\smod_\C
    \]
    is full for all $m,n \in \N$.
\end{prop}

\begin{proof}
    This follows from \cref{pound,sunrise,pyramid}.
\end{proof}

\subsection{Fullness over the real numbers}

In this subsection, we prove one of our main results: the oriented incarnation superfunctor of \cref{tiger} is full when $\kk=\R$ and $A$ is a central real division superalgebra.

Suppose $\DD$ is a central real division superalgebra and recall \cref{amongus}.  For $V = \DD^{m|n}$, we have canonical isomorphisms
\begin{equation} \label{fish}
    V^\C = (\DD^{m|n})^\C \xrightarrow{\cong} (\DD^\C)^{m|n}
    \quad \text{and} \quad
    V^{*,\C} = \left( (\DD^{m|n})^* \right)^\C \xrightarrow{\cong} \left( (\DD^\C)^{m|n} \right)^*.
\end{equation}
The next result shows that the diagram
\[
    \begin{tikzcd}
        \OB_\R(\DD^\op)^\C \arrow[rr, "\sR", "\cong"'] \arrow[d, "\sG_{m|n}^\C"']
        & & \OB_\C(\DD^{\op,\C}) \arrow[d, "\sG_{m|n}"]
        \\
        (\fgl(m|n,\DD)\smod_\R)^\C \arrow[rr, "\sC_{\fgl(m|n,\DD)}", "\cong"']
        & & \fgl(m|n,\DD^\C)\smod_\C
    \end{tikzcd}
\]
commutes up to supernatural isomorphism, where $\sC_{\fgl(m|n,\DD)}$ is defined in \cref{golem}, and $\sR$ is the superfunctor of \cref{crystal}.

\begin{prop} \label{jack}
    This is a monoidal supernatural isomorphism of superfunctors
    \[
        \sC_{\fgl(m|n,\DD)} \sG_{m|n}^\C \xrightarrow{\cong} \sG_{m|n} \sR
    \]
    determined by \cref{fish}.
\end{prop}

\begin{proof}
    To simplify notation, we set
    \[
        \sG = \sG_{m|n},\quad
        \sG^\C = \sG_{m|n}^\C,\quad
        \sC = \sC_{\fgl(m|n,\DD)},\quad
        W = (\DD^\C)^{m|n}.
    \]
    Let $\omega$ be the monoidal supernatural isomorphism determined by \cref{fish}.  For each generating morphism $f \in \{\upcross, \leftcap, \leftcup, \rightcap, \rightcup, \uptokstrand[a^\op] : a \in A\}$, we must show that
    \[
        \omega_Y \circ \sC \sG^\C(f) = \sG \sR(f) \circ \omega_X,
    \]
    where $X$ and $Y$ are the domain and codomain of $f$, respectively.

    We have
    \[
        \omega_{\upobj \otimes \upobj} \circ \sC \sG^\C (\upcross)
        = \omega_{\upobj \otimes \upobj} \circ \flip_{V^\C,V^\C}
        = \flip_{W,W} \circ \omega_{\upobj \otimes \upobj}
        = \sG \sR (\upcross) \circ \omega_{\upobj \otimes \upobj}.
    \]
    For $a \in \DD$, $v \in V$, and $y,z \in \C$, we have
    \[
        \omega_\upobj \circ \sC \sG^\C(\uptokstrand[a^\op] \otimes y)
        \colon v \otimes z \mapsto (-1)^{\bar{a} \bar{v}} \omega_\upobj(va \otimes yz)
    \]
    and
    \[
        \sG \sR (\uptokstrand[a^\op] \otimes y) \circ \omega_\upobj
        = \sG (\uptokstrand[a^\op \otimes y]) \circ \omega_\upobj
        \colon v \otimes z \mapsto (-1)^{\bar{a} \bar{v}} \omega_\upobj (va \otimes yz).
    \]
    For $f \in V^*$, $v \in V$, and $y,z \in \C$, we have
    \[
        \omega_\one \circ \sC \sG^\C (\leftcap)
        \colon (f \otimes y) \otimes (v \otimes z)
        \mapsto f(v) yz
    \]
    and
    \[
        \sG \sR (\leftcap) \circ \omega_{\downobj \otimes \upobj}
        = \sG (\leftcap) \circ \omega_{\downobj \otimes \upobj}
        \colon (f \otimes y) \otimes (v \otimes z)
        \mapsto f(v) yz.
    \]
    Finally, we have
    \[
        \omega_{\upobj \otimes \downobj} \circ \sC \sG^\C (\leftcup)
        \colon 1 \mapsto \sum_{v \in \bB_V^\R} v \otimes v^*
        \qquad \text{and} \qquad
        \sG \sR (\leftcup) \circ \omega_\one
        \colon 1 \mapsto \sum_{v \in \bB_{V^\C}^{\C}} v \otimes v^*.
    \]
    These are equal since any $\R$-basis of $V$ is also a $\C$-basis of $V^\C$.  The verifications for the remaining generating morphisms $\rightcap$ and $\rightcup$ are similar.
\end{proof}

\begin{theo} \label{OBrealfull}
    When $\kk=\R$ and $A$ is a central real division superalgebra, the oriented incarnation superfunctor $\sG_{m|n}$ of \cref{tiger} is full for all $m,n \in \N$.
\end{theo}

\begin{proof}
    Suppose $A = \DD$ is a central real division superalgebra and $m,n \in \N$.  We wish to show that, for all objects $X,Y$ in $\OB_\kk(\DD^\op)$, the $\R$-linear map
    \[
        \sG \colon \Hom_{\OB_\kk(\DD^\op;m-n)}(X,Y) \to \Hom_{\fgl(m|n,\DD)}(\sG(X),\sG(Y))
    \]
    is surjective.  As explained in \cref{subsec:complexification}, this map is surjective if and only if the complexified map
    \[
        \sG^\C \colon \Hom_{\OB_\kk(\DD^\op;m-n)}(X,Y)^\C \to \Hom_{\fgl(m|n,\DD)}(\sG(X),\sG(Y))^\C
    \]
    is surjective.  This follows from \cref{jack,sail,picnic}.
\end{proof}

\subsection{Consequences for real Lie groups\label{subsec:OBnonsuper}}

We assume throughout this subsection that $\DD \in \{\R,\HH\}$.  Then $\OB_\R(\DD)$ is a monoidal category, and there is no need to work in the setting of supercategories.  In fact, $\OB_\R(\DD)$ is a spherical pivotal category.  (We refer the reader to \cite[\S 4.3]{Sel11} for the definition of spherical pivotal category.)

In any spherical pivotal category $\cC$, we have a trace map $\operatorname{Tr} \colon \bigoplus_{X \in \cC} \End_\cC(X) \to \End_\cC(\one)$.  In terms of string diagrams, this corresponds to closing a diagram off to the right or left:
\[
    \operatorname{Tr}
    \left(
        \begin{tikzpicture}[centerzero]
            \draw[line width=2] (0,-0.5) -- (0,0.5);
            \filldraw[fill=white,draw=black] (-0.25,0.2) rectangle (0.25,-0.2);
            \node at (0,0) {$\scriptstyle{f}$};
        \end{tikzpicture}
    \right)
    =
    \begin{tikzpicture}[centerzero]
        \draw[line width=2] (0,0.2) arc(180:0:0.3) -- (0.6,-0.2) arc(360:180:0.3);
        \filldraw[fill=white,draw=black] (-0.25,0.2) rectangle (0.25,-0.2);
        \node at (0,0) {$\scriptstyle{f}$};
    \end{tikzpicture}
    =
    \begin{tikzpicture}[centerzero]
        \draw[line width=2] (0,0.2) arc(0:180:0.3) -- (-0.6,-0.2) arc(180:360:0.3);
        \filldraw[fill=white,draw=black] (-0.25,0.2) rectangle (0.25,-0.2);
        \node at (0,0) {$\scriptstyle{f}$};
    \end{tikzpicture}
    \ ,
\]
where the second equality follows from the axioms of a spherical category.  We say that a morphism $f \in \Hom_\cC(X,Y)$ is \emph{negligible} if $\operatorname{Tr}(f \circ g) = 0$ for all $g \in \Hom_\cC(Y,X)$.  The negligible morphisms form a two-sided tensor ideal $\cN$ of $\cC$, and the quotient $\cC/\cN$ is called the \emph{semisimplification} of $\cC$.  For $m \in \N$, let $\cON_\R(\DD;m)$ denote the tensor ideal of negligible morphisms of $\Kar(\OB_\R(\DD;m))$.

For an associative $\kk$-algebra $A$ and $m \in \N$, let $\fgl(m,A)\tmod_\kk$ denote the monoidal category of tensor $\fgl(m,A)$-modules.  By definition, this is the full subcategory of $\fgl(m,A)\md_\kk$ whose objects are direct summands of $V^{\otimes r} \otimes (V^*)^{\otimes s}$, $r,s \in \N$, where $V = A^m$ is the natural module.

\begin{theo} \label{storm}
    For $\DD \in \{\R,\HH\}$ and  $m \in \N$, the oriented incarnation functor induces an equivalence of monoidal categories
    \[
        \Kar(\OB_\R(\DD^\op;m))/\cON_\R(\DD^\op;m) \to \fgl(m,\DD)\tmod_\R.
    \]
\end{theo}

\begin{proof}
    Let $\DD \in \{\R,\HH\}$.  Since the category $\fgl(m,\DD)\tmod_\R$ is idempotent complete, $\sG_m = \sG_{m|0}$ induces a monoidal functor
    \[
        \Kar(\sG_m) \colon \Kar(\OB_\R(\DD^\op;m)) = \fgl(m,\DD)\tmod_\R.
    \]
    By \cref{OBrealfull}, this functor is full.

    Every object in $\fgl(m,\DD)\tmod_\R$ is completely reducible, since its complexification is a tensor module for $\fgl(m,\DD)^\C \cong \fgl(m,\DD^\C)$, hence completely reducible.  Thus, the category $\fgl(m,\DD)\tmod_\R$ is semisimple.
    \details{
        If $\fg$ is a Lie algebra, then a finite-dimensional $\fg$-module $V$ is completely reducible if and only if $\End_\fg(V)$ is semisimple.  If $R$ is an associative algebra over some field $\kk$, and $R \otimes_\kk \kk'$ is semisimple for some field $\kk'$ containing $\kk$, then $R$ is semisimple.  (See, for example, Proposition~I.5.11 of the lecture notes \href{https://www.jmilne.org/math/CourseNotes/LAG.pdf}{Lie algebras, algebraic groups, and Lie groups} by J.S.\ Milne.)  If $V$ is a tensor $\fgl(m,\DD)$-module, then $\End_{\fgl(m,\DD)}(V)^\C \cong \End_{\fgl(m,\DD^\C)}(V^\C)$ is semisimple.  Thus, $\End_{\fgl(m,\DD)}(V)$ is semisimple, and so $V$ is a completely reducible $\fgl(m,\DD)$-module.
    }
    In addition, by \cref{Obasisthm}, $\End_{\OB_\R(\DD^\op;m)}(\one) = \R 1_\one$.  Thus, by \cite[Prop.~6.9]{SW22}, the kernel of $\Kar(\sG_m)$ is equal to $\cON_\R(\DD^\op;m)$.  Therefore, $\Kar(\sG_m)$ induces a full and faithful functor
    \[
        \OB_\R(\DD^\op;m)/\cON_\R(\DD^\op;m) \to \fgl(m,\DD)\tmod_\R.
    \]

    Finally, since the image of $\Kar(\sG_m)$ contains all summands of tensor powers of the natural module $\DD^m$ and its dual (i.e.\ all tensor modules), it is essentially surjective, hence an equivalence of categories.
\end{proof}

\section{Superhermitian forms over involutive superalgebras\label{sec:hermitian}}

In \cref{sec:unoriented}, we will introduce our second main diagrammatic supercategory.  Then, in \cref{sec:Unic}, we will define the corresponding incarnation superfunctor.  These constructions will depend on superhermitian forms over involutive superalgebras.  In the current section, we cover the important properties of these forms.  Then, in \cref{sec:hermdiv}, we further specialize to the case where the superalgebra is a real division superalgebra.

\subsection{Involutive superalgebras\label{subsec:invalg}}

An \emph{anti-involution} of a superalgebra $A$ is a homomorphism of associative superalgebras $A \to A^\op$ squaring to the identity.  Equivalently, it is a $\kk$-linear map $\star \colon A \to A$, $a \mapsto a^\star$, such that
\begin{equation} \label{galaxy}
    (ab)^\star = (-1)^{\bar{a} \bar{b}} b^\star a^\star
    \quad \text{and} \quad
    (a^\star)^\star = a
    \qquad \text{for all } a,b \in A.
\end{equation}
An \emph{involutive superalgebra} is a pair $(A,\star)$, where $\star$ is an anti-involution of an associative superalgebra $A$.

We will typically use the notation $\star$ or $\inv$ for anti-involutions.  If $(A,\star)$ is an involutive superalgebra and $V$ is a right $A$-supermodule, then we let $V^\star$ denote the \emph{left} $A$-supermodule that is equal to $V$ as a $\kk$-supermodule, and with $A$-action given by
\begin{equation} \label{swap}
    a \cdot v := (-1)^{\bar{a}\bar{v}} v a^\star.
\end{equation}

Recall the definition of the Nakayama automorphism $\Nak$ of a Frobenius superalgebra from \cref{subsec:FrobAlg}.  An \emph{involutive Frobenius superalgebra} is a triple $(A,\star,\form)$ such that $(A,\form)$ is Frobenius superalgebra, $(A,\star)$ is an involutive superalgebra, and
\begin{equation} \label{rainbow}
    \Nak^2(a) = a
    \quad \text{and} \quad
    \form(a^\star) = \form(a)
    \qquad \text{for all } a \in A.
\end{equation}

We will discuss our main examples of interest, the involutive real division superalgebras, in \cref{sec:hermdiv}.  However, let us mention here some other important examples.

\begin{egs} \label{sky}
    \begin{enumerate}
        \item \label{aqua} The identity map is an anti-involution of any supercommutative Frobenius superalgebra.
        \item As a special case of \cref{aqua}, if $A = \kk[x]/(x^n)$ for some $n \ge 1$, we can take $\star$ to be the identity map and $\form(\sum_{r=0}^{n-1} a_r x^r) = a_{n-1}$.
        \item If $G$ is a finite group, we can take $A = \kk G$, $g^\star = g^{-1}$ for all $g \in G$, and $\form$ to be projection onto the identity element.
    \end{enumerate}
\end{egs}

\begin{lem}
    If $(A,\star,\form)$ is an involutive Frobenius superalgebra, then
    \begin{equation} \label{mouse}
        \Nak(a^\star) = \Nak(a)^\star
        \qquad \text{for all } a \in A.
    \end{equation}
\end{lem}

\begin{proof}
    For all $a,b \in A$, we have
    \[
        \form(ab)
        \overset{\cref{rainbow}}{=} \form((ab)^\star)
        = (-1)^{\bar{a}\bar{b}} \form(b^\star a^\star)
        \overset{\cref{Nakayama}}{=} (-1)^{\bar{a}\bar{b}} \form(a^\star \Nak(b^\star))
    \]
    and
    \[
        \form(ab)
        \overset{\cref{Nakayama}}{\underset{\cref{rainbow}}{=}} (-1)^{\bar{a}\bar{b}} \form(\Nak(b)a)
        \overset{\cref{rainbow}}{=} (-1)^{\bar{a}\bar{b}} \form ((\Nak(b)a)^\star)
        = (-1)^{\bar{a}\bar{b}} \form(a^\star \Nak(b)^\star).
    \]
    Then \cref{mouse} follows from the nondegeneracy of the Frobenius form.
\end{proof}

The following corollary will play an important role; see \cref{crescent}.

\begin{cor}
    If $(A,\star,\form)$ is an involutive Frobenius algebra, then so is $(A,\inv,\form)$, where $a^\inv = \Nak(a)^\star$.
\end{cor}

\begin{lem}
    Suppose $(A,\star,\form)$ is an involutive Frobenius superalgebra with Nakayama automorphism $\Nak$, and let $\bB_A$ be a homogeneous $\kk$-basis of $A$.  Then the left dual basis to $\bB_A^\star := \{b^\star : b \in \bB_A\}$ is given by
    \begin{equation} \label{nova}
        (b^\star)^\vee = \Nak(b^\vee)^\star.
    \end{equation}
\end{lem}

\begin{proof}
    For $b,c \in \bB_A$, we have
    \[
        \form \left( \Nak(c^\vee)^\star b^\star \right)
        \overset{\cref{galaxy}}{\underset{\cref{rainbow}}{=}} (-1)^{\bar{b}\bar{c}} \form (b \Nak(c^\vee))
        \overset{\cref{Nakayama}}{=} \form (c^\vee b)
        = \delta_{bc}.
        \qedhere
    \]
\end{proof}

\begin{lem}
    If $(A,\star,\form)$ is an involutive Frobenius superalgebra, then
    \begin{equation} \label{snow}
        \str_A(a) = \str_A(a^\star)
        \qquad \text{for all } a \in A,
    \end{equation}
    where $\str_A$ is the supertrace map defined in \cref{crazy}.
\end{lem}

\begin{proof}
    Using \cref{essex}, we compute
    \[
        \str_A(a)
        = \sum_{b \in \bB_A} (-1)^{\bar{b}} \form(b^\vee b a)
        \overset{\cref{rainbow}}{=} \sum_{b \in \bB_A} \form \left( a^\star b^\star (b^\vee)^\star \right)
        = \sum_{b \in \bB_A} (-1)^{\bar{b}} \form(a^\star b^\vee b)
        = \str_A(a^\star),
    \]
    where, in the second-to-last equality we changed to a sum over $\{ (b^\vee)^\star : b \in \bB_A \}$ and used that, for $b,c \in \bB_A$,
    \[
        (-1)^{\bar{c}} \form(c^\star (b^\vee)^\star)
        = \form(b^\vee c)
        = \delta_{bc},
    \]
    and so $\left( (b^\vee)^\star \right)^\vee = (-1)^{\bar{b}} b^\star$.
\end{proof}

\subsection{Supersymmetric forms}

Let $(A,\inv)$ denote an involutive superalgebra.

\begin{defin}
    For $\nu \in \{\pm 1\}$, a \emph{$(\nu,\inv)$-supersymmetric form} on a right $A$-supermodule $V$ is a homogeneous $\kk$-bilinear form $\Phi \colon V \times V \to \kk$ such that
    \begin{equation} \label{supersymmetric}
        \Phi(v,w) = \nu (-1)^{\bar{v}\bar{w}} \Phi(w,v)
        \qquad \text{and} \qquad
        \Phi(va,w) = (-1)^{\bar{a}\bar{w}} \Phi(v,wa^\inv)
    \end{equation}
    for all $v,w \in V$ and $a \in A$.
\end{defin}

If $\Phi$ is a nondegenerate $(\nu,\inv)$-supersymmetric form on $V$ and $\bB_V$ is a $\kk$-basis of $V$, then the left dual basis $\bB_V^\vee = \{v^\vee : v \in \bB_V\}$ of $V$ is defined by
\[
    \Phi(v^\vee,w) = \delta_{vw}.
\]
Note that $\overline{v^\vee} = \overline{v} + \bar{\Phi}$.

Recall the left $A$-actions on $V^*$ and $V^\inv$ given in \cref{swap,dualaction}.  A $(\nu,\inv)$-supersymmetric form $\Phi$ of parity $\sigma$ induces a parity-preserving homomorphism of left $A$-supermodules
\begin{equation} \label{step}
    V^\inv \to \Pi^\sigma V^*,\qquad
    v \mapsto \pi^\sigma \Phi^v,\qquad
    \text{where} \quad \Phi^v(w) = \Phi(v,w).
\end{equation}
\details{
    Let $g$ denote the map \cref{step}.  For $a \in A$ and $v,w \in V$, we have
    \begin{multline*}
        g(a \cdot v)(w)
        = (-1)^{\bar{a}\bar{v}} g(va^\inv)(w)
        = (-1)^{\bar{a}\bar{v}} \pi^\sigma \Phi^{v a^\inv}(w)
        = (-1)^{\bar{a}\bar{v}} \pi^\sigma \Phi(v a^\inv, w)
        \\
        = (-1)^{\bar{a}(\bar{v}+\bar{w})} \pi^\sigma \Phi(v,wa)
        = (-1)^{\bar{a}(\bar{v}+\bar{w})} \pi^\sigma \Phi^v(wa)
        = (-1)^{\bar{a}(\bar{v}+\bar{w})} g(v)(wa)
        = \big( a g(v) \big)(w),
    \end{multline*}
    as desired.
}
This is an isomorphism if and only if $\Phi$ is nondegenerate.

\begin{lem} \label{snooze}
    If a right $A$-supermodule $V$ admits a $(\nu,\inv)$-supersymmetric form, then
    \begin{equation}
        \str_V(a^\inv) = \str_V(a)
        \qquad \text{for all } a \in A.
    \end{equation}
\end{lem}

\begin{proof}
    Let $\Phi$ be a $(\nu,\inv)$-supersymmetric form on $V$, let $\bB_V$ be a $\kk$-basis of $V$, and let $\bB_V^\vee = \{v^\vee : v \in \bB_V\}$ denote the left dual basis with respect to $\Phi$.  Then
    \[
        \Phi(w,v^\vee)
        \overset{\cref{supersymmetric}}{=} \nu (-1)^{\bar{w}(\bar{v}+\bar{\Phi})} \Phi(v^\vee,w)
        = \nu (-1)^{\bar{v}+\bar{v}\bar{\Phi}} \delta_{vw}.
    \]
    Thus $(v^\vee)^\vee = (-1)^{\bar{v}+\bar{v}\bar{\Phi}} v$.  Then we compute
    \begin{multline*}
        \str_V(a^\inv)
        = \sum_{v \in \bB_V} (-1)^{\bar{v}} v^*(va^\inv)
        = \sum_{v \in \bB_V} (-1)^{\bar{v}} \Phi(v^\vee, v a^\inv)
        \overset{\cref{supersymmetric}}{=} \sum_{v \in \bB_V} (-1)^{\bar{v}+\bar{a}\bar{v}} \Phi(v^\vee a, v)
        \\
        \overset{\cref{supersymmetric}}{=} \nu \sum_{v \in \bB_V} (-1)^{\bar{v}\bar{\Phi}} \Phi(v,v^\vee a)
        = \sum_{v \in \bB_V} (-1)^{\bar{v}} \Phi(v^\vee,va)
        = \sum_{v \in \bB_V} \str_V(a),
    \end{multline*}
    where, in the second-to-last equality, we changed to a sum over the basis $\bB_V^\vee$.
\end{proof}

\subsection{Superhermitian forms}

Our main source of examples of $(\nu,\inv)$-supersymmetric forms will come from superhermitian forms over involutive Frobenius superalgebras.  Let $V$ be a finitely-generated right supermodule over an involutive superalgebra $(A,\star)$.  A homogeneous map $\varphi \colon V \times V \to A$ is a \emph{$\star$-sesquilinear form} if it is $\kk$-bilinear and
\begin{equation} \label{sesquilinear}
    \varphi(va,wb) = (-1)^{\bar{a} (\bar{\varphi}+\bar{v})} a^\star \varphi(v,w) b,\quad \text{for all } a,b \in A,\ v,w \in V.
\end{equation}
(In our cases of interest, $A$ will be purely even whenever $\bar{\varphi}=1$.)  If, in addition, there exists $\nu \in \{\pm 1\}$ such that
\begin{equation} \label{hermitian}
    \varphi(v,w) = \nu (-1)^{\bar{v} \bar{w}} \varphi(w,v)^\star \quad \text{for all } v,w \in V,
\end{equation}
then we say that $\varphi$ is a \emph{$(\nu,\star)$-superhermitian form}.  We say that $\varphi$ is \emph{unimodular} if the map
\[
    V \to \Hom_A(V,A),\quad v \mapsto \varphi(v,-),\quad v \in V,
\]
is an isomorphism of $\kk$-supermodules.  If $A$ is a division superalgebra, then $\varphi$ is unimodular if and only if it is nondegenerate.

We say that two $(\nu,\star)$-superhermitian forms $\varphi_1$ and $\varphi_2$ are \emph{equivalent} if there exists a homogeneous $f \in \Aut_A(V)$ such that
\[
    \varphi_2(v,w) = \varphi_1(f(v),f(w))
    \qquad \text{for all } v,w \in V.
\]
Note that, when $A=\kk$, a $(\nu,\id)$-superhermitian form is the same as a $(\nu,\id)$-supersymmetric form.

\begin{eg}
    If $A = \C$, $\star$ is complex conjugation, and $V$ is purely even, then a $(1,\star)$-superhermitian form is the familiar notion of a hermitian form, while a $(-1,\star)$-superhermitian form is a skew-hermitian form.  On the other hand, a $(1,\id)$-superhermitian form is a symmetric $\C$-bilinear form, while a $(-1,\id)$-superhermitian form is a skew-symmetric $\C$-bilinear form.
\end{eg}

\begin{rem}
    An \emph{even} $(\nu,\star)$-superhermitian form on $V$ is equivalent to an even $(-\nu,\star)$-superhermitian form on the parity-shifted supermodule $\Pi V$.  Thus, for even forms, one can assume $\nu=1$ without losing any generality.  However, this is \emph{not} the case for odd forms.  An \emph{odd} $(\nu,\star)$-superhermitian form on $V$ is equivalent to an odd $(\nu,\star)$-superhermitian form on the parity shift $\Pi V$.  Since odd forms are important for the periplectic Lie superalgebras we wish to include, we consider general $\nu$ in the current paper.
\end{rem}

\begin{lem} \label{crescent}
    If $(A,\star,\form)$ is an involutive Frobenius superalgebra, and $\varphi$ is a nondegenerate $(\nu,\star)$-superhermitian form on $V$, then the composite
    \begin{equation} \label{moon}
        \Phi = \form \circ \varphi \colon V \times V \to \kk
    \end{equation}
    is a nondegenerate $(\nu,\inv)$-supersymmetric form on $V$, with $a^\inv = \Nak(a)^\star$.
\end{lem}

\begin{proof}
    For all $v,w \in V$, we have
    \[
        \Phi(v,w)
        = \form(\varphi(v,w))
        \overset{\cref{hermitian}}{=} \nu (-1)^{\bar{v}\bar{w}} \form(\varphi(w,v)^\star)
        \overset{\cref{rainbow}}{=} \nu (-1)^{\bar{v}\bar{w}} \form(\varphi(w,v))
        = \nu(-1)^{\bar{v}\bar{w}} \Phi(w,v)
    \]
    and
    \begin{multline*}
        \Phi(va,w)
        = \form(\varphi(va,w))
        \overset{\cref{sesquilinear}}{=} (-1)^{\bar{a}(\bar{\varphi}+\bar{v})} \form(a^\star \varphi(v,w))
        \\
        \overset{\cref{Nakayama}}{=} (-1)^{\bar{a}\bar{w}} \form(\varphi(v,w)\Nak(a^\star))
        \overset{\cref{sesquilinear}}{=} (-1)^{\bar{a}\bar{w}} \form(\varphi(v,w\Nak(a^\star)))
        \overset{\cref{mouse}}{=} (-1)^{\bar{a}\bar{w}} \Phi(v,w \Nak(a)^\star).
    \end{multline*}
    Thus $\Phi$ is $(\nu,\inv)$-supersymmetric, with $a^\inv = \Nak(a)^\star$.

    It remains to show that $\Phi$ is nondegenerate.  Suppose $v \in V$ is nonzero.  Then, since $\varphi$ is nondegenerate, there exists $w \in V$ such that $\varphi(v,w) \ne 0$.  Since $\form$ is nondegenerate, there exists $a \in A$ such that
    \[
        0 \ne \form(\varphi(v,w)a)
        = \form(\varphi(v,wa))
        = \Phi(v,wa).
        \qedhere
    \]
\end{proof}

\subsection{Adjoint operators\label{subsec:adjoint}}

Suppose that $(A,\star)$ is an involutive superalgebra, and that $\varphi$ is a unimodular $(\nu,\star)$-superhermitian form on a right $A$-supermodule $V$.

\begin{lem}
    For all $X \in \End_A(V)$, there exists a unique $X^\dagger \in \End_A(V)$ such that
    \[
        \varphi(v,Xw) = (-1)^{\bar{X} \bar{v}} \varphi(X^\dagger v,w) \qquad \text{for all } v,w \in V.
    \]
\end{lem}

\begin{proof}
    Fix $v \in V$.  The map $w \mapsto (-1)^{\bar{X}\bar{v}} \varphi(v,Xw)$ is an element of $\Hom_A(V,A)$.
    \details{
        This map is clearly additive.  We also have
        \[
            wa
            \mapsto (-1)^{\bar{X}\bar{v}} \varphi(v, Xwa)
            = (-1)^{\bar{X}\bar{v}} \varphi(v,Xw)a.
        \]
        Hence the map is a morphism of right $A$-supermodules.
    }
    Thus, since $\varphi$ is unimodular, there exists a unique $v' \in V$ such that
    \[
        \varphi(v',w)
        = (-1)^{\bar{X}\bar{v}} \varphi(v,Xw) \qquad \text{for all } v \in V.
    \]
    We define $X^\dagger v = v'$.  It is then straightforward to verify that $X^\dagger \in \End_A(V)$.
    \details{
        For homogeneous $u,v,w \in V$ and $a \in A$ such that $\bar{v} = \bar{a} + \bar{u}$, we have
        \begin{align*}
            \varphi(X^\dagger(ua+v),w)
            &= (-1)^{\bar{X} \bar{v}} \varphi(ua+v,Xw) \\
            &= (-1)^{\bar{X} \bar{v} + \bar{a} \bar{u}} a^\star \varphi(u,Xw) + (-1)^{\bar{X} \bar{v}} \varphi(v,Xw) \\
            &= \varphi( (X^\dagger u)a + X^\dagger v, w ),
        \end{align*}
        and hence $X^\dagger(ua+v) = (X^\dagger u)a + X^\dagger v$.
    }
\end{proof}

The element $X^\dagger$ is called the \emph{adjoint} to $X$.

\begin{lem} \label{hilt}
    The map $X \mapsto X^\dagger$ is an anti-involution on $\End_A(V)$.  In particular, $(X^\dagger)^\dagger = X$, and
    \begin{equation} \label{blade}
        (XY)^\dagger = (-1)^{\bar{X} \bar{Y}} Y^\dagger X^\dagger
        \quad \text{for all } X,Y \in \End_A(V).
    \end{equation}
\end{lem}

\begin{proof}
    This is a straightforward verification.
    \details{
        For all $v,w \in V$, we have
        \begin{align*}
            \varphi(v,Xw)
            &= (-1)^{\bar{X}\bar{v}} \varphi(X^\dagger v, w) \\
            &= (-1)^{\bar{X}(\bar{v}+\bar{w}) + \bar{v}\bar{w}} \varphi(w,X^\dagger v)^\star \\
            &= (-1)^{(\bar{X}+\bar{w})\bar{v}} \varphi((X^\dagger)^\dagger w, v)^\star \\
            &= \varphi(v,(X^\dagger)^\dagger w).
        \end{align*}
        Thus $(X^\dagger)^\dagger = X$.  We also have
        \begin{align*}
            (-1)^{(\bar{X} + \bar{Y})\bar{v}} \varphi((XY)^\dagger v, w)
            &= \varphi(v,XYw) \\
            &= (-1)^{\bar{X}\bar{v}} \varphi(X^\dagger v, Y w) \\
            &= (-1)^{(\bar{X}+\bar{v})\bar{Y} + \bar{X}\bar{v}} \varphi(Y^\dagger X^\dagger v, w),
        \end{align*}
        and so \cref{blade} holds.
    }
\end{proof}

Now suppose that $(A,\star,\form)$ is an involutive Frobenius superalgebra, and consider the nondegenerate $(\nu,\inv)$-supersymmetric form $\Phi = \form \circ \varphi$ as in \cref{crescent}.  Then, for all $X \in \End_A(V)$, we have
\[
    \Phi(v,Xw) = (-1)^{\bar{X}\bar{v}} \Phi(X^\dagger v, w)
    \qquad \text{for all } v,w \in V.
\]
It follows that the definition of $X^\dagger$ is the same if we use the $(\nu,\star)$-superhermitian form $\varphi$ or the corresponding $(\nu,\inv)$-supersymmetric form $\Phi$.

\subsection{Harish-Chandra superpairs associated to superhermitian forms\label{subsec:HCform}}

For $\varphi$ either a unimodular $(\nu,\star)$-superhermitian form or a nondegenerate (hence unimodular) $(\nu,\inv)$-supersymmetric form, define
\begin{align*}
    G_\rd(\varphi) &= \{ X \in \Aut_A(V)_0 : \varphi(Xv,Xw) = \varphi(v,w) \text{ for all } v,w \in V\}, \\
    \fg(\varphi) &= \{ X \in \End_A(V) : \varphi(Xv,w) = -(-1)^{\bar{X} \bar{v}} \varphi(v,Xw) \text{ for all } v,w \in V\} \\
    &= \{ X \in \End_A(V) : X^\dagger = - X\}.
\end{align*}
We have that $\fg(\varphi)$ is a Lie superalgebra with the usual bracket:
\[
    [X,Y] = XY - (-1)^{\bar{X}\bar{Y}} YX.
\]
The pair $G(\varphi) := (G_\rd(\varphi),\fg(\varphi))$ is a Harish-Chandra superpair; see \cref{subsec:HCpair}.

If $\varphi_1$ and $\varphi_2$ are equivalent forms, then the groups $G_\rd(\varphi_1)$ and $G_\rd(\varphi_2)$ are isomorphic, as are the Lie superalgebras $\fg(\varphi_1)$ and $\fg(\varphi_2)$.

If $(A,\star,\form)$ is an involutive Frobenius superalgebra and $\varphi$ is a unimodular $(\nu,\star)$-superhermitian form, then we have the nondegenerate $(\nu,\inv)$-supersymmetric form $\Phi = \form \circ \varphi$ from \cref{crescent}.  It follows from the nondegeneracy of $\form$ that
\[
    \fg(\varphi) = \fg(\Phi)
    \qquad \text{and} \qquad
    G_\rd(\varphi) = G_\rd(\Phi).
\]
\details{
    The first equality follows from the fact that the definition of $X^\dagger$ is the same if we use $\varphi$ or $\Phi$, as noted in \cref{subsec:adjoint}

    It is clear that $G_\rd(\varphi) \subseteq G_\rd(\Phi)$.  Now assume that $X \in G_\rd(\Phi)$.  Then, for all $a \in A$ and $v,w \in V$, we have
    \[
        \form(\varphi(Xv,Xw)a)
        = \form(\varphi(Xv,Xwa))
        = \Phi(Xv,Xwa)
        = \Phi(v,wa)
        = \form(\varphi(v,wa))
        = \form(\varphi(v,w)a).
    \]
    Since $\form$ is nondegenerate, we have $\varphi(Xv,Xw) = \varphi(v,w)$.
}

\section{Superhermitian forms over involutive real division superalgebras\label{sec:hermdiv}}

For our purposes, the most important examples of involutive superalgebras come from real division superalgebras.  In this section, we examine some important properties of the Lie superalgebras associated to superhermitian forms over involutive real division superalgebras.

For the quaternions, we have the anti-involution
\[
    \star \colon \HH \to \HH,\qquad
    (a+bi+cj+dk)^\star = a - bi - cj - dk,\quad
    a,b,c,d \in \R,
\]
of quaternionic conjugation.  This restricts to complex conjugation on $\C$ and the identity map on $\R$.  The complex Clifford superalgebra $\Cl(\C)$ has anti-involution
\begin{equation} \label{CClinv}
    \star \colon \Cl(\C) \to \Cl(\C),\qquad
    (a + \varepsilon b)^\star = a^\star + \varepsilon b^\star i,\quad
    a,b \in \C,
\end{equation}
where, on the right-hand side, $\star$ denotes complex conjugation.  Note that this is an anti-involution of \emph{real} superalgebras, but not of \emph{complex} superalgebras.  In fact, there are no $\C$-linear anti-involutions of $\Cl(\C)$.
\details{
    Suppose $\star \colon \Cl(\C) \to \Cl(\C)$ is a $\C$-linear anti-involution.  Since $\star$ is parity-preserving, we have $\varepsilon^* = \varepsilon z$ for some $z \in \C$.  Then
    \[
        1 = 1^\star
        = (\varepsilon^2)^\star
        = - (\varepsilon^\star)^2
        = - (\varepsilon z)^2
        = - \varepsilon^2 z^2
        = - z^2.
    \]
    Thus $z = \pm i$.  But then
    \[
        (\varepsilon^\star)^\star
        = (\varepsilon z)^\star
        = \varepsilon^\star z
        = \varepsilon z^2
        = - \varepsilon
        \ne \varepsilon.
    \]
}
The notation $\star$ will always refer to the above involutions when working with the real division superalgebras $\R$, $\C$, $\HH$, and $\Cl(\C)$.  Note that $(\DD,\star,\RP)$ is an involutive Frobenius superalgebra for $\DD \in \{\R,\C,\HH,\Cl(\C)\}$, as is $(\C,\id,\RP)$, where $\RP(a)$ is the real part of the even part of $a$.  None of the other real division superalgebras $\Cl_r(\R)$, $r \in \{1,2,3,5,6,7\}$, admit anti-involutions, since they are not isomorphic to their opposite superalgebras.

If $(A,\star)$ is an involutive superalgebra, then the complexification $A^\C$ is also an involutive superalgebra, with involution (which we continue to denote by the same symbol)
\begin{equation} \label{complexstar}
    \star \colon A^\C \to A^\C,\quad
    (a \otimes z)^\star = a^\star \otimes z,\qquad
    a \in A,\ z \in \C.
\end{equation}
Similarly, if $(A,\star,\form)$ is an involutive Frobenius superalgebra, then so is $(A^\C,\star,\form)$; see \cref{floor}.

\subsection{Real case}

In this subsection we work over the involutive real division superalgebra $(\R,\id)$.  If
\[
    \varphi \colon V \times V \to \R
\]
is a nondegenerate $(\nu,\id)$-superhermitian form on an $\R$-supermodule $V$, then its complexification
\[
    \varphi^\C \colon V^\C \times V^\C \to \C
\]
is a nondegenerate $(\nu,\id)$-superhermitian form on $V^\C$.

\begin{prop} \label{compRform}
    We have an isomorphism of complex Lie superalgebras
    \[
        \fg(\varphi)^\C \cong \fg(\varphi^\C).
    \]
\end{prop}

\begin{proof}
    It is straightforward to verify that the map $X \otimes a \mapsto Xa$, $X \in \fg(\varphi)$, $a \in \C$, gives the desired isomorphism.
\end{proof}

\subsection{Complex cases}

In this subsection we work over the involutive real division superalgebra $(\DD,\star)$, where $\DD \in \{\C,\Cl(\C)\}$.

\begin{prop} \label{compCform}
    Suppose $\varphi$ is a nondegenerate $(\nu,\star)$-superhermitian form on $\DD^{m|n}$.  Then we have an isomorphism of complex Lie superalgebras
    \[
        \fg(\varphi)^\C \cong \fgl(m|n,\DD).
    \]
\end{prop}

\begin{proof}
    It follows from \cref{blade} with $Y = aI$ that $(aX)^\dagger = a^\star X^\dagger$ for all $a \in \DD$ and $X \in \fg(\varphi)$.
    Multiplication by $i$ gives an isomorphism of $\C$-supermodules
    \[
        \fg(\varphi) \xrightarrow{\cong} \{X \in \fgl(m|n,\DD) : X^\dagger = X\}.
    \]
    Therefore, for every $X \in \fgl(m|n,\DD)$, we have
    \begin{gather*}
        X = \tfrac{1}{2}(X^\dagger - X) + \tfrac{1}{2}(X^\dagger + X),
        \qquad \text{with}
        \\
        \tfrac{1}{2}(X^\dagger - X) \in \fg(\varphi)
        \qquad \text{and} \qquad
        \tfrac{1}{2}(X^\dagger + X) \in \fg(\varphi) i.
    \end{gather*}
    The result follows.
\end{proof}

\subsection{Quaternionic case\label{subsec:compHform}}

In this subjection we work over the real involutive division superalgebra $(\HH,\star)$ and we fix a nondegenerate $(\nu,\star)$-superhermitian form $\varphi$ on $\HH^{m|n}$.   We will often view the quaternions as
\begin{equation} \label{jcong}
    \HH = \C[j],\quad j^2 = -1,\quad jzj^{-1} = z^\star,\ z \in \C.
\end{equation}
Choosing the $\C$-basis $\{1,j\}$ for $\HH$, we will identify $\HH^{m|n}$ with $\C^{2m|2n}$.  Similarly, under the inclusion $\imath$ of \cref{Pauli}, we will identify $\Mat_{m|n}(\HH)$ with a subring of $\Mat_{2m|2n}(\C)$.  In particular, we have
\begin{equation} \label{dutch}
    \Mat_{m|n}(\HH)
    = \{ X \in \Mat_{2m|2n}(\C) : X(vj) = (Xv)j \text{ for all } v \in \HH^{m|n} = \C^{2m|2n} \}.
\end{equation}
Note that, for $v \in \HH^{m|n} = \C^{2m|2n}$, we have
\begin{equation} \label{mayor}
    vj = J v^\star,\qquad \text{where} \quad
    J = J_{m+n} =
    \begin{pmatrix}
        J_1 & 0 & \cdots & 0 \\
        0 & J_1 & \ddots & \vdots \\
        \vdots & \ddots & \ddots & 0 \\
        0 & \cdots & 0 & J_1
    \end{pmatrix},
    \quad
    J_1 = \begin{pmatrix} 0 & -1 \\ 1 & 0 \end{pmatrix}.
\end{equation}
It follows that, for $X \in \Mat_{2m|2n}(\C)$,
\begin{equation} \label{aiel}
    X \in \Mat_{m|n}(\HH)
    \iff J X^\star J^{-1} = X.
\end{equation}
\details{
    We identify $\Mat_{m|n}(\HH)$ with $\End_\HH(\HH^{m|n}) \subseteq \End_\C(\C^{2m|2n}) = \Mat_{2m|2n}(\C)$.  Then, by \cref{dutch}, we have
    \[
        X \in \Mat_{m|n}(\HH)
        \iff X J v^\star = J(Xv)^\star = J X^\star v^\star \ \forall\ v \in \C^{2m|2n}
        \iff J X^\star J^{-1} = X.
    \]
}

Consider the $\C$-linear maps
\[
    \proj^\HH_\C,\proj^\HH_{j\C} \colon \HH \to \C,\quad
    \proj^\HH_\C(a+jb) = a,\quad
    \proj^\HH_{j\C}(a+jb) = b,\qquad
    a,b \in \C.
\]
It is straightforward to verify that
\begin{equation} \label{reveal}
    \proj^\HH_{j\C}(zj) = \proj^\HH_\C(z)^\star
    \quad \text{for all } z \in \C.
\end{equation}
\details{
    We have
    \begin{multline*}
        \proj^\HH_{j\C}((a+ib+jc+kd)j)
        = \proj^\HH_{j\C}(aj + ibj + jcj + kdj)
        \\
        \overset{\cref{jcong}}{=} \proj^\HH_{j\C}(j a^\star + k b^\star - c^\star - i d^\star)
        = a^\star
        = \proj^\HH_\C(a+ib+jc+kd)^\star.
    \end{multline*}
}
Define
\[
    \varphi^1 := \proj^\HH_\C \circ \varphi,\qquad
    \varphi^j := \proj^\HH_{j\C} \circ \varphi.
\]
These are precisely the components of $\varphi$ with respect to the $\C$-basis $\{1,j\}$ of $\HH$.  The following lemma gives a precise relationship between $\varphi^1$ and $\varphi^j$.

\begin{lem}
    For all $v,w \in \HH^{m|n} = \C^{2m|2n}$, we have
    \begin{align} \label{boar}
        \varphi^j(vj,w) &= - \varphi^1(v,w),&
        \varphi^1(vj,w) &= \varphi^j(v,w),
        \\ \label{rabbit}
        \varphi^j(v,wj) &= \varphi^1(v,w)^\star,&
        \varphi^1(v,wj) &= - \varphi^j(v,w)^\star.
    \end{align}
\end{lem}

\begin{proof}
    We have
    \[
        \varphi^1(vj,w)  + j \varphi^j(vj,w)
        = \varphi(vj,w)
        = -j \varphi(v,w)
        = -j \varphi^1(v,w) + \varphi^j(v,w).
    \]
    Comparing components gives \cref{boar}.  Similarly
    \[
        \varphi^1(v,wj) + j \varphi^j(v,wj)
        = \varphi(v,wj)
        = \varphi(v,w)j
        \overset{\cref{jcong}}{=} j \varphi^1(v,w)^\star - \varphi^j(v,w)^\star
    \]
    implies \cref{rabbit}.
\end{proof}

\begin{lem} \label{fries}
    The form
    \[
        \varphi^j \colon \C^{2m|2n} \times \C^{2m|2n} \to \C
    \]
    is nondegenerate and $(-\nu,\id)$-supersymmetric.
\end{lem}

\begin{proof}
    It follows from \cref{jcong} that
    \[
        \varphi^j(av,wb) = a \varphi^j(v,w) b,\qquad
        v,w \in V,\ a,b \in \C.
    \]
    Also,
    \begin{multline*}
        \varphi(v,w)
        = \nu (-1)^{\bar{v}\bar{w}} \varphi(w,v)^\star
        = \nu (-1)^{\bar{v}\bar{w}} \big( \varphi^1(w,v) + j \varphi^j(w,v) \big)^\star
        \\
        = \nu (-1)^{\bar{v}\bar{w}} \big( \varphi^1(w,v)^\star + \varphi^j(w,v)^\star j^\star \big)
        = \nu (-1)^{\bar{v}\bar{w}} \varphi^1(w,v)^\star - j \nu (-1)^{\bar{v}\bar{w}} \varphi^j(w,v).
    \end{multline*}
    Thus, $\varphi^j \colon \C^{2m|2n} \times \C^{2m|2n} \to \C$ is a $(-\nu,\id)$-supersymmetric form.  It follows from \cref{rabbit} that, for all $v,w \in \C^{2m|2n}$, we have
    \[
        \varphi(v,w) = \varphi^j(v,wj)^\star + j \varphi^j(v,w).
    \]
    Thus, $\varphi^j$ is nondegenerate, since $\varphi$ is.
\end{proof}

\begin{cor}
    We have
    \begin{equation} \label{house}
        \varphi^j(v,wj) - \varphi^j(vj,w) = 2 \RP \varphi(v,w).
    \end{equation}
\end{cor}

\begin{proof}
    Using \cref{boar,rabbit}, we have
    \[
        \varphi^j(v,wj) - \varphi^j(vj,w)
        = \varphi^1(v,w)^\star + \varphi^1(v,w)
        = 2 \RP \varphi(v,w).
        \qedhere
    \]
\end{proof}

\begin{prop} \label{compHform}
    We have an isomorphism of complex Lie superalgebras
    \[
        \fg(\varphi)^\C \cong \fg(\varphi^j).
    \]
\end{prop}

\begin{proof}
    For all $X \in \Mat_{m|n}(\HH)$ satisfying $\varphi^j(Xv,w) = - (-1)^{\bar{X}\bar{v}} \varphi^j(v,Xw)$ for all $v,w \in \C^{2m|2n}$, we have
    \begin{multline*}
        \varphi^1(Xv,w)
        \overset{\cref{rabbit}}{=} \varphi^j(Xv,wj)^\star
        = -(-1)^{\bar{X}\bar{v}} \varphi^j(v,X(wj))^\star
        \\
        \overset{\cref{dutch}}{=} -(-1)^{\bar{X}\bar{v}} \varphi^j(v,(Xw)j)^\star
        \overset{\cref{rabbit}}{=} -(-1)^{\bar{X}\bar{v}} \varphi^1(v,Xw).
    \end{multline*}
    Thus
    \[
        \fg(\varphi)
        = \{ X \in \fgl(m|n,\HH) : \varphi^j(Xv,w) = -(-1)^{\bar{X}\bar{v}} \varphi^j(v,Xw) \}.
    \]
    Furthermore, using \cref{fries}, for $X \in \fg(\varphi)$, $a \in \C$, and $v,w \in \C^{2m|2n}$, we have
    \[
        \varphi^j((Xa)v, w)
        = -(-1)^{\bar{X}\bar{v}} a \varphi^j(v,Xw)
        = -(-1)^{\bar{X}\bar{v}} \varphi^j(v,Xw) a
        = -(-1)^{\bar{X}\bar{v}} \varphi^j(v,(Xa)w).
    \]
    It then follows from \cref{glcomplex}\ref{glcomplex:H} that
    \[
        \fg(\varphi)^\C
        = \{ X \in \fgl(2m|2n,\C) : \varphi^j(Xv,w) = -(-1)^{\bar{X}\bar{v}} \varphi^j(v,Xw) \}
        = \fg(\varphi^j).
        \qedhere
    \]
\end{proof}

\section{The unoriented supercategory\label{sec:unoriented}}

In this section, we introduce the second of our two main diagrammatic supercategories.  After defining the supercategory, we deduce some of its additional properties.  We then prove a basis theorem for morphism spaces.  Throughout this section, $\kk$ is an arbitrary field, and $(A,\inv)$ is an involutive superalgebra.

\subsection{Definition of the supercategory}

\begin{defin} \label{FBC}
    For $\sigma \in \Z_2$, we define $\Brauer_\kk^\sigma(A,\inv)$ to be the strict monoidal supercategory generated by one object $\go$ and morphisms
    \begin{gather*}
        \crossmor \colon \go^{\otimes 2} \to \go^{\otimes 2},\qquad
        \capmor \colon \go^{\otimes 2} \to \one,\qquad
        \cupmor \colon \one \to \go^{\otimes 2},\qquad
        \tokstrand \colon \go \to \go,\quad a \in A,
    \end{gather*}
    subject to the relations
    \begin{gather} \label{brauer}
        \begin{tikzpicture}[centerzero]
            \draw (0.2,-0.4) to[out=135,in=down] (-0.15,0) to[out=up,in=225] (0.2,0.4);
            \draw (-0.2,-0.4) to[out=45,in=down] (0.15,0) to[out=up,in=-45] (-0.2,0.4);
        \end{tikzpicture}
        \ =\
        \begin{tikzpicture}[centerzero]
            \draw (-0.2,-0.4) -- (-0.2,0.4);
            \draw (0.2,-0.4) -- (0.2,0.4);
        \end{tikzpicture}
        \ ,\quad
        \begin{tikzpicture}[centerzero]
            \draw (0.4,-0.4) -- (-0.4,0.4);
            \draw (0,-0.4) to[out=135,in=down] (-0.32,0) to[out=up,in=225] (0,0.4);
            \draw (-0.4,-0.4) -- (0.4,0.4);
        \end{tikzpicture}
        \ =\
        \begin{tikzpicture}[centerzero]
            \draw (0.4,-0.4) -- (-0.4,0.4);
            \draw (0,-0.4) to[out=45,in=down] (0.32,0) to[out=up,in=-45] (0,0.4);
            \draw (-0.4,-0.4) -- (0.4,0.4);
        \end{tikzpicture}
        \ ,\quad
        \begin{tikzpicture}[centerzero]
            \draw (-0.3,0.4) -- (-0.3,0) arc(180:360:0.15) arc(180:0:0.15) -- (0.3,-0.4);
        \end{tikzpicture}
        =
        \begin{tikzpicture}[centerzero]
            \draw (0,-0.4) -- (0,0.4);
        \end{tikzpicture}
        = (-1)^\sigma\
        \begin{tikzpicture}[centerzero]
            \draw (-0.3,-0.4) -- (-0.3,0) arc(180:0:0.15) arc(180:360:0.15) -- (0.3,0.4);
        \end{tikzpicture}
        \ ,\quad
        \begin{tikzpicture}[anchorbase]
            \draw (-0.15,-0.4) to[out=60,in=-90] (0.15,0) arc(0:180:0.15) to[out=-90,in=120] (0.15,-0.4);
        \end{tikzpicture}
        = \,
        \capmor
        \ ,\quad
        \begin{tikzpicture}[centerzero]
            \draw (-0.2,-0.3) -- (-0.2,-0.1) arc(180:0:0.2) -- (0.2,-0.3);
            \draw (-0.3,0.3) \braiddown (0,-0.3);
        \end{tikzpicture}
        =
        \begin{tikzpicture}[centerzero]
            \draw (-0.2,-0.3) -- (-0.2,-0.1) arc(180:0:0.2) -- (0.2,-0.3);
            \draw (0.3,0.3) \braiddown (0,-0.3);
        \end{tikzpicture}
        \ ,
        \\ \label{tokrel}
        \begin{tikzpicture}[centerzero]
            \draw (0,-0.35) -- (0,0.35);
            \token{0,0}{east}{1};
        \end{tikzpicture}
        =
        \begin{tikzpicture}[centerzero]
            \draw (0,-0.35) -- (0,0.35);
        \end{tikzpicture}
        \ ,\quad
        \lambda\
        \begin{tikzpicture}[centerzero]
            \draw (0,-0.35) -- (0,0.35);
            \token{0,0}{west}{a};
        \end{tikzpicture}
        + \mu\
        \begin{tikzpicture}[centerzero]
            \draw (0,-0.35) -- (0,0.35);
            \token{0,0}{west}{b};
        \end{tikzpicture}
        =
        \begin{tikzpicture}[centerzero]
            \draw (0,-0.35) -- (0,0.35);
            \token{0,0}{west}{\lambda a + \mu b};
        \end{tikzpicture}
        ,\quad
        \begin{tikzpicture}[centerzero]
            \draw (0,-0.35) -- (0,0.35);
            \token{0,-0.15}{east}{b};
            \token{0,0.15}{east}{a};
        \end{tikzpicture}
        =
        \begin{tikzpicture}[centerzero]
            \draw (0,-0.35) -- (0,0.35);
            \token{0,0}{west}{ab};
        \end{tikzpicture}
        \ ,\quad
        \begin{tikzpicture}[centerzero]
            \draw (-0.35,-0.35) -- (0.35,0.35);
            \draw (0.35,-0.35) -- (-0.35,0.35);
            \token{-0.17,-0.17}{east}{a};
        \end{tikzpicture}
        =
        \begin{tikzpicture}[centerzero]
            \draw (-0.35,-0.35) -- (0.35,0.35);
            \draw (0.35,-0.35) -- (-0.35,0.35);
            \token{0.17,0.17}{west}{a};
        \end{tikzpicture}
        \ ,\quad
        \begin{tikzpicture}[anchorbase]
            \draw (-0.2,-0.2) -- (-0.2,0) arc (180:0:0.2) -- (0.2,-0.2);
            \token{-0.2,0}{east}{a};
        \end{tikzpicture}
        \ =\
        \begin{tikzpicture}[anchorbase]
            \draw (-0.2,-0.2) -- (-0.2,0) arc (180:0:0.2) -- (0.2,-0.2);
            \token{0.2,0}{west}{a^\inv};
        \end{tikzpicture}
        ,
    \end{gather}
    for all $a,b \in A$ and $\lambda,\mu \in \kk$.  The parity of $\tokstrand$ is $\bar{a}$, the morphisms $\cupmor$ and $\capmor$ both have parity $\sigma$, and $\crossmor$ is even.  We refer to the morphisms $\tokstrand$ as \emph{tokens}.

    For $d \in \kk$, we define $\Brauer_\kk^\sigma(A,\inv;d)$ to be the quotient of $\Brauer_\kk^\sigma(A,\inv)$ by the additional relations
    \begin{equation} \label{burst}
        \bubble{a} = d \str_A(a) 1_\one,\qquad a \in A,
    \end{equation}
    where $\str_A$ is given by \cref{crazy}.  We call $d$ the \emph{specialization parameter}.
\end{defin}

\begin{eg} \label{Brauercase}
    When $A=\kk$ and $\inv = \id$, we have $\tokstrand = a\, \idstrand$ for all $a \in \kk$.  Thus, we can omit the generators $\tokstrand$ and all the relations involving them.  Then, if $\sigma=0$, the relations \cref{brauer,burst} are the defining relations of the \emph{Brauer category}, as given in \cite[Th.~2.6]{LZ15}.  For general $\sigma$, $\Brauer_\kk^\sigma(\kk,\id;d)$ is isomorphic to the \emph{marked Brauer category} of \cite{KT17}, although the description there looks somewhat different since the authors do not use the concept of a monoidal supercategory.
\end{eg}

\begin{prop} \label{silver}
    The following relations hold in $\Brauer_\kk^\sigma(A,\inv)$ for all $a \in A$:
    \begin{gather} \label{mirror}
        \begin{tikzpicture}[anchorbase]
            \draw (-0.15,0.4) to[out=-60,in=90] (0.15,0) arc(360:180:0.15) to[out=90,in=240] (0.15,0.4);
        \end{tikzpicture}
        = (-1)^{\sigma}\ \cupmor
        \ ,\qquad
        \begin{tikzpicture}[centerzero]
            \draw (-0.2,0.3) -- (-0.2,0.1) arc(180:360:0.2) -- (0.2,0.3);
            \draw (-0.3,-0.3) \braidup (0,0.3);
        \end{tikzpicture}
        =
        \begin{tikzpicture}[centerzero]
            \draw (-0.2,0.3) -- (-0.2,0.1) arc(180:360:0.2) -- (0.2,0.3);
            \draw (0.3,-0.3) \braidup (0,0.3);
        \end{tikzpicture}
        \ ,\qquad
        \begin{tikzpicture}[centerzero]
            \draw (-0.35,-0.35) -- (0.35,0.35);
            \draw (0.35,-0.35) -- (-0.35,0.35);
            \token{-0.17,0.17}{east}{a};
        \end{tikzpicture}
        =
        \begin{tikzpicture}[centerzero]
            \draw (-0.35,-0.35) -- (0.35,0.35);
            \draw (0.35,-0.35) -- (-0.35,0.35);
            \token{0.17,-0.17}{west}{a};
        \end{tikzpicture}
        ,\qquad
        \begin{tikzpicture}[anchorbase]
            \draw (-0.2,0.2) -- (-0.2,0) arc (180:360:0.2) -- (0.2,0.2);
            \token{0.2,0}{west}{a};
        \end{tikzpicture}
        =
        \begin{tikzpicture}[anchorbase]
            \draw (-0.2,0.2) -- (-0.2,0) arc (180:360:0.2) -- (0.2,0.2);
            \token{-0.2,0}{east}{a^\inv};
        \end{tikzpicture}
        \, .
    \end{gather}
\end{prop}

\begin{proof}
    For the second relation in \cref{mirror}, we have
    \[
        \begin{tikzpicture}[centerzero]
            \draw (-0.2,0.3) -- (-0.2,0.1) arc(180:360:0.2) -- (0.2,0.3);
            \draw (-0.3,-0.3) \braidup (0,0.3);
        \end{tikzpicture}
        =
        \begin{tikzpicture}[anchorbase]
            \draw (-1,0.5) -- (-1,0.2) arc(180:360:0.2) arc(180:0:0.2) -- (-0.2,0) arc(180:360:0.2) -- (0.2,0.5);
            \draw (-0.5,-0.3) \braidup (0,0.5);
        \end{tikzpicture}
        \overset{\cref{intlaw}}{=} (-1)^\sigma\
        \begin{tikzpicture}[anchorbase]
            \draw (-1,0.5) -- (-1,0) arc(180:360:0.2) -- (-0.6,0.2) arc(180:0:0.2) arc(180:360:0.2) -- (0.2,0.5);
            \draw (-0.3,-0.3) \braidup (0,0.5);
        \end{tikzpicture}
        = (-1)^\sigma\
        \begin{tikzpicture}[anchorbase]
            \draw (-1,0.5) -- (-1,0) arc(180:360:0.2) -- (-0.6,0.2) arc(180:0:0.2) arc(180:360:0.2) -- (0.2,0.5);
            \draw (-0.3,-0.3) \braidup (-0.8,0.5);
        \end{tikzpicture}
       =
        \begin{tikzpicture}[centerzero]
            \draw (-0.2,0.3) -- (-0.2,0.1) arc(180:360:0.2) -- (0.2,0.3);
            \draw (0.3,-0.3) to[out=up,in=down] (0,0.3);
        \end{tikzpicture}
        \ ,
    \]
    where, for the unadorned equalities, we have used the third, sixth, and fourth equalities in \cref{brauer}, in that order.  Then, for the first relation in \cref{mirror}, we compute
    \[
        \begin{tikzpicture}[anchorbase]
            \draw (-0.15,0.4) to[out=-60,in=90] (0.15,0) arc(360:180:0.15) to[out=90,in=240] (0.15,0.4);
        \end{tikzpicture}
        = (-1)^\sigma\
        \begin{tikzpicture}[anchorbase]
            \draw (-0.15,0.4) to[out=-60,in=90] (0.15,0) arc(360:180:0.15) to[out=90,in=180] (0.15,0.3) arc (90:0:0.15) arc(180:360:0.15) -- (0.6,0.4);
        \end{tikzpicture}
        = (-1)^\sigma
        \begin{tikzpicture}[anchorbase]
            \draw (0.1,0.4) to[out=down,in=right] (-0.2,-0.3) to[out=left,in=left] (-0.2,0.2) to[out=right,in=left] (0.2,-0.15) to[out=right,in=down] (0.4,0.4);
        \end{tikzpicture}
        \overset{\cref{intlaw}}{=}
        \begin{tikzpicture}[anchorbase]
            \draw (0.1,0.4) to[out=down,in=right] (-0.2,-0.15) to[out=left,in=left] (-0.2,0.2) to[out=right,in=left] (0.2,-0.3) to[out=right,in=down] (0.4,0.4);
        \end{tikzpicture}
        =
        \begin{tikzpicture}[anchorbase]
            \draw (0.5,0.1) to[out=down,in=right] (0.3,-0.4) to[out=left,in=down] (0,0) arc(180:0:0.15) \braiddown (0,-0.4) to[out=down,in=down] (0.8,-0.4) -- (0.8,0.1);
        \end{tikzpicture}
        = \
        \begin{tikzpicture}[anchorbase]
            \draw (0,0.3) -- (0,0) arc(360:180:0.15) arc(0:180:0.15) to[out=down,in=down,looseness=1.5] (0.3,0) -- (0.3,0.3);
        \end{tikzpicture}
        =  (-1)^\sigma\ \cupmor
        \, ,
    \]
    where, for the unadorned equalities, we use fourth relation in \cref{brauer}, the sixth relation in \cref{brauer}, the second relation in \cref{mirror}, the fifth relation in \cref{brauer}, and finally the fourth relation in \cref{brauer}.

    The third relation in \cref{mirror} follows from the fourth relation in \cref{tokrel} after composing on the top and bottom with the crossing $\crossmor$ and using the first relation in \cref{brauer}.  For the last relation in \cref{mirror}, we have
    \[
        \begin{tikzpicture}[anchorbase]
            \draw (-0.2,0.2) -- (-0.2,0) arc(180:360:0.2) -- (0.2,0.2);
            \token{0.2,0}{west}{a};
        \end{tikzpicture}
        \overset{\cref{brauer}}{=} (-1)^\sigma\
        \begin{tikzpicture}[anchorbase]
            \draw (-0.75,0.8) -- (-0.75,0) arc(180:360:0.25) -- ++(0,0.5) arc(180:0:0.25) arc(180:360:0.25) -- (0.75,0.8);
            \token{-0.25,0}{west}{a};
        \end{tikzpicture}
        \overset{\cref{intlaw}}{=} (-1)^{\sigma + \sigma \bar{a}}\
        \begin{tikzpicture}[anchorbase]
            \draw (-0.75,0.8) -- (-0.75,0) arc(180:360:0.25) -- ++(0,0.5) arc(180:0:0.25) arc(180:360:0.25) -- (0.75,0.8);
            \token{-0.25,0.5}{east}{a};
        \end{tikzpicture}
        \overset{\cref{tokrel}}{=} (-1)^{\sigma + \sigma \bar{a}}\
        \begin{tikzpicture}[anchorbase]
            \draw (-0.75,0.8) -- (-0.75,0) arc(180:360:0.25) -- ++(0,0.5) arc(180:0:0.25) arc(180:360:0.25) -- (0.75,0.8);
            \token{0.25,0.5}{west}{a^\inv};
        \end{tikzpicture}
        \overset{\cref{brauer}}{\underset{\cref{intlaw}}{=}}
        \begin{tikzpicture}[anchorbase]
            \draw (-0.2,0.2) -- (-0.2,0) arc (180:360:0.2) -- (0.2,0.2);
            \token{-0.2,0}{east}{a^\inv};
        \end{tikzpicture}
        \ .
        \qedhere
    \]
\end{proof}

It follows from the defining relations that ``bubbles'' are central in $\Brauer_\kk^\sigma(A,\inv)$:
\[
    \bubble{a}
    \begin{tikzpicture}[anchorbase]
        \draw (0,-0.3) -- (0,0.3);
    \end{tikzpicture}
    =
    \begin{tikzpicture}[anchorbase]
        \draw (0,-0.3) -- (0,0.3);
    \end{tikzpicture}
    \
    \bubble{a}
    \qquad \text{for all } a \in A.
\]
Note also that
\begin{equation} \label{Eugene}
    \bubble{ab}
    \overset{\cref{tokrel}}{=}
    \begin{tikzpicture}[centerzero]
        \draw (0.2,-0.15) -- (0.2,0.15) arc(0:180:0.2) -- (-0.2,-0.15) arc(180:360:0.2);
        \token{0.2,0.15}{west}{a};
        \token{0.2,-0.15}{west}{b};
    \end{tikzpicture}
    \overset{\cref{tokrel}}{=}
    \begin{tikzpicture}[centerzero]
        \draw (0.2,-0.15) -- (0.2,0.15) arc(0:180:0.2) -- (-0.2,-0.15) arc(180:360:0.2);
        \token{-0.2,0.15}{east}{a^\inv};
        \token{0.2,-0.15}{west}{b};
    \end{tikzpicture}
    \overset{\cref{intlaw}}{=} (-1)^{\bar{a}\bar{b}}\
    \begin{tikzpicture}[centerzero]
        \draw (0.2,-0.15) -- (0.2,0.15) arc(0:180:0.2) -- (-0.2,-0.15) arc(180:360:0.2);
        \token{-0.2,-0.15}{east}{a^\inv};
        \token{0.2,0.15}{west}{b};
    \end{tikzpicture}
    \overset{\cref{mirror}}{=} (-1)^{\bar{a}\bar{b}}\
    \begin{tikzpicture}[centerzero]
        \draw (0.2,-0.15) -- (0.2,0.15) arc(0:180:0.2) -- (-0.2,-0.15) arc(180:360:0.2);
        \token{0.2,0.15}{west}{b};
        \token{0.2,-0.15}{west}{a};
    \end{tikzpicture}
    \overset{\cref{tokrel}}{=} (-1)^{\bar{a}\bar{b}}\ \bubble{ba}
\end{equation}
and
\begin{equation} \label{Oregon}
    \bubble{a}
    \overset{\cref{mirror}}{=} (-1)^\sigma\
    \begin{tikzpicture}[centerzero]
        \draw (0.2,-0.2) \braidup (-0.2,0.2) arc(180:0:0.2) \braiddown (-0.2,-0.2) arc(180:360:0.2);
        \token{0.2,0.2}{west}{a};
    \end{tikzpicture}
    \overset{\cref{tokrel}}{=} (-1)^\sigma\
    \begin{tikzpicture}[centerzero]
        \draw (0.2,-0.2) \braidup (-0.2,0.2) arc(180:0:0.2) \braiddown (-0.2,-0.2) arc(180:360:0.2);
        \token{-0.2,-0.2}{east}{a};
    \end{tikzpicture}
    \overset{\cref{brauer}}{=} (-1)^\sigma
    \begin{tikzpicture}[centerzero]
        \draw (0.2,0) arc(0:360:0.2);
        \token{-0.2,0}{east}{a};
    \end{tikzpicture}
    \overset{\cref{tokrel}}{=}  (-1)^\sigma\ \bubble{a^\inv}.
\end{equation}

\begin{rem} \label{belgium}
    Suppose that $(A,\inv,\form)$ is an involutive Frobenius superalgebra, $\sigma=1$, and $d \ne 0$.  In $\Brauer_\kk^1(A,\inv;d)$, we have
    \[
        d \str_A(a) 1_\one
        = \bubble{a}
        \overset{\cref{Oregon}}{=} -\, \bubble{a^\inv}
        = - d \str_A(a^\inv) 1_\one
        \overset{\cref{snow}}{=} - d \str_A(a) 1_\one
    \]
    for all $a \in A$.  In applications to representation theory, we will assume that the characteristic of $\kk$ is not two. In this case, it follows that $\str_A=0$ or $1_\one = 0$.  In the former case, we have $\Brauer_\kk^1(A,\inv;d) = \Brauer_\kk^1(A,\inv;0)$.  In the latter case, $\Brauer_\kk^1(A,\inv;d)$ is the zero supercategory.  Thus, when $\sigma=1$ and $(A,\inv)$ can be endowed with the structure of an involutive Frobenius superalgebra, we will usually assume that $d=0$.
\end{rem}

For any $d \in \kk$, we have isomorphisms of monoidal supercategories
\begin{equation} \label{Xinv}
    \Xi_\inv \colon \Brauer_\kk^\sigma(A,\inv) \to \Brauer_\kk^\sigma(A^\op,\inv)
    \quad \text{and} \quad
    \Xi_\inv \colon \Brauer_\kk^\sigma(A,\inv;d) \to \Brauer_\kk^\sigma(A^\op,\inv;d),
\end{equation}
given by applying the involution $\inv$ to all tokens.

\subsection{The basis theorem\label{subsec:Bbasis}}

This subsection is dedicated to the proof of a basis theorem giving bases for the morphisms spaces of the $\Brauer_\kk^\sigma(A,\inv;d)$.  Our method involves embedding this supercategory into the superadditive envelope of the \emph{oriented} supercategory $\OB_\kk(A)$.  Even in the case $A=\kk$, when $\Brauer_\kk^0(\kk,\id;d)$ is the usual Brauer category, this method of proof is new.

Recall, from \cref{sec:monsupcat}, that, for a monoidal supercategory $\cC$, we let $\Add(\cC_\pi)$ denote its superadditive envelope.  The objects of $\Add(\cC_\pi)$ are formal direct sums of objects of the $\Pi$-envelope $\cC_\pi$, and morphisms in $\Add(\cC_\pi)$ are matrices of morphisms in $\cC_\pi$, which we will write as sums of their components.  For example,
\[
    \begin{tikzpicture}[centerzero]
        \draw[->] (-0.2,-0.2) -- (0.2,0.2);
        \draw[->] (0.2,-0.2) -- (-0.2,0.2);
        \shiftline{-0.3,0.2}{0.3,0.2}{0};
        \shiftline{-0.3,-0.2}{0.3,-0.2}{0};
    \end{tikzpicture}
    +
    \begin{tikzpicture}[centerzero]
        \draw[->] (-0.2,-0.2) -- (0.2,0.2);
        \draw[<-] (0.2,-0.2) -- (-0.2,0.2);
        \shiftline{-0.3,0.2}{0.3,0.2}{\sigma};
        \shiftline{-0.3,-0.2}{0.3,-0.2}{\sigma};
    \end{tikzpicture}
    +
    \begin{tikzpicture}[centerzero]
        \draw[<-] (-0.2,-0.2) -- (0.2,0.2);
        \draw[->] (0.2,-0.2) -- (-0.2,0.2);
        \shiftline{-0.3,0.2}{0.3,0.2}{\sigma};
        \shiftline{-0.3,-0.2}{0.3,-0.2}{\sigma};
    \end{tikzpicture}
    + (-1)^\sigma\,
    \begin{tikzpicture}[centerzero]
        \draw[<-] (-0.2,-0.2) -- (0.2,0.2);
        \draw[<-] (0.2,-0.2) -- (-0.2,0.2);
        \shiftline{-0.3,0.2}{0.3,0.2}{0};
        \shiftline{-0.3,-0.2}{0.3,-0.2}{0};
    \end{tikzpicture}
    \colon (\upobj \oplus \Pi^\sigma \downobj) \otimes (\upobj \oplus \Pi^\sigma \downobj) \to (\upobj \oplus \Pi^\sigma \downobj) \otimes (\upobj \oplus \Pi^\sigma \downobj)
\]
is a morphism in $\Add(\OB_\kk(A))$ with components
\begin{align*}
    \begin{tikzpicture}[centerzero]
        \draw[->] (-0.2,-0.2) -- (0.2,0.2);
        \draw[->] (0.2,-0.2) -- (-0.2,0.2);
        \shiftline{-0.3,0.2}{0.3,0.2}{0};
        \shiftline{-0.3,-0.2}{0.3,-0.2}{0};
    \end{tikzpicture}
    \colon \upobj \otimes \upobj &\to \upobj \otimes \upobj,
    &
    \begin{tikzpicture}[centerzero]
        \draw[->] (-0.2,-0.2) -- (0.2,0.2);
        \draw[<-] (0.2,-0.2) -- (-0.2,0.2);
        \shiftline{-0.3,0.2}{0.3,0.2}{\sigma};
        \shiftline{-0.3,-0.2}{0.3,-0.2}{\sigma};
    \end{tikzpicture}
    \colon \Pi^\sigma \upobj \otimes \downobj &\to \Pi^\sigma \downobj \otimes \upobj,
    \\
    \begin{tikzpicture}[centerzero]
        \draw[<-] (-0.2,-0.2) -- (0.2,0.2);
        \draw[->] (0.2,-0.2) -- (-0.2,0.2);
        \shiftline{-0.3,0.2}{0.3,0.2}{\sigma};
        \shiftline{-0.3,-0.2}{0.3,-0.2}{\sigma};
    \end{tikzpicture}
    \colon \Pi^\sigma \downobj \otimes \upobj &\to \Pi^\sigma \upobj \otimes \downobj,
    &
    (-1)^\sigma\,
    \begin{tikzpicture}[centerzero]
        \draw[<-] (-0.2,-0.2) -- (0.2,0.2);
        \draw[<-] (0.2,-0.2) -- (-0.2,0.2);
        \shiftline{-0.3,0.2}{0.3,0.2}{0};
        \shiftline{-0.3,-0.2}{0.3,-0.2}{0};
    \end{tikzpicture}
    \colon \downobj \otimes \downobj &\to \downobj \otimes \downobj,
\end{align*}
and all other components equal to zero.

\begin{theo} \label{bulb}
    Fix $\sigma \in \Z_2$.  There exists a unique monoidal superfunctor $\sD \colon \Brauer_\kk^\sigma(A,\inv) \to \Add(\OB_\kk(A)_\pi)$ such that $\sD(\go) = \upobj \oplus \Pi^\sigma \downobj$ and
    \begin{gather*}
        \sD \left( \crossmor \right) =
        \begin{tikzpicture}[centerzero]
            \draw[->] (-0.2,-0.2) -- (0.2,0.2);
            \draw[->] (0.2,-0.2) -- (-0.2,0.2);
            \shiftline{-0.3,0.2}{0.3,0.2}{0};
            \shiftline{-0.3,-0.2}{0.3,-0.2}{0};
        \end{tikzpicture}
        +
        \begin{tikzpicture}[centerzero]
            \draw[->] (-0.2,-0.2) -- (0.2,0.2);
            \draw[<-] (0.2,-0.2) -- (-0.2,0.2);
            \shiftline{-0.3,0.2}{0.3,0.2}{\sigma};
            \shiftline{-0.3,-0.2}{0.3,-0.2}{\sigma};
        \end{tikzpicture}
        +
        \begin{tikzpicture}[centerzero]
            \draw[<-] (-0.2,-0.2) -- (0.2,0.2);
            \draw[->] (0.2,-0.2) -- (-0.2,0.2);
            \shiftline{-0.3,0.2}{0.3,0.2}{\sigma};
            \shiftline{-0.3,-0.2}{0.3,-0.2}{\sigma};
        \end{tikzpicture}
        + (-1)^\sigma\,
        \begin{tikzpicture}[centerzero]
            \draw[<-] (-0.2,-0.2) -- (0.2,0.2);
            \draw[<-] (0.2,-0.2) -- (-0.2,0.2);
            \shiftline{-0.3,0.2}{0.3,0.2}{0};
            \shiftline{-0.3,-0.2}{0.3,-0.2}{0};
        \end{tikzpicture}
        ,
        \\
        \sD \left( \capmor \right) =
        \begin{tikzpicture}[anchorbase]
            \draw[<-] (-0.15,-0.15) -- (-0.15,0) arc(180:0:0.15) -- (0.15,-0.15);
            \shiftline{-0.25,0.25}{0.25,0.25}{0};
            \shiftline{-0.25,-0.15}{0.25,-0.15}{\sigma};
        \end{tikzpicture}
        +
        \begin{tikzpicture}[anchorbase]
            \draw[->] (-0.15,-0.15) -- (-0.15,0) arc(180:0:0.15) -- (0.15,-0.15);
            \shiftline{-0.25,0.25}{0.25,0.25}{0};
            \shiftline{-0.25,-0.15}{0.25,-0.15}{\sigma};
        \end{tikzpicture}
        ,\qquad
        \sD \left( \cupmor \right) =
        \begin{tikzpicture}[anchorbase]
            \draw[<-] (-0.15,0.15) -- (-0.15,0) arc(180:360:0.15) -- (0.15,0.15);
            \shiftline{-0.25,0.15}{0.25,0.15}{\sigma};
            \shiftline{-0.25,-0.25}{0.25,-0.25}{0};
        \end{tikzpicture}
        + (-1)^\sigma
        \begin{tikzpicture}[anchorbase]
            \draw[->] (-0.15,0.15) -- (-0.15,0) arc(180:360:0.15) -- (0.15,0.15);
            \shiftline{-0.25,0.15}{0.25,0.15}{\sigma};
            \shiftline{-0.25,-0.25}{0.25,-0.25}{0};
        \end{tikzpicture}
        ,
        \\
        \sD \left( \tokstrand \right) =
        \begin{tikzpicture}[centerzero]
            \draw[->] (0,-0.2) -- (0,0.2);
            \token{0,0}{east}{a};
            \shiftline{-0.1,0.2}{0.1,0.2}{0};
            \shiftline{-0.1,-0.2}{0.1,-0.2}{0};
        \end{tikzpicture}
        + (-1)^{\sigma\bar{a}}
        \begin{tikzpicture}[centerzero]
            \draw[<-] (0,-0.2) -- (0,0.2);
            \token{0,0}{east}{a^\inv};
            \shiftline{-0.1,0.2}{0.1,0.2}{\sigma};
            \shiftline{-0.1,-0.2}{0.1,-0.2}{\sigma};
        \end{tikzpicture}
        ,\quad a \in A.
    \end{gather*}
    If $(A,\inv)$ can be endowed with the structure of an involutive Frobenius superalgebra, then, for all $d \in \kk$, this induces a monoidal superfunctor
    \[
        \sD \colon \Brauer_\kk^\sigma(A,\inv;2d) \to \Add(\OB_\kk(A;d)_\pi),
    \]
    where we assume that $d=0$ if $\sigma=1$.  (See \cref{belgium}.)
\end{theo}

\begin{proof}
    We must verify that $\sD$ respects the defining relations of $\Brauer_\kk(A,\inv)$ from \cref{FBC}.  For the first relation in \cref{brauer}, we compute
    \[
        \sD
        \left(
            \begin{tikzpicture}[centerzero]
                \draw (0.2,-0.4) to[out=135,in=down] (-0.15,0) to[out=up,in=225] (0.2,0.4);
                \draw (-0.2,-0.4) to[out=45,in=down] (0.15,0) to[out=up,in=-45] (-0.2,0.4);
            \end{tikzpicture}
        \right)
        =
        \begin{tikzpicture}[centerzero]
            \draw[->] (0.2,-0.4) to[out=135,in=down] (-0.15,0) to[out=up,in=225] (0.2,0.4);
            \draw[->] (-0.2,-0.4) to[out=45,in=down] (0.15,0) to[out=up,in=-45] (-0.2,0.4);
            \shiftline{-0.3,0.4}{0.3,0.4}{0};
            \shiftline{-0.3,-0.4}{0.3,-0.4}{0};
        \end{tikzpicture}
        +
        \begin{tikzpicture}[centerzero]
            \draw[->] (0.2,-0.4) to[out=135,in=down] (-0.15,0) to[out=up,in=225] (0.2,0.4);
            \draw[<-] (-0.2,-0.4) to[out=45,in=down] (0.15,0) to[out=up,in=-45] (-0.2,0.4);
            \shiftline{-0.3,0.4}{0.3,0.4}{\sigma};
            \shiftline{-0.3,-0.4}{0.3,-0.4}{\sigma};
        \end{tikzpicture}
        +
        \begin{tikzpicture}[centerzero]
            \draw[<-] (0.2,-0.4) to[out=135,in=down] (-0.15,0) to[out=up,in=225] (0.2,0.4);
            \draw[->] (-0.2,-0.4) to[out=45,in=down] (0.15,0) to[out=up,in=-45] (-0.2,0.4);
            \shiftline{-0.3,0.4}{0.3,0.4}{\sigma};
            \shiftline{-0.3,-0.4}{0.3,-0.4}{\sigma};
        \end{tikzpicture}
        +
        \begin{tikzpicture}[centerzero]
            \draw[<-] (0.2,-0.4) to[out=135,in=down] (-0.15,0) to[out=up,in=225] (0.2,0.4);
            \draw[<-] (-0.2,-0.4) to[out=45,in=down] (0.15,0) to[out=up,in=-45] (-0.2,0.4);
            \shiftline{-0.3,0.4}{0.3,0.4}{0};
            \shiftline{-0.3,-0.4}{0.3,-0.4}{0};
        \end{tikzpicture}
        =
        \begin{tikzpicture}[centerzero]
            \draw[->] (0.2,-0.4) -- (0.2,0.4);
            \draw[->] (-0.2,-0.4) -- (-0.2,0.4);
            \shiftline{-0.3,0.4}{0.3,0.4}{0};
            \shiftline{-0.3,-0.4}{0.3,-0.4}{0};
        \end{tikzpicture}
        +
        \begin{tikzpicture}[centerzero]
            \draw[->] (0.2,-0.4) -- (0.2,0.4);
            \draw[<-] (-0.2,-0.4) -- (-0.2,0.4);
            \shiftline{-0.3,0.4}{0.3,0.4}{\sigma};
            \shiftline{-0.3,-0.4}{0.3,-0.4}{\sigma};
        \end{tikzpicture}
        +
        \begin{tikzpicture}[centerzero]
            \draw[<-] (0.2,-0.4) -- (0.2,0.4);
            \draw[->] (-0.2,-0.4) -- (-0.2,0.4);
            \shiftline{-0.3,0.4}{0.3,0.4}{\sigma};
            \shiftline{-0.3,-0.4}{0.3,-0.4}{\sigma};
        \end{tikzpicture}
        +
        \begin{tikzpicture}[centerzero]
            \draw[<-] (0.2,-0.4) -- (0.2,0.4);
            \draw[<-] (-0.2,-0.4) -- (-0.2,0.4);
            \shiftline{-0.3,0.4}{0.3,0.4}{0};
            \shiftline{-0.3,-0.4}{0.3,-0.4}{0};
        \end{tikzpicture}
        = \sD
        \left(
            \begin{tikzpicture}[centerzero]
                \draw (-0.2,-0.4) -- (-0.2,0.4);
                \draw (0.2,-0.4) -- (0.2,0.4);
            \end{tikzpicture}
        \right).
    \]
    The proof of the second relation in \cref{brauer} is similar.

    For the third relation in \cref{brauer}, we have
    \begin{align*}
        \sD
        \left(
            \begin{tikzpicture}[centerzero]
                \draw (-0.3,0.4) -- (-0.3,0) arc(180:360:0.15) arc(180:0:0.15) -- (0.3,-0.4);
            \end{tikzpicture}
        \right)
        &=
        \sD ( \, \idstrand \otimes \capmor ) \circ \sD(\cupmor \otimes \idstrand \, )
        \\
        &\overset{\mathclap{\cref{slush}}}{=}\
        \left(
            \begin{tikzpicture}[anchorbase]
                \draw[->] (-0.2,-0.1) -- (-0.2,0.3);
                \draw[<-] (0,-0.1) -- (0,0) arc(180:0:0.15) -- (0.3,-0.1);
                \shiftline{-0.3,0.3}{0.4,0.3}{0};
                \shiftline{-0.3,-0.1}{0.4,-0.1}{\sigma};
            \end{tikzpicture}
            +
            \begin{tikzpicture}[anchorbase]
                \draw[->] (-0.2,-0.1) -- (-0.2,0.3);
                \draw[->] (0,-0.1) -- (0,0) arc(180:0:0.15) -- (0.3,-0.1);
                \shiftline{-0.3,0.3}{0.4,0.3}{0};
                \shiftline{-0.3,-0.1}{0.4,-0.1}{\sigma};
            \end{tikzpicture}
            + (-1)^\sigma\,
            \begin{tikzpicture}[anchorbase]
                \draw[<-] (-0.2,-0.1) -- (-0.2,0.3);
                \draw[<-] (0,-0.1) -- (0,0) arc(180:0:0.15) -- (0.3,-0.1);
                \shiftline{-0.3,0.3}{0.4,0.3}{\sigma};
                \shiftline{-0.3,-0.1}{0.4,-0.1}{0};
            \end{tikzpicture}
            + (-1)^\sigma\,
            \begin{tikzpicture}[anchorbase]
                \draw[<-] (-0.2,-0.1) -- (-0.2,0.3);
                \draw[->] (0,-0.1) -- (0,0) arc(180:0:0.15) -- (0.3,-0.1);
                \shiftline{-0.3,0.3}{0.4,0.3}{\sigma};
                \shiftline{-0.3,-0.1}{0.4,-0.1}{0};
            \end{tikzpicture}
        \right)
        \\ &\qquad \qquad
        \circ
        \left(
            \begin{tikzpicture}[anchorbase]
                \draw[<-] (-0.15,0.1) -- (-0.15,0) arc(180:360:0.15) -- (0.15,0.1);
                \draw[->] (0.35,-0.3) -- (0.35,0.1);
                \shiftline{-0.25,0.1}{0.45,0.1}{\sigma};
                \shiftline{-0.25,-0.3}{0.45,-0.3}{0};
            \end{tikzpicture}
            + (-1)^\sigma\,
            \begin{tikzpicture}[anchorbase]
                \draw[->] (-0.15,0.1) -- (-0.15,0) arc(180:360:0.15) -- (0.15,0.1);
                \draw[->] (0.35,-0.3) -- (0.35,0.1);
                \shiftline{-0.25,0.1}{0.45,0.1}{\sigma};
                \shiftline{-0.25,-0.3}{0.45,-0.3}{0};
            \end{tikzpicture}
            +
            \begin{tikzpicture}[anchorbase]
                \draw[<-] (-0.15,0.1) -- (-0.15,0) arc(180:360:0.15) -- (0.15,0.1);
                \draw[<-] (0.35,-0.3) -- (0.35,0.1);
                \shiftline{-0.25,0.1}{0.45,0.1}{0};
                \shiftline{-0.25,-0.3}{0.45,-0.3}{\sigma};
            \end{tikzpicture}
            + (-1)^\sigma\,
            \begin{tikzpicture}[anchorbase]
                \draw[->] (-0.15,0.1) -- (-0.15,0) arc(180:360:0.15) -- (0.15,0.1);
                \draw[<-] (0.35,-0.3) -- (0.35,0.1);
                \shiftline{-0.25,0.1}{0.45,0.1}{0};
                \shiftline{-0.25,-0.3}{0.45,-0.3}{\sigma};
            \end{tikzpicture}
        \right)
        \\
        &=
        \begin{tikzpicture}[centerzero]
            \draw[<-] (-0.3,0.4) -- (-0.3,0) arc(180:360:0.15) arc(180:0:0.15) -- (0.3,-0.4);
            \shiftline{-0.4,0.4}{0.4,0.4}{0};
            \shiftline{-0.4,-0.4}{0.4,-0.4}{0};
        \end{tikzpicture}
        +
        \begin{tikzpicture}[centerzero]
            \draw[->] (-0.3,0.4) -- (-0.3,0) arc(180:360:0.15) arc(180:0:0.15) -- (0.3,-0.4);
            \shiftline{-0.4,0.4}{0.4,0.4}{\sigma};
            \shiftline{-0.4,-0.4}{0.4,-0.4}{\sigma};
        \end{tikzpicture}
        =
        \begin{tikzpicture}[centerzero]
            \draw[->] (0,-0.4) -- (0,0.4);
            \shiftline{-0.2,0.4}{0.2,0.4}{0};
            \shiftline{-0.2,-0.4}{0.2,-0.4}{0};
        \end{tikzpicture}
        +
        \begin{tikzpicture}[centerzero]
            \draw[<-] (0,-0.4) -- (0,0.4);
            \shiftline{-0.2,0.4}{0.2,0.4}{\sigma};
            \shiftline{-0.2,-0.4}{0.2,-0.4}{\sigma};
        \end{tikzpicture}
        = \sD
        \left(\
            \begin{tikzpicture}[centerzero]
                \draw (0,-0.4) -- (0,0.4);
            \end{tikzpicture}\
        \right).
    \end{align*}
    Similarly, for the fourth relation in \cref{brauer}, we have
    \begin{align*}
        \sD
        \left(
            \begin{tikzpicture}[centerzero]
                \draw (-0.3,-0.4) -- (-0.3,0) arc(180:0:0.15) arc(180:360:0.15) -- (0.3,0.4);
            \end{tikzpicture}
        \right)
        &=
        \sD ( \capmor \otimes \idstrand\, ) \circ \sD(\, \idstrand \otimes \cupmor )
        \\
        &\overset{\mathclap{\cref{slush}}}{=}\
        \left(
            \begin{tikzpicture}[anchorbase]
                \draw[<-] (-0.3,-0.1) -- (-0.3,0) arc(180:0:0.15) -- (0,-0.1);
                \draw[->] (0.2,-0.1) -- (0.2,0.3);
                \shiftline{-0.4,0.3}{0.3,0.3}{0};
                \shiftline{-0.4,-0.1}{0.3,-0.1}{\sigma};
            \end{tikzpicture}
            +
            \begin{tikzpicture}[anchorbase]
                \draw[->] (-0.3,-0.1) -- (-0.3,0) arc(180:0:0.15) -- (0,-0.1);
                \draw[->] (0.2,-0.1) -- (0.2,0.3);
                \shiftline{-0.4,0.3}{0.3,0.3}{0};
                \shiftline{-0.4,-0.1}{0.3,-0.1}{\sigma};
            \end{tikzpicture}
            +
            \begin{tikzpicture}[anchorbase]
                \draw[<-] (-0.3,-0.1) -- (-0.3,0) arc(180:0:0.15) -- (0,-0.1);
                \draw[<-] (0.2,-0.1) -- (0.2,0.3);
                \shiftline{-0.4,0.3}{0.3,0.3}{\sigma};
                \shiftline{-0.4,-0.1}{0.3,-0.1}{0};
            \end{tikzpicture}
            +
            \begin{tikzpicture}[anchorbase]
                \draw[->] (-0.3,-0.1) -- (-0.3,0) arc(180:0:0.15) -- (0,-0.1);
                \draw[<-] (0.2,-0.1) -- (0.2,0.3);
                \shiftline{-0.4,0.3}{0.3,0.3}{\sigma};
                \shiftline{-0.4,-0.1}{0.3,-0.1}{0};
            \end{tikzpicture}
        \right)
        \\ &\qquad \qquad
        \circ
        \left(
            \begin{tikzpicture}[anchorbase]
                \draw[<-] (-0.15,0.1) -- (-0.15,0) arc(180:360:0.15) -- (0.15,0.1);
                \draw[->] (-0.35,-0.3) -- (-0.35,0.1);
                \shiftline{-0.45,0.1}{0.25,0.1}{\sigma};
                \shiftline{-0.45,-0.3}{0.25,-0.3}{0};
            \end{tikzpicture}
            + (-1)^\sigma\,
            \begin{tikzpicture}[anchorbase]
                \draw[->] (-0.15,0.1) -- (-0.15,0) arc(180:360:0.15) -- (0.15,0.1);
                \draw[->] (-0.35,-0.3) -- (-0.35,0.1);
                \shiftline{-0.45,0.1}{0.25,0.1}{\sigma};
                \shiftline{-0.45,-0.3}{0.25,-0.3}{0};
            \end{tikzpicture}
            + (-1)^\sigma\,
            \begin{tikzpicture}[anchorbase]
                \draw[<-] (-0.15,0.1) -- (-0.15,0) arc(180:360:0.15) -- (0.15,0.1);
                \draw[<-] (-0.35,-0.3) -- (-0.35,0.1);
                \shiftline{-0.45,0.1}{0.25,0.1}{0};
                \shiftline{-0.45,-0.3}{0.25,-0.3}{\sigma};
            \end{tikzpicture}
            +
            \begin{tikzpicture}[anchorbase]
                \draw[->] (-0.15,0.1) -- (-0.15,0) arc(180:360:0.15) -- (0.15,0.1);
                \draw[<-] (-0.35,-0.3) -- (-0.35,0.1);
                \shiftline{-0.45,0.1}{0.25,0.1}{0};
                \shiftline{-0.45,-0.3}{0.25,-0.3}{\sigma};
            \end{tikzpicture}
        \right)
        \\
        &= (-1)^\sigma
        \left(
            \begin{tikzpicture}[centerzero]
                \draw[->] (-0.3,-0.4) -- (-0.3,0) arc(180:0:0.15) arc(180:360:0.15) -- (0.3,0.4);
                \shiftline{-0.4,0.4}{0.4,0.4}{0};
                \shiftline{-0.4,-0.4}{0.4,-0.4}{0};
            \end{tikzpicture}
            +
            \begin{tikzpicture}[centerzero]
                \draw[<-] (-0.3,-0.4) -- (-0.3,0) arc(180:0:0.15) arc(180:360:0.15) -- (0.3,0.4);
                \shiftline{-0.4,0.4}{0.4,0.4}{\sigma};
                \shiftline{-0.4,-0.4}{0.4,-0.4}{\sigma};
            \end{tikzpicture}
        \right)
        = (-1)^\sigma
        \left(
            \begin{tikzpicture}[centerzero]
                \draw[->] (0,-0.4) -- (0,0.4);
                \shiftline{-0.2,0.4}{0.2,0.4}{0};
                \shiftline{-0.2,-0.4}{0.2,-0.4}{0};
            \end{tikzpicture}
            +
            \begin{tikzpicture}[centerzero]
                \draw[<-] (0,-0.4) -- (0,0.4);
                \shiftline{-0.2,0.4}{0.2,0.4}{\sigma};
                \shiftline{-0.2,-0.4}{0.2,-0.4}{\sigma};
            \end{tikzpicture}
        \right)
        = (-1)^\sigma \sD
        \left(\
            \begin{tikzpicture}[centerzero]
                \draw (0,-0.4) -- (0,0.4);
            \end{tikzpicture}\
        \right).
    \end{align*}

    The fifth relation in \cref{brauer} is straightforward.  For the sixth relation in \cref{brauer}, we compute
    \begin{align*}
        \sD
        &\left(
            \begin{tikzpicture}[centerzero]
                \draw (-0.2,-0.3) -- (-0.2,-0.1) arc(180:0:0.2) -- (0.2,-0.3);
                \draw (-0.3,0.3) \braiddown (0,-0.3);
            \end{tikzpicture}
        \right)
        = \sD(\, \idstrand\, \otimes \capmor) \circ \sD ( \crossmor \otimes \, \idstrand\, )
        \\
        &\overset{\mathclap{\cref{slush}}}{=}\
        \left(
            \begin{tikzpicture}[anchorbase]
                \draw[->] (-0.2,-0.1) -- (-0.2,0.3);
                \draw[<-] (0,-0.1) -- (0,0) arc(180:0:0.15) -- (0.3,-0.1);
                \shiftline{-0.3,0.3}{0.4,0.3}{0};
                \shiftline{-0.3,-0.1}{0.4,-0.1}{\sigma};
            \end{tikzpicture}
            +
            \begin{tikzpicture}[anchorbase]
                \draw[->] (-0.2,-0.1) -- (-0.2,0.3);
                \draw[->] (0,-0.1) -- (0,0) arc(180:0:0.15) -- (0.3,-0.1);
                \shiftline{-0.3,0.3}{0.4,0.3}{0};
                \shiftline{-0.3,-0.1}{0.4,-0.1}{\sigma};
            \end{tikzpicture}
            + (-1)^\sigma\,
            \begin{tikzpicture}[anchorbase]
                \draw[<-] (-0.2,-0.1) -- (-0.2,0.3);
                \draw[<-] (0,-0.1) -- (0,0) arc(180:0:0.15) -- (0.3,-0.1);
                \shiftline{-0.3,0.3}{0.4,0.3}{\sigma};
                \shiftline{-0.3,-0.1}{0.4,-0.1}{0};
            \end{tikzpicture}
            + (-1)^\sigma\,
            \begin{tikzpicture}[anchorbase]
                \draw[<-] (-0.2,-0.1) -- (-0.2,0.3);
                \draw[->] (0,-0.1) -- (0,0) arc(180:0:0.15) -- (0.3,-0.1);
                \shiftline{-0.3,0.3}{0.4,0.3}{\sigma};
                \shiftline{-0.3,-0.1}{0.4,-0.1}{0};
            \end{tikzpicture}
        \right)
        \\ &\quad
        \circ
        \left(
            \begin{tikzpicture}[centerzero]
                \draw[->] (-0.4,-0.2) -- (0,0.2);
                \draw[->] (0,-0.2) -- (-0.4,0.2);
                \draw[->] (0.2,-0.2) -- (0.2,0.2);
                \shiftline{-0.5,0.2}{0.3,0.2}{0};
                \shiftline{-0.5,-0.2}{0.3,-0.2}{0};
            \end{tikzpicture}
            +
            \begin{tikzpicture}[centerzero]
                \draw[->] (-0.4,-0.2) -- (0,0.2);
                \draw[<-] (0,-0.2) -- (-0.4,0.2);
                \draw[->] (0.2,-0.2) -- (0.2,0.2);
                \shiftline{-0.5,0.2}{0.3,0.2}{\sigma};
                \shiftline{-0.5,-0.2}{0.3,-0.2}{\sigma};
            \end{tikzpicture}
            +
            \begin{tikzpicture}[centerzero]
                \draw[<-] (-0.4,-0.2) -- (0,0.2);
                \draw[->] (0,-0.2) -- (-0.4,0.2);
                \draw[->] (0.2,-0.2) -- (0.2,0.2);
                \shiftline{-0.5,0.2}{0.3,0.2}{\sigma};
                \shiftline{-0.5,-0.2}{0.3,-0.2}{\sigma};
            \end{tikzpicture}
            + (-1)^\sigma\,
            \begin{tikzpicture}[centerzero]
                \draw[<-] (-0.4,-0.2) -- (0,0.2);
                \draw[<-] (0,-0.2) -- (-0.4,0.2);
                \draw[->] (0.2,-0.2) -- (0.2,0.2);
                \shiftline{-0.5,0.2}{0.3,0.2}{0};
                \shiftline{-0.5,-0.2}{0.3,-0.2}{0};
            \end{tikzpicture}
            +
            \begin{tikzpicture}[centerzero]
                \draw[->] (-0.4,-0.2) -- (0,0.2);
                \draw[->] (0,-0.2) -- (-0.4,0.2);
                \draw[<-] (0.2,-0.2) -- (0.2,0.2);
                \shiftline{-0.5,0.2}{0.3,0.2}{\sigma};
                \shiftline{-0.5,-0.2}{0.3,-0.2}{\sigma};
            \end{tikzpicture}
            +
            \begin{tikzpicture}[centerzero]
                \draw[->] (-0.4,-0.2) -- (0,0.2);
                \draw[<-] (0,-0.2) -- (-0.4,0.2);
                \draw[<-] (0.2,-0.2) -- (0.2,0.2);
                \shiftline{-0.5,0.2}{0.3,0.2}{0};
                \shiftline{-0.5,-0.2}{0.3,-0.2}{0};
            \end{tikzpicture}
            +
            \begin{tikzpicture}[centerzero]
                \draw[<-] (-0.4,-0.2) -- (0,0.2);
                \draw[->] (0,-0.2) -- (-0.4,0.2);
                \draw[<-] (0.2,-0.2) -- (0.2,0.2);
                \shiftline{-0.5,0.2}{0.3,0.2}{0};
                \shiftline{-0.5,-0.2}{0.3,-0.2}{0};
            \end{tikzpicture}
            + (-1)^\sigma\,
            \begin{tikzpicture}[centerzero]
                \draw[<-] (-0.4,-0.2) -- (0,0.2);
                \draw[<-] (0,-0.2) -- (-0.4,0.2);
                \draw[<-] (0.2,-0.2) -- (0.2,0.2);
                \shiftline{-0.5,0.2}{0.3,0.2}{\sigma};
                \shiftline{-0.5,-0.2}{0.3,-0.2}{\sigma};
            \end{tikzpicture}
        \right)
        \\
        &=
        \begin{tikzpicture}[centerzero]
            \draw[<-] (-0.2,-0.3) -- (-0.2,-0.1) arc(180:0:0.2) -- (0.2,-0.3);
            \draw[<-] (-0.3,0.3) \braiddown (0,-0.3);
            \shiftline{-0.4,0.3}{0.3,0.3}{0};
            \shiftline{-0.4,-0.3}{0.3,-0.3}{\sigma};
        \end{tikzpicture}
        +
        \begin{tikzpicture}[centerzero]
            \draw[->] (-0.2,-0.3) -- (-0.2,-0.1) arc(180:0:0.2) -- (0.2,-0.3);
            \draw[<-] (-0.3,0.3) \braiddown (0,-0.3);
            \shiftline{-0.4,0.3}{0.3,0.3}{0};
            \shiftline{-0.4,-0.3}{0.3,-0.3}{\sigma};
        \end{tikzpicture}
        +
        \begin{tikzpicture}[centerzero]
            \draw[<-] (-0.2,-0.3) -- (-0.2,-0.1) arc(180:0:0.2) -- (0.2,-0.3);
            \draw[->] (-0.3,0.3) \braiddown (0,-0.3);
            \shiftline{-0.4,0.3}{0.3,0.3}{\sigma};
            \shiftline{-0.4,-0.3}{0.3,-0.3}{0};
        \end{tikzpicture}
        + (-1)^\sigma\,
        \begin{tikzpicture}[centerzero]
            \draw[->] (-0.2,-0.3) -- (-0.2,-0.1) arc(180:0:0.2) -- (0.2,-0.3);
            \draw[->] (-0.3,0.3) \braiddown (0,-0.3);
            \shiftline{-0.4,0.3}{0.3,0.3}{\sigma};
            \shiftline{-0.4,-0.3}{0.3,-0.3}{0};
        \end{tikzpicture}
    \end{align*}
    and
    \begin{align*}
        \sD
        &\left(
            \begin{tikzpicture}[centerzero]
                \draw (-0.2,-0.3) -- (-0.2,-0.1) arc(180:0:0.2) -- (0.2,-0.3);
                \draw (0.3,0.3) \braiddown (0,-0.3);
            \end{tikzpicture}
        \right)
        = \sD(\capmor \otimes \, \idstrand\, ) \circ \sD ( \, \idstrand\, \otimes \crossmor )
        \\
        &\overset{\mathclap{\cref{slush}}}{=}\
        \left(
            \begin{tikzpicture}[anchorbase]
                \draw[<-] (-0.3,-0.1) -- (-0.3,0) arc(180:0:0.15) -- (0,-0.1);
                \draw[->] (0.2,-0.1) -- (0.2,0.3);
                \shiftline{-0.4,0.3}{0.3,0.3}{0};
                \shiftline{-0.4,-0.1}{0.3,-0.1}{\sigma};
            \end{tikzpicture}
            +
            \begin{tikzpicture}[anchorbase]
                \draw[->] (-0.3,-0.1) -- (-0.3,0) arc(180:0:0.15) -- (0,-0.1);
                \draw[->] (0.2,-0.1) -- (0.2,0.3);
                \shiftline{-0.4,0.3}{0.3,0.3}{0};
                \shiftline{-0.4,-0.1}{0.3,-0.1}{\sigma};
            \end{tikzpicture}
            +
            \begin{tikzpicture}[anchorbase]
                \draw[<-] (-0.3,-0.1) -- (-0.3,0) arc(180:0:0.15) -- (0,-0.1);
                \draw[<-] (0.2,-0.1) -- (0.2,0.3);
                \shiftline{-0.4,0.3}{0.3,0.3}{\sigma};
                \shiftline{-0.4,-0.1}{0.3,-0.1}{0};
            \end{tikzpicture}
            +
            \begin{tikzpicture}[anchorbase]
                \draw[->] (-0.3,-0.1) -- (-0.3,0) arc(180:0:0.15) -- (0,-0.1);
                \draw[<-] (0.2,-0.1) -- (0.2,0.3);
                \shiftline{-0.4,0.3}{0.3,0.3}{\sigma};
                \shiftline{-0.4,-0.1}{0.3,-0.1}{0};
            \end{tikzpicture}
        \right)
        \\ &\quad
        \circ
        \left(
            \begin{tikzpicture}[centerzero]
                \draw[->] (-0.2,-0.2) -- (0.2,0.2);
                \draw[->] (0.2,-0.2) -- (-0.2,0.2);
                \draw[->] (-0.4,-0.2) -- (-0.4,0.2);
                \shiftline{-0.5,0.2}{0.3,0.2}{0};
                \shiftline{-0.5,-0.2}{0.3,-0.2}{0};
            \end{tikzpicture}
            +
            \begin{tikzpicture}[centerzero]
                \draw[->] (-0.2,-0.2) -- (0.2,0.2);
                \draw[<-] (0.2,-0.2) -- (-0.2,0.2);
                \draw[->] (-0.4,-0.2) -- (-0.4,0.2);
                \shiftline{-0.5,0.2}{0.3,0.2}{\sigma};
                \shiftline{-0.5,-0.2}{0.3,-0.2}{\sigma};
            \end{tikzpicture}
            +
            \begin{tikzpicture}[centerzero]
                \draw[<-] (-0.2,-0.2) -- (0.2,0.2);
                \draw[->] (0.2,-0.2) -- (-0.2,0.2);
                \draw[->] (-0.4,-0.2) -- (-0.4,0.2);
                \shiftline{-0.5,0.2}{0.3,0.2}{\sigma};
                \shiftline{-0.5,-0.2}{0.3,-0.2}{\sigma};
            \end{tikzpicture}
            + (-1)^\sigma\,
            \begin{tikzpicture}[centerzero]
                \draw[<-] (-0.2,-0.2) -- (0.2,0.2);
                \draw[<-] (0.2,-0.2) -- (-0.2,0.2);
                \draw[->] (-0.4,-0.2) -- (-0.4,0.2);
                \shiftline{-0.5,0.2}{0.3,0.2}{0};
                \shiftline{-0.5,-0.2}{0.3,-0.2}{0};
            \end{tikzpicture}
            +
            \begin{tikzpicture}[centerzero]
                \draw[->] (-0.2,-0.2) -- (0.2,0.2);
                \draw[->] (0.2,-0.2) -- (-0.2,0.2);
                \draw[<-] (-0.4,-0.2) -- (-0.4,0.2);
                \shiftline{-0.5,0.2}{0.3,0.2}{\sigma};
                \shiftline{-0.5,-0.2}{0.3,-0.2}{\sigma};
            \end{tikzpicture}
            +
            \begin{tikzpicture}[centerzero]
                \draw[->] (-0.2,-0.2) -- (0.2,0.2);
                \draw[<-] (0.2,-0.2) -- (-0.2,0.2);
                \draw[<-] (-0.4,-0.2) -- (-0.4,0.2);
                \shiftline{-0.5,0.2}{0.3,0.2}{0};
                \shiftline{-0.5,-0.2}{0.3,-0.2}{0};
            \end{tikzpicture}
            +
            \begin{tikzpicture}[centerzero]
                \draw[<-] (-0.2,-0.2) -- (0.2,0.2);
                \draw[->] (0.2,-0.2) -- (-0.2,0.2);
                \draw[<-] (-0.4,-0.2) -- (-0.4,0.2);
                \shiftline{-0.5,0.2}{0.3,0.2}{0};
                \shiftline{-0.5,-0.2}{0.3,-0.2}{0};
            \end{tikzpicture}
            + (-1)^\sigma\,
            \begin{tikzpicture}[centerzero]
                \draw[<-] (-0.2,-0.2) -- (0.2,0.2);
                \draw[<-] (0.2,-0.2) -- (-0.2,0.2);
                \draw[<-] (-0.4,-0.2) -- (-0.4,0.2);
                \shiftline{-0.5,0.2}{0.3,0.2}{\sigma};
                \shiftline{-0.5,-0.2}{0.3,-0.2}{\sigma};
            \end{tikzpicture}
        \right)
        \\
        &=
        \begin{tikzpicture}[centerzero]
            \draw[<-] (-0.2,-0.3) -- (-0.2,-0.1) arc(180:0:0.2) -- (0.2,-0.3);
            \draw[<-] (0.3,0.3) \braiddown (0,-0.3);
            \shiftline{-0.3,0.3}{0.4,0.3}{0};
            \shiftline{-0.3,-0.3}{0.4,-0.3}{\sigma};
        \end{tikzpicture}
        +
        \begin{tikzpicture}[centerzero]
            \draw[->] (-0.2,-0.3) -- (-0.2,-0.1) arc(180:0:0.2) -- (0.2,-0.3);
            \draw[<-] (0.3,0.3) \braiddown (0,-0.3);
            \shiftline{-0.3,0.3}{0.4,0.3}{0};
            \shiftline{-0.3,-0.3}{0.4,-0.3}{\sigma};
        \end{tikzpicture}
        +
        \begin{tikzpicture}[centerzero]
            \draw[<-] (-0.2,-0.3) -- (-0.2,-0.1) arc(180:0:0.2) -- (0.2,-0.3);
            \draw[->] (0.3,0.3) \braiddown (0,-0.3);
            \shiftline{-0.3,0.3}{0.4,0.3}{\sigma};
            \shiftline{-0.3,-0.3}{0.4,-0.3}{0};
        \end{tikzpicture}
        + (-1)^\sigma\,
        \begin{tikzpicture}[centerzero]
            \draw[->] (-0.2,-0.3) -- (-0.2,-0.1) arc(180:0:0.2) -- (0.2,-0.3);
            \draw[->] (0.3,0.3) \braiddown (0,-0.3);
            \shiftline{-0.3,0.3}{0.4,0.3}{\sigma};
            \shiftline{-0.3,-0.3}{0.4,-0.3}{0};
            \end{tikzpicture}
        \overset{\cref{ruby}}{=}
        \begin{tikzpicture}[centerzero]
            \draw[<-] (-0.2,-0.3) -- (-0.2,-0.1) arc(180:0:0.2) -- (0.2,-0.3);
            \draw[<-] (-0.3,0.3) \braiddown (0,-0.3);
            \shiftline{-0.4,0.3}{0.3,0.3}{0};
            \shiftline{-0.4,-0.3}{0.3,-0.3}{\sigma};
        \end{tikzpicture}
        +
        \begin{tikzpicture}[centerzero]
            \draw[->] (-0.2,-0.3) -- (-0.2,-0.1) arc(180:0:0.2) -- (0.2,-0.3);
            \draw[<-] (-0.3,0.3) \braiddown (0,-0.3);
            \shiftline{-0.4,0.3}{0.3,0.3}{0};
            \shiftline{-0.4,-0.3}{0.3,-0.3}{\sigma};
        \end{tikzpicture}
        +
        \begin{tikzpicture}[centerzero]
            \draw[<-] (-0.2,-0.3) -- (-0.2,-0.1) arc(180:0:0.2) -- (0.2,-0.3);
            \draw[->] (-0.3,0.3) \braiddown (0,-0.3);
            \shiftline{-0.4,0.3}{0.3,0.3}{\sigma};
            \shiftline{-0.4,-0.3}{0.3,-0.3}{0};
        \end{tikzpicture}
        + (-1)^\sigma\,
        \begin{tikzpicture}[centerzero]
            \draw[->] (-0.2,-0.3) -- (-0.2,-0.1) arc(180:0:0.2) -- (0.2,-0.3);
            \draw[->] (-0.3,0.3) \braiddown (0,-0.3);
            \shiftline{-0.4,0.3}{0.3,0.3}{\sigma};
            \shiftline{-0.4,-0.3}{0.3,-0.3}{0};
        \end{tikzpicture}
        \ .
    \end{align*}

    The first, second, and fourth relations in \cref{tokrel} are straightforward.  For the third relation in \cref{tokrel}, we compute
    \[
        \sD
        \left(
            \begin{tikzpicture}[centerzero]
                \draw (0,-0.4) -- (0,0.4);
                \token{0,0.15}{east}{a};
                \token{0,-0.15}{east}{b};
            \end{tikzpicture}
        \right)
        =
        \begin{tikzpicture}[centerzero]
            \draw[->] (0,-0.4) -- (0,0.4);
            \token{0,0.15}{east}{a};
            \token{0,-0.15}{east}{b};
            \shiftline{-0.2,0.4}{0.2,0.4}{0};
            \shiftline{-0.2,-0.4}{0.2,-0.4}{0};
        \end{tikzpicture}
        +
        \begin{tikzpicture}[centerzero]
            \draw[<-] (0,-0.4) -- (0,0.4);
            \token{0,0.15}{east}{a^\inv};
            \token{0,-0.15}{east}{b^\inv};
            \shiftline{-0.2,0.4}{0.2,0.4}{\sigma};
            \shiftline{-0.2,-0.4}{0.2,-0.4}{\sigma};
        \end{tikzpicture}
        =
        \begin{tikzpicture}[centerzero]
            \draw[->] (0,-0.4) -- (0,0.4);
            \token{0,0}{east}{ab};
            \shiftline{-0.2,0.4}{0.2,0.4}{0};
            \shiftline{-0.2,-0.4}{0.2,-0.4}{0};
        \end{tikzpicture}
        + (-1)^{\bar{a}\bar{b}}
        \begin{tikzpicture}[centerzero]
            \draw[<-] (0,-0.4) -- (0,0.4);
            \token{0,0}{east}{b^\inv a^\inv};
            \shiftline{-0.2,0.4}{0.2,0.4}{\sigma};
            \shiftline{-0.2,-0.4}{0.2,-0.4}{\sigma};
        \end{tikzpicture}
        \overset{\cref{galaxy}}{=}
        \begin{tikzpicture}[centerzero]
            \draw[->] (0,-0.4) -- (0,0.4);
            \token{0,0}{east}{ab};
            \shiftline{-0.2,0.4}{0.2,0.4}{0};
            \shiftline{-0.2,-0.4}{0.2,-0.4}{0};
        \end{tikzpicture}
        +
        \begin{tikzpicture}[centerzero]
            \draw[<-] (0,-0.4) -- (0,0.4);
            \token{0,0}{east}{(ab)^\inv};
            \shiftline{-0.2,0.4}{0.2,0.4}{\sigma};
            \shiftline{-0.2,-0.4}{0.2,-0.4}{\sigma};
        \end{tikzpicture}
        = \sD
        \left(
            \begin{tikzpicture}[centerzero]
                \draw (0,-0.4) -- (0,0.4);
                \token{0,0}{east}{ab};
            \end{tikzpicture}
        \right).
    \]
    Finally, for the last relation in \cref{tokrel}, we have
    \begin{align*}
        \sD
        \left(
            \begin{tikzpicture}[anchorbase]
                \draw (-0.2,-0.2) -- (-0.2,0) arc (180:0:0.2) -- (0.2,-0.2);
                \token{-0.2,0}{east}{a};
            \end{tikzpicture}
        \right)
        &=
        \sD( \capmor ) \circ \sD( \tokstrand[a] \otimes\ \idstrand\ )
        \\
        &\overset{\mathclap{\cref{slush}}}{=}\
        \left(
            \begin{tikzpicture}[anchorbase]
                \draw[<-] (-0.15,-0.15) -- (-0.15,0) arc(180:0:0.15) -- (0.15,-0.15);
                \shiftline{-0.25,0.25}{0.25,0.25}{0};
                \shiftline{-0.25,-0.15}{0.25,-0.15}{\sigma};
            \end{tikzpicture}
            +
            \begin{tikzpicture}[anchorbase]
                \draw[->] (-0.15,-0.15) -- (-0.15,0) arc(180:0:0.15) -- (0.15,-0.15);
                \shiftline{-0.25,0.25}{0.25,0.25}{0};
                \shiftline{-0.25,-0.15}{0.25,-0.15}{\sigma};
            \end{tikzpicture}
        \right)
        \circ
        \left(
            \begin{tikzpicture}[centerzero]
                \draw[->] (-0.15,-0.25) -- (-0.15,0.25);
                \token{-0.15,0}{east}{a};
                \draw[->] (0.15,-0.25) -- (0.15,0.25);
                \shiftline{-0.25,0.25}{0.25,0.25}{0};
                \shiftline{-0.25,-0.25}{0.25,-0.25}{0};
            \end{tikzpicture}
            + (-1)^{\sigma \bar{a}}
            \begin{tikzpicture}[centerzero]
                \draw[->] (-0.15,-0.25) -- (-0.15,0.25);
                \token{-0.15,0}{east}{a};
                \draw[<-] (0.15,-0.25) -- (0.15,0.25);
                \shiftline{-0.25,0.25}{0.25,0.25}{\sigma};
                \shiftline{-0.25,-0.25}{0.25,-0.25}{\sigma};
            \end{tikzpicture}
            + (-1)^{\sigma \bar{a}}
            \begin{tikzpicture}[centerzero]
                \draw[<-] (-0.15,-0.25) -- (-0.15,0.25);
                \token{-0.15,0}{east}{a^\inv};
                \draw[->] (0.15,-0.25) -- (0.15,0.25);
                \shiftline{-0.25,0.25}{0.25,0.25}{\sigma};
                \shiftline{-0.25,-0.25}{0.25,-0.25}{\sigma};
            \end{tikzpicture}
            +
            \begin{tikzpicture}[centerzero]
                \draw[<-] (-0.15,-0.25) -- (-0.15,0.25);
                \token{-0.15,0}{east}{a^\inv};
                \draw[<-] (0.15,-0.25) -- (0.15,0.25);
                \shiftline{-0.25,0.25}{0.25,0.25}{0};
                \shiftline{-0.25,-0.25}{0.25,-0.25}{0};
            \end{tikzpicture}
        \right)
        \\
        &= (-1)^{\sigma\bar{a}}
        \left(
            \begin{tikzpicture}[anchorbase]
                \draw[<-] (-0.2,-0.2) -- (-0.2,0) arc (180:0:0.2) -- (0.2,-0.2);
                \token{-0.2,0}{east}{a^\inv};
                \shiftline{-0.3,0.4}{0.3,0.4}{0};
                \shiftline{-0.3,-0.2}{0.3,-0.2}{\sigma};
            \end{tikzpicture}
            +
            \begin{tikzpicture}[anchorbase]
                \draw[->] (-0.2,-0.2) -- (-0.2,0) arc (180:0:0.2) -- (0.2,-0.2);
                \token{-0.2,0}{east}{a};
                \shiftline{-0.3,0.4}{0.3,0.4}{0};
                \shiftline{-0.3,-0.2}{0.3,-0.2}{\sigma};
            \end{tikzpicture}
        \right)
        = (-1)^{\sigma\bar{a}}
        \left(
            \begin{tikzpicture}[anchorbase]
                \draw[<-] (-0.2,-0.2) -- (-0.2,0) arc (180:0:0.2) -- (0.2,-0.2);
                \token{0.2,0}{west}{a^\inv};
                \shiftline{-0.3,0.4}{0.3,0.4}{0};
                \shiftline{-0.3,-0.2}{0.3,-0.2}{\sigma};
            \end{tikzpicture}
            +
            \begin{tikzpicture}[anchorbase]
                \draw[->] (-0.2,-0.2) -- (-0.2,0) arc (180:0:0.2) -- (0.2,-0.2);
                \token{0.2,0}{west}{a};
                \shiftline{-0.3,0.4}{0.3,0.4}{0};
                \shiftline{-0.3,-0.2}{0.3,-0.2}{\sigma};
            \end{tikzpicture}
        \right)
    \end{align*}
    and
    \begin{align*}
        \sD
        \left(
            \begin{tikzpicture}[anchorbase]
                \draw (-0.2,-0.2) -- (-0.2,0) arc (180:0:0.2) -- (0.2,-0.2);
                \token{0.2,0}{west}{a^\inv};
            \end{tikzpicture}
        \right)
        &= \sD(\capmor) \circ \sD(\ \idstrand\, \otimes \tokstrand[a^\inv])
        \\
        &\overset{\mathclap{\cref{slush}}}{=}\
        \left(
            \begin{tikzpicture}[anchorbase]
                \draw[<-] (-0.15,-0.15) -- (-0.15,0) arc(180:0:0.15) -- (0.15,-0.15);
                \shiftline{-0.25,0.25}{0.25,0.25}{0};
                \shiftline{-0.25,-0.15}{0.25,-0.15}{\sigma};
            \end{tikzpicture}
            +
            \begin{tikzpicture}[anchorbase]
                \draw[->] (-0.15,-0.15) -- (-0.15,0) arc(180:0:0.15) -- (0.15,-0.15);
                \shiftline{-0.25,0.25}{0.25,0.25}{0};
                \shiftline{-0.25,-0.15}{0.25,-0.15}{\sigma};
            \end{tikzpicture}
        \right)
        \circ
        \left(
            \begin{tikzpicture}[centerzero]
                \draw[->] (-0.15,-0.25) -- (-0.15,0.25);
                \token{0.15,0}{west}{a^\inv};
                \draw[->] (0.15,-0.25) -- (0.15,0.25);
                \shiftline{-0.25,0.25}{0.25,0.25}{0};
                \shiftline{-0.25,-0.25}{0.25,-0.25}{0};
            \end{tikzpicture}
            + (-1)^{\sigma \bar{a}}\,
            \begin{tikzpicture}[centerzero]
                \draw[->] (-0.15,-0.25) -- (-0.15,0.25);
                \token{0.15,0}{west}{a};
                \draw[<-] (0.15,-0.25) -- (0.15,0.25);
                \shiftline{-0.25,0.25}{0.25,0.25}{\sigma};
                \shiftline{-0.25,-0.25}{0.25,-0.25}{\sigma};
            \end{tikzpicture}
            + (-1)^{\sigma \bar{a}}\,
            \begin{tikzpicture}[centerzero]
                \draw[<-] (-0.15,-0.25) -- (-0.15,0.25);
                \token{0.15,0}{west}{a^\inv};
                \draw[->] (0.15,-0.25) -- (0.15,0.25);
                \shiftline{-0.25,0.25}{0.25,0.25}{\sigma};
                \shiftline{-0.25,-0.25}{0.25,-0.25}{\sigma};
            \end{tikzpicture}
            +
            \begin{tikzpicture}[centerzero]
                \draw[<-] (-0.15,-0.25) -- (-0.15,0.25);
                \token{0.15,0}{west}{a};
                \draw[<-] (0.15,-0.25) -- (0.15,0.25);
                \shiftline{-0.25,0.25}{0.25,0.25}{0};
                \shiftline{-0.25,-0.25}{0.25,-0.25}{0};
            \end{tikzpicture}
        \right)
        \\
        &= (-1)^{\sigma\bar{a}}
        \left(
            \begin{tikzpicture}[anchorbase]
                \draw[<-] (-0.2,-0.2) -- (-0.2,0) arc (180:0:0.2) -- (0.2,-0.2);
                \token{0.2,0}{west}{a^\inv};
                \shiftline{-0.3,0.4}{0.3,0.4}{0};
                \shiftline{-0.3,-0.2}{0.3,-0.2}{\sigma};
            \end{tikzpicture}
            +
            \begin{tikzpicture}[anchorbase]
                \draw[->] (-0.2,-0.2) -- (-0.2,0) arc (180:0:0.2) -- (0.2,-0.2);
                \token{0.2,0}{west}{a};
                \shiftline{-0.3,0.4}{0.3,0.4}{0};
                \shiftline{-0.3,-0.2}{0.3,-0.2}{\sigma};
            \end{tikzpicture}
        \right).
    \end{align*}

    It remains to prove the final statement of the theorem.  This statement is clear if $d=0$, and so it suffices to assume $\sigma=0$.  In this case, we have
    \[
        \sD \left( \bubble{a} \right)
        = \ccbubble{a} + \cbubble{a^\inv}
        = \ccbubble{a + a^\inv}
    \]
    where, in the last equality, we used \cite[(4.24)]{MS21} to convert the clockwise bubble to a counterclockwise one.  Then the assertion follows from the fact that
    \[
        \str_A(a + a^\inv)
        \overset{\cref{snow}}{=} 2 \str_A(a).
        \qedhere
    \]
\end{proof}

We are now ready to state and prove the basis theorem for $\Brauer_\kk^\sigma(A,\inv;d)$.  For $r,s \in \N$, an \emph{$(r,s)$-Brauer diagram} is a string diagram representing a morphism in $\Hom_{\Brauer_\kk^\sigma(A,\inv;d)}(\go^{\otimes r}, \go^{\otimes s})$ such that:
\begin{itemize}
    \item there no tokens on any string and no closed strings (i.e.\ strings with no endpoints);
    \item no string has more than one critical point;
    \item there are no self-intersections of strings and no two strings cross each other more than once.
\end{itemize}
A \emph{perfect matching} of a finite set is a partition of that set into subsets of size $2$.  Numbering the bottom endpoints $1,\dotsc,r$ from left to right and the top endpoints $r+1,\dotsc,r+s$ from left to right, each $(r,s)$-Brauer diagram induces a perfect matching of $\{1,\dotsc,r+s\}$.  Let $\bD(r,s)$ denote a set of $(r,s)$-Brauer diagrams, with each perfect matching of $\{1,\dotsc,r+s\}$ induced by exactly one element of $\bD(r,s)$.

For $r,s \in \N$, let $\bD^\bullet(r,s)$ denote the set of all morphisms that can be obtained from elements of $\bD(r,s)$ by adding exactly one token to each string according to \cref{jiggy}.  For example,
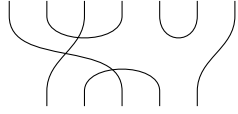
\[
    \begin{tikzpicture}[anchorbase]
        \draw (0.5,-0.2) -- (0.5,0) to[out=up,in=up] (1.5,0) -- (1.5,-0.2);
        \draw (1,-0.2) -- (1,0) \braidup (-0.5,1) -- (-0.5,1.2);
        \draw (0,-0.2) -- (0,0) \braidup (0.5,1) -- (0.5,1.2);
        \draw (0,1.2) -- (0,1) to[out=down,in=down] (1,1) -- (1,1.2);
        \draw (2,-0.2) -- (2,0) \braidup (2.5,1) -- (2.5,1.2);
        \draw (1.5,1.2) -- (1.5,1) to[out=down,in=down,looseness=2] (2,1) -- (2,1.2);
    \end{tikzpicture}
\]
is a possible element of $\bD(5,7)$ and
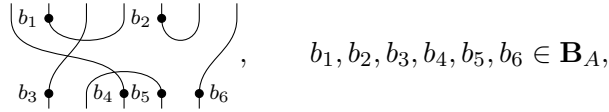
\[
    \begin{tikzpicture}[anchorbase]
        \draw (0.5,-0.2) -- (0.5,0) to[out=up,in=up] (1.5,0) -- (1.5,-0.2);
        \draw (1,-0.2) -- (1,0) \braidup (-0.5,1) -- (-0.5,1.2);
        \draw (0,-0.2) -- (0,0) \braidup (0.5,1) -- (0.5,1.2);
        \draw (0,1.2) -- (0,1) to[out=down,in=down] (1,1) -- (1,1.2);
        \draw (2,-0.2) -- (2,0) \braidup (2.5,1) -- (2.5,1.2);
        \draw (1.5,1.2) -- (1.5,1) to[out=down,in=down,looseness=2] (2,1) -- (2,1.2);
        \token{0,1}{east}{b_1};
        \token{1.5,1}{east}{b_2};
        \token{0,0}{east}{b_3};
        \token{1,0}{east}{b_4};
        \token{1.5,0}{east}{b_5};
        \token{2,0}{west}{b_6};
    \end{tikzpicture}
    ,\qquad b_1,b_2,b_3,b_4,b_5,b_6 \in \bB_A,
\]
are the corresponding elements of $\bD^\bullet(5,7)$.

We expect the following theorem to hold for an arbitrary involutive superalgebra $(A,\inv)$.  However, since our proof relies on \cref{Obasisthm}, it assumes that $A$ also admits the structure of a Frobenius superalgebra.  As explained in \cref{amongus}, this assumption holds whenever $A$ is a real or complex division superalgebra.

\begin{theo} \label{basisthm}
    For all $r,s \in \N$ and $d \in \kk$, the set $\bD^\bullet(r,s)$ is $\kk$-basis for the morphism space $\Hom_{\Brauer_\kk^\sigma(A,\inv;d)}(\go^{\otimes r}, \go^{\otimes s})$
\end{theo}

\begin{proof}
    We first prove that the elements of $\bD^\bullet(r,s)$ span $\Hom_{\Brauer_\kk^\sigma(A,\inv;d)}(\go^{\otimes r}, \go^{\otimes s})$ as $\kk$-supermodule.  Using the last two relations in \cref{tokrel} and the last two relations in \cref{mirror}, tokens on strings with endpoints can be moved near the appropriate endpoints.  Then, using the third relation in \cref{tokrel}, we can reduce the number of tokens to precisely one on each string.  (Recall that a string with no token is the same as a string with a token labelled by the identity element of $A$, using the first relation in \cref{tokrel}.)  Next, using the second relation in \cref{tokrel}, we can write any diagram as a linear combination of ones where the tokens are labelled by elements of $\bB_A$.  Finally, using the relations \cref{brauer} and the first two relations in \cref{mirror}, we can move the strings so that they are positioned to agree with some element of $\bD(r,s)$, together with some bubbles to the right of this element.  Then we can evaluate all bubbles using \cref{burst}.

    It remains to prove linear independence of $\bD^\bullet(r,s)$.  Under the functor $\sD$ of \cref{bulb}, each element of $\bD_\bullet(r,s)$ is sent, up to sign and parity shift, to a sum over all possible orientations of the strands, with the map $a \mapsto \pm a^\inv$ applied to labels of tokens on downward pointing strands.  It follows from \cref{Obasisthm} that these images are linearly independent.  Therefore, $\bD^\bullet(r,s)$ is linearly independent.
\end{proof}

\section{The unoriented incarnation superfunctor\label{sec:Unic}}

In this section, we introduce the main application of the supercategory $\Brauer_\kk^\sigma(A,\inv)$ to the representation theory of supergroups.  We begin by defining a very general \emph{unoriented incarnation superfunctor} and proving an asymptotic faithfulness result.  We then turn our attention to the special cases where $\kk \in \{\R,\C\}$ and $A$ is a division superalgebra over $\kk$.  When $\kk=\C$, fullness of the incarnation functor follows from known results.  When $\kk=\R$, we state the fullness result, whose proof is split into three cases, proved in \cref{sec:real,sec:complex,sec:quaternionic}.

\subsection{Definition of the superfunctor}

In this subsection we work over an arbitrary field $\kk$.  Let $(A,\inv)$ be an involutive superalgebra, let $V$ be a right $A$-supermodule, and let $\Phi$ be a nondegenerate $(\nu,\inv)$-supersymmetric form of parity $\sigma$ on $V$.  Fix a homogeneous $\kk$-basis $\bB_V$ of $V$, and let $\bB_V^\vee = \{b^\vee : b \in \bB_V\}$ be the left dual basis with respect to $\Phi$.  Recall the notation $\flip$ and $\rho_a$ from \cref{subsec:supermodules}.  It follows from \cref{galaxy} that
\begin{equation} \label{prodigy}
    \rho_{a^\inv} \rho_{b^\inv} = \rho_{(ab)^\inv},\qquad a,b \in A.
\end{equation}

\begin{theo} \label{incarnation}
    There exists a unique monoidal superfunctor, which we call the \emph{unoriented incarnation superfunctor} associated to $\Phi$,
    \[
        \sF_\Phi \colon \Brauer_\kk^\sigma(A,\inv) \to G(\Phi)\smod_\kk
    \]
    such that $\sF_\Phi(\go)=V$ and
    \begin{equation} \label{shark}
        \sF_\Phi \left( \crossmor \right) = \nu \flip,\qquad
        \sF_\Phi \left( \capmor \right) = \Phi,\qquad
        \sF_\Phi (\tokstrand[a]) = \rho_{a^\inv},\quad a \in A.
    \end{equation}
    This superfunctor also satisfies the following:
    \begin{gather} \label{jellyfish}
        \sF_\Phi \left( \cupmor \right) = \Phi'
        \colon \kk \to V \otimes V,\qquad
        1 \mapsto \sum_{v \in \bB_V} (-1)^{\sigma \bar{v}} v \otimes v^\vee,
        \qquad \text{and}
        \\ \label{japan}
        \sF_\Phi \left( \bubble{a} \right) = \str_V(a) 1_\one
        \qquad \text{for all } a \in A.
    \end{gather}
    If $V = A^{m|n}$ for some $m,n \in \N$, then $\sF_\Phi$ induces a monoidal superfunctor
    \[
        \sF_\Phi \colon \Brauer_\kk^\sigma(A,\inv;\nu(m-n)) \to G(\Phi)\smod_\kk.
    \]
\end{theo}

\begin{proof}
    We first show that \cref{shark,jellyfish} indeed yield a superfunctor $\sF_\Phi$.  We must show that it respects the relations \cref{brauer,tokrel}.  The first two relations in \cref{brauer} are clear.  For the third equality in \cref{brauer}, we compute
    \[
        \sF_\Phi
        \left(
            \begin{tikzpicture}[centerzero]
                \draw (-0.3,0.4) -- (-0.3,0) arc(180:360:0.15) arc(180:0:0.15) -- (0.3,-0.4);
            \end{tikzpicture}
        \right)
        \colon v
        \xmapsto{\Phi' \otimes 1_V} \sum_{w \in \bB_V} (-1)^{\sigma \bar{w}} w \otimes w^\vee \otimes v
        \xmapsto{1_V \otimes \Phi} \sum_{w \in \bB_V} \Phi(w^\vee, v) v
        = v.
    \]
    For the fourth equality in \cref{brauer}, we compute
    \[
        \sF_\Phi
        \left(
            \begin{tikzpicture}[centerzero]
                \draw (-0.3,-0.4) -- (-0.3,0) arc(180:0:0.15) arc(180:360:0.15) -- (0.3,0.4);
            \end{tikzpicture}
        \right)
        \colon v
        \xmapsto{1_V \otimes \Phi'} \sum_{w \in \bB_V} (-1)^{\sigma(\bar{v}+\bar{w})} v \otimes w \otimes w^\vee
        \xmapsto{\Phi \otimes 1_V} (-1)^\sigma \sum_{w \in \bB_V} \Phi(v,w) w^\vee
        = (-1)^\sigma v,
    \]
    where, to simplify the sign, we used the fact that $\Phi(v,w)=0$ unless $\bar{v}+\bar{w}=\sigma$.

    The fifth equality in \cref{brauer} follows immediately from the fact that $\Phi$ is a $(\nu,\inv)$-supersymmetric $\kk$-bilinear form.  For the sixth equality in \cref{brauer}, we compute
    \[
        \sF_\Phi
        \left(
            \begin{tikzpicture}[centerzero]
                \draw (-0.2,-0.3) -- (-0.2,-0.1) arc(180:0:0.2) -- (0.2,-0.3);
                \draw (-0.3,0.3) \braiddown (0,-0.3);
            \end{tikzpicture}
        \right)
        \colon u \otimes v \otimes w
        \xmapsto{\nu \flip \otimes 1_V} \nu (-1)^{\bar{u}\bar{v}} v \otimes u \otimes w
        \xmapsto{1_V \otimes \Phi} \nu (-1)^{(\sigma+\bar{u})\bar{v}} \Phi(u,w) v,
    \]
    and
    \[
        \sF_\Phi
        \left(
            \begin{tikzpicture}[centerzero]
                \draw (-0.2,-0.3) -- (-0.2,-0.1) arc(180:0:0.2) -- (0.2,-0.3);
                \draw (0.3,0.3) \braiddown (0,-0.3);
            \end{tikzpicture}
        \right)
        u \otimes v \otimes w
        \xmapsto{1_V \otimes \nu \flip} \nu (-1)^{\bar{v}\bar{w}} u \otimes w \otimes v
        \xmapsto{\Phi \otimes 1_V} \nu (-1)^{\bar{v}\bar{w}} \Phi(u,w) v
        = \nu (-1)^{(\sigma+\bar{u})\bar{v}} \Phi(u,w) v,
    \]
    where, in the final equality, we used the fact that $\Phi(u,w)=0$ unless $\bar{w} = \sigma +\bar{u}$.

    The first two relations in \cref{tokrel} are straightforward, while the third follows from \cref{prodigy}.  For the fourth relation in \cref{tokrel}, we compute
    \[
        \sF_\Phi
        \left(
            \begin{tikzpicture}[centerzero]
                \draw (-0.35,-0.35) -- (0.35,0.35);
                \draw (0.35,-0.35) -- (-0.35,0.35);
                \token{-0.17,-0.17}{east}{a};
            \end{tikzpicture}
        \right)
        \colon u \otimes v
        \xmapsto{\rho_a \otimes 1_V} (-1)^{\bar{a}\bar{u}} ua^\star \otimes v
        \xmapsto{\nu \flip} \nu (-1)^{\bar{u}\bar{v} + \bar{a}(\bar{u}+\bar{v})} v \otimes u a^\star,
    \]
    and
    \[
        \sF_\Phi
        \left(
            \begin{tikzpicture}[centerzero]
                \draw (-0.35,-0.35) -- (0.35,0.35);
                \draw (0.35,-0.35) -- (-0.35,0.35);
                \token{0.17,0.17}{west}{a};
            \end{tikzpicture}
        \right)
        \colon u \otimes v
        \xmapsto{\nu \flip} \nu (-1)^{\bar{u}\bar{v}} v \otimes u
        \xmapsto{1_V \otimes \rho_a} \nu (-1)^{\bar{u}\bar{v} + \bar{a}(\bar{u}+\bar{v})} v \otimes u a^\star.
    \]
    Finally, for the fifth relation in \cref{tokrel}, we compute
    \[
        \sF_\Phi
        \left(
            \begin{tikzpicture}[anchorbase]
                \draw (-0.2,-0.2) -- (-0.2,0) arc (180:0:0.2) -- (0.2,-0.2);
                \token{-0.2,0}{east}{a};
            \end{tikzpicture}
        \right)
        \colon u \otimes v
        \xmapsto{\rho_a \otimes 1_V} (-1)^{\bar{a}\bar{u}} ua^\inv \otimes v
        \xmapsto{\Phi} (-1)^{\bar{a}\bar{u}} \Phi(ua^\inv,v)
        \overset{\cref{supersymmetric}}{=}
        (-1)^{\bar{a}(\bar{u}+\bar{v})} \Phi(u,v a),
    \]
    and
    \[
        \sF_\Phi
        \left(
            \begin{tikzpicture}[anchorbase]
                \draw (-0.2,-0.2) -- (-0.2,0) arc (180:0:0.2) -- (0.2,-0.2);
                \token{0.2,0}{west}{a^\inv};
            \end{tikzpicture}
        \right)
        \colon u \otimes v
        \xmapsto{1_V \otimes \rho_{a^\inv}} (-1)^{\bar{a}(\bar{u}+\bar{v})} u \otimes v a
        \xmapsto{\Phi} (-1)^{\bar{a}(\bar{u}+\bar{v})} \Phi(u,v a).
    \]

    Next we prove \cref{japan}.  Using \cref{crescent}, we have
    \[
        \Phi(w,v^\vee)
        \overset{\cref{supersymmetric}}{=} \nu (-1)^{\bar{w}(\bar{v}+\sigma)} \Phi(v^\vee, w)
        = \nu (-1)^{\bar{v} + \sigma \bar{v}} \delta_{vw}.
    \]
    Thus, the $\kk$-basis of $V$ left dual to $\bB_V^\vee$ is given by $(v^\vee)^\vee = \nu (-1)^{\bar{v}+\sigma \bar{v}} v$.  Therefore,
    \begin{multline*}
        \sF_\Phi \left( \bubble{a} \right)
        \colon 1
        \xmapsto{\Phi'} \nu \sum_{v \in \bB_V} (-1)^{\bar{v}} v^\vee \otimes v
        \xmapsto{1_V \otimes \rho_a} \nu \sum_{v \in \bB_V} (-1)^{\bar{v}} v^\vee \otimes v a^\inv
        \\
        \xmapsto{\Phi} \nu \sum_{v \in \bB_V} (-1)^{\bar{v}} \Phi(v^\vee, v a^\inv)
        = \nu \str_V(a),
    \end{multline*}
    where, in the final equality, we used \cref{snooze}.  The fact that $\sF_\Phi$ factors through $\Brauer_\kk^\sigma(A,\inv;\nu(m-n))$ when $V = A^{m|n}$ then follows from \cref{break}.

    It remains to prove that, for any functor as in the first sentence of the theorem, we have $\sF_\Phi(\cupmor) = \Phi'$.  Suppose that
    \[
        \sF_\Phi(\cupmor) \colon 1
        \mapsto \sum_{u,v \in \bB_V} a_{uv} u \otimes v^\vee,\qquad
        a_{uv} \in \kk.
    \]
    Then, for all $v \in \bB_V$, we have
    \[
        v =
        \sF_\Phi
        \left(\
            \begin{tikzpicture}[centerzero]
                \draw (0,-0.4) -- (0,0.4);
            \end{tikzpicture}
        \ \right)
        (v)
        =
        \sF_\Phi
        \left(
            \begin{tikzpicture}[centerzero]
                \draw (-0.3,0.4) -- (-0.3,0) arc(180:360:0.15) arc(180:0:0.15) -- (0.3,-0.4);
            \end{tikzpicture}
        \right)
        (v)
        \xmapsto{\sF_\Phi \left( \cupmor\, \otimes\ \idstrand\ \right)}
        \sum_{u,w \in \bB_V} a_{uw} u \otimes w^\vee \otimes v
        \xmapsto{1_V \otimes \Phi}
        \sum_{u \in \bB_V} (-1)^{\sigma \bar{u}} a_{uv}  u.
    \]
    It follows that $a_{uv} = (-1)^{\sigma \bar{v}} \delta_{uv}$ for all $u,v \in \bB_V$, and so $\sF_\Phi(\cupmor) = \Phi'$.
\end{proof}

\subsection{Asymptotic faithfulness}

For the remainder of this section, we assume that $V = A^{m|n}$.  Fix a $\kk$-basis $\bB_A$ of $A$ with the property that $b^\inv = \pm b$ for all $b \in \bB_A$.

\begin{prop} \label{butte}
    If $2m+2n \ge r+s$, then the elements $\sF_\Phi(f)$, $f \in \bD^\bullet(r,s)$, are linearly independent, over $\kk$, in $\Hom_{G(\Phi)}(V^{\otimes r}, V^{\otimes s})$.
\end{prop}

\begin{proof}
    We have a commutative diagram
    \[
        \begin{tikzcd}
            \Hom_{\Brauer_\kk^\sigma(A,\inv)}(\go^{\otimes r}, \go^{\otimes s}) \arrow[d, "\sF_\Phi"] \arrow[r, "\cong"]
            & \Hom_{\Brauer_\kk^\sigma(A,\inv)}(\go^{\otimes (r+s)}, \one) \arrow[d, "\sF_\Phi"]
            \\
            \Hom_{G(\Phi)} \big( V^{\otimes r}, V^{\otimes s} \big) \arrow[r, "\cong"]
            & \Hom_{G(\Phi)} \big( V^{\otimes (r+s)}, \kk \big)
        \end{tikzcd}
    \]
    where the horizontal maps are the usual isomorphisms that hold in any rigid monoidal supercategory.  In particular, the top horizontal map is the $\kk$-linear isomorphism given on diagrams by
    \[
        \begin{tikzpicture}[centerzero]
            \draw (0,-0.2) rectangle (1,0.2);
            \draw (0.2,0.2) -- (0.2,0.7);
            \node at (0.52,0.45) {$\cdots$};
            \draw (0.8,0.2) -- (0.8,0.7);
            \draw (0.2,-0.2) -- (0.2,-0.7);
            \draw (0.8,-0.2) -- (0.8,-0.7);
            \node at (0.52,-0.45) {$\cdots$};
        \end{tikzpicture}
        \mapsto
        \begin{tikzpicture}[centerzero]
            \draw (0,-0.2) rectangle (1,0.2);
            \draw (0.2,0.2) arc(0:180:0.2) -- (-0.2,-0.7);
            \node at (0,0.8) {$\vdots$};
            \draw (0.8,0.2) arc(0:180:0.8) -- (-0.8,-0.7);
            \draw (0.2,-0.2) -- (0.2,-0.7);
            \draw (0.8,-0.2) -- (0.8,-0.7);
            \node at (0.52,-0.45) {$\cdots$};
            \node at (-0.48,-0.45) {$\cdots$};
        \end{tikzpicture}
        \qquad \text{with inverse} \qquad
        \begin{tikzpicture}[centerzero]
            \draw (-1,-0.2) rectangle (1,0.2);
            \draw (-0.8,-0.2) -- (-0.8,-1);
            \draw (-0.2,-0.2) -- (-0.2,-1);
            \node at (-0.48,-0.6) {$\cdots$};
            \draw (0.2,-0.2) -- (0.2,-1);
            \draw (0.8,-0.2) -- (0.8,-1);
            \node at (0.52,-0.6) {$\cdots$};
        \end{tikzpicture}
        \mapsto \pm\,
        \begin{tikzpicture}[centerzero]
            \draw (-1,-0.2) rectangle (1,0.2);
            \draw (-0.8,-0.2) arc(360:180:0.2) -- (-1.2,1);
            \draw (-0.2,-0.2) arc(360:180:0.8) -- (-1.8,1);
            \node at (-1.45,0.6) {$\cdots$};
            \node at (-1,-0.6) {$\vdots$};
            \draw (0.2,-0.2) -- (0.2,-1);
            \draw (0.8,-0.2) -- (0.8,-1);
            \node at (0.52,-0.6) {$\cdots$};
        \end{tikzpicture}
        \ ,
    \]
    where the rectangles denote some diagram, and the sign (which is needed only when $\sigma=1$) is determined by the parity of this diagram.  Applying the top horizontal map to an element of $\bD^\bullet(r,s)$, then sliding tokens along strands to the correct position, yields an element of $\bD^\bullet(r+s,0)$ up to sign.  Therefore, it suffices to prove the theorem in the case where $s=0$.

    Suppose $m,n \in \N$ satisfy $2m+2n \ge r$.  If $r$ is odd, then $\bD^\bullet(r,0) = \varnothing$, and the proposition holds trivially.  Therefore, we suppose that $r$ is even.  In what follows we number the strand endpoints $1,2,\dotsc,r$ from left to right.  Given $f \in \bD^\bullet(r,0)$, we enumerate the strands in $f$ in order of the numbering of their right endpoint.  Let $b_i$, $1 \le i \le \frac{r}{2}$, denote the label of the token on the $i$-th strand of $f$.  For $1 \le i \le r$, define
    \[
        v_i =
        \begin{cases}
            e_j & \text{if the $j$-th strand in $f$ has right endpoint in position $i$} \\
            (e_j b_j)^\vee & \text{if the $j$-th strand in $f$ has left endpoint in position $i$},
        \end{cases}
    \]
    where $e_1,\dotsc,e_{m+n}$ denotes the standard $A$-basis of $V = A^{m|n}$, and where
    \[
        \{(e_ib)^\vee : 1 \le i \le m+n,\ b \in \bB_A\}
    \]
    denotes the basis of $V$ left dual to $\{e_i b : 1 \le i \le m+n,\ b \in \bB_A\}$ with respect to $\Phi$.  Now define
    \[
        v_f = v_1 \otimes v_2 \otimes \dotsb \otimes v_r.
    \]
    For example,
    \begin{gather*}
        \text{if }
        f =
        \begin{tikzpicture}[baseline={([yshift=10pt]current bounding box.south)}]
            \draw (0,-0.3) -- (0,0) to[out=up,in=up] (3.5,0) -- (3.5,-0.3);
            \draw (0.5,-0.3) -- (0.5,0) to[out=up,in=up] (1.5,0) -- (1.5,-0.3);
            \draw (1,-0.3) -- (1,0) to[out=up,in=up] (2,0) -- (2,-0.3);
            \draw (2.5,-0.3) -- (2.5,0) to[out=up,in=up] (3,0) -- (3,-0.3);
            \token{1.5,-0.1}{east}{b_1};
            \token{2,-0.1}{west}{b_2};
            \token{3,-0.1}{east}{b_3};
            \token{3.5,-0.1}{west}{b_4};
        \end{tikzpicture}
        \ ,\qquad \text{then}
        \\
        v_f = (e_4 b_4)^\vee \otimes (e_1 b_1)^\vee \otimes (e_2 b_2)^\vee \otimes e_1 \otimes e_2 \otimes (e_3 b_3)^\vee \otimes e_3 \otimes e_4.
    \end{gather*}
    It is straightforward to verify that
    \[
        \sF_\Phi(f)(v_g) = \pm \delta_{f,g},\qquad \text{for all } f,g \in \bD^\bullet(r,0).
    \]
    It follows that the elements of $\sF_\Phi(f)$, $f \in \bD^\bullet(r,0)$, are linearly independent, as desired.
\end{proof}

\begin{prop}
    If $2m+2n \ge r+s$, then the induced $\kk$-supermodule homomorphism
    \[
        \sF_\Phi \colon \Hom_{\Brauer_\kk^\sigma(A,\inv;\nu(m-n))}(\go^{\otimes r}, \go^{\otimes s}) \to \Hom_{G(\Phi)}(V^{\otimes r}, V^{\otimes s})
    \]
    is injective.
\end{prop}

\begin{proof}
    This follows immediately from \cref{basisthm,butte}.
\end{proof}

\subsection{Fullness over the real and complex numbers}

We are now ready to state the last of our main results: the fullness of the unoriented incarnation functor in the case of a central real division superalgebra.  We begin by stating the fullness result over the complex numbers.

\begin{prop} \label{soup}
    If $\Phi$ is a nondegenerate $(\nu,\id)$-supersymmetric form on $\C^{m|n}$ of parity $\sigma$, then the unoriented incarnation superfunctor
    \[
        \sF_\Phi \colon \Brauer_\C^\sigma(\C,\id;\nu(m-n)) \to G(\Phi)\smod_\C
    \]
    of \cref{incarnation} is full.
\end{prop}

\begin{proof}
    First consider the case $\sigma=0$.  As explained in \cref{subsec:hermit-prelim}, we may assume that $\nu = 1$.  Then $\Brauer_\C^0(\C,\id;m-n)$ is the usual Brauer category and the result was proved in \cite[Th.~5.6]{LZ17}; see also \cite{ES16}.  Next consider the case $\sigma=1$, in which case we must have $m=n$, as explained in \cref{subsec:hermit-periplectic}.  By \cref{wine}, we may again assume that $\nu=1$.  Then $\Brauer_\C^0(\C,\id;0)$ is the periplectic Brauer category, introduced in \cite{KT17} as $\mathcal{B}(0,-1)$.  In this case, the result was proved in \cite[Th.~6.2.1]{CE21}, with the key ingredient being \cite[\S4.9]{DLZ18}.
\end{proof}

In the remainder of the paper, we will be mostly concerned with the case where $A$ is a real division superalgebra (see \cref{amongus}), and where the $(\nu,\inv)$-supersymmetric form comes from a $(\nu,\star)$-superhermitian form, as in \cref{crescent}.  To simplify notation we define, for $(\DD,\star)$ an involutive real division superalgebra, and $d \in \R$,
\begin{equation} \label{sand}
    \Brauer_\R^\sigma(\DD) := \Brauer_\R^\sigma(\DD,\inv),\qquad
    \Brauer_\R^\sigma(\DD;d) := \Brauer_\R^\sigma(\DD,\inv;d),
\end{equation}
where $a^\inv = (-1)^{\bar{a}} a^\star$.

\begin{theo} \label{divfull}
    Suppose $(\DD,\star)$ is an involutive central real division superalgebra, and $\varphi$ is a nondegenerate $(\nu,\star)$-superhermitian form on $\DD^{m|n}$ of parity $\sigma$.  Let $\Phi$ be the corresponding nondegenerate $(\nu,\inv)$-supersymmetric form on $\DD^{m|n}$, as in \cref{crescent}.  Then the unoriented incarnation superfunctor
    \[
        \sF_\Phi \colon \Brauer_\R^\sigma(\DD;\nu(m-n)) \to G(\Phi)\smod_\R
    \]
    of \cref{incarnation} is full.
\end{theo}

The proof of \cref{divfull} will be broken into three parts:
\begin{itemize}
    \item we prove it holds for $(\DD,\star) = (\R,\id)$ in \cref{divfull:R},
    \item we prove it holds for $(\DD,\star) = (\C,\star)$ and $(\DD,\star) = (\Cl(\C),\star)$ in \cref{divfull:C}, and
    \item we prove it holds for $(\DD,\star) = (\HH,\star)$ in \cref{divfull:H}.
\end{itemize}

\begin{rem} \label{river}
    As explained in \cref{subsec:HCpair}, the forgetful functor $G(\Phi)\smod_\R \to \fg(\Phi)\smod_\R$ is full and faithful when $G_\rd(\Phi)$ is connected.  In this case, we can replace the target supercategories $G(\Phi)\smod_\C$ and $G(\Phi)\smod_\R$ in \cref{soup,divfull} by $\fg(\Phi)\smod_\C$ and $\fg(\Phi)\smod_\R$, respectively.  It follows from the descriptions of $G_\rd(\Phi)$ in \cref{app:hermit} that we can make this replacement whenever $\sigma=1$.  In addition, when $\sigma=0$, we can make this replacement in \cref{divfull} when $(\DD,\star) \in \{(\C,\star),(\Cl(\C),\star)\}$.
\end{rem}

\subsection{Consequences for the semisimple cases\label{subsec:Bnonsuper}}

For $(\DD,\star) \in \{(\R,\id), (\C,\star), (\HH,\star)\}$, we see that $\Brauer_\kk(\DD)$ is a monoidal category (as opposed to a monoidal \emph{super}category).  In fact, the category $\Brauer_\kk(\DD)$ is a spherical pivotal category.  We also have $\inv = \star$.  For $d \in \Z$, we let $\cN_\R(\DD;d)$ denote the tensor ideal of negligible morphisms of $\Kar(\Brauer_\R(\DD;d))$; see \cref{subsec:OBnonsuper}.

For a $(\nu,\star)$-supersymmetric form $\Phi$ on $\DD^{m|n}$, we let $G(\Phi)\tsmod_\R$ denote the monoidal supercategory of tensor $G(\Phi)$-supermodules.  By definition, this is the full sub-supercategory of $G(\Phi)\smod_\R$ whose objects are direct summands of $V^{\otimes r}$, $r \in \N$, where $V = \DD^{m|n}$ is the natural supermodule.  We let $G(\Phi)\tsmod_\R'$ denote the \emph{underlying category} of $G(\Phi)\tsmod_\R$.  By definition, this is the category with the same objects as $G(\Phi)\tsmod_\R$, but whose morphisms are the \emph{even} $G$-supermodule homomorphisms.  So $G(\Phi)\tsmod_\R'$ is a monoidal category (as opposed to a monoidal \emph{super}category).  If $n=0$, so that $G(\Phi)$ is a real group (as opposed to a \emph{super}group), then we write $G(\Phi)\tmod_\R$ for the category of tensor modules, defined in the same way.

\begin{theo} \label{pike}
    For $(\DD,\star) \in \{(\R,\id), (\C,\star), (\HH,\star)\}$ and $p,q,m,n \in \N$, $p+q=m$, the unoriented incarnation functor induces equivalences of monoidal categories
    \begin{align*}
        \Kar(\Brauer_\R(\R;m))/\cN_\R(\R;m) &\to \rO(p,q)\tmod_\R, \\
        \Kar(\Brauer_\R(\R;-2n))/\cN_\R(\R;-2n) &\to \OSp(0|2n,\R)\tsmod_\R', \\
        \Kar(\Brauer_\R(\R;1-2n))/\cN_\R(\R;1-2n) &\to \OSp(1|2n,\R)\tsmod_\R', \\
        \Kar(\Brauer_\R(\C;m))/\cN_\R(\C;m) &\to \rU(p,q)\tmod_\R, \\
        \Kar(\Brauer_\R(\HH;m))/\cN_\R(\HH;m) &\to \Sp(p,q)\tmod_\R, \\
        \Kar(\Brauer_\R(\HH;-n))/\cN_\R(\HH;-n) &\to \rO(m,\HH)\tsmod_\R'.
    \end{align*}
\end{theo}

Note that, while $\OSp(0|2n,\R)$ is isomorphic to $\Sp(2n,\R)$, we write $\OSp(0|2n,\R)\tsmod_\R'$ in \cref{pike}, instead of $\Sp(2n,\R)\tmod_\R$, since we view the natural module $\R^{2n}$ as purely odd.  Similarly, in $\rO(n,\HH)\tsmod_\R'$, we view the natural module $\HH^n$ as purely odd; see \cref{subsec:hermformH}.

\begin{proof}
    The proof is similar to that of \cref{storm}.  To simplify notation, we give the proof for $\rO(p,q)$, since the other cases are analogous; see \cref{bivalent}\ref{bivalent1}.  Let $\Phi = \varphi_{p,q|2n}$ be the $(\nu,\star)$-supersymmetric form on $\R^m$ defined in \cref{hermformR}, so that $G(\Phi) = \rO(p,q)$.  Since the category $\rO(p,q)\tmod_\R$ is idempotent complete, $\sF_\Phi$ induces a monoidal functor
    \[
        \Kar(\sF_\Phi) \colon \Kar(\Brauer_\R(\R;m)) \to \rO(p,q)\tmod_\R.
    \]
    By \cref{divfull}, this functor is full.

    Every object in $\rO(p,q)\tmod_\R$ is completely reducible, since its complexification is a tensor module for $\rO(m,\C)$, hence completely reducible.  Thus, the category $\rO(p,q)\tmod_\R$ is semisimple.
    \details{
        See the details in the proof of \cref{storm}.  The fact that the category $\rO(m,\C)\tmod_\C$ is semisimple follows, for example, from \cite[Th.~9.6]{Del07}.
    }
    In addition, by \cref{basisthm}, $\End_{\Brauer_\R(\R;m)}(\one) = \R 1_\one$.  Thus, by \cite[Prop.~6.9]{SW22}, the kernel of $\Kar(\sF_\Phi)$ is equal to $\cN_\R(\R;m)$.  Therefore, $\Kar(\sF_\Phi)$ induces a full and faithful functor
    \[
        \Brauer_\R(\R;m)/\cN_\R(\R;m) \to \rO(p,q)\tmod_\R.
    \]

    Finally, since the image of $\Kar(\sF_\Phi)$ contains all summands of tensor powers of the natural module $\R^m$ (i.e.\ all tensor modules), it is essentially surjective, hence an equivalence of categories.
\end{proof}

\begin{rem} \label{bivalent}
    \begin{enumerate}[wide]
        \item \label{bivalent1} \Cref{pike} involves precisely the supergroups with Lie superalgebras that are real forms of complex Lie superalgebras whose finite-dimensional modules are all semisimple; see \cite[Th.~4.1]{DH76}.

        \item As explained in \cref{river}, we can replace $\OSp(0|2n,\R)$, $\rU(p,q)$, and $\Sp(p,q)$ in \cref{pike} by $\fsp(0|2n,\R)$, $\fu(p,q)$, and $\fsp(p,q)$, respectively.  However, we \emph{cannot} replace $\rO(p,q)$, $\OSp(1|2n,\R)$, or $\rO(n,\HH)$ by their Lie superalgebras, since the orthogonal groups are not connected.
    \end{enumerate}
\end{rem}

\begin{cor} \label{surprise}
    If $p,p',q,q' \in \N$ satisfy $p+q=p'+q'$, then we have equivalences of monoidal categories
    \begin{align*}
        \rO(p,q)\tmod_\R &\simeq \rO(p',q')\tmod_\R,\\
        \rU(p,q)\tmod_\R &\simeq \rU(p',q')\tmod_\R,\\
        \Sp(p,q)\tmod_\R &\simeq \Sp(p',q')\tmod_\R,
    \end{align*}
    sending the natural supermodule to the natural supermodule.
\end{cor}

We will extend \cref{surprise} to equivalences of more general supergroups in \cref{iceR,iceC,iceH}.

\begin{rem}
    It is crucial that \cref{surprise} involves $\rO(p,q)$ and $\rU(p,q)$, as opposed to $\SO(p,q)$ and $\SU(p,q)$.  For example, let $V = \C^2 \cong \R^4$  denote the natural $\rU(2)$-module.  By restriction, this is also the natural module for $\SU(2)$.  Direct computation shows that
    \[
        \C \cong \End_{\rU(2)}(V) \subseteq \End_{\SU(2)}(V) \cong \HH.
    \]
    On the other hand, all irreducible modules of $\SU(1,1) \cong \SL(2,\R)$ have endomorphism algebra isomorphic to $\R$.  Thus, the categories $\SU(1,1)\tmod_\R$ and $\SU(2)\tmod_\R$ are \emph{not} equivalent.  If $W = \C^2 \cong \R^4$ denotes the natural module of $\rU(1,1)$, which is also the natural module for $\SU(1,1)$ by restriction, we have
    \[
        \C \cong \End_{\rU(1,1)}(W) \subseteq \End_{\SU(1,1)}(W) \cong \Mat_2(\R).
    \]
    The equivalence $\rU(2)\tmod_\R \simeq \rU(1,1)\tmod_\R$ of \cref{surprise} sends $V$ to $W$.  Both modules have endomorphism algebra isomorphic to $\C$ only if we use the full unitary groups.
\end{rem}

\section{Unoriented fullness: real case\label{sec:real}}

In this section, we prove \cref{divfull} in the case $(\DD,\star) = (\R,\id)$.  Note that, since $\R$ is purely even, we have $\inv = \star$.  Recall, from \cref{sec:hermitian}, that, if $(A,\star)$ is an involutive superalgebra, then so is $(A^\C,\star)$.  The proof of the following result is analogous to that of \cref{crystal}.

\begin{prop} \label{prato}
    For any real involutive superalgebra $(A,\star)$ and $\sigma \in \Z_2$, there is an isomorphism of monoidal supercategories
    \[
        \Brauer_\R^\sigma(A,\star)^\C \xrightarrow{\cong} \Brauer_\C^\sigma(A^\C,\star)
    \]
    given on objects by $\go \mapsto \go$ and on morphisms by
    \[
        \crossmor \mapsto \crossmor,\qquad
        \capmor \mapsto \capmor,\qquad
        \cupmor \mapsto \cupmor,\qquad
        \tokstrand \mapsto \tokstrand[a \otimes 1],\quad a \in A.
    \]
    For all $d \in \kk$, this induces an isomorphism of monoidal supercategories
    \[
        \Brauer_\R^\sigma(A,\star;d)^\C \xrightarrow{\cong} \Brauer_\C^\sigma(A^\C,\star;d)
    \]
\end{prop}

Fix $m,n \in \N$, set $V = \R^{m|n}$, and let $\Phi$ be a nondegenerate $(\nu,\id)$-supersymmetric form on $V$ of parity $\sigma$.  (See \cref{subsec:hermformR,subsec:hermit-periplectic} for a classification.)  We have a natural identification $V^\C = \C^{m|n}$, and we extend $\Phi$ to a nondegenerate $(\nu,\id)$-supersymmetric form
\[
    \Phi^\C \colon V^\C \times V^\C \to \C.
\]

\begin{lem} \label{light}
    For all $G(\Phi)$-supermodules $U$ and $W$, we have an isomorphism of $\C$-supermodules
    \[
        \Hom_{G(\Phi)}(U,W)^\C \xrightarrow{\cong} \Hom_{G(\Phi^\C)}(U^\C,W^\C),\quad
        f \otimes a \mapsto f \otimes a,\quad
        f \in \Hom_{G(\Phi)}(U,W),\ a \in \C.
    \]
\end{lem}

\begin{proof}
    First suppose that $\sigma=0$.  Then, as explained in \cref{subsec:hermit-prelim}, we may assume that $\nu=1$.  In this case, the classification of the nondegenerate $(1,\id)$-supersymmetric forms is recalled in \cref{subsec:hermformR}.  We see that $n$ must be even, and $G_\rd(\Phi) = \rO(p,q) \times \Sp(n,\R)$ for some $p,q \in \N$, $p+q=m$.  The group $\Sp(n,\R)$ is connected.  On the other hand, if $m \ge 1$, then $\rO(p,q)$ has four connected components if $p,q \ge 1$, and two connected components if $p=0$ or $q=0$.  Suppose $p,q \ge 1$, the proof in the other case being analogous.  Then choose elements $X_1,X_2,X_3$ in $\rO(p,q)$, one from each of its connected components not containing the identity.  By \cref{skool2,compRform}, we have an isomorphism
    \[
        \Hom_{G(\Phi)}(U,W)^\C \cong \Hom_{X_1,X_2,X_3,\fg(\Phi^\C)}(U^\C,W^\C).
    \]
    Now, $G_\rd(\Phi^\C) \cong \rO(m,\C) \times \Sp(n,\C)$.  The group $\Sp(n,\C)$ is connected, while $\rO(m,\C)$ has two connected components.  Reordering if necessary, we may assume that $\det(X_1) = \det(X_2) = -1 = - \det(X_3)$.  By \cref{breath}, we have
    \[
        \Hom_{X_1,X_2,X_3,\fg(\Phi^\C)}(U^\C,W^\C)
        = \Hom_{X_1,\fg(\Phi^\C)}(U^\C,W^\C)
        = \Hom_{G(\Phi^\C)}(U^\C, W^\C).
    \]

    The case $\sigma=1$ is easier.  As explained in \cref{subsec:hermit-periplectic}, we have $G(\Phi) = \GL(m,\R)$ and $G(\Phi^\C) = \GL(m,\C)$, which are both connected.  Then, using \cref{skool}, we have
    \[
        \Hom_{G(\Phi)}(U,W)^\C
        = \Hom_{\fg(\Phi)}(U,W)^\C
        \cong \Hom_{\fg(\Phi^\C)}(U^\C,W^\C)
        = \Hom_{G(\Phi^\C)}(U^\C,W^\C).
        \qedhere
    \]
\end{proof}

It follows from \cref{light} that we have a canonical full and faithful superfunctor
\begin{equation} \label{bright}
    \sE_\R \colon (G(\Phi)\smod_\R)^\C \to G(\Phi^\C)\smod_\C
\end{equation}
sending $V$ to $V^\C$.  Since $\R$ is commutative, we can identify $\R$ and $\R^\op$.  Let $\sS_\R$ denote the isomorphism of \cref{prato} when $(A,\star) = (\R,\id)$.

\begin{prop} \label{embers}
    The diagram
    \[
        \begin{tikzcd}
            \Brauer_\R^{\sigma} (\R;\nu(m-n))^\C \arrow[rr, "\sS_\R"] \arrow[d, "\sF_\Phi^\C"']
            & & \Brauer_\C^{\sigma}(\C;\nu(m-n)) \arrow[d, "\sF_{\Phi^\C}"]
            \\
            \left( G(\Phi)\smod_\R \right)^\C \arrow[rr, "\sE_\R"]
            & & G(\Phi^\C)\smod_\C
        \end{tikzcd}
    \]
    commutes.
\end{prop}

\begin{proof}
    To simplify notation, we set $\sS = \sS_\R$, $\sE = \sE_\R$, $\sF = \sF_{\Phi^\C}$, and $\sF^\C = \sF_\Phi^\C$.  On objects, we have
    \[
        \sE \sF^\C (\go) = V^\C = \sF(\go) = \sF \sS(\go).
    \]

    For morphisms, we need to show that
    \[
        \sE \sF^\C (f)
        = \sF \sS (f)
        \quad \text{for } f \in \left\{ \crossmor, \capmor, \cupmor \right\},
    \]
    where $X$ and $Y$ are the domain and codomain of $f$, respectively.

    For $f = \crossmor$, we have
    \[
        \sE \sF^\C (\crossmor)
        = \nu \flip_{V^\C,V^\C}
        = \sF \sS(\crossmor).
    \]
    For $f = \capmor$, we have
    \[
        \sE \sF^\C(\capmor) = \sF \sS (\capmor)
        \colon V^\C \otimes_\C V^\C \to \C,\quad
        u \otimes v \mapsto \Phi^\C(u,v).
    \]
    Finally, we consider $f = \cupmor$.  Let $\bB_V$ denote an $\R$-basis of $V$, which we also view as a $\C$-basis of $V^\C$.  Then we have
    \[
        \sE \sF^\C(\cupmor)
        = \sF \sS(\cupmor)
        \colon \C \to V^\C \otimes_\C V^\C,\qquad
        1 \mapsto \sum_{v \in \bB_V} v \otimes v^\vee.
        \qedhere
    \]
\end{proof}

\begin{prop} \label{divfull:R}
    \Cref{divfull} holds when $(\DD,\star) = (\R,\id)$.
\end{prop}

\begin{proof}
    We wish to show that the $\R$-linear map
    \[
        \sF_\Phi \colon \Hom_{\Brauer_\R^\sigma(\R;\nu(m-n))}(\go^{\otimes r}, \go^{\otimes s}) \to \Hom_{G(\Phi)}(V^{\otimes r}, V^{\otimes s})
    \]
    is surjective.  This map is surjective if and only if the induced map
    \begin{equation} \label{lilly}
        \sF_\Phi^\C \colon \Hom_{\Brauer_\R^\sigma(\R;\nu(m-n))}(\go^{\otimes r}, \go^{\otimes s})^\C
        \to \Hom_{G(\Phi)}(V^{\otimes r}, V^{\otimes s})^\C
    \end{equation}
    is surjective.  To show that \cref{lilly} is surjective, it suffices to show that the diagram
    \begin{equation} \label{deep}
        \begin{tikzcd}
            \Hom_{\Brauer_\R^\sigma(\R;\nu(m-n))}(\go^{\otimes r}, \go^{\otimes s})^\C \arrow[d, "\sS_\R"', "\cong"] \arrow[r, "\sF_\Phi^\C"]
            & \Hom_{G(\Phi)}(V^{\otimes r}, V^{\otimes s})^\C \arrow[d,"\cong", "\sE_\R"']
            \\
            \Hom_{\Brauer_\C^\sigma(\C);\nu(m-n)} (\go^{\otimes r}, \go^{\otimes s})
            \arrow[r, two heads, "\sF_{\Phi^\C}"]
            & \Hom_{G(\Phi^\C)} \left( (V^\C)^{\otimes r}, (V^\C)^{\otimes s} \right)
        \end{tikzcd}
    \end{equation}
    commutes, where surjectivity of the bottom horizontal map follows from \cref{soup}.  Commutativity of this diagram follows from \cref{embers}.
\end{proof}

If $\varphi_{p,q|2n}$ is the form defined in \cref{hermformR}, then $G(\varphi_{p,q|2n}) = \OSp(p,q|2n,\R)$ is the indefinite orthosymplectic supergroup.  Recall the definition $G(\Phi)\tsmod_\R$ of the monoidal supercategory of tensor $G(\Phi)$-supermodules from \cref{subsec:Bnonsuper}.

\begin{prop} \label{iceR}
    If $p,p',q,q',n \in \N$ satisfy $p+q=p'+q'$, then we have an equivalence of monoidal supercategories,
    \[
        \OSp(p,q|2n,\R)\tsmod_\R \simeq \OSp(p',q'|2n,\R)\tsmod_\R
    \]
    sending the natural supermodule to the natural supermodule.
\end{prop}

\begin{proof}
    Viewing $\Brauer_\R^0(\R;m-n)$ as a subcategory of $\Brauer_\R^0(\R;m-n)^\C$, it follows from \cref{funk} and the commutativity of \cref{deep}, with $\sigma=0$ and $\nu=1$, that $\ker(\sF_\Phi) = \sS_\R^{-1}(\ker(\sF_{\Phi^\C})) \cap \Brauer_\R^0(\R,m-n)$.  In particular, $\ker(\sF_\Phi)$ depends only on $\Phi^\C$.  As noted in \cref{subsec:hermformCid}, $\Phi^\C$ depends only on $p+q$ and $n$, up to equivalence.  Hence, both $\OSp(p,q|2n)\tsmod_\R$ and $\OSp(p',q'|2n)\tsmod_\R$ are equivalent to the quotient of $\Brauer_\R^0(\R,m-n)$ by this common kernel.
\end{proof}

\section{Unoriented fullness: complex cases\label{sec:complex}}

In this section, we prove \cref{divfull} in the case where $(\DD,\star)$ is either $(\C,\star)$ or $(\Cl(\C),\star)$.  Throughout this section, we assume that $(\DD,\star)$ is one of these two complex involutive superalgebras.  Recall, from \cref{amongus}, that $\DD$ is also a real Frobenius superalgebra, with Nakayama automorphism $\Nak$ given by \cref{divNak}, so that
\[
    a^\inv = (-1)^{\bar{a}} a^\star,\qquad a \in \DD.
\]
By \cref{delay}, we can assume that the specialization parameter $d$ is zero when $\DD = \Cl(\C)$.

If $V$ is a complex vector superspace, we let $V^\star$ denote the conjugate complex vector superspace, which is a special case of the construction described in \cref{subsec:invalg}.  Precisely, $V^\star$ is equal to $V$ as an $\R$-vector superspace, but the $\C$-action is given by
\begin{equation} \label{stellar}
    V^\star \times \C \to V^\star,\qquad
    (v,a) \mapsto v a^\star.
\end{equation}
Above, and elsewhere, the juxtaposition $vb$, for $b \in \C$ and $v \in V$ or $V^\star$, will always denote the $\C$-action on $V$ (as opposed to the $\C$-action on $V^\star$).  Recall the notion of complexification from \cref{subsec:complexification}.  We have an isomorphism of $\C$-vector spaces
\begin{equation} \label{dub}
    V^\C
    \xrightarrow{\cong} V \oplus V^\star,\qquad
    v \mapsto \tfrac{1}{\sqrt{2}}(v,v),\quad
    v \in V.
\end{equation}
Note that $\C$-linearity implies that
\[
    v \otimes a \mapsto \tfrac{1}{\sqrt{2}}(va,va^\star),\qquad
    v \in V,\ a \in \C.
\]
\details{
    The map has inverse given by
    \[
        (u,v) \mapsto \tfrac{1}{\sqrt{2}} \big( (u+v) \otimes 1 + (v-u) i \otimes i \big).
    \]
}

Recall, from \cref{sec:monsupcat}, the superadditive envelope $\Add(\cC_\pi)$ of a supercategory $\cC$.

\begin{prop} \label{deacon}
    Fix $d \in \kk$ and $\sigma \in \Z_2$.  There exists a unique $\C$-linear monoidal superfunctor
    \[
        \sS_\DD \colon \Brauer_\R^\sigma(\DD;d)^\C \to \Add(\OB_\C(\DD;d)_\pi)
    \]
    such that $\sS_\DD(\go) = \upobj \oplus \Pi^\sigma \downobj$ and
    \begin{gather} \label{deacon1}
        \sS_\DD \left( \crossmor \right) =
        \begin{tikzpicture}[centerzero]
            \draw[->] (-0.2,-0.2) -- (0.2,0.2);
            \draw[->] (0.2,-0.2) -- (-0.2,0.2);
            \shiftline{-0.3,0.2}{0.3,0.2}{0};
            \shiftline{-0.3,-0.2}{0.3,-0.2}{0};
        \end{tikzpicture}
        +
        \begin{tikzpicture}[centerzero]
            \draw[->] (-0.2,-0.2) -- (0.2,0.2);
            \draw[<-] (0.2,-0.2) -- (-0.2,0.2);
            \shiftline{-0.3,0.2}{0.3,0.2}{\sigma};
            \shiftline{-0.3,-0.2}{0.3,-0.2}{\sigma};
        \end{tikzpicture}
        +
        \begin{tikzpicture}[centerzero]
            \draw[<-] (-0.2,-0.2) -- (0.2,0.2);
            \draw[->] (0.2,-0.2) -- (-0.2,0.2);
            \shiftline{-0.3,0.2}{0.3,0.2}{\sigma};
            \shiftline{-0.3,-0.2}{0.3,-0.2}{\sigma};
        \end{tikzpicture}
        + (-1)^\sigma\,
        \begin{tikzpicture}[centerzero]
            \draw[<-] (-0.2,-0.2) -- (0.2,0.2);
            \draw[<-] (0.2,-0.2) -- (-0.2,0.2);
            \shiftline{-0.3,0.2}{0.3,0.2}{0};
            \shiftline{-0.3,-0.2}{0.3,-0.2}{0};
        \end{tikzpicture}
        ,
        \\ \label{deacon2}
        \sS_\DD \left( \capmor \right) =
        \begin{tikzpicture}[anchorbase]
            \draw[<-] (-0.15,-0.15) -- (-0.15,0) arc(180:0:0.15) -- (0.15,-0.15);
            \shiftline{-0.25,0.25}{0.25,0.25}{0};
            \shiftline{-0.25,-0.15}{0.25,-0.15}{\sigma};
        \end{tikzpicture}
        +
        \begin{tikzpicture}[anchorbase]
            \draw[->] (-0.15,-0.15) -- (-0.15,0) arc(180:0:0.15) -- (0.15,-0.15);
            \shiftline{-0.25,0.25}{0.25,0.25}{0};
            \shiftline{-0.25,-0.15}{0.25,-0.15}{\sigma};
        \end{tikzpicture}
        ,\qquad
        \sS_\DD \left( \cupmor \right) =
        \begin{tikzpicture}[anchorbase]
            \draw[<-] (-0.15,0.15) -- (-0.15,0) arc(180:360:0.15) -- (0.15,0.15);
            \shiftline{-0.25,0.15}{0.25,0.15}{\sigma};
            \shiftline{-0.25,-0.25}{0.25,-0.25}{0};
        \end{tikzpicture}
        + (-1)^\sigma
        \begin{tikzpicture}[anchorbase]
            \draw[->] (-0.15,0.15) -- (-0.15,0) arc(180:360:0.15) -- (0.15,0.15);
            \shiftline{-0.25,0.15}{0.25,0.15}{\sigma};
            \shiftline{-0.25,-0.25}{0.25,-0.25}{0};
        \end{tikzpicture}
        ,
        \\ \label{deacon3}
        \sS_\DD \left( \tokstrand \right) =
        \begin{tikzpicture}[centerzero]
            \draw[->] (0,-0.2) -- (0,0.2);
            \token{0,0}{east}{a};
            \shiftline{-0.1,0.2}{0.1,0.2}{0};
            \shiftline{-0.1,-0.2}{0.1,-0.2}{0};
        \end{tikzpicture}
        + (-1)^{\sigma\bar{a}}
        \begin{tikzpicture}[centerzero]
            \draw[<-] (0,-0.2) -- (0,0.2);
            \token{0,0}{east}{a^\inv};
            \shiftline{-0.1,0.2}{0.1,0.2}{\sigma};
            \shiftline{-0.1,-0.2}{0.1,-0.2}{\sigma};
        \end{tikzpicture}
        ,\quad a \in \DD.
    \end{gather}
    The functor $\sS_\DD$ is full and faithful.
\end{prop}

\begin{proof}
    The superfunctor $\sS_\DD$ is the complexification $\sD^\C$ of the superfunctor of \cref{bulb}, with $d$ replaced by $d/2$, followed by the superfunctor
    \[
        \Add(\OB_\R(\DD;d/2))^\C \to \Add(\OB_\C(\DD;d))
    \]
    given by imposing the relations
    \[
        \uptokstrand[a] = a\ \upstrand,\qquad a \in \C.
    \]
    Note that the doubling of the specialization parameter comes from the fact that, by \cref{delay},
    \[
        \str_\DD^\R(1)
        = \sdim_\R \DD
        = 2(\sdim_\C \DD)
        = 2 \str_\DD^\C(1).
    \]

    It remains to prove that $\sS_\DD$ is full and faithful.  For $r,s \in \N$, the functor $\sS_\DD$ induces a $\C$-linear map
    \begin{multline} \label{gravel}
        \Hom_{\Brauer_\R^\sigma(\DD;d)}(\go^{\otimes r}, \go^{\otimes s})^\C
        \to \Hom_{\Add(\OB_\C^\sigma(\DD;d)_\pi)}((\upobj \oplus \Pi^\sigma \downobj)^{\otimes r}, (\upobj \otimes \downobj)^{\otimes s})
        \\
        \cong \bigoplus_{X_1,\dotsc,X_r,Y_1,\dotsc,Y_s \in \{\upobj,\Pi^\sigma \downobj\}} \Hom_{\OB_\C^\sigma(\DD;d)_\pi}(X_1 \otimes \dotsb \otimes X_r, Y_1 \otimes \dotsb \otimes Y_s).
    \end{multline}
    By \cref{basisthm}, $\Hom_{\Brauer_\R^\sigma(\DD;d)}(\go^{\otimes r}, \go^{\otimes s})^\C$ has $\R$-basis $\bD^\bullet(r,s)$.  Thus, it has dimension
    \[
        (\dim_\R \DD)^{(r+s)/2} (r+s-1)!!
    \]
    if $r+s$ is even and dimension zero if $r+s$ is odd.  (We use here the fact that the number of perfect matchings of a set of size $2n$ is $(2n-1)!! := (2n-1)(2n-3) \dotsm 1$.)  On the other hand, by \cref{basisthm}, $\bigoplus_{X_1,\dotsc,X_r,Y_1,\dotsc,Y_s \in \{\upobj,\Pi^\sigma \downobj\}} \Hom_{\OB_\C(\DD)}(X_1 \otimes \dotsb \otimes X_r, Y_1 \otimes \dotsb \otimes Y_s)$ has the same dimension since, when $r+s$ is even, there are $(r+s-1)!!$ perfect matchings of the $r+s$ endpoints, $2^{(r+s)/2}$ choices for the orientations of the strands, and then $(\dim_\C \DD)^{(r+s)/2}$ ways to put a token labelled $b \in \bB_\DD^\C$ on each strand.  Thus, to prove that \cref{gravel} is an isomorphism, it suffices to prove that it is surjective.  To do this, it is enough to show that each generating morphism of $\OB_\C^\sigma(\DD;d)$, with appropriate parity shifts, is in the image of $\sS_\DD$.  Noting that
    \[
        \sS_\DD \left( \tfrac{1}{2}\ \idstrand - \tfrac{i}{2}\ \tokstrand[i] \right) =
        \begin{tikzpicture}[centerzero]
            \draw[->] (0,-0.2) -- (0,0.2);
            \shiftline{-0.1,0.2}{0.1,0.2}{0};
            \shiftline{-0.1,-0.2}{0.1,-0.2}{0};
        \end{tikzpicture}
        \qquad \text{and} \qquad
        \sS_\DD \left( \tfrac{1}{2}\ \idstrand + \tfrac{i}{2}\ \tokstrand[i] \right) =
        \begin{tikzpicture}[centerzero]
            \draw[<-] (0,-0.2) -- (0,0.2);
            \shiftline{-0.1,0.2}{0.1,0.2}{\sigma};
            \shiftline{-0.1,-0.2}{0.1,-0.2}{\sigma};
        \end{tikzpicture}
        ,
    \]
    we have
    \begin{align*}
        \tfrac{1}{4} \sS_\DD
        \left(
            \crossmor
            - i
            \begin{tikzpicture}[anchorbase]
                \draw (-0.2,-0.2) -- (0.2,0.2);
                \draw (0.2,-0.2) -- (-0.2,0.2);
                \token{-0.1,-0.1}{east}{i};
            \end{tikzpicture}
            - i\
            \begin{tikzpicture}[anchorbase]
                \draw (-0.2,-0.2) -- (0.2,0.2);
                \draw (0.2,-0.2) -- (-0.2,0.2);
                \token{0.1,-0.1}{west}{i};
            \end{tikzpicture}
            -
            \begin{tikzpicture}[anchorbase]
                \draw (-0.2,-0.2) -- (0.2,0.2);
                \draw (0.2,-0.2) -- (-0.2,0.2);
                \token{-0.1,-0.1}{east}{i};
                \token{0.1,-0.1}{west}{i};
            \end{tikzpicture}
        \right)
        &=
        \begin{tikzpicture}[centerzero]
            \draw[->] (-0.2,-0.2) -- (0.2,0.2);
            \draw[->] (0.2,-0.2) -- (-0.2,0.2);
            \shiftline{-0.3,0.2}{0.3,0.2}{0};
            \shiftline{-0.3,-0.2}{0.3,-0.2}{0};
        \end{tikzpicture}
        ,&
        \tfrac{1}{2} \sS_\DD
        \left(
            \tokstrand[a] - i\ \tokstrand[ia]
        \right)
        &=
        \begin{tikzpicture}[centerzero]
            \draw[->] (0,-0.2) -- (0,0.2);
            \token{0,0}{east}{a};
            \shiftline{-0.1,0.2}{0.1,0.2}{0};
            \shiftline{-0.1,-0.2}{0.1,-0.2}{0};
        \end{tikzpicture}
        ,\quad a \in \DD,
        \\
        \tfrac{1}{2} \sS_\DD
        \left(
            \capmor - i\
            \begin{tikzpicture}[anchorbase]
                \draw[-] (-0.15,-0.15) -- (-0.15,0) arc(180:0:0.15) -- (0.15,-0.15);
                \token{0.15,0}{west}{i};
            \end{tikzpicture}
        \right)
        &=
        \begin{tikzpicture}[anchorbase]
            \draw[<-] (-0.15,-0.15) -- (-0.15,0) arc(180:0:0.15) -- (0.15,-0.15);
            \shiftline{-0.25,0.25}{0.25,0.25}{0};
            \shiftline{-0.25,-0.15}{0.25,-0.15}{\sigma};
        \end{tikzpicture}
        \, ,&
        \tfrac{1}{2} \sS_\DD
        \left(
            \capmor + i\
            \begin{tikzpicture}[anchorbase]
                \draw[-] (-0.15,-0.15) -- (-0.15,0) arc(180:0:0.15) -- (0.15,-0.15);
                \token{0.15,0}{west}{i};
            \end{tikzpicture}
        \right)
        &=
        \begin{tikzpicture}[anchorbase]
            \draw[->] (-0.15,-0.15) -- (-0.15,0) arc(180:0:0.15) -- (0.15,-0.15);
            \shiftline{-0.25,0.25}{0.25,0.25}{0};
            \shiftline{-0.25,-0.15}{0.25,-0.15}{\sigma};
        \end{tikzpicture}
        \, ,
        \\
        \tfrac{1}{2} \sS_\DD
        \left(
            \cupmor + i\
            \begin{tikzpicture}[anchorbase]
                \draw (-0.15,0.15) -- (-0.15,0) arc(180:360:0.15) -- (0.15,0.15);
                \token{0.15,0}{west}{i};
            \end{tikzpicture}
        \right)
        &=
        \begin{tikzpicture}[anchorbase]
            \draw[<-] (-0.15,0.15) -- (-0.15,0) arc(180:360:0.15) -- (0.15,0.15);
            \shiftline{-0.25,0.15}{0.25,0.15}{\sigma};
            \shiftline{-0.25,-0.25}{0.25,-0.25}{0};
        \end{tikzpicture}
        \, ,&
        (-1)^\sigma \tfrac{1}{2} \sS_\DD
        \left(
            \cupmor - i\
            \begin{tikzpicture}[anchorbase]
                \draw (-0.15,0.15) -- (-0.15,0) arc(180:360:0.15) -- (0.15,0.15);
                \token{0.15,0}{west}{i};
            \end{tikzpicture}
        \right)
        &=
        \begin{tikzpicture}[anchorbase]
            \draw[->] (-0.15,0.15) -- (-0.15,0) arc(180:360:0.15) -- (0.15,0.15);
            \shiftline{-0.25,0.15}{0.25,0.15}{\sigma};
            \shiftline{-0.25,-0.25}{0.25,-0.25}{0};
        \end{tikzpicture}
        \, .
        \qedhere
    \end{align*}
\end{proof}

Fix $m,n \in \N$, and set $V = \DD^{m|n}$.  If $\DD = \Cl(\C)$, we assume that $n=0$; see \cref{fold}.  Note that, if $\varphi$ is a $(\nu,\star)$-superhermitian form on $V$, then $i \varphi$ is a $(-\nu,\star)$-superhermitian form on $V$.  Therefore, without loss of generality, we let $\varphi$ be a nondegenerate $(1,\star)$-superhermitian form on $V$ of parity $\sigma$.  (See \cref{subsec:hermformCstar,subsec:hermformClstar} for a classification.)  Let $\Phi = \form \circ \varphi$ be the corresponding nondegenerate $(\nu,\inv)$-supersymmetric form, with $a^\inv = (-1)^{\bar{a}} a^\star$; see \cref{crescent,divNak}.  Recall, from \cref{subsec:HCform}, that $G(\Phi)=G(\varphi)$.

\begin{lem} \label{white}
    For all $G(\Phi)$-supermodules $U$ and $W$, we have an isomorphism of $\C$-supermodules
    \[
        \Hom_{G(\Phi)}(U,W)^\C
        \xrightarrow{\cong} \Hom_{\fgl(m|n,\DD)}(U^\C,W^\C),\quad
        f \otimes a \mapsto f \otimes a,\quad
        f \in \Hom_{G(\Phi)}(U,W),\ a \in \C.
    \]
\end{lem}

\begin{proof}
    As explained in \cref{subsec:hermformCstar,subsec:hermformClstar,subsec:hermit-periplectic}, $G_\rd(\Phi)$ is connected.  Thus
    \[
        \Hom_{G(\Phi)}(U,W)^\C
        \cong \Hom_{\fg(\Phi)}(U,W)^\C
        \overset{\cref{skool}}{\cong} \Hom_{\fgl(m|n,\DD)}(U^\C,W^\C),
    \]
    where we used \cref{compCform} in the final isomorphism.
\end{proof}

It follows from \cref{white} that we have a canonical full and faithful superfunctor
\begin{equation}
    \sE_\DD \colon (G(\Phi)\smod_\R)^\C \to \fgl(m|n,\DD)\smod_\C
\end{equation}
sending $V$ to $V^\C$.

Since $\DD$ is complex division superalgebra, $V$ is naturally a  complex vector superspace, and hence the $\fg(\Phi)$-supermodule $V = \DD^{m|n}$ is naturally a supermodule over $\fg(\Phi)^\C \cong \fgl(m|n,\DD)$.  Recall the isomorphism $\Xi_\inv$ of \cref{Xinv}.  The next result shows that the diagram
\[
    \begin{tikzcd}
        \Brauer_\R^\sigma(\DD;m-n)^\C \arrow[r, "\sS_\DD"] \arrow[d, "\sF_\Phi^\C"']
        & \Add(\OB_\C(\DD;m-n)_\pi) \arrow[r, "\Xi_\inv"]
        & \Add(\OB_\C(\DD^\op;m-n)_\pi) \arrow[d, two heads, "\sG_{m|n}"]
        \\
        \left( G(\Phi)\smod_\R \right)^\C \arrow[rr, "\sE_\DD"]
        & & \fgl(m|n,\DD)\smod_\C
    \end{tikzcd}
\]
commutes up to supernatural isomorphism.  Recall the notation $\Phi^v$ introduced in \cref{step}, and the notation \cref{apple} for elements of a parity shift.

\begin{prop} \label{fire}
    There is an monoidal supernatural isomorphism of superfunctors
    \[
        \eta \colon \sE_\DD \sF_\Phi^\C \xrightarrow{\cong} \sG_{m|n} \Xi_\inv \sS_\DD
    \]
    determined by
    \begin{equation} \label{swirl}
        \eta_\go \colon V^\C \xrightarrow{\cong} V \oplus \Pi^\sigma V^*,\quad
        \eta_\go(v) = \tfrac{1}{\sqrt{2}} \left( v, \pi^\sigma \Phi^v \right).
    \end{equation}
\end{prop}

\begin{proof}
    To simplify notation, we set $\sG = \sG_{m|n}$, $\sS = \sS_\DD$, $\sE = \sE_{\fg}$, $\sF^\C = \sF_\Phi^\C$, and $\Xi = \Xi_\inv$.  First note that $\eta_\go$ is the composition of the isomorphisms of $\C$-supermodules
    \[
        V^\C
        \xrightarrow[\cref{dub}]{\cong} V \oplus V^\star
        \xrightarrow{\cong} V \oplus \Pi^\sigma V^*,
    \]
    where the second isomorphism uses the restriction of \cref{step} to $\C$-supermodules, noting that $V^\inv = V^\star$ as $\C$-supermodules.  Thus $\eta_\go$ is a parity-preserving isomorphism of $\C$-vector superspaces.  It is straightforward to verify that it is also a homomorphism of $\fg(\Phi)^\C$-supermodules.
    \details{
        For $X \in \fg(\Phi)$, $a \in \C$, and $v \in V$, we have
        \[
            (X \otimes a)(v \otimes 1)
            = Xv \otimes a
            \xmapsto{\cref{dub}} \tfrac{1}{\sqrt{2}} (Xva, Xva^\star)
            \mapsto \tfrac{1}{\sqrt{2}} (Xva, \pi^\sigma \Phi^{Xva^\star}).
        \]
        Now, for $w \in V$,
        \begin{multline*}
            \pi^\sigma \Phi^{Xva^\star}(w)
            = \pi^\sigma \Phi(Xva^\star,w)
            = - (-1)^{\bar{X}\bar{v}} \pi^\sigma \Phi(v,Xwa)
            = - (-1)^{\bar{X}\bar{v}} \pi^\sigma \Phi(v,(X \otimes a)w)
            \\
            = - (-1)^{\bar{X}\bar{v}} \pi^\sigma \Phi^v((X \otimes a)w)
            = \big( (X \otimes a) \pi^\sigma \Phi^v \big)(w).
        \end{multline*}
        Since we also have $Xva = (X \otimes a)v$, it follows that
        \[
            \eta_\go \big( (X \otimes a)(v \otimes 1) \big)
            = (X \otimes a) \eta_\go (v \otimes 1),
        \]
        as desired.
    }

    On objects, we have
    \[
        \sE \sF^\C(\go)
        = V^\C
        \xrightarrow[\cong]{\eta_\go}
        V \oplus \Pi^\sigma V^*
        = \sG (\upobj \oplus \Pi^\sigma \downobj)
        = \sG \Xi \sS(\go).
    \]
    For morphisms, we need to show that
    \[
        \eta_Y \circ \sE \sF^\C (f)
        = \sG \Xi \sS (f) \circ \eta_X
        \quad \text{for } f \in \left\{ \crossmor, \capmor, \cupmor, \tokstrand[a] : a \in \DD \right\},
    \]
    where $X$ and $Y$ are the domain and codomain of $f$, respectively.

    For $f = \crossmor$, we have
    \begin{multline*}
        \eta_{\go \otimes \go} \circ \sE \sF^\C (\crossmor)
        = \eta_{\go \otimes \go} \circ \flip_{V^\C,V^\C}
        \\
        = (\flip_{V \otimes V} + \Pi^\sigma \flip_{V \otimes V^*} + \Pi^\sigma \flip_{V \otimes V^*} + (-1)^\sigma \flip_{V^* \otimes V^*}) \circ \eta_{\go \otimes \go}
        = \sG \Xi \sS(\crossmor) \circ \eta_{\go \otimes \go},
    \end{multline*}
    where the sign of $(-1)^\sigma$ arises from the isomorphism \cref{quirk}.
    \details{
        When $\sigma = 1$,
        \[
            \sG
            \left(
                \begin{tikzpicture}[centerzero]
                    \draw[<-] (-0.2,-0.2) -- (0.2,0.2);
                    \draw[<-] (0.2,-0.2) -- (-0.2,0.2);
                    \shiftline{-0.3,0.2}{0.3,0.2}{0};
                    \shiftline{-0.3,-0.2}{0.3,-0.2}{0};
                \end{tikzpicture}
            \right)
            = \flip \colon V^* \otimes V^* \to V^* \otimes V^*.
        \]
        Thus, using the coherence maps, as in \cref{dubpi}, we have the composition
        \begin{gather*}
            \Pi V^* \otimes \Pi V^*
            \xrightarrow{\cong} V^* \otimes V^*
            \xrightarrow{\flip} V^* \otimes V^*
            \xrightarrow{\cong} \Pi V^* \otimes \Pi V^*,
            \\
            \pi f \otimes \pi g
            \mapsto - (-1)^{\bar{f}} f \otimes g
            \mapsto - (-1)^{\bar{f} + \bar{f}\bar{g}} g \otimes f
            \mapsto (-1)^{\bar{f}+\bar{g}+\bar{f}\bar{g}} \pi f \otimes \pi g.
        \end{gather*}
        On the other hand
        \begin{gather*}
            \Pi V^* \otimes \Pi V^*
            \xrightarrow{\flip} \Pi V^* \otimes \Pi V^*,
            \\
            \pi f \otimes \pi g
            \mapsto (-1)^{(\bar{f}+1)(\bar{g}+1)} \pi g \otimes \pi f
            = - (-1)^{\bar{f}+\bar{g}+\bar{f}\bar{g}} \pi f \otimes \pi g.
        \end{gather*}
    }

    For $f = \capmor$, noting that $\eta_\one$ is the identity map $\C \to \C$, we have, for $v,w \in V$,
    \[
        \eta_\one \circ \sE \sF^\C(\capmor)
        \colon V^\C \otimes_\C V^\C \to \C,\qquad
        v \otimes w \mapsto \Phi(v,w),
    \]
    and
    \begin{gather*}
        \sG \Xi \sS (\capmor) \circ \eta_{\go \otimes \go}
        = \sG
        \left(
            \begin{tikzpicture}[anchorbase]
                \draw[<-] (-0.15,-0.15) -- (-0.15,0) arc(180:0:0.15) -- (0.15,-0.15);
                \shiftline{-0.25,0.25}{0.25,0.25}{0};
                \shiftline{-0.25,-0.15}{0.25,-0.15}{\sigma};
            \end{tikzpicture}
            +
            \begin{tikzpicture}[anchorbase]
                \draw[->] (-0.15,-0.15) -- (-0.15,0) arc(180:0:0.15) -- (0.15,-0.15);
                \shiftline{-0.25,0.25}{0.25,0.25}{0};
                \shiftline{-0.25,-0.15}{0.25,-0.15}{\sigma};
            \end{tikzpicture}
        \right)
        \circ \eta_{\go \otimes \go}
        \colon V^\C \otimes V^\C \to \C,
        \\
        v \otimes w \mapsto \tfrac{1}{2}(v, \pi^\sigma \Phi^v) \otimes (w, \pi^\sigma \Phi^w)
        \mapsto \tfrac{1}{2} \left( \Phi(v,w) + (-1)^{\bar{v}\bar{w}} \Phi(w,v) \right)
        \overset{\cref{supersymmetric}}{=} \Phi(v,w),
    \end{gather*}
    as desired.  (Above, we have used the fact that the maps are uniquely determined by their values on $v \otimes w$, for $v,w \in V \subseteq V^\C$.)

    Next we consider $f = \cupmor$.  Let $\bB_V^\C$ be a homogeneous $\C$-basis for $V$.  Then $\bB_V^\C \cup \bB_V^\C i$ is a homogeneous $\R$-basis for $V$.  It is straightforward to verify that, for all $w \in \bB_V^\C$, we have
    \[
        (wi)^\vee = w^\vee i
        \quad \text{and} \quad
        \Phi^{w^\vee} = w^*,
    \]
    where $\vee$ denotes left duals with respect to $\Phi$.
    \details{
        For $v,w \in \bB_V^\C$, we have
        \[
            \Phi(v^\vee i, w i)
            \overset{\cref{supersymmetric}}{=} \Phi(v^\vee, w i (-i))
            = \Phi(v^\vee,w)
            = \delta_{w,v}
        \]
        and
        \[
            \Phi(v^\vee i, w)
            \overset{\cref{supersymmetric}}{=} - \Phi(v^\vee, wi)
            = 0.
        \]
        Also,
        \[
            \Phi^{w^\vee}(v)
            = \Phi(w^\vee,v)
            = \delta_{wv}.
        \]
    }
    Identifying $v \in V$ and $\pi^\sigma f \in \Pi^\sigma V^*$ with $(v,0),(0,\pi^\sigma f) \in V \oplus \Pi^\sigma V^*$, we can write $(v,\pi^\sigma f)$ as $v + \pi^\sigma f$.  Using this convention,
    \begin{align*}
        \eta_{\go \otimes \go} \circ \sE \sF^\C(\cupmor)
        \colon \C &\to (V \oplus \Pi^\sigma V^*) \otimes_\C (V \oplus \Pi^\sigma V^*)
        \\
        &\xrightarrow{\cong} (V \otimes_\C V) \oplus \Pi^\sigma(V \otimes_\C V^*) \oplus \Pi^\sigma(V^* \otimes_\C V) \oplus (V^* \otimes_\C V^*)
    \end{align*}
    is the map given by
    \begin{align*}
        1 &\mapsto \frac{1}{2} \sum_{w \in \bB_V^\C} (-1)^{\sigma\bar{w}} (w + \pi^\sigma \Phi^w) \otimes (w^\vee + \pi^\sigma \Phi^{w^\vee}) + \frac{1}{2} \sum_{w \in \bB_V^\C} (-1)^{\sigma\bar{w}} (iw + \pi^\sigma \Phi^{iw}) \otimes (w^\vee i + \pi^\sigma \Phi^{w^\vee i})
        \\
        &= \frac{1}{2} \sum_{w \in \bB_V^\C} (-1)^{\sigma\bar{w}} (w + \pi^\sigma \Phi^w) \otimes (w^\vee + \pi^\sigma \Phi^{w^\vee}) + \frac{1}{2} \sum_{w \in \bB_V^\C} (-1)^{\sigma\bar{w}} (w i - \pi^\sigma \Phi^w i) \otimes (w^\vee i - \pi^\sigma \Phi^{w^\vee} i)
        \\
        &= \sum_{w \in \bB_V^\C} (-1)^{\sigma\bar{w}} w \otimes \pi^\sigma \Phi^{w^\vee} + \sum_{w \in \bB_V^\C} (-1)^{\sigma\bar{w}} \pi^\sigma \Phi^w \otimes w^\vee
        \\
        &= \sum_{w \in \bB_V^\C} (-1)^{\sigma\bar{w}} w \otimes \pi^\sigma w^* + \sum_{w \in \bB_V^\C} (-1)^{\bar{w}} \pi^\sigma w^* \otimes w
        \\
        &\mapsto \sum_{w \in \bB_V^\C} \pi^\sigma w \otimes w^* + \sum_{w \in \bB_V^\C} (-1)^{\bar{w}} \pi^\sigma w^* \otimes w,
    \end{align*}
    where, in the final equality, we used the fact that the last sum is independent of the choice of basis to sum over the basis $\bB_V^{\C,\vee}$, and the fact that $(w^\vee)^\vee = (-1)^{\sigma \bar{w} + \bar{w}} w$.
    \details{
        For $v,w \in \bB_V^\C$, we have
        \[
            \Phi(v,w^\vee)
            = (-1)^{\bar{v}(\sigma+\bar{w})} \Phi(w^\vee,v)
            = (-1)^{\bar{v}(\sigma+\bar{w})} \delta_{vw}
            = (-1)^{\sigma \bar{w} + \bar{w}}.
        \]
    }
    On the other hand, we have
    \begin{gather*}
        \sG \Xi \sS (\cupmor) \circ \eta_\one
        = \sG
        \left(
            \begin{tikzpicture}[anchorbase]
                \draw[<-] (-0.15,0.15) -- (-0.15,0) arc(180:360:0.15) -- (0.15,0.15);
                \shiftline{-0.25,0.15}{0.25,0.15}{\sigma};
                \shiftline{-0.25,-0.25}{0.25,-0.25}{0};
            \end{tikzpicture}
            + (-1)^\sigma
            \begin{tikzpicture}[anchorbase]
                \draw[->] (-0.15,0.15) -- (-0.15,0) arc(180:360:0.15) -- (0.15,0.15);
                \shiftline{-0.25,0.15}{0.25,0.15}{\sigma};
                \shiftline{-0.25,-0.25}{0.25,-0.25}{0};
            \end{tikzpicture}
        \right)
        \colon \C \to \Pi^\sigma (V \otimes V^*) \oplus \Pi^\sigma (V^* \otimes_\C V),
        \\
        1 \mapsto \sum_{w \in \bB_V^\C} \pi^\sigma w \otimes w^*
        + \sum_{w \in \bB_V^\C} (-1)^{\bar{w}} \pi^\sigma w^* \otimes w.
    \end{gather*}

    Finally, for $f = \tokstrand[a]$, $a \in \DD$, we have
    \[
        \eta_\go \circ \sE \sF^\C(\tokstrand[a])
        \colon V^\C \to V \oplus \Pi^\sigma V^*,\quad
        v
        \mapsto (-1)^{\bar{a}\bar{v}} v a^\inv
        \mapsto (-1)^{\bar{a}\bar{v}} \tfrac{1}{\sqrt{2}} (v a^\inv, \pi^\sigma \Phi^{v a^\inv})
    \]
    and
    \begin{gather*}
        \sG \Xi \sS(\tokstrand[a]) \circ \eta_\go
        = \sG
        \left(
            \begin{tikzpicture}[centerzero]
                \draw[->] (0,-0.2) -- (0,0.2);
                \token{0,0}{east}{(a^\inv)^\op};
                \shiftline{-0.1,0.2}{0.1,0.2}{0};
                \shiftline{-0.1,-0.2}{0.1,-0.2}{0};
            \end{tikzpicture}
            + (-1)^{\sigma\bar{a}}
            \begin{tikzpicture}[centerzero]
                \draw[<-] (0,-0.2) -- (0,0.2);
                \token{0,0}{east}{a^\op};
                \shiftline{-0.1,0.2}{0.1,0.2}{\sigma};
                \shiftline{-0.1,-0.2}{0.1,-0.2}{\sigma};
            \end{tikzpicture}
        \right)
        \colon V^\C \to V \oplus \Pi^\sigma V^*,
        \\
        v
        \mapsto \tfrac{1}{\sqrt{2}} (v,\pi^\sigma \Phi^v)
        \xmapsto{\cref{batty}} \tfrac{1}{\sqrt{2}} \left( (-1)^{\bar{a}\bar{v}} v a^\inv, (-1)^{\sigma\bar{a}} a \pi^\sigma \Phi^v \right)
        = (-1)^{\bar{a}\bar{v}} \tfrac{1}{\sqrt{2}} (v a^\inv, \pi^\sigma \Phi^{v a^\inv}).
        \qedhere
    \end{gather*}
\end{proof}

\begin{prop} \label{divfull:C}
    \Cref{divfull} holds when $(\DD,\star)$ is equal to $(\C,\star)$ or $(\Cl(\C),\star)$.
\end{prop}

\begin{proof}
    To show that the superfunctor $\sF_\Phi$ is full, we must show that the $\R$-linear map
    \[
        \sF_\Phi \colon \Hom_{\Brauer_\R^\sigma(\DD;m-n)}(\go^{\otimes r}, \go^{\otimes s})
        \to \Hom_{G(\Phi)}(V^{\otimes r}, V^{\otimes s})
    \]
    is surjective.  This map is surjective if and only if the induced map
    \begin{equation} \label{fernC}
        \sF_\Phi^\C \colon \Hom_{\Brauer_\R^\sigma(\DD;m-n)}(\go^{\otimes r}, \go^{\otimes s})^\C
        \to \Hom_{G(\Phi)}(V^{\otimes r}, V^{\otimes s})^\C
    \end{equation}
    is surjective.  To show that \cref{fernC} is surjective, it suffices to show that the diagram
    \[
        \begin{tikzcd}
            \Hom_{\Brauer_\R^\sigma(\DD;m-n)}(\go^{\otimes r}, \go^{\otimes s})^\C \arrow[d, "\sS_\DD"', "\cong"] \arrow[r, "\sF_\Phi^\C"]
            & \Hom_{G(\Phi)}(V^{\otimes r}, V^{\otimes s})^\C \arrow[d,"\cong", "\sE_\DD"']
            \\
            \Hom_{\Add(\OB_\C(\DD;m-n))} \left( (\upobj \oplus \Pi^\sigma \downobj)^{\otimes r}, (\upobj \oplus \Pi^\sigma \downobj)^{\otimes s} \right) \arrow[d, "\cong", "\Xi_\inv"']
            & \Hom_{\fgl(m|n,\DD)} \left( (V^\C)^{\otimes r}, (V^\C)^{\otimes s} \right)
            \\
            \Hom_{\Add(\OB_\C(\DD^\op;m-n))} \left( (\upobj \oplus \Pi^\sigma \downobj)^{\otimes r}, (\upobj \oplus \Pi^\sigma \downobj)^{\otimes s} \right)
            \arrow[r, "\sG_{m|n}"]
            & \Hom_{\fgl(m|n,\DD)} \left( (V \oplus \Pi^\sigma V^*)^{\otimes r}, (V \oplus \Pi^\sigma V^*)^{\otimes s} \right) \arrow[u,"\cong"']
        \end{tikzcd}
    \]
    commutes, where the bottom-right vertical isomorphism is induced by the isomorphism \cref{swirl}.  Commutativity of this diagram follows from \cref{fire}.
\end{proof}

As a special case of $G(\Phi)$, we have the supergroups $\rU(p,q|r,s)$ and $\rUQ(p,q)$; see \cref{subsec:hermformCstar,subsec:hermformClstar}.  Recall the definition $G(\Phi)\tsmod_\R$ of the monoidal supercategory of tensor $G(\Phi)$-supermodules from \cref{subsec:Bnonsuper}.

\begin{prop} \label{iceC}
    If $p,p',q,q',r,r',s,s' \in \N$ satisfy $p+q=p'+q'$ and $r+s=r'+s'$, then we have equivalences of monoidal supercategories
    \begin{align*}
        \rU(p,q|r,s)\tsmod_\R &\simeq \rU(p',q'|r',s')\tsmod_\R \qquad \text{and} \\
        \rUQ(p,q)\tsmod_\R &\simeq \rUQ(p',q')\tsmod_\R,
    \end{align*}
    sending the natural supermodule to the natural supermodule.
\end{prop}

\begin{proof}
    The proof is analogous to that of \cref{iceR}, using the commutative diagram appearing in the proof of \cref{divfull:C}.
\end{proof}

\section{Unoriented fullness: quaternionic case\label{sec:quaternionic}}

In this section, we prove \cref{divfull} in the case where $(\DD,\star) = (\HH,\star)$.  We will naturally view $\C = \R + \R i$ as a subalgebra of $\HH$.

Suppose that $V$ is an $\HH$-vector superspace.  Then, using \cref{jcong}, we have an isomorphism of $\C$-vector spaces
\begin{equation}
    V \xrightarrow{\cong} V^\star,\qquad
    v \mapsto vj,
\end{equation}
where $V^\star$ denotes the conjugate $\C$-vector superspace, as in \cref{stellar}.  Combining with \cref{dub}, this yields an isomorphism of $\C$-vector spaces
\begin{equation} \label{Hdub}
    V^\C \xrightarrow{\cong} V \oplus V,\qquad
    v \mapsto \tfrac{1}{\sqrt{2}}(v,vj), \quad v \in V.
\end{equation}
Note that $\C$-linearity implies that
\[
    v \otimes a \mapsto \tfrac{1}{\sqrt{2}}(va,vja)
    = \tfrac{1}{\sqrt{2}}(va,va^\star j),\quad
    v \in V,\ a \in \C.
\]

Recall that the additive envelope of a linear category $\cC$ is the category $\Add(\cC)$ whose objects are formal direct sums of objects in $\cC$, and whose morphisms are identified with matrices of morphisms in $\cC$ in the usual way.   We write morphisms in additive envelopes as sums of their components, as in \cref{subsec:Bbasis}.  In what follows, we will often be considering the object $\go \oplus \go \in \Add(\Brauer_\C)$.  In order to distinguish the two copies of $\go$, we will color the first blue and the second red.  We will then color diagrams in such a way that the color of their endpoints indicate which copy of $\go$ is involved.  In order to make the diagrams also readable without color, blue strands will be made thick and the blue copy of $\go$ will be written in bold.  Thus, for example,
\[
    \idstrandb =
    \begin{pmatrix}
        1_\go & 0 \\
        0 & 0
    \end{pmatrix}
    ,\
    \idstrandr =
    \begin{pmatrix}
        0 & 0 \\
        0 & 1_\go
    \end{pmatrix}
    ,\
    \idstrandbr =
    \begin{pmatrix}
        0 & 0 \\
        1_\go & 0
    \end{pmatrix}
    ,\
    \idstrandrb =
    \begin{pmatrix}
        0 & 1_\go \\
        0 & 0
    \end{pmatrix}
    \in \Hom_{\Add(\Brauer_\C)}(\gob \oplus \gor, \gob \oplus \gor),
\]
And
\[
    \crossmorbr
    =
    \begin{pmatrix}
        0 & 0 & 0 & 0 \\
        0 & 0 & 0 & 0 \\
        0 & \crossmor & 0 & 0 \\
        0 & 0 & 0 & 0
    \end{pmatrix}
    \in \End_{\Add(\Brauer_\C)}((\gob \oplus \gor)^{\otimes 2})
    = \End_{\Add(\Brauer_\C)}((\gob \otimes \gob) \oplus (\gob \otimes \gor) \oplus (\gor \otimes \gob) \oplus (\gor \otimes \gor)).
\]

\begin{prop} \label{space}
    There exists a unique $\C$-linear monoidal functor
    \[
        \sS_\HH \colon \Brauer_\R^\sigma(\HH;d)^\C \to \Add(\Brauer_\C^\sigma(\C;-2d))
    \]
    such that $\sS_\HH(\go) = \gob \oplus \gor$ and
    \begin{gather}
        \sS_\HH \left( \crossmor \right) = - \crossmorbb - \crossmorrr - \crossmorbr - \crossmorrb,\qquad
        \sS_\HH \left( \capmor \right) = \capmorbr - \capmorrb\, ,\qquad
        \sS_\HH \left( \cupmor \right) = \cupmorrb - \cupmorbr\, ,
        \\
        \sS_\HH \left( \tokstrand[i] \right) = i\ \idstrandr - i\ \idstrandb\, ,\qquad
        \sS_\HH \left( \tokstrand[j] \right) = \idstrandbr - \idstrandrb\, ,\qquad
        \sS_\HH \left( \tokstrand[k] \right) = i\ \idstrandbr + i\ \idstrandrb\, .
    \end{gather}
    The functor $\sS_\HH$ is full and faithful.
\end{prop}

\begin{proof}
    To show that $\sS_\HH$ is well defined, we must show that it respects the relations \cref{brauer,tokrel}.  For the first relation in \cref{brauer}, we verify that
    \[
        \sS_\HH
        \left(
            \begin{tikzpicture}[centerzero]
                \draw (0.2,-0.4) to[out=135,in=down] (-0.15,0) to[out=up,in=225] (0.2,0.4);
                \draw (-0.2,-0.4) to[out=45,in=down] (0.15,0) to[out=up,in=-45] (-0.2,0.4);
            \end{tikzpicture}
        \right)
        =
        \begin{tikzpicture}[centerzero]
            \draw[bcolor] (0.2,-0.4) to[out=135,in=down] (-0.15,0) to[out=up,in=225] (0.2,0.4);
            \draw[bcolor] (-0.2,-0.4) to[out=45,in=down] (0.15,0) to[out=up,in=-45] (-0.2,0.4);
        \end{tikzpicture}
        +
        \begin{tikzpicture}[centerzero]
            \draw[rcolor] (0.2,-0.4) to[out=135,in=down] (-0.15,0) to[out=up,in=225] (0.2,0.4);
            \draw[rcolor] (-0.2,-0.4) to[out=45,in=down] (0.15,0) to[out=up,in=-45] (-0.2,0.4);
        \end{tikzpicture}
        +
        \begin{tikzpicture}[centerzero]
            \draw[bcolor] (0.2,-0.4) to[out=135,in=down] (-0.15,0) to[out=up,in=225] (0.2,0.4);
            \draw[rcolor] (-0.2,-0.4) to[out=45,in=down] (0.15,0) to[out=up,in=-45] (-0.2,0.4);
        \end{tikzpicture}
        +
        \begin{tikzpicture}[centerzero]
            \draw[rcolor] (0.2,-0.4) to[out=135,in=down] (-0.15,0) to[out=up,in=225] (0.2,0.4);
            \draw[bcolor] (-0.2,-0.4) to[out=45,in=down] (0.15,0) to[out=up,in=-45] (-0.2,0.4);
        \end{tikzpicture}
        =
        \begin{tikzpicture}[centerzero]
            \draw[bcolor] (-0.2,-0.4) -- (-0.2,0.4);
            \draw[bcolor] (0.2,-0.4) -- (0.2,0.4);
        \end{tikzpicture}
        +
        \begin{tikzpicture}[centerzero]
            \draw[rcolor] (-0.2,-0.4) -- (-0.2,0.4);
            \draw[rcolor] (0.2,-0.4) -- (0.2,0.4);
        \end{tikzpicture}
        +
        \begin{tikzpicture}[centerzero]
            \draw[bcolor] (-0.2,-0.4) -- (-0.2,0.4);
            \draw[rcolor] (0.2,-0.4) -- (0.2,0.4);
        \end{tikzpicture}
        +
        \begin{tikzpicture}[centerzero]
            \draw[rcolor] (-0.2,-0.4) -- (-0.2,0.4);
            \draw[bcolor] (0.2,-0.4) -- (0.2,0.4);
        \end{tikzpicture}
        =
        \sS_\HH
        \left(
            \begin{tikzpicture}[centerzero]
                \draw (-0.2,-0.4) -- (-0.2,0.4);
                \draw (0.2,-0.4) -- (0.2,0.4);
            \end{tikzpicture}
        \right).
    \]
    The proof of the second relation in \cref{brauer} is similar, since both sides are mapped by $\sS_\HH$ to the negative of the sum over all possible colorings of the strands.

    For the third relation in \cref{brauer}, we have
    \[
        \sS_\HH
        \left(
            \begin{tikzpicture}[centerzero]
                \draw (-0.3,0.4) -- (-0.3,0) arc(180:360:0.15) arc(180:0:0.15) -- (0.3,-0.4);
            \end{tikzpicture}
        \right)
        =
        \begin{tikzpicture}[centerzero]
            \draw[bcolor] (-0.3,0.4) -- (-0.3,0) arc(180:270:0.15);
            \draw[rcolor] (-0.15,-0.15) arc(270:360:0.15) arc(180:90:0.15);
            \draw[bcolor] (0.15,0.15) arc(90:0:0.15) -- (0.3,-0.4);
        \end{tikzpicture}
        +
        \begin{tikzpicture}[centerzero]
            \draw[rcolor] (-0.3,0.4) -- (-0.3,0) arc(180:270:0.15);
            \draw[bcolor] (-0.15,-0.15) arc(270:360:0.15) arc(180:90:0.15);
            \draw[rcolor] (0.15,0.15) arc(90:0:0.15) -- (0.3,-0.4);
        \end{tikzpicture}
        =\,
        \begin{tikzpicture}[centerzero]
            \draw[bcolor] (0,-0.4) -- (0,0.4);
        \end{tikzpicture}
        \, +\,
        \begin{tikzpicture}[centerzero]
            \draw[rcolor] (0,-0.4) -- (0,0.4);
        \end{tikzpicture}
        \, = \sS_\HH
        \left(\
            \begin{tikzpicture}[centerzero]
                \draw (0,-0.4) -- (0,0.4);
            \end{tikzpicture}\
        \right).
    \]
    The proof of the fourth relation in \cref{brauer} is analogous.
    \details{
        We have
        \[
            \sS_\HH
            \left(
                \begin{tikzpicture}[centerzero]
                    \draw (-0.3,-0.4) -- (-0.3,0) arc(180:0:0.15) arc(180:360:0.15) -- (0.3,0.4);
                \end{tikzpicture}
            \right)
            =
            \left(
                \begin{tikzpicture}[centerzero]
                    \draw[bcolor] (-0.3,-0.4) -- (-0.3,0) arc(180:90:0.15);
                    \draw[rcolor] (-0.15,0.15) arc(90:0:0.15) arc(180:270:0.15);
                    \draw[bcolor] (0.15,-0.15) arc(-90:0:0.15) -- (0.3,0.4);
               \end{tikzpicture}
               +
               \begin{tikzpicture}[centerzero]
                    \draw[rcolor] (-0.3,-0.4) -- (-0.3,0) arc(180:90:0.15);
                    \draw[bcolor] (-0.15,0.15) arc(90:0:0.15) arc(180:270:0.15);
                    \draw[rcolor] (0.15,-0.15) arc(-90:0:0.15) -- (0.3,0.4);
               \end{tikzpicture}
            \right)
            = (-1)^\sigma
            \left(
                \begin{tikzpicture}[centerzero]
                    \draw[bcolor] (0,-0.4) -- (0,0.4);
                \end{tikzpicture}
                \, +\,
                \begin{tikzpicture}[centerzero]
                    \draw[rcolor] (0,-0.4) -- (0,0.4);
                \end{tikzpicture}
            \right)
            \, = \sS_\HH
            \left( (-1)^\sigma\
                \begin{tikzpicture}[centerzero]
                    \draw (0,-0.4) -- (0,0.4);
                \end{tikzpicture}\
            \right).
        \]
    }

    To verify the fifth relation in \cref{brauer}, we compute
    \[
        \sS_\HH
        \left(
            \begin{tikzpicture}[anchorbase]
                \draw (-0.15,-0.4) to[out=60,in=-90] (0.15,0) arc(0:180:0.15) to[out=-90,in=120] (0.15,-0.4);
            \end{tikzpicture}
        \right)
        =
        \begin{tikzpicture}[anchorbase]
            \draw[bcolor] (-0.15,-0.4) to[out=60,in=-90] (0.15,0) arc(0:90:0.15);
            \draw[rcolor] (0,0.15) arc(90:180:0.15) to[out=-90,in=120] (0.15,-0.4);
        \end{tikzpicture}
        -
        \begin{tikzpicture}[anchorbase]
            \draw[rcolor] (-0.15,-0.4) to[out=60,in=-90] (0.15,0) arc(0:90:0.15);
            \draw[bcolor] (0,0.15) arc(90:180:0.15) to[out=-90,in=120] (0.15,-0.4);
        \end{tikzpicture}
        =
        \begin{tikzpicture}[anchorbase]
            \draw[bcolor] (-0.15,-0.6) -- (-0.15,-0.4);
            \draw[rcolor] (-0.15,-0.4) to[out=up,in=-90] (0.15,0) arc(0:180:0.15) to[out=-90,in=up] (0.15,-0.4) -- (0.15,-0.6);
        \end{tikzpicture}
        -
        \begin{tikzpicture}[anchorbase]
            \draw[rcolor] (-0.15,-0.6) -- (-0.15,-0.4);
            \draw[bcolor] (-0.15,-0.4) to[out=up,in=-90] (0.15,0) arc(0:180:0.15) to[out=-90,in=up] (0.15,-0.4) -- (0.15,-0.6);
        \end{tikzpicture}
        = \left( \capmorbr - \capmorrb \right)
        = \sS_\HH \left( \capmor \right).
    \]
    For the sixth relation in \cref{brauer}, we have
    \[
        \sS_\HH
        \left(
            \begin{tikzpicture}[centerzero]
                \draw (-0.2,-0.3) -- (-0.2,-0.1) arc(180:0:0.2) -- (0.2,-0.3);
                \draw (-0.3,0.3) \braiddown (0,-0.3);
            \end{tikzpicture}
        \right)
        =
        \begin{tikzpicture}[centerzero]
            \draw[rcolor] (-0.2,-0.3) -- (-0.2,-0.1) arc(180:90:0.2);
            \draw[bcolor] (0,0.1) arc(90:0:0.2) -- (0.2,-0.3);
            \draw[bcolor] (-0.3,0.3) \braiddown (0,-0.3);
        \end{tikzpicture}
        -
        \begin{tikzpicture}[centerzero]
            \draw[bcolor] (-0.2,-0.3) -- (-0.2,-0.1) arc(180:90:0.2);
            \draw[rcolor] (0,0.1) arc(90:0:0.2) -- (0.2,-0.3);
            \draw[bcolor] (-0.3,0.3) \braiddown (0,-0.3);
        \end{tikzpicture}
        +
        \begin{tikzpicture}[centerzero]
            \draw[rcolor] (-0.2,-0.3) -- (-0.2,-0.1) arc(180:90:0.2);
            \draw[bcolor] (0,0.1) arc(90:0:0.2) -- (0.2,-0.3);
            \draw[rcolor] (-0.3,0.3) \braiddown (0,-0.3);
        \end{tikzpicture}
        -
        \begin{tikzpicture}[centerzero]
            \draw[bcolor] (-0.2,-0.3) -- (-0.2,-0.1) arc(180:90:0.2);
            \draw[rcolor] (0,0.1) arc(90:0:0.2) -- (0.2,-0.3);
            \draw[rcolor] (-0.3,0.3) \braiddown (0,-0.3);
        \end{tikzpicture}
        =
        \begin{tikzpicture}[centerzero]
            \draw[rcolor] (-0.2,-0.3) -- (-0.2,-0.1) arc(180:90:0.2);
            \draw[bcolor] (0,0.1) arc(90:0:0.2) -- (0.2,-0.3);
            \draw[bcolor] (0.3,0.3) \braiddown (0,-0.3);
        \end{tikzpicture}
        -
        \begin{tikzpicture}[centerzero]
            \draw[bcolor] (-0.2,-0.3) -- (-0.2,-0.1) arc(180:90:0.2);
            \draw[rcolor] (0,0.1) arc(90:0:0.2) -- (0.2,-0.3);
            \draw[bcolor] (0.3,0.3) \braiddown (0,-0.3);
        \end{tikzpicture}
        +
        \begin{tikzpicture}[centerzero]
            \draw[rcolor] (-0.2,-0.3) -- (-0.2,-0.1) arc(180:90:0.2);
            \draw[bcolor] (0,0.1) arc(90:0:0.2) -- (0.2,-0.3);
            \draw[rcolor] (0.3,0.3) \braiddown (0,-0.3);
        \end{tikzpicture}
        -
        \begin{tikzpicture}[centerzero]
            \draw[bcolor] (-0.2,-0.3) -- (-0.2,-0.1) arc(180:90:0.2);
            \draw[rcolor] (0,0.1) arc(90:0:0.2) -- (0.2,-0.3);
            \draw[rcolor] (0.3,0.3) \braiddown (0,-0.3);
        \end{tikzpicture}
        = \sS_\HH
        \left(
            \begin{tikzpicture}[centerzero]
                \draw (-0.2,-0.3) -- (-0.2,-0.1) arc(180:0:0.2) -- (0.2,-0.3);
                \draw (0.3,0.3) \braiddown (0,-0.3);
            \end{tikzpicture}
        \right).
    \]

    The first two relations in \cref{tokrel} are straightforward.  For the third relation in \cref{tokrel}, it suffices to consider the cases where $a,b \in \{1,i,j,k\}$.  These are all straightforward computations.  For example, we have
    \begin{align*}
        \sS_\HH
        \left(
            \begin{tikzpicture}[centerzero]
                \draw (0,-0.35) -- (0,0.35);
                \token{0,-0.15}{east}{i};
                \token{0,0.15}{east}{i};
            \end{tikzpicture}
        \right)
        &= - \idstrandr - \idstrandb
        = \sS_\HH
        \left(
            \begin{tikzpicture}[centerzero]
                \draw (0,-0.35) -- (0,0.35);
                \token{0,0}{west}{-1};
            \end{tikzpicture}
        \right),
        &
        \sS_\HH
        \left(
            \begin{tikzpicture}[centerzero]
                \draw (0,-0.35) -- (0,0.35);
                \token{0,-0.15}{east}{j};
                \token{0,0.15}{east}{j};
            \end{tikzpicture}
        \right)
        &= -\,
        \begin{tikzpicture}[centerzero]
            \draw[bcolor] (0,-0.4) -- (0,-0.15);
            \draw[rcolor] (0,-0.15) -- (0,0.15);
            \draw[bcolor] (0,0.15) -- (0,0.4);
        \end{tikzpicture}
        \, -\,
        \begin{tikzpicture}[centerzero]
            \draw[rcolor] (0,-0.4) -- (0,-0.15);
            \draw[bcolor] (0,-0.15) -- (0,0.15);
            \draw[rcolor] (0,0.15) -- (0,0.4);
        \end{tikzpicture}
        = - \idstrandb - \idstrandr
        = \sS_\HH
        \left(
            \begin{tikzpicture}[centerzero]
                \draw (0,-0.35) -- (0,0.35);
                \token{0,0}{west}{-1};
            \end{tikzpicture}
        \right),
        \\
        \sS_\HH
        \left(
            \begin{tikzpicture}[centerzero]
                \draw (0,-0.35) -- (0,0.35);
                \token{0,-0.15}{east}{j};
                \token{0,0.15}{east}{i};
            \end{tikzpicture}
        \right)
        &= i\ \idstrandbr + i\ \idstrandrb\,
        = \sS_\HH
        \left(
            \begin{tikzpicture}[centerzero]
                \draw (0,-0.35) -- (0,0.35);
                \token{0,0}{west}{k};
            \end{tikzpicture}
        \right),
        &
        \sS_\HH
        \left(
            \begin{tikzpicture}[centerzero]
                \draw (0,-0.35) -- (0,0.35);
                \token{0,-0.15}{east}{i};
                \token{0,0.15}{east}{j};
            \end{tikzpicture}
        \right)
        &= - i\ \idstrandrb - i\ \idstrandbr\,
        = \sS_\HH
        \left(
            \begin{tikzpicture}[centerzero]
                \draw (0,-0.35) -- (0,0.35);
                \token{0,0}{west}{-k};
            \end{tikzpicture}
        \right).
    \end{align*}
    The other cases are analogous.  It is also straightforward to verify the fourth relation in \cref{tokrel} by considering the cases $a \in \{1,i,j,k\}$.  For the fifth relation in \cref{tokrel}, we again consider the cases $a \in \{1,i,j,k\}$.  When $a=i$, we have
    \[
        \sS_\HH
        \left(
            \begin{tikzpicture}[anchorbase]
                \draw (-0.2,-0.2) -- (-0.2,0) arc (180:0:0.2) -- (0.2,-0.2);
                \token{-0.2,0}{east}{i};
            \end{tikzpicture}
        \right)
        = -i\ \capmorrb - i\ \capmorbr
        = \sS_\HH
        \left(
            \begin{tikzpicture}[anchorbase]
                \draw (-0.2,-0.2) -- (-0.2,0) arc (180:0:0.2) -- (0.2,-0.2);
                \token{0.2,0}{west}{-i};
            \end{tikzpicture}
        \right)
        = \sS_\HH
        \left(
            \begin{tikzpicture}[anchorbase]
                \draw (-0.2,-0.2) -- (-0.2,0) arc (180:0:0.2) -- (0.2,-0.2);
                \token{0.2,0}{west}{i^\star};
            \end{tikzpicture}
        \right).
    \]
    The other cases are analogous.

    Finally, we show that $\sS_\HH$ respects \cref{burst}.  First note that, by \cref{Oregon}, any bubble with a purely imaginary token is zero.  By \cref{toklin}, any bubble with a token labelled by $a \in \R$ is equal to $a$ times a bubble with no token.  Then the fact that $\sS_\HH$ respects \cref{burst} follows from the computation
    \[
        \sS_\HH
        \left(
            \begin{tikzpicture}[centerzero]
                \bub{0,0};
            \end{tikzpicture}
        \right)
        = -
        \begin{tikzpicture}[centerzero]
            \draw[bcolor] (0,-0.2) arc(-90:90:0.2);
            \draw[rcolor] (0,0.2) arc(90:270:0.2);
        \end{tikzpicture}
        -
        \begin{tikzpicture}[centerzero]
            \draw[rcolor] (0,-0.2) arc(-90:90:0.2);
            \draw[bcolor] (0,0.2) arc(90:270:0.2);
        \end{tikzpicture}
        = -2\
        \begin{tikzpicture}[centerzero]
            \bub{0,0};
        \end{tikzpicture}
        \ ,
    \]
    and the fact that, by \cref{delay},
    \[
        \str_\HH^\R(a) = 4 a
        \qquad \text{and} \qquad
        \str_\C^\C(a) = a
        \qquad \text{for all } a \in \R.
    \]

    It remains to prove that $\sS_\HH$ is full and faithful.  For $r,s \in \N$, the functor $\sS_\HH$ induces a $\C$-linear map
    \begin{equation} \label{piano}
        \Hom_{\Brauer_\R^\sigma(\HH;d)}(\go^{\otimes r}, \go^{\otimes s})^\C
        \to \Hom_{\Add(\Brauer_\C^\sigma(\C;-2d))}((\gob \oplus \gor)^{\otimes r}, (\gob \oplus \gor)^{\otimes s}).
    \end{equation}
    By \cref{basisthm}, $\Hom_{\Brauer_\R^\sigma(\HH;d)}(\go^{\otimes r}, \go^{\otimes s})^\C$ has $\C$-basis $\bD^\bullet(r,s)$.  Thus, it has dimension
    \[
        4^{(r+s)/2} (r+s-1)!! = 2^{r+s} (r+s-1)!!
    \]
    if $r+s$ is even and dimension zero if $r+s$ is odd.  (We use here the fact that the number of perfect matchings of a set of size $2n$ is $(2n)!! := (2n-1)(2n-3) \dotsm 1$, and that $\dim_\R \HH = 4$.)  On the other hand, $\Hom_{\Add(\Brauer_\C^\sigma(\C;-2d))}((\gob \oplus \gor)^{\otimes r}, (\gob \oplus \gor)^{\otimes s})$ has the same dimension, since, when $r+s$ is even, there are $(r+s-1)!!$ perfect matchings of the $r+s$ endpoints, and then $2^{r+s}$ ways to choose one of the colors blue or red for the endpoints of the strings.  Thus, to prove that \cref{piano} is an isomorphism, it suffices to prove that it is surjective.  To do this, it is enough to show that all possible colorings of the generating morphisms of $\Brauer_\C^\sigma(\C;-2d)$ lie in the image of $\sS_\HH$.

    First note that
    \begin{align}
        \sS_\HH \left( \tfrac{1}{2}\ \idstrand + \tfrac{i}{2}\ \tokstrand[i] \right) &= \idstrandb\, ,&
        \sS_\HH \left( \tfrac{1}{2}\ \idstrand - \tfrac{i}{2}\ \tokstrand[i] \right) &= \idstrandr\, ,
        \\ \label{mirage}
        \sS_\HH \left( - \tfrac{1}{2}\ \tokstrand[j] - \tfrac{i}{2}\ \tokstrand[k] \right) &= \idstrandrb\, ,&
        \sS_\HH \left( \tfrac{1}{2}\ \tokstrand[j] - \tfrac{i}{2}\ \tokstrand[k] \right) &= \idstrandbr\, .
    \end{align}
    Next, we compute
    \begin{align*}
        - \tfrac{1}{4} \sS_\HH
        \left(
            \crossmor
            + i
            \begin{tikzpicture}[anchorbase]
                \draw (-0.2,-0.2) -- (0.2,0.2);
                \draw (0.2,-0.2) -- (-0.2,0.2);
                \token{-0.1,-0.1}{east}{i};
            \end{tikzpicture}
            + i\
            \begin{tikzpicture}[anchorbase]
                \draw (-0.2,-0.2) -- (0.2,0.2);
                \draw (0.2,-0.2) -- (-0.2,0.2);
                \token{0.1,-0.1}{west}{i};
            \end{tikzpicture}
            -
            \begin{tikzpicture}[anchorbase]
                \draw (-0.2,-0.2) -- (0.2,0.2);
                \draw (0.2,-0.2) -- (-0.2,0.2);
                \token{-0.1,-0.1}{east}{i};
                \token{0.1,-0.1}{west}{i};
            \end{tikzpicture}
        \right)
        &= \crossmorbb,
        &
        \sS_\HH \left( \tfrac{1}{2}\ \tokstrand[a] + \tfrac{i}{2}\ \tokstrand[ia] \right)
        &=
        \begin{tikzpicture}[centerzero]
            \draw[bcolor] (0,-0.2) -- (0,0.2);
            \tokenb{0,0}{west}{a};
        \end{tikzpicture}
        ,\ a \in \C,
        \\
        \tfrac{1}{2} \sS_\HH
        \left(
            \begin{tikzpicture}[anchorbase]
                \draw (-0.15,-0.15) -- (-0.15,0) arc(180:0:0.15) -- (0.15,-0.15);
                \token{0.15,0}{west}{j};
            \end{tikzpicture}
            - i\
            \begin{tikzpicture}[anchorbase]
                \draw (-0.15,-0.15) -- (-0.15,0) arc(180:0:0.15) -- (0.15,-0.15);
                \token{0.15,0}{west}{k};
            \end{tikzpicture}
        \right)
        &=
        \begin{tikzpicture}[anchorbase]
            \draw[bcolor] (-0.15,-0.35) -- (-0.15,0) arc(180:90:0.15);
            \draw[rcolor] (0,0.15) arc(90:0:0.15) -- (0.15,-0.15);
            \draw[bcolor] (0.15,-0.15) -- (0.15,-0.35);
        \end{tikzpicture}
        = \capmorbb\, ,
        &
        \tfrac{1}{2} \sS_\HH
        \left(
            \begin{tikzpicture}[anchorbase]
                \draw[-] (-0.15,0.15) -- (-0.15,0) arc(180:360:0.15) -- (0.15,0.15);
                \token{0.15,0}{west}{j};
            \end{tikzpicture}
            + i\
            \begin{tikzpicture}[anchorbase]
                \draw[-] (-0.15,0.15) -- (-0.15,0) arc(180:360:0.15) -- (0.15,0.15);
                \token{0.15,0}{west}{k};
            \end{tikzpicture}
        \right)
        &=
        \begin{tikzpicture}[anchorbase]
            \draw[bcolor] (-0.15,0.35) -- (-0.15,0) arc(180:270:0.15);
            \draw[rcolor] (0,-0.15) arc(270:360:0.15) -- (0.15,0.15);
            \draw[bcolor] (0.15,0.15) --(0.15,0.35);
        \end{tikzpicture}
        = \cupmorbb.
    \end{align*}
    Then, composing with the morphisms in \cref{mirage} to change colors of strands, we see that all possible colorings of the generating morphisms lie in the image of $\sS_\HH$, as desired.
\end{proof}

Fix $m,n \in \N$, set $V = \HH^{m|n}$, and let $\varphi$ be a nondegenerate $(\nu,\star)$-superhermitian form on $V$ of parity $\sigma$.  (See \cref{subsec:hermformH,subsec:hermit-periplectic} for a classification.)  Recall, from \cref{fries}, that the induced form $\varphi^j \colon \C^{2m|2n} \times \C^{2m|2n} \to \C$ is nondegenerate, $(-\nu,\id)$-supersymmetric, and of parity $\sigma$.  Let $\Phi$ denote the $(\nu,\star)$-supersymmetric form defined as in \cref{moon}, where we recall that the Frobenius form $\form$ on $\HH$ is projection onto the real part.  Recall, from \cref{subsec:HCform}, that $G(\Phi)=G(\varphi)$ and $\fg(\Phi) = \fg(\varphi)$.

\begin{lem} \label{blue}
    For all $G(\Phi)$-supermodules $U$ and $W$, we have an isomorphism of $\C$-supermodules
    \[
        \Hom_{G(\Phi)}(U,W)^\C
        \xrightarrow{\cong} \Hom_{G(\varphi^j)}(U^\C,W^\C),\quad
        f \otimes a \mapsto f \otimes a,\quad
        f \in \Hom_{G(\Phi)}(U,W),\ a \in \C.
    \]
\end{lem}

\begin{proof}
    First suppose that $\sigma=0$.  Then, as explained in \cref{subsec:hermit-prelim}, we may assume that $\nu=1$.  By the results of \cref{subsec:hermformH}, we see that $G_\rd(\Phi)$ has two connected components when $n \ge 1$.  (The case $n=0$ is easier, and similar to the $\sigma=1$ case discussed below.)  Fix $X \in G_\rd(\Phi)$ with $\det(X)=-1$, so that $X$ is in the connected component of $G_\rd(\Phi)$ not containing the identity.  Then, using \cref{compHform} we have
    \[
        \Hom_{G(\Phi)}(U,W)^\C
        \overset{\cref{skool2}}{=} \Hom_{X,\fg(\varphi)^\C}(U^\C,W^\C)
        \cong \Hom_{X,\fg(\varphi^j)}(U^\C,W^\C)
        \overset{\cref{breath}}{=} \Hom_{G(\varphi^j)}(U^\C,W^\C),
    \]
    where, in the final equality, we used the fact that $G_\rd(\varphi^j)$ has two connected components, and $X$ lies in the connected component not containing the identity, as explained in \cref{subsec:hermformCid}.

    The case $\sigma=1$ is easier.  As explained in \cref{subsec:hermit-periplectic}, $G(\Phi)$ and $G(\varphi^j)$ are both connected.  Then, using \cref{skool}, we have
    \[
        \Hom_{G(\Phi)}(U,W)^\C
        = \Hom_{\fg(\Phi)}(U,W)^\C
        \cong \Hom_{\fg(\varphi^j)}(U^\C,W^\C)
        = \Hom_{G(\varphi^j)}(U^\C,W^\C).
        \qedhere
    \]
\end{proof}

It follows from \cref{blue} that we have a canonical full and faithful superfunctor
\begin{equation}
    \sE_\HH \colon (G(\Phi)\smod_\R)^\C \to G(\varphi^j)\smod_\C
\end{equation}
sending $V$ to $V^\C$.  The next result shows that the diagram
\[
    \begin{tikzcd}
        \Brauer_\R^\sigma(\HH,\nu(m-n))^\C \arrow[rr, "\sS_\HH"] \arrow[d, "\sF_\Phi^\C"']
        & & \Add(\Brauer_\C^\sigma(\C,\nu(2n-2m))) \arrow[d, two heads, "\sF_{\varphi^j}"]
        \\
        \left( G(\Phi)\smod_\R \right)^\C \arrow[rr, "\sE_{\HH}"]
        & & G(\varphi^j)\smod_\C
    \end{tikzcd}
\]
commutes up to supernatural isomorphism.

\begin{prop} \label{smoke}
    There is a monoidal supernatural isomorphism of functors
    \[
        \theta \colon \sE_\HH \sF_\Phi^\C \xrightarrow{\cong} \sF_{\varphi^j} \sS_\HH
    \]
    determined by
    \begin{equation} \label{divide}
        \theta_\go \colon V^\C \xrightarrow{\cong} V \oplus V,\qquad
        \theta_\go(v) = \tfrac{1}{\sqrt{2}} (v, vj),\quad v \in V.
    \end{equation}
\end{prop}

\begin{proof}
    To simplify notation, we set $\sS = \sS_\HH$ and $\sE = \sE_\HH$.  First note that $\theta_\go$ is the isomorphism \cref{Hdub}.  On objects, we have
    \[
        \sE \sF_\Phi^\C (\go)
        = V^\C
        \xrightarrow[\cong]{\theta_\go} \blue{V} \oplus \red{V}
        = \sF_{\varphi^j}(\gob \oplus \gor)
        = \sF_{\varphi^j} \sS_\HH(\go).
    \]
    Here, and it what follows, in the isomorphism $V^\C \cong \blue{V} \oplus \red{V}$, the first copy of $V$ and its elements are denoted by bold blue characters and the second copy of $V$ and its elements are denoted by red non-bold characters.  This matches our diagrammatic conventions introduced earlier.

    For morphisms, we need to show that
    \[
        \theta_Y \circ \sE \sF_\Phi^\C (f)
        = \sF_{\varphi^j} \sS (f) \circ \theta_X
        \quad \text{for } f \in \left\{ \crossmor, \capmor, \cupmor, \tokstrand[i], \tokstrand[j] \right\},
    \]
    where $X$ and $Y$ are the domain and codomain of $f$, respectively.  (Since $\tokstrand[k] = \tokstrand[i] \circ \tokstrand[j]$, we do not need to check $f = \tokstrand[k]$.)

    For $f = \crossmor$, we have
    \[
        \theta_{\go \otimes \go} \circ \sE \sF_\Phi^\C (\crossmor)
        = \nu \theta_{\go \otimes \go} \circ \flip_{V^\C,V^\C}
        = \nu \flip_{\blue{V} \oplus \red{V}, \blue{V} \oplus \red{V}} \circ \theta_{\go \otimes \go}
        = \sF_{\varphi^j} \sS(\crossmor) \circ \theta_{\go \otimes \go}.
    \]
    Note that the negative sign appearing in the definition of $\sS(\crossmor)$ cancels with the negative sign arising from the fact that $\varphi^j$ is $(-\nu,\id)$-supersymmetric.

    For $f = \capmor$, noting that $\theta_\one$ is the identity map $\C \to \C$, we have, for $v,w \in V$,
    \[
        \theta_\one \circ \sE \sF_\Phi^\C(\capmor)
        \colon V^\C \otimes_\C V^\C \to \C,\qquad
        v \otimes w \mapsto \RP \varphi(v,w),
    \]
    and
    \begin{gather*}
        \sF_{\varphi^j} \sS (\capmor) \circ \theta_{\go \otimes \go}
        = \sF_{\varphi^j} (\capmorbr - \capmorrb) \circ \theta_{\go \otimes \go}
        \colon V^\C \otimes_\C V^\C \to \C,
        \\
        v \otimes w
        \mapsto \tfrac{1}{2} (\blue{v}, \red{vj}) \otimes (\blue{w}, \red{wj})
        \mapsto \tfrac{1}{2} (\varphi^j(v,wj) - \varphi^j(vj,w))
        \overset{\cref{house}}{=} \RP \varphi(v,w).
    \end{gather*}
    (Above, we have used the fact that the maps are uniquely determined by their values on $v \otimes w$, for $v,w \in V \subseteq V^\C$.)

    Next we consider $f = \cupmor$.  Choose a $\C$-basis $\bB_V^\C$ of $V$.  Then $\{v, vi : v \in \bB_V^\C \}$ is an $\R$-basis of $V$, and it is straightforward to verify that $(vi)^\vee = v^\vee i$.
    \details{
        For $v,w \in \bB_V^\C$, we have
        \[
            \Phi(w^\vee i, v i)
            = \RP \varphi(w^\vee i, v i)
            = \RP \varphi^1(w^\vee i, v i)
            = \RP (-i \varphi^1(w^\vee,v) i)
            = \RP(\delta_{wv})
            = \delta_{wv}
        \]
        and
        \[
            \Phi(w^\vee i, v)
            = \RP \varphi(w^\vee i,v)
            = \RP \varphi^1(w^\vee i,v)
            = \RP (-i \varphi^1(w^\vee v))
            = \RP (-i \delta_{wv})
            = 0.
        \]
        Thus $(vi)^\vee = v^\vee i$, as desired.
    }
    Identifying $\blue{w} \in \blue{V}$ and $\red{v} \in \red{V}$ and with $(\blue{w},\red{0}), (\blue{0},\red{v}) \in \blue{V} \oplus \red{V}$, we can write $(\blue{w},\red{v})$ as $\blue{w} + \red{v}$.  Using this convention, we have that
    \[
        \theta_{\go \otimes \go} \circ \sE \sF_\Phi^\C (\cupmor)
        \colon \C \to (\blue{V} \oplus \red{V}) \otimes_\C (\blue{V} \oplus \red{V})
    \]
    is the map given by
    \begin{align*}
        1 &\mapsto \sum_{v \in \bB_V^\R} (-1)^{\sigma\bar{v}} v \otimes v^\vee
        = \sum_{v \in \bB_V^\C} (-1)^{\sigma\bar{v}} (v \otimes v^\vee + v i \otimes v^\vee i)
        \\
        &\mapsto \frac{1}{2} \sum_{v \in \bB_V^\C} (-1)^{\sigma\bar{v}} \big( (\blue{v}+\red{vj}) \otimes (\blue{v^\vee}+\red{v^\vee j}) + (\blue{vi}+\red{vij}) \otimes (\blue{v^\vee i} + \red{v^\vee ij}) \big)
        \\
        &= \sum_{v \in \bB_V^\C} (-1)^{\sigma\bar{v}} \blue{v} \otimes \red{v^\vee j} + \sum_{v \in \bB_V^\C} (-1)^{\sigma\bar{v}} \red{vj} \otimes \blue{v^\vee}.
    \end{align*}
    On the other hand, for all $v,w \in \bB_V^\C$, we have
    \[
        \varphi^j(w^\vee j, v)
        \overset{\cref{boar}}{=} - \varphi^1(w^\vee,v)
        = - \delta_{vw}
        \qquad \text{and} \qquad
        \varphi^j(w^\vee, vj)
        \overset{\cref{rabbit}}{=} \varphi^1(w^\vee,v)^\star
        = \delta_{vw}.
    \]
    Thus, the $\C$-bases left dual to $\bB_V^\C$ and $\{vj : v \in \bB_V^\C\}$ with respect to $\varphi^j$ are $\{-v^\vee j : v \in \bB_V^\C\}$ and $\{v^\vee : v \in \bB_V^\C\}$, respectively.  Therefore,
    \begin{gather*}
        \sF_{\varphi_j} \sS ( \cupmor ) \circ \eta_\one
        = \sF_{\varphi_j} (\cupmorrb - \cupmorbr) \colon \C \to (\red{V} \otimes_\C \blue{V}) \oplus (\blue{V} \otimes_\C \red{V}),
        \\
        1 \mapsto \sum_{v \in \bB_V^\C} (-1)^{\sigma\bar{v}} \red{vj} \otimes \blue{v^\vee} + \sum_{v \in \bB_V^\C} (-1)^{\sigma\bar{v}} \blue{v} \otimes \red{v^\vee j}.
    \end{gather*}

    For $f = \tokstrand[i]$, we have
    \[
        \theta_\go \circ \sE \sF_\Phi^\C \left( \tokstrand[i] \right)
        \colon V^\C \to \blue{V} \oplus \red{V},\quad
        v \mapsto -vi \mapsto \tfrac{1}{\sqrt{2}} (\blue{-vi}, \red{-vij})
        = \tfrac{1}{\sqrt{2}} (\blue{-vi}, \red{vji})
    \]
    and
    \[
        \sF_{\varphi^j} \sS \left( \tokstrand[i] \right) \circ \theta_\go
        = \sF_{\varphi^j} \left( i\ \idstrandr - i\ \idstrandb\, \right) \circ \eta_\go
        \colon V^\C \to \blue{V} \oplus \red{V},\quad
        v \mapsto \tfrac{1}{\sqrt{2}} (\blue{v},\red{vj})
        \mapsto \tfrac{1}{\sqrt{2}} (\blue{-vi}, \red{vji}).
    \]

    For $f = \tokstrand[j]$, we have
    \[
        \theta_\go \circ \sE \sF_\Phi^\C \left( \tokstrand[j] \right)
        \colon V^\C \to \blue{V} \oplus \red{V},\quad
        v \mapsto -vj
        \mapsto \tfrac{1}{\sqrt{2}} (\blue{-vj}, \red{v})
    \]
    and
    \[
        \sF_{\varphi^j} \sS \left( \tokstrand[j] \right) \circ \theta_\go
        = \sF_{\varphi^j} \left( \idstrandbr - \idstrandrb\, \right) \circ \eta_\go
        \colon V^\C \to \blue{V} \oplus \red{V},\quad
        v \mapsto \tfrac{1}{\sqrt{2}} (\blue{v},\red{vj})
        \mapsto \tfrac{1}{\sqrt{2}} (\blue{-vj}, \red{v}).
    \]
    \details{
        As a consistency check, we verify that, for $f = \tokstrand[k]$, we have
        \[
            \theta_\go \circ \sE \sF_\Phi^\C \left( \tokstrand[k] \right)
            \colon V^\C \to \blue{V} \oplus \red{V},\quad
            v \mapsto -vk
            \mapsto \tfrac{1}{\sqrt{2}} (\blue{-vk}, \red{-vkj})
            = \tfrac{1}{2} (\blue{vji}, \red{vi})
        \]
        and
        \[
            \sF_{\varphi^j} \sS \left( \tokstrand[k] \right) \circ \theta_\go
            = \sF_{\varphi^j} \left( i\ \idstrandbr + i\ \idstrandrb\, \right) \circ \eta_\go
            \colon V^\C \to \blue{V} \oplus \red{V},\quad
            v \mapsto \tfrac{1}{\sqrt{2}} (\blue{v},\red{vj})
            \mapsto \tfrac{1}{\sqrt{2}} (\blue{vji}, \red{vi}).
        \]
    }
\end{proof}

\begin{prop} \label{divfull:H}
    \Cref{divfull} holds when $(\DD,\star) = (\HH,\star)$.
\end{prop}

\begin{proof}
    We wish to show that, for all $r,s \in \N$, the $\R$-linear map
    \[
        \sF_\Phi \colon \Hom_{\Brauer_\R^\sigma(\HH;\nu(m-n))}(\go^{\otimes r}, \go^{\otimes s})
        \to \Hom_{G(\Phi)}(V^{\otimes r}, V^{\otimes s})
    \]
    is surjective.  This map is surjective if and only if the map
    \begin{equation} \label{maple}
        \sF_\Phi^\C \colon \Hom_{\Brauer_\R^\sigma(\HH,\star;\nu(m-n))}(\go^{\otimes r}, \go^{\otimes s})^\C
        \to \Hom_{G(\Phi)}(V^{\otimes r}, V^{\otimes s})^\C
    \end{equation}
    is surjective.  To show that \cref{maple} is surjective, it suffices to show that the diagram
    \[
        \begin{tikzcd}
            \Hom_{\Brauer_\R^\sigma(\HH;\nu(m-n))}(\go^{\otimes r}, \go^{\otimes s})^\C \arrow[d, "\sS_\HH"', "\cong"] \arrow[r, "\sF_\Phi^\C"]
            & \Hom_{G(\Phi)}(V^{\otimes r}, V^{\otimes s}) \otimes_\R \C \arrow[d,"\cong", "\sE_\HH"']
            \\
            \Hom_{\Add(\Brauer_\C^\sigma(\C;\nu(2n-2m)))} \left( (\go\oplus \go)^{\otimes r}, (\go \oplus \go)^{\otimes s} \right) \arrow[d, two heads, "\sF_{\varphi^j}"']
            & \Hom_{G(\varphi^j)} \left( (V^\C)^{\otimes r}, (V^\C)^{\otimes s} \right) \arrow[dl,"\cong"]
            \\
            \Hom_{G(\varphi^j)}((V \oplus V)^{\otimes r}, (V \oplus V)^{\otimes s})
        \end{tikzcd}
    \]
    commutes, where the diagonal isomorphism is induced by the isomorphism \cref{divide}, and surjectivity of the bottom-left vertical arrow follows from \cref{soup}.  Commutativity of this diagram follows from \cref{smoke}.
\end{proof}

As a special case of $G(\Phi)$, we have the quaternionic orthosymplectic supergroups $\OSp^*(n|p,q)$; see \cref{subsec:hermformH}.  Recall the definition $G(\Phi)\tsmod_\R$ of the monoidal supercategory of tensor $G(\Phi)$-supermodules from \cref{subsec:Bnonsuper}.

\begin{prop} \label{iceH}
    If $p,p',q,q',n \in \N$ satisfy $p+q=p'+q'$, then we have an equivalence of monoidal supercategories
    \[
        \OSp^*(n|p,q)\tsmod_\R \simeq \OSp^*(n|p',q')\tsmod_\R
    \]
    sending the natural supermodule to the natural supermodule.
\end{prop}

\begin{proof}
    The proof is analogous to that of \cref{iceR}, using the commutative diagram appearing in the proof of \cref{divfull:H}.
\end{proof}

\appendix
\section{Classification of superhermitian forms\label{app:hermit}}

In this appendix, we give the classification of $(\nu,\star)$-superhermitian forms for the various choices of involutive division superalgebra.  In each case, we also give the corresponding Harish-Chandra superpair.  The explicit descriptions given in this appendix show that the Lie superalgebras of the form $\fg(\varphi)$, together with the Lie superalgebras $\fgl(r|s,\DD)$ for a real division superalgebra $\DD$, give \emph{all} the real forms of $\fgl(m|n,\C)$, $\osp(m|2n,\C)$, $\fp(m,\C)$, and $\fq(m,\C)$; see \cref{breakthrough}.

\subsection{Preliminaries\label{subsec:hermit-prelim}}

For any supermatrix $X$, let
\[
    X^\sharp := (X^\star)^\st = (X^\st)^\star.
\]
The isomorphism \cref{hopping}, together with the isomorphism of superalgebras $\Mat_{m|n}(A^\op) \xrightarrow{\cong} \Mat_{m|n}(A)$, $X_\op \mapsto X^\star$, shows that
\begin{equation} \label{sharpie}
    (XY)^\sharp = (-1)^{\bar{X} \bar{Y}} Y^\sharp X^\sharp
\end{equation}
whenever the product $XY$ is defined.  It follows from \cref{quiet} that, for $X \in \Mat_{(m|n) \times (r|s)}(A)$, we have
\[
    (X^\sharp)^\sharp = S_{m|n} X S_{r|s}
    \qquad \text{where} \quad
    S_{p|q} =
    \begin{pmatrix}
        I_p & 0 \\
        0 & -I_q
    \end{pmatrix}.
\]
We will often omit the subscripts on $S_{p,q}$ when its size is clear from the context.  Since we identify $A^{m|n} = \Mat_{(m|n) \times (1|0)}(A)$, and $S_{1|0} = I_1$, we have
\begin{equation} \label{underground}
    \begin{pmatrix} v_0 \\ v_1 \end{pmatrix}^\sharp = \begin{pmatrix} v_0^{\star,\tr} & -v_1^{\star,\tr} \end{pmatrix}
    \qquad \text{and} \qquad
    (v^\sharp)^\sharp = Sv
    \qquad \text{for} \quad
    v = \begin{pmatrix} v_0 \\ v_1 \end{pmatrix} \in A^{m|n}.
\end{equation}

Every sesquilinear form on $A^{m|n}$ is of the form
\begin{equation} \label{German}
    \varphi(v,w) = v^\sharp M w
    \qquad \text{where} \quad
    M \in \Mat_{m|n}(A)
    \text{ is homogeneous}.
\end{equation}

\begin{lem} \label{drop}
    The form $\varphi$ given by \cref{German} is $(\nu,\star)$-superhermitian if and only if $M^\sharp = \nu (-1)^{\bar{M}} MS$.
\end{lem}

\begin{proof}
    For $v,w \in A^{m|n}$, we have
    \begin{equation} \label{rock}
        \varphi(w,v)^\star
        = \varphi(w,v)^\sharp
        \overset{\cref{sharpie}}{=} (-1)^{\bar{v}\bar{w} + \bar{M}(\bar{v}+\bar{w})} v^\sharp M^\sharp (w^\sharp)^\sharp
        \overset{\cref{underground}}{=} (-1)^{\bar{v}\bar{w} + \bar{M}(\bar{v}+\bar{w})} v^\sharp M^\sharp S w.
    \end{equation}
    Now, if $\bar{\varphi} = \bar{M} = 0$, then $\varphi(w,v) = 0 = \varphi(v,w)$ unless $\bar{v} + \bar{w} = 0$.  Then \cref{rock} implies that $\varphi$ is $(\star,\nu)$-superhermitian if and only if $M = \nu M^\sharp S$.  On the other hand, if $\bar{\varphi} = \bar{M} = 1$, then $\varphi(w,v) = 0 = \varphi(v,w)$ unless $\bar{v} + \bar{w} = 1$.  Then \cref{rock} implies that $\varphi$ is $(\star,\nu)$-superhermitian if and only if $M = - \nu M^\sharp S$.  Combining both cases, and using the fact that $S^2 = I$, the result follows.
\end{proof}

\begin{lem} \label{purple}
    For $\varphi$ as in \cref{German}, we have
    \[
        X^\dagger = (M^\sharp S)^{-1} X^\sharp M^\sharp S.
    \]
\end{lem}

\begin{proof}
    For $v,w \in A^{m|n}$ and $X \in \fgl(m|n,A)$, we have
    \[
        (-1)^{\bar{X}\bar{v}} \varphi(X^\dagger v, w)
        \overset{\cref{sharpie}}{=} v^\sharp (X^\dagger)^\sharp M w
        \qquad \text{and} \qquad
        \varphi(v,Xw) = v^\sharp MX w.
    \]
    Thus
    \[
        (X^\dagger)^\sharp M = M X
        \implies M^\sharp S X^\dagger S = X^\sharp M^\sharp
        \implies X^\dagger = (M^\sharp S)^{-1} X^\sharp M^\sharp S.
        \qedhere
    \]
\end{proof}

For the remainder of this section $(\DD,\star)$, will denote an involutive division superalgebra over $\kk \in \{\R,\C\}$.  Our goal is to classify the $(\nu,\star)$-superhermitian forms over such superalgebras.  An \emph{even} $(\nu,\star)$-superhermitian form $V$ is equivalent to an even $(-\nu,\star)$-superhermitian form on $\Pi V$.  Thus, for even forms, we will only treat the $(1,\star)$-superhermitian case.  For \emph{odd} forms, we will need to treat both the $\nu=1$ and $\nu=-1$ cases.  (However, these lead to isomorphic Lie algebras; see \cref{wine}.)

An even $(1,\star)$-superhermitian form on a $\DD$-supermodule $V = V_0 \oplus V_1$ corresponds, when viewing $V$ as a (non-super) $\kk$-module, to a superhermitian form on $V_0$ and an skew-superhermitian form on $V_1$.  This allows us to use well-known results classifying such forms up to equivalence.  (See, for example, \cite{Lew82}.)  They are typically classified by dimension or signature.

\subsection{Case $(\C,\id)$, $\kk = \C$, even form\label{subsec:hermformCid}}

There are no even nondegenerate $(1,\id)$-superhermitian forms on $\C^{m|n}$ when $n \in 2\Z + 1$.  Every even nondegenerate $(1,\id)$-superhermitian form on $\C^{m|2n}$ is equivalent to
\begin{equation} \label{hermformCid}
    \varphi_{m|2n} \colon (v,w) \mapsto
    v^\transpose
    \begin{pmatrix}
        I_m & 0 & 0 \\
        0 & 0 & I_n \\
        0 & -I_n & 0
    \end{pmatrix}
    w.
\end{equation}
Then $G(\varphi_{m|2n})$ is the \emph{complex orthosymplectic supergroup} $\OSp(m|2n,\C)$.  We have
\[
    G_\rd(\varphi_{m|2n}) = \rO(m,\C) \times \Sp(2n,\C)
    \qquad \text{and} \qquad
    \fg(\varphi_{m|2n}) = \osp(m|2n,\C).
\]
The group $\Sp(2n,\C)$ is connected, whereas, when $m \ge 1$, $\rO(m,\C)$ has two connected components: the identity component $\SO(n,\C)$ and the elements of $\rO(m,\C)$ with determinant $-1$.  It follows that, for any nondegenerate even $(\nu,\id)$-superhermitian form $\varphi$ on $\C^{m|n}$, the group $G_\rd(\varphi)$ has two connected components.  Any $X \in G_\rd(\varphi)$ with $\det(X)=-1$ is in the connected component not containing the identity.

\subsection{Case $(\R,\id)$, $\kk=\R$, even form\label{subsec:hermformR}}

There are no even nondegenerate  $(1,\id)$-superhermitian forms on $\R^{m|n}$ when $n \in 2\Z + 1$.  Every even nondegenerate  $(1,\id)$-superhermitian form on $\R^{m|2n}$ is equivalent to
\begin{equation} \label{hermformR}
    \varphi_{p,q|2n} \colon (v,w) \mapsto
    v^\transpose
    \begin{pmatrix}
        I_p & 0 & 0 & 0 \\
        0 & -I_q & 0 & 0 \\
        0 & 0 & 0 & I_n \\
        0 & 0 & -I_n & 0
    \end{pmatrix}
    w,\qquad v,w \in \R^{m|2n},
\end{equation}
for some $p,q \in \N$, $p+q=m$.  Furthermore, $\varphi_{p,q|2n} \sim \varphi_{p',q'|2n}$ if and only if $(p',q') = (p,q)$ or $(p',q') = (q,p)$.  The supergroup $G(\varphi_{p,q|2n})$ is the \emph{indefinite orthosymplectic supergroup} $\OSp(p,q|2n,\R)$.  We have
\[
    G_\rd(\varphi_{p,q|2n}) = \rO(p,q) \times \Sp(2n,\R)
    \qquad \text{and} \qquad
    \fg(\varphi_{p,q|2n}) = \osp(p,q|2n,\R).
\]
The real symplectic group $\Sp(2n,\R)$ is connected.  When $p,q \ge 1$, the indefinite orthogonal group $\rO(p,q)$ has four connected components, corresponding to the two choices $\pm 1$ of determinant for the restriction to the two subspaces $\R^p \times 0^q$ and $0^p \times \R^q$ of $\R^m$.  When $p=0$ or $q=0$, but $p+q \ge 1$, $\rO(p,q)$ has two connected components.

Since $\varphi_{p,q|2n}^\C$ is equivalent to the form $\varphi_{m|2n}$ of \cref{hermformCid}, it follows from \cref{compRform} that
\[
    \osp(p,q|2n,\R)^\C
    = \fg(\varphi_{p,q|2n})^\C
    \cong \fg(\varphi_{p,q|2n}^\C)
    \cong \osp(m|2n,\C).
\]

\subsection{Case $(\C,\star)$, $\kk = \R$, even form\label{subsec:hermformCstar}}

Every even nondegenerate  $(1,\star)$-superhermitian form on $\C^{m|n}$ is equivalent to
\begin{equation} \label{form:Cstar}
    \varphi_{p,q|r,s} \colon (v,w) \mapsto
     v^\sharp
    \begin{pmatrix}
        I_p & 0 & 0 & 0 \\
        0 & -I_q & 0 & 0 \\
        0 & 0 & iI_r & 0 \\
        0 & 0 & 0 & -iI_s
    \end{pmatrix}
    w
\end{equation}
for some $p,q,r,s \in \N$, $p+q=m$, $r+s=n$.  Furthermore, $\varphi_{p,q|r,s} \sim \varphi_{p',q'|r',s'}$ if and only if $(p',q',r',s') \in \{(p,q,r,s), (q,p,s,r), (r,s,p,q), (s,r,q,p)\}$.  We call $\rU(p,q|r,s) := G(\varphi_{p,q|r,s})$ the \emph{indefinite unitary supergroup}.  We have
\[
    G_\rd(\varphi_{p,q|r,s}) = \rU(p,q) \times \rU(r,s)
    \qquad \text{and} \qquad
    \fg(\varphi_{p,q|r,s}) = \fu(p,q|r,s).
\]
Since the indefinite unitary group $U(p,q)$ is connected for all $p,q \in \N$, it follows that $G_\rd(\varphi)$ is connected for any nondegenerate even $(\nu,\star)$-superhermitian form $\varphi$ on $\C^{m|n}$.
\details{
    For any matrix $X \in \rU(p,q)$, there exists a path to an element of $\mathrm{SU}(p,q)$.  Indeed, we can take the path $(\det X)^{-t} X$, $0 \le t \le 1$.  By \cite[Prop.~1.145]{Kna02}, $\mathrm{SU}(p,q)$ is connected.  Hence $\rU(p,q)$ is connected.
}

It follows from \cref{compCform} that we have an isomorphism of complex Lie superalgebras
\[
    \fu(p,q|r,s)^\C
    = \fg(\varphi_{p,q|r,s})^\C
    \cong \fgl(m|n,\C).
\]

\subsection{Case $(\HH,\star)$, $\kk=\R$, even form\label{subsec:hermformH}}

Every even nondegenerate  $(1,\star)$-superhermitian form on $\HH^{m|n}$ is equivalent to
\begin{equation} \label{form:H}
    \varphi_{p,q|n} \colon (v,w) \mapsto
     v^\sharp
    \begin{pmatrix}
        I_p & 0 & 0 \\
        0 & -I_q & 0 \\
        0 & 0 & j I_n
    \end{pmatrix}
    w
\end{equation}
for some $p,q \in \N$, $p+q=m$.  (See, for example, \cite[\S 5, \S 6]{Lew82}.)  We have
\[
    \varphi_{p,q|n} \sim \varphi_{p',q'|n'}
    \iff
    (p',q',n) \in \{(p,q,n), (q,p,n)\}.
\]
We call $\OSp^*(n|p,q) := G(\varphi_{p,q|n})$ the \emph{quaternionic orthosymplectic supergroup}.  We have
\[
    G_\rd(\varphi_{p,q|n}) = \rO(n,\HH) \times \rU(p,q,\HH)
    \qquad \text{and} \qquad
    \fg(\varphi_{p,q|n}) = \osp^*(n|p,q),
\]
where $\rO(n,\HH)$ is the quaternionic orthogonal group, sometimes denoted $\rO^*(2n)$ in the literature, and $\rU(p,q,\HH)$ is the indefinite quaternionic unitary group, which is equal to the indefinite symplectic group $\Sp(p,q)$.
\details{
    We have
    \begin{align*}
        \begin{pmatrix}
            X & 0 \\
            0 & Y
        \end{pmatrix}
        \in G_\rd(\varphi_{p,q|n})
        &\iff
        \begin{pmatrix}
            X^\sharp & 0 \\
            0 & Y^\sharp
        \end{pmatrix}
        \begin{pmatrix}
            I_p & 0 & 0 \\
            0 & -I_q & 0 \\
            0 & 0 & j I_n
        \end{pmatrix}
        \begin{pmatrix}
            X & 0 \\
            0 & Y
        \end{pmatrix}
        =
        \begin{pmatrix}
            I_p & 0 & 0 \\
            0 & -I_q & 0 \\
            0 & 0 & j I_n
        \end{pmatrix}
        \\ &\iff
        X^\sharp
        \begin{pmatrix}
            I_p & 0 \\
            0 & -I_q
        \end{pmatrix}
        X
        =
        \begin{pmatrix}
            I_p & 0 \\
            0 & -I_q
        \end{pmatrix}
        \quad \text{and} \quad
        Y^\sharp j Y = jI_n.
    \end{align*}
    The first equality is equivalent to $X \in U(p,q,\HH)$.  Viewing $Y$ as an element of $\Mat_{2n}(\C)$ and using \cref{mayor}, we have
    \[
        Y^\sharp j Y = jI_n
        \iff J Y^\transpose Y = J
        \iff Y^\transpose Y = I
        \iff Y \in \rO(n,\HH).
    \]
}
The notation for $\fg(\varphi_{p,q|n})$ is not consistent in the literature.  For example, it is denoted $\osp^*(n|2m,2p)$ in \cite{Ser83}.  The indefinite symplectic group $\Sp(p,q)$ is connected; see \cite[Prop.~1.145]{Kna02}.  The quaternionic orthogonal group $\rO(n,\HH)$ has two connected components when $n \ge 1$.  Viewing elements of $\rO(n,\HH)$ as $2n \times 2n$ complex matrices, the identity component, $\SO(n,\HH)$, of $\rO(n,\HH)$ consists of those elements with determinant $1$, while the other component consists of those elements with determinant $-1$.
\details{
    An element $X \in \GL(2n,\C)$ lies in $\rO(n,\HH)$ if and only if $J X^\star J^{-1} = X$ (see \cref{aiel}) and
    \begin{align*}
        (Xv)^\sharp (jI) Xw &= v^\sharp (jI) w \qquad \forall\ v,w \in \C^{2n}
        \\ \iff
        v^\sharp X^\sharp (jI) Xw &= v^\sharp (jI) w \qquad \forall\ v,w \in \C^{2n}
        \\ \overset{\cref{mayor}}{\iff}
        J v^\transpose X^\transpose Xw &= J v^\transpose w \qquad \forall\ v,w \in \C^{2n}
        \\ \iff
        X^\transpose X &= I.
    \end{align*}
    The condition $X^\transpose X = I$ forces $\det(X) = \pm 1$.  The subgroup $\SO(m,\HH)$ is connected by \cite[Prop.~1.145]{Kna02}.  If $X,Y \in \rO(n,\HH)$ satisfy $\det(X) = \det(Y) = -1$, then we have $X = (XY^{-1})Y$.  Since $\det(XY^{-1}) = \det(X)\det(Y)^{-1} = 1$, there is path in $\SO(m,\HH)$ from $XY^{-1}$ to the identity.  This gives a path from $X$ to $Y$.  Hence, $\{X \in \rO(m,\HH) : \det(X)=-1\}$ is connected.
}
It follows that, for any nondegenerate even $(\nu,\star)$-superhermitian form $\varphi$ on $\HH^{m|n}$, $n \ge 1$, the group $G_\rd(\varphi)$ has two connected components.  Any $X \in G_{\rd}(\varphi)$ with $\det(X)=-1$ is in the connected component not containing the identity.

\begin{lem}
    We have an isomorphism of complex Lie superalgebras
    \[
        \fg(\varphi_{p,q|n})^\C  \cong \osp(2n|2m,\C).
    \]
    (Note the reversal in the order of $m$ and $n$ on the right-hand side.)
\end{lem}

\begin{proof}
    Consider the form $(\varphi_{p,q|n})^j$ in the notation of \cref{subsec:compHform}.  Since a skew-supersymmetric form on $\C^{2m|2n}$ is equivalent to a supersymmetric form on $\C^{2n|2m}$ (via the parity shift map $\C^{2m|2n} \to \C^{2n|2m}$), we see that $(\varphi_{p,q|n})^j$ is equivalent to the form $\varphi_{2n|2m}$ given by replacing $m$ and $n$ in \cref{hermformCid} by $2n$ and $2m$, respectively.  Thus, the result follows from \cref{compHform}.
\end{proof}

\subsection{Case $(\Cl(\C),\star)$, $\kk=\R$, even form\label{subsec:hermformClstar}}

In light of \cref{fold}, we assume in this case that $n=0$.  Every even nondegenerate  $(1,\star)$-superhermitian form on $\Cl(\C)^m$ is equivalent to
\begin{equation} \label{form:Clstar}
    \varphi_{p,q} \colon (v,w) \mapsto
    v^\sharp
    \begin{pmatrix}
        I_p & 0 \\
        0 & -I_q
    \end{pmatrix}
    w
\end{equation}
for unique $p,q \in \N$, $p+q=m$.
\details{
    The restriction to the even part $\C^m \subseteq \Cl(\C)^m$ must give a $(1,\star)$-superhermitian form on $\C^m$, which must of the form \cref{form:Clstar}; see \cref{subsec:hermformCstar}.  Then, for $v,w \in \C^m$, we have
    \[
        \varphi(\varepsilon v, \varepsilon w)
        = \varepsilon i \varphi(v,w) \varepsilon
        = i \varphi(v,w)
        = i v^\sharp
        \begin{pmatrix}
            I_p & 0 \\
            0 & -I_q
        \end{pmatrix}
        w
        = (\varepsilon v)^\sharp
        \begin{pmatrix}
            I_p & 0 \\
            0 & -I_q
        \end{pmatrix}
        (\varepsilon w).
    \]
    Hence $\varphi$ is of the form \cref{form:Clstar} on all of $\Cl(\C)^m$.
}
We call $\rUQ(p,q) := G(\varphi_{p,q})$ the \emph{indefinite isomeric unitary supergroup}.  We have
\[
    G_\rd(\varphi_{p,q}) = \rU(p,q).
\]
The notation for $\fg(\varphi_{p,q})$ is not consistent in the literature.  It is sometimes denoted by $\mathfrak{uq}(p,q)$.  Its simple quotient is denoted $upsq(n,p)$ in \cite{Ser83}.  Since the indefinite unitary group $\rU(p,q)$ is connected, it follows that, for any nondegenerate $(\nu,\star)$-superhermitian form on $\Cl(\C)^m$, the group $G_\rd(\varphi)$ is connected.

It follows from \cref{compCform} that we have an isomorphism of complex Lie superalgebras
\[
    \fg(\varphi_{p,q})^\C \cong \fgl(m,\Cl(\C)) = \fq(m,\C).
\]

\subsection{Odd forms\label{subsec:hermit-periplectic}}

Let $\kk \in \{\R,\C\}$ and let $(\DD,\star)$ be an arbitrary involutive division $\kk$-superalgebra.  If $\kk=\R$ and $\varphi$ is a nondegenerate $(\nu,\star)$-superhermitian form on $\Cl(\C)^{m|n}$, then $\varepsilon(1+i) \varphi$ is a nondegenerate $(\nu,\star)$-superhermitian form on $\Cl(\C)$ of parity $\bar{\varphi}+1$.
\details{
    Suppose $\varphi$ is a nondegenerate $(\nu,\star)$-superhermitian form on $\Cl(\C)^{m|n}$.  Then, for all $v,w \in \Cl(\C)^{m|n}$, we have
    \[
        \varepsilon(1+i)\varphi(v,w)
        = \nu (-1)^{\bar{v}\bar{w}} \varepsilon(1+i) \varphi(w,v)^\star
        = \nu (-1)^{\bar{v}\bar{w}} (\varepsilon(1+i) \varphi(w,v))^\star.
    \]
}
Since we already treated the even forms above, we assume in this subsection that $\DD \in \{\R,\C,\HH\}$.

An odd form on $\DD^{m|n}$ can only be nondegenerate when $m=n$.  Any odd nondegenerate $(\nu,\star)$-superhermitian form on $\DD^{m|m}$ is equivalent to
\begin{equation} \label{form:odd}
    \varphi_m^\nu \colon (v,w) \mapsto v^\sharp
    \begin{pmatrix}
        0 & I_m \\
        - \nu I_m & 0
    \end{pmatrix}
    w.
\end{equation}
\details{
    \Cref{drop} shows that $\varphi$ is odd and $(\nu,\star)$-superhermitian if and only if
    \[
        M =
        \begin{pmatrix}
            0 & M_{01} \\
            - \nu M_{01}^{\star,\transpose} & 0
        \end{pmatrix}
    \]
    form some $M_{01} \in \Mat_m(A)$.  Then we can change basis to make $M_{01} = I_m$.
}
(Recall, from \cref{underground}, the sign appearing in $v^\sharp$.)  With this form, it follows from \cref{purple} that
\[
    \begin{pmatrix}
        X_{00} & X_{01} \\
        X_{10} & X_{11}
    \end{pmatrix}^\dagger
    =
    \begin{pmatrix}
        X_{11}^\sharp & - \nu X_{01}^\sharp \\
        \nu X_{10}^\sharp & X_{00}^\sharp.
    \end{pmatrix}.
\]
\details{
    We have
    \[
        \begin{pmatrix}
            0 & I_m \\
            - \nu I_m & 0
        \end{pmatrix}^\sharp
        S_{m|m}
        =
        \begin{pmatrix}
            0 & - \nu I_m \\
            I_m & 0
        \end{pmatrix}.
    \]
    Hence, by \cref{purple},
    \begin{align*}
        \begin{pmatrix}
            X_{00} & X_{01} \\
            X_{10} & X_{11}
        \end{pmatrix}^\dagger
        &=
        \begin{pmatrix}
            0 & I_m \\
            - \nu I_m & 0
        \end{pmatrix}
        \begin{pmatrix}
            X_{00} & X_{01} \\
            X_{10} & X_{11}
        \end{pmatrix}^\sharp
        \begin{pmatrix}
            0 & - \nu I_m \\
            I_m & 0
        \end{pmatrix}
        \\
        &=
        \begin{pmatrix}
            0 & I_m \\
            - \nu I_m & 0
        \end{pmatrix}
        \begin{pmatrix}
            X_{00}^\sharp & -X_{10}^\sharp \\
            X_{01}^\sharp & X_{11}^\sharp
        \end{pmatrix}
        \begin{pmatrix}
            0 & - \nu I_m \\
            I_m & 0
        \end{pmatrix}
        \\
        &=
        \begin{pmatrix}
            0 & I_m \\
            - \nu I_m & 0
        \end{pmatrix}
        \begin{pmatrix}
            - X_{10}^\sharp & - \nu X_{00}^\sharp \\
            X_{11}^\sharp & - \nu X_{01}^\sharp
        \end{pmatrix}
        =
        \begin{pmatrix}
            X_{11}^\sharp & -\nu X_{01}^\sharp \\
            \nu X_{10}^\sharp & X_{00}^\sharp.
        \end{pmatrix}.
    \end{align*}
}
Thus,
\[
    \fg(\varphi_m^\nu)
    =
    \left\{
        \begin{pmatrix}
            X & Y \\
            Z & - X^\sharp
        \end{pmatrix}
        :
        X,Y,Z \in \Mat_m(\DD),\ Y = \nu Y^\sharp,\ Z = -\nu Z^\sharp
    \right\}.
\]
In particular, we have an isomorphism of Lie superalgebras
\begin{equation} \label{wine}
    \fg(\varphi_m^\nu) \xrightarrow{\cong} \fg(\varphi_m^{-\nu}),\qquad
    X \mapsto X^\#.
\end{equation}
We also have
\[
    G_\rd(\varphi_m^\nu)
    =
    \left\{
        \begin{pmatrix}
            X & 0 \\
            0 & (X^\sharp)^{-1}
        \end{pmatrix}
        : A = \GL(m,\DD)
    \right\}
    \cong \GL(m,\DD).
\]
\details{
    Since $\DD_1=0$, the elements of $G_\rd(\varphi_m^\nu)$ are block diagonal.  We have
    \begin{align*}
        \begin{pmatrix}
            X & 0 \\
            0 & Y
        \end{pmatrix}
        \in G_\rd(\varphi_m^\nu)
        &\iff
        \begin{pmatrix}
            X^\sharp & 0 \\
            0 & Y^\sharp
        \end{pmatrix}
        \begin{pmatrix}
            0 & I_m \\
            -\nu I_m & 0
        \end{pmatrix}
        \begin{pmatrix}
            X & 0 \\
            0 & Y
        \end{pmatrix}
        =
        \begin{pmatrix}
            0 & I_m \\
            -\nu I_m & 0
        \end{pmatrix}
        \\ &\iff
        \begin{pmatrix}
            0 & X^\sharp Y \\
            - \nu Y^\sharp X & 0
        \end{pmatrix}
        =
        \begin{pmatrix}
            0 & I_m \\
            -\nu I_m
        \end{pmatrix}
        \\ &\iff
        Y = (X^\sharp)^{-1}.
    \end{align*}
}
It follows that $G_\rd(\varphi)$ is connected for any nondegenerate odd $(\nu,\star)$-superhermitian form $\varphi$ on $\DD^{m|n}$.  When $(\DD,\star)$ is equal to $(\R,\id)$ or $(\C,\id)$, the Lie superalgebras
\[
    \fg(\varphi_m^\nu) = \fp(m,\R)
    \qquad \text{and} \qquad
    \fg(\varphi_m^\nu) = \fp(m,\C)
\]
are the real and complex periplectic Lie superalgebras, respectively.  The notation for the other cases is less standard.  When $(\DD,\star) = (\C,\star)$, the Lie superalgebra $\fg(\varphi_m)$ is sometimes denoted $\mathfrak{up}(m)$.  When $(\DD,\star) = (\HH,\star)$, $\fg(\varphi_m^\nu)$ is sometimes denoted $\fp^*(m)$.  Their simple quotients are denoted $us\pi(m)$ and $s\pi^*(m)$, respectively, in \cite{Ser83}.

It follows from \cref{compRform,compCform,compHform,wine} that we have isomorphisms of complex Lie superalgebras
\[
    \fg(\varphi_m^\nu)^\C \cong
    \begin{cases}
        \fp(m,\C) & \text{if } (\DD,\star) = (\R,\id), \\
        \fgl(m|m,\C) & \text{if } (\DD,\star) = (\C,\star), \\
        \fp(2m,\C) & \text{if } (\DD,\star) = (\HH,\star).
    \end{cases}
\]

\section{Classification of real forms}

A classification of the real forms of the classical Lie algebras can be found in \cite[\S26.1]{FH91}.  The classification of the real simple Lie superalgebras was first given in \cite{Ser83}.  In particular, \cite[Table~3]{Ser83} lists all the real forms of the simple subquotients of $\fgl(m|n,\C)$, $\osp(m|2n,\C)$, $\fp(m,\C)$, and $\fq(m,\C)$.  However, because $\fgl(m|n,\C)$, $\fp(m,\C)$ and $\fq(m,\C)$ are \emph{not} simple, they are not covered by this classification.  Since the real forms of these Lie superalgebras do not seem to have appeared in the literature, we give a classification here.

Throughout this subsection $\fg$ denotes one of the superalgebras $\fgl(m|n,\C)$, $\fq(m,\C)$, or $\fp(m,\C)$, $m,n \in \N$.  Let
\[
    \fg' = [\fg,\fg]
    \qquad \text{and} \qquad
    \fg'' = \fg'/Z(\fg'),
\]
where $[\fg,\fg]$ denotes the ideal of $\fg$ generated by $[X,Y]$, $X,Y \in \fg$, and $Z(\fg') = \{ X \in \fg' : [X,\fg'] = 0\}$ denotes the center of $\fg'$.  For $m,n \in \N$,
\begin{align*}
    \fgl(m|n,\C)'
    &= \fsl(m|n,\C)
    = \{ X \in \fgl(m|n,\C) : \str(X) = 0 \},
    \\
    \fq(m,\C)'
    &= \{X \in \fq(m,\C) = \fgl(m,\Cl(\C)) : \tr(X)_1 = 0 \},
    \\
    \fp(m,\C)'
    &=
    \left\{
        \begin{pmatrix} X_{00} & X_{01} \\ X_{10} & -X_{00}^\transpose \end{pmatrix} \in \fgl(m|m,\C)
        : \tr(X_{00}) = 0,\ X_{01}^\transpose = X_{01},\ X_{10}^\transpose = - X_{10}
    \right\}.
\end{align*}
Since, for $m \ne n$,
\[
    Z(\fsl(m|n,\C)) = 0,\quad
    Z(\fsl(m|m,C)) = \C I_{2m},\quad
    Z(\fq(m,\C)') = \C I_m,\quad
    Z(\fp(m,\C)') = 0,
\]
we have
\begin{align*}
    \fgl(m|n,\C)'' &= \fsl(m|n,\C),
    \\
    \fgl(m|m,\C)'' &= \fsl(m|m,\C)/\C I_{2m} = \fpsl(m|m,\C),
    \\
    \fq(m,\C)'' &= \fq(m,\C)'/\C I_m,
    \\
    \fp(m,\C)'' &= \fp(m,\C)'.
\end{align*}

The adjoint action of $\fg$ on itself is given by
\[
    \Ad \colon \fg \to \End_\fg(\fg),\qquad
    \Ad(X)(Y) = [X,Y],\quad X,Y \in \fg.
\]
This restricts to give an action $\Ad' \colon \fg \to \End_\fg(\fg')$ of $\fg$ on $\fg'$.

\begin{lem} \label{hype}
    We have $\ker(\Ad') = Z(\fg)$.
\end{lem}

\begin{proof}
    We show this for the case $\fg = \fgl(m|n,\C)$, since the other cases are analogous.  Suppose $X = \sum_{r,s=1}^{m+n} a_{rs} E_{rs}$, $a_{rs} \in \C$, is a homogeneous element of $\ker(\Ad')$, where $E_{rs}$ denotes the matrix with a $1$ in position $(r,s)$, and a $0$ in all other positions.  Then, for all $1 \le t,u \le m+n$, we have
    \[
        0 = [X,E_{tu}]
        = \sum_{r=1}^{m+n} a_{rt} E_{ru} \pm \sum_{s=1}^{m+n} a_{us} E_{ts}.
    \]
    Thus, $a_{rs}=0$ whenever $r \ne s$.  So $X$ is diagonal, hence even.  Then, for all $1 \le t,u \le m+n$, we have
    \[
        0 = [X,E_{tu}]
        = (a_{tt} - a_{uu}) E_{tu}.
    \]
    Thus $a_{tt} = a_{uu}$, and so $X \in \C I_{m+n} = Z(\fgl(m|n,\C))$.
    \details{
        For $\fg = \fq_m(\C) = \fgl(m,\Cl(\C))$, the first part of the argument shows that any $X \in \ker(\Ad')$ must be diagonal.  In the second part, the matrices $E_{tu}$ are all even (even though $X$ may not be), and so the argument remains valid.

        For
        \[
            \fp(m,\C) =
            \left\{
                \begin{pmatrix} X_{00} & X_{01} \\ X_{10} & -X_{00}^\transpose \end{pmatrix}
                : X_{01}^\transpose = X_{01},\ X_{10}^\transpose = - X_{10}
            \right\},
        \]
        we only consider $1 \le t,u \le m$ and $m < t,u \le 2m$ in the first part of the argument.  This still shows that any $X \in \ker(\Ad')$ must be diagonal.  Similarly, the second part for the argument shows that $a_{tt} = a_{uu}$ if $1 \le t,u \le m$ or $m < t,u \le 2m$.  So it remains to show that $a_{mm} = a_{m+1,m+1}$.  To see this, note that $E_{m,m+1} + E_{1,2m} \in \fp(m,\C)'$ and
        \[
            0 = [X,E_{m,m+1} + E_{1,2m}]
            = (a_{mm} - a_{m+1,m+1}) E_{m,m+1} + (a_{11} - a_{2m,2m}) E_{1,2m}.
        \]
        This implies that $a_{mm} = a_{m+1,m+1}$, as desired.
    }
\end{proof}

A \emph{conjugate-linear involution} of $\fg$ (also called a \emph{real structure} on $\fg$) is an automorphism $\kappa$ of $\fg$, considered as a \emph{real} Lie superalgebra, satisfying $\kappa^2 = \id$ and $\kappa(aX) = a^\star \kappa(X)$ for all $a \in \C$.  Every real form of $\fg$ is isomorphic to
\[
    \fg^\kappa = \{ X \in \fg : \kappa(X) = X \}
\]
for some conjugate-linear involution $\kappa$ of $\fg$.

If $\kappa$ is a conjugate-linear involution of $\fg$, then $\kappa$ restricts to a conjugate-linear involution $\kappa'$ of $\fg'$, which, in turn, induces a conjugate-linear involution $\kappa''$ of $\fg''$.

\begin{lem} \label{sweep}
    Suppose $\kappa$ and $\chi$ are conjugate-linear involutions of $\fg$ such that $\kappa'' = \chi''$.  Then $\kappa' = \chi'$ and $\kappa(X) - \chi(Z) \in Z(\fg)$ for all $X \in \fg$.
\end{lem}

\begin{proof}
    Since the odd part of $\fg''$ is equal to the odd part of $\fg$, we have $\kappa|_{\fg_1} = \chi|_{\fg_1}$.  We see by inspection that $[\fg_1,\fg_1] = \fg'$.  Thus, $\kappa' = \chi'$.  Then, for all $X \in \fg$, $Y \in \fg'$, we have
    \[
        [\kappa(X), \chi(Y)]
        = [\kappa(X), \kappa(Y)]
        = \kappa([X,Y])
        = \chi([X,Y])
        = [\chi(X),\chi(Y)].
    \]
    It follows that $\Ad'(\kappa(X)) = \Ad'(\chi(X))$.  Hence, by \cref{hype}, we have $\kappa(X) - \chi(X) \in Z(\fg)$.
\end{proof}

\begin{prop} \label{breakthrough}
    Every real form of $\fgl(m|n,\C)$, $\osp(m|2n,\C)$, $\fp(m,\C)$, $\fq(m,\C)$, $m,n \in \N$, is isomorphic to either
    \begin{itemize}
        \item $\fgl(r|s;\DD)$ for a real division superalgebra $\DD$, or
        \item $\fg(\varphi)$ for a $(\nu,\star)$-superhermitian form $\varphi$ on $\DD^{r|s}$, where $(\DD,\star)$ is an involutive real division superalgebra,
    \end{itemize}
    for some $r,s \in \N$.
\end{prop}

\begin{proof}
    In the case of $\osp(m|2n,\C)$, which is simple, we see from \cite[Table~3]{Ser83} that the real forms are $\osp(p,q|2n,\R)$, $p+q=m$ and, when $m$ is even, $\osp^*(\frac{m}{2}|p,q)$, $p+q=n$.  Then the result follows from \cref{subsec:hermformR,subsec:hermformH}.

    Now assume that $\fg$ is one of the Lie superalgebras $\fgl(m|n,\C)$, $\fq(m,\C)$, or $\fp(m,\C)$, $m,n \in \N$.  All of the real forms described in \cref{subsec:hermformR,subsec:hermformCstar,subsec:hermformH,subsec:hermformClstar,subsec:hermit-periplectic} induce real forms of $\fg''$.  Comparing to the real forms given \cite[Table~3]{Ser83} shows that we obtain \emph{all} real forms of $\fg''$ in this way.  Thus, it remains to show that, up to isomorphism, a real form of $\fg$ is determined by the corresponding real form of $\fg''$.  Equivalently, it suffices to show that, if $\kappa$ and $\chi$ are conjugate-linear involutions of $\fg$ such that $\kappa'' = \chi''$, then $\kappa = \chi$.

    Suppose $\kappa$ and $\chi$ are conjugate-linear involutions of $\fg$ such that $\kappa'' = \chi''$.  Then, by \cref{sweep}, $\kappa' = \chi'$.  If $\fg = \fp(m,\C)$, then $Z(\fg)=0$, and it follows from \cref{sweep} that $\kappa = \chi$, and we are done.

    Now suppose that $\fg = \fg(m|n,\C)$, $m \ne n$, or $\fg = \fq(m,\C)$.  Then $\fg = \fg' \oplus \C I$, where $I$ denotes the identity matrix.  Both $\kappa$ and $\chi$ must leave $Z(\fg) = \C I$ invariant, hence the corresponding real forms must be isomorphic to $(\fg')^\kappa \oplus \R = (\fg')^\chi \oplus \R$, where $\R$ denotes the one-dimensional abelian real Lie algebra.

    Finally, suppose that $\fg = \fgl(m|m,\C)$.  We have a short exact sequence of Lie superalgebras
    \[
        0 \to \fsl(m|m,\C) \to \fgl(m|m,\C) \xrightarrow{\str} Z(\fg) = \C \to 0.
    \]
    Since $\kappa$ and $\chi$ agree on $\fsl(m|m,\C)$, it follows from \cref{sweep} that there exists $a \in \C$ such that
    \[
        \kappa(X) = \chi(X) + a \str(X) I,\qquad X \in \fg.
    \]
    Then, for all $X \in \fg$, we have
    \[
        X = \kappa^2(X)
        = \kappa \left( \chi(X) + a \str(X) I \right) \\
        = X + a^\star \str(X)^\star \chi(I) + a \str(X) I.
    \]
    Thus,
    \[
        a^\star \str(X)^\star \chi(I) + a \str(X) = 0 \qquad
        \text{for all } X \in \fg.
    \]
    Choosing $X = E_{11}$ and $X = i E_{11}$, implies that $a^\star \chi(I) + a = 0$ and $a^\star \chi(I) - a = 0$.  Hence $a=0$, and so $\kappa = \chi$, as desired.
\end{proof}


\bibliographystyle{alphaurl}
\bibliography{DiagSup}

\begin{thebibliography}{BCNR17}

\bibitem[Bae20]{Bae20}
J.~Baez.
\newblock The tenfold way.
\newblock {\em Notices Amer. Math. Soc.}, November:1599--1601, 2020.
\newblock \href {http://arxiv.org/abs/2011.14234} {\path{arXiv:2011.14234}},
  \href {https://doi.org/10.1090/noti2167} {\path{doi:10.1090/noti2167}}.

\bibitem[BCK19]{BCK19}
J.~Brundan, J.~Comes, and J.~R. Kujawa.
\newblock A basis theorem for the degenerate affine oriented
  {B}rauer-{C}lifford supercategory.
\newblock {\em Canad. J. Math.}, 71(5):1061--1101, 2019.
\newblock \href {http://arxiv.org/abs/1706.09999} {\path{arXiv:1706.09999}},
  \href {https://doi.org/10.4153/cjm-2018-030-8}
  {\path{doi:10.4153/cjm-2018-030-8}}.

\bibitem[BCNR17]{BCNR17}
J.~Brundan, J.~Comes, D.~Nash, and A.~Reynolds.
\newblock A basis theorem for the affine oriented {B}rauer category and its
  cyclotomic quotients.
\newblock {\em Quantum Topol.}, 8(1):75--112, 2017.
\newblock \href {http://arxiv.org/abs/1404.6574} {\path{arXiv:1404.6574}},
  \href {https://doi.org/10.4171/QT/87} {\path{doi:10.4171/QT/87}}.

\bibitem[BE17]{BE17}
J.~Brundan and A.~P. Ellis.
\newblock Monoidal supercategories.
\newblock {\em Comm. Math. Phys.}, 351(3):1045--1089, 2017.
\newblock \href {http://arxiv.org/abs/1603.05928} {\path{arXiv:1603.05928}},
  \href {https://doi.org/10.1007/s00220-017-2850-9}
  {\path{doi:10.1007/s00220-017-2850-9}}.

\bibitem[BS12]{BS12}
J.~Brundan and C.~Stroppel.
\newblock Gradings on walled {B}rauer algebras and {K}hovanov's arc algebra.
\newblock {\em Adv. Math.}, 231(2):709--773, 2012.
\newblock \href {http://arxiv.org/abs/1107.0999} {\path{arXiv:1107.0999}},
  \href {https://doi.org/10.1016/j.aim.2012.05.016}
  {\path{doi:10.1016/j.aim.2012.05.016}}.

\bibitem[BSW21]{BSW-foundations}
J.~Brundan, A.~Savage, and B.~Webster.
\newblock Foundations of {F}robenius {H}eisenberg categories.
\newblock {\em J. Algebra}, 578:115--185, 2021.
\newblock \href {http://arxiv.org/abs/2007.01642} {\path{arXiv:2007.01642}},
  \href {https://doi.org/10.1016/j.jalgebra.2021.02.025}
  {\path{doi:10.1016/j.jalgebra.2021.02.025}}.

\bibitem[Cal22]{Cal22}
K.~Calvert.
\newblock Compact {S}chur-{W}eyl duality: real {L}ie groups and the cyclotomic
  {B}rauer algebra.
\newblock {\em J. Pure Appl. Algebra}, 226(11):Paper No. 107082, 15, 2022.
\newblock \href {http://arxiv.org/abs/2003.09319} {\path{arXiv:2003.09319}},
  \href {https://doi.org/10.1016/j.jpaa.2022.107082}
  {\path{doi:10.1016/j.jpaa.2022.107082}}.

\bibitem[CE21]{CE21}
K.~Coulembier and M.~Ehrig.
\newblock The periplectic {B}rauer algebra {III}: {T}he {D}eligne category.
\newblock {\em Algebr. Represent. Theory}, 24(4):993--1027, 2021.
\newblock \href {http://arxiv.org/abs/1704.07547} {\path{arXiv:1704.07547}},
  \href {https://doi.org/10.1007/s10468-020-09976-8}
  {\path{doi:10.1007/s10468-020-09976-8}}.

\bibitem[CH17]{CH17}
J.~Comes and T.~Heidersdorf.
\newblock Thick ideals in {D}eligne's category {$\underline{\rm Re}{\rm
  p}(O_\delta)$}.
\newblock {\em J. Algebra}, 480:237--265, 2017.
\newblock \href {http://arxiv.org/abs/1507.06728} {\path{arXiv:1507.06728}},
  \href {https://doi.org/10.1016/j.jalgebra.2017.01.050}
  {\path{doi:10.1016/j.jalgebra.2017.01.050}}.

\bibitem[CW12]{CW12}
J.~Comes and B.~Wilson.
\newblock Deligne's category {$\underline{\rm{Rep}}(GL_\delta)$} and
  representations of general linear supergroups.
\newblock {\em Represent. Theory}, 16:568--609, 2012.
\newblock \href {http://arxiv.org/abs/1108.0652} {\path{arXiv:1108.0652}},
  \href {https://doi.org/10.1090/S1088-4165-2012-00425-3}
  {\path{doi:10.1090/S1088-4165-2012-00425-3}}.

\bibitem[Del07]{Del07}
P.~Deligne.
\newblock La cat\'{e}gorie des repr\'{e}sentations du groupe sym\'{e}trique
  {$S_t$}, lorsque {$t$} n'est pas un entier naturel.
\newblock In {\em Algebraic groups and homogeneous spaces}, volume~19 of {\em
  Tata Inst. Fund. Res. Stud. Math.}, pages 209--273. Tata Inst. Fund. Res.,
  Mumbai, 2007.

\bibitem[DH76]{DH76}
D.~\v{Z}. Djokovi\'{c} and G.~Hochschild.
\newblock Semisimplicity of {$2$}-graded {L}ie algebras. {II}.
\newblock {\em Illinois J. Math.}, 20(1):134--143, 1976.
\newblock URL: \url{http://projecteuclid.org/euclid.ijm/1256050167}.

\bibitem[DLZ18]{DLZ18}
P.~Deligne, G.~I. Lehrer, and R.~B. Zhang.
\newblock The first fundamental theorem of invariant theory for the
  orthosymplectic super group.
\newblock {\em Adv. Math.}, 327:4--24, 2018.
\newblock \href {http://arxiv.org/abs/1508.04202} {\path{arXiv:1508.04202}},
  \href {https://doi.org/10.1016/j.aim.2017.06.009}
  {\path{doi:10.1016/j.aim.2017.06.009}}.

\bibitem[DM99]{DM99}
P.~Deligne and J.~W. Morgan.
\newblock Notes on supersymmetry (following {J}oseph {B}ernstein).
\newblock In {\em Quantum fields and strings: a course for mathematicians,
  {V}ol. 1, 2 ({P}rinceton, {NJ}, 1996/1997)}, pages 41--97. Amer. Math. Soc.,
  Providence, RI, 1999.

\bibitem[ES16]{ES16}
M.~Ehrig and C.~Stroppel.
\newblock Schur-{W}eyl duality for the {B}rauer algebra and the
  ortho-symplectic {L}ie superalgebra.
\newblock {\em Math. Z.}, 284(1-2):595--613, 2016.
\newblock \href {http://arxiv.org/abs/1412.7853} {\path{arXiv:1412.7853}},
  \href {https://doi.org/10.1007/s00209-016-1669-y}
  {\path{doi:10.1007/s00209-016-1669-y}}.

\bibitem[FH91]{FH91}
W.~Fulton and J.~Harris.
\newblock {\em Representation theory}, volume 129 of {\em Graduate Texts in
  Mathematics}.
\newblock Springer-Verlag, New York, 1991.
\newblock A first course, Readings in Mathematics.
\newblock \href {https://doi.org/10.1007/978-1-4612-0979-9}
  {\path{doi:10.1007/978-1-4612-0979-9}}.

\bibitem[Gav20]{Gav20}
F.~Gavarini.
\newblock A new equivalence between super {H}arish-{C}handra pairs and {L}ie
  supergroups.
\newblock {\em Pacific J. Math.}, 306(2):451--485, 2020.
\newblock \href {https://doi.org/10.2140/pjm.2020.306.451}
  {\path{doi:10.2140/pjm.2020.306.451}}.

\bibitem[Hei17]{Hei17}
T.~Heidersdorf.
\newblock Mixed tensors of the general linear supergroup.
\newblock {\em J. Algebra}, 491:402--446, 2017.
\newblock \href {http://arxiv.org/abs/1406.0444} {\path{arXiv:1406.0444}},
  \href {https://doi.org/10.1016/j.jalgebra.2017.08.012}
  {\path{doi:10.1016/j.jalgebra.2017.08.012}}.

\bibitem[Kel05]{Kel05}
G.~M. Kelly.
\newblock Basic concepts of enriched category theory.
\newblock {\em Repr. Theory Appl. Categ.}, (10):vi+137, 2005.
\newblock Reprint of the 1982 original [Cambridge Univ. Press, Cambridge;
  MR0651714].

\bibitem[Kna02]{Kna02}
A.~W. Knapp.
\newblock {\em Lie groups beyond an introduction}, volume 140 of {\em Progress
  in Mathematics}.
\newblock Birkh\"{a}user Boston, Inc., Boston, MA, second edition, 2002.

\bibitem[KT17]{KT17}
J.~R. Kujawa and B.~C. Tharp.
\newblock The marked {B}rauer category.
\newblock {\em J. Lond. Math. Soc. (2)}, 95(2):393--413, 2017.
\newblock \href {http://arxiv.org/abs/1411.6929} {\path{arXiv:1411.6929}},
  \href {https://doi.org/10.1112/jlms.12015} {\path{doi:10.1112/jlms.12015}}.

\bibitem[Lew82]{Lew82}
D.~W. Lewis.
\newblock The isometry classification of {H}ermitian forms over division
  algebras.
\newblock {\em Linear Algebra Appl.}, 43:245--272, 1982.
\newblock \href {https://doi.org/10.1016/0024-3795(82)90258-0}
  {\path{doi:10.1016/0024-3795(82)90258-0}}.

\bibitem[LSM02]{LSM02}
C.~Lee~Shader and D.~Moon.
\newblock Mixed tensor representations and rational representations for the
  general linear {L}ie superalgebras.
\newblock {\em Comm. Algebra}, 30(2):839--857, 2002.
\newblock \href {https://doi.org/10.1081/AGB-120013185}
  {\path{doi:10.1081/AGB-120013185}}.

\bibitem[LZ15]{LZ15}
G.~I. Lehrer and R.~B. Zhang.
\newblock The {B}rauer category and invariant theory.
\newblock {\em J. Eur. Math. Soc. (JEMS)}, 17(9):2311--2351, 2015.
\newblock \href {http://arxiv.org/abs/1207.5889} {\path{arXiv:1207.5889}},
  \href {https://doi.org/10.4171/JEMS/558} {\path{doi:10.4171/JEMS/558}}.

\bibitem[LZ17]{LZ17}
G.~I. Lehrer and R.~B. Zhang.
\newblock The first fundamental theorem of invariant theory for the
  orthosymplectic supergroup.
\newblock {\em Comm. Math. Phys.}, 349(2):661--702, 2017.
\newblock \href {http://arxiv.org/abs/1401.7395} {\path{arXiv:1401.7395}},
  \href {https://doi.org/10.1007/s00220-016-2731-7}
  {\path{doi:10.1007/s00220-016-2731-7}}.

\bibitem[LZ21]{LZ21}
G.~I. Lehrer and R.~B. Zhang.
\newblock The second fundamental theorem of invariant theory for the
  orthosymplectic supergroup.
\newblock {\em Nagoya Math. J.}, 242:52--76, 2021.
\newblock \href {http://arxiv.org/abs/1407.1058} {\path{arXiv:1407.1058}},
  \href {https://doi.org/10.1017/nmj.2019.25} {\path{doi:10.1017/nmj.2019.25}}.

\bibitem[Moo03]{Moo03}
D.~Moon.
\newblock Tensor product representations of the {L}ie superalgebra
  {${\mathfrak{p}}(n)$} and their centralizers.
\newblock {\em Comm. Algebra}, 31(5):2095--2140, 2003.
\newblock \href {https://doi.org/10.1081/AGB-120018988}
  {\path{doi:10.1081/AGB-120018988}}.

\bibitem[MS23]{MS21}
A.~McSween and A.~Savage.
\newblock Affine oriented {F}robenius {B}rauer categories.
\newblock {\em Comm. Algebra}, 51(2):742--756, 2023.
\newblock \href {http://arxiv.org/abs/2101.04582} {\path{arXiv:2101.04582}},
  \href {https://doi.org/10.1080/00927872.2022.2113401}
  {\path{doi:10.1080/00927872.2022.2113401}}.

\bibitem[PS16]{PS16}
J.~Pike and A.~Savage.
\newblock Twisted {F}robenius extensions of graded superrings.
\newblock {\em Algebr. Represent. Theory}, 19(1):113--133, 2016.
\newblock \href {http://arxiv.org/abs/1502.00590} {\path{arXiv:1502.00590}},
  \href {https://doi.org/10.1007/s10468-015-9565-4}
  {\path{doi:10.1007/s10468-015-9565-4}}.

\bibitem[RS19]{RS19}
H.~Rui and L.~Song.
\newblock Affine {B}rauer category and parabolic category {$\mathcal{O}$} in
  types {$B$}, {$C$}, {$D$}.
\newblock {\em Math. Z.}, 293(1-2):503--550, 2019.
\newblock \href {https://doi.org/10.1007/s00209-018-2207-x}
  {\path{doi:10.1007/s00209-018-2207-x}}.

\bibitem[Sav19]{Sav19}
A.~Savage.
\newblock Frobenius {H}eisenberg categorification.
\newblock {\em Algebr. Comb.}, 2(5):937--967, 2019.
\newblock \href {http://arxiv.org/abs/1802.01626} {\path{arXiv:1802.01626}},
  \href {https://doi.org/10.5802/alco.73} {\path{doi:10.5802/alco.73}}.

\bibitem[Sel11]{Sel11}
P.~Selinger.
\newblock A survey of graphical languages for monoidal categories.
\newblock In {\em New structures for physics}, volume 813 of {\em Lecture Notes
  in Phys.}, pages 289--355. Springer, Heidelberg, 2011.
\newblock \href {http://arxiv.org/abs/0908.3347} {\path{arXiv:0908.3347}},
  \href {https://doi.org/10.1007/978-3-642-12821-9\_4}
  {\path{doi:10.1007/978-3-642-12821-9\_4}}.

\bibitem[Ser83]{Ser83}
V.~V. Serganova.
\newblock Classification of simple real {L}ie superalgebras and symmetric
  superspaces.
\newblock {\em Funktsional. Anal. i Prilozhen.}, 17(3):46--54, 1983.

\bibitem[SS22]{Sam22}
S.~Samchuck-Schnarch.
\newblock Frobenius {B}rauer categories.
\newblock M.{S}c. thesis, University of Ottawa, 2022.
\newblock \href {https://doi.org/http://dx.doi.org/10.20381/ruor-28137}
  {\path{doi:http://dx.doi.org/10.20381/ruor-28137}}.

\bibitem[SW22]{SW22}
A.~Savage and B.~Westbury.
\newblock Quantum diagrammatics for ${F}_4$.
\newblock 2022.
\newblock \href {http://arxiv.org/abs/2204.11976} {\path{arXiv:2204.11976}}.

\bibitem[Wal64]{Wal64}
C.~T.~C. Wall.
\newblock Graded {B}rauer groups.
\newblock {\em J. Reine Angew. Math.}, 213:187--199, 1963/64.
\newblock \href {https://doi.org/10.1515/crll.1964.213.187}
  {\path{doi:10.1515/crll.1964.213.187}}.

\end{thebibliography}

\end{document}